\documentclass[11pt]{article}
\usepackage[
    top = 2cm, 
    left = 2cm, 
    right = 2cm, 
    bottom = 2.5cm, 
    footskip = 1.25cm, 
]{geometry}

\usepackage{csquotes}

\usepackage[sorting = none, maxnames = 100]{biblatex}
\usepackage[calcwidth]{titlesec}
\titleformat{\section}{\Large\bfseries 
}{\thesection. 
}{
    0.2em 
}{}[
    \vspace{-1em}{\color{seeblau}\rule{\titlewidth}{1.2pt}} 
]
\titleformat{\subsection}{\large\bfseries 
}{\thesubsection. 
}{
    0.2em 
}{}[
    \vspace{-0.8em}{\color{seeblau}\rule{\titlewidth}{1pt}} 
]
  
\usepackage[english]{babel}

\usepackage[]{algpseudocodex}
\tikzset{algpxIndentLine/.style = {draw = seeblau}}

\usepackage{xcolor}
\definecolor{seeblau}{RGB}{150, 150, 150}

\usepackage[hidelinks]{hyperref}
\hypersetup{
    colorlinks, 
    citecolor = red, 
    linkcolor = red, 
    urlcolor = red
}

\usepackage{amsthm} 

\usepackage{amsmath}
\usepackage{amssymb}
\usepackage{amsfonts}

\usepackage{booktabs}
\usepackage{multirow}                                       

\usepackage{mathrsfs}
\usepackage{mathtools} 

\usepackage[most]{tcolorbox}
\tcbuselibrary{theorems}

\usepackage{tikz}

\tcbset{
    myTheoremStyle/.style = {
        enhanced,
        colback = seeblau!5,
        colframe = seeblau,
        coltitle = black,
        fonttitle = \bfseries,
        boxed title style = {colback = seeblau!25, arc = 0mm}, 
        attach boxed title to top left = {
            yshift = -2mm, 
            xshift = 5mm
        }, 
        breakable, 
        separator sign = {\hspace{0.4em}--\hspace{0.05em}},
        description delimiters parenthesis, 
        top = 3mm, 
        arc = 1mm, 
        before skip = 0.75cm, 
        after skip = 0.75cm
    }
}

\newtcbtheorem[auto counter]{theorem}{Theorem}{myTheoremStyle}{THM}

\newtcbtheorem[auto counter]{proposition}{Proposition}{myTheoremStyle}{PRO}

\newtcbtheorem[auto counter]{lemma}{Lemma}{myTheoremStyle}{LEM}

\newtcbtheorem[auto counter]{definition}{Definition}{myTheoremStyle}{DEF}

\newtcbtheorem[auto counter]{assumption}{Assumption}{myTheoremStyle}{ASS}

\newtcbtheorem[auto counter]{algorithmOuter}{Algorithm}{myTheoremStyle}{ALG}

\newtcbtheorem[auto counter]{remark}{Remark}{myTheoremStyle}{REM}

\tcolorboxenvironment{proof}{
    blanker, 
    breakable, 
    before skip = 15pt, 
    after skip = 15pt, 
    borderline west = {
        0.5mm 
    }{-10pt}{
    seeblau!50 
    },
    top = 0.5mm, 
    left = 0mm, 
    bottom = 0.5mm
}

\newcommand{\Julia}{%
  \raisebox{-0.18\height}{\includegraphics[height=2.75ex]{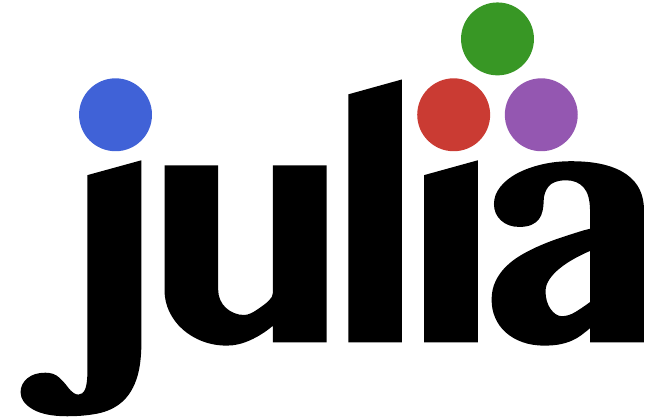}}%
  \hspace{0.025em}
}

\newcommand{\BB}{\mathbb{B}}
\newcommand{\EE}{\mathbb{E}} 
\newcommand{\NN}{\mathbb{N}} 
\newcommand{\PP}{\mathbb{P}} 
\newcommand{\QQ}{\mathbb{Q}} 
\newcommand{\RR}{\mathbb{R}} 
\newcommand{\VV}{\mathbb{V}} 

\renewcommand{\SS}{\mathbb{S}}
\newcommand{\GG}{\mathbb{G}}
\newcommand{\WW}{\mathbb{W}}

\DeclareMathOperator{\graph}{Gr}

\DeclareMathOperator{\convex}{conv}
\DeclareMathOperator{\dom}{dom}

\DeclareMathOperator{\interior}{int}
\DeclareMathOperator{\relinterior}{relint}
\DeclareMathOperator{\bd}{bd}
\DeclareMathOperator{\cl}{cl}
\DeclareMathOperator{\relbd}{relbd}
\DeclareMathOperator{\Unif}{Uniform}

\DeclareMathOperator*{\argmin}{arg\,min}
\DeclareMathOperator*{\argmax}{arg\,max}

\DeclareMathOperator{\diam}{diam}

\newcommand{\Ab}{\mathbf{A}}
\newcommand{\Bb}{\mathbf{B}}
\newcommand{\Cb}{\mathbf{C}}
\newcommand{\Db}{\mathbf{D}}
\newcommand{\Eb}{\mathbf{E}}
\newcommand{\Fb}{\mathbf{F}}
\newcommand{\Gb}{\mathbf{G}}

\newcommand{\Lb}{\mathbf{L}}

\newcommand{\Nb}{\mathbf{N}}

\newcommand{\Pb}{\mathbf{P}}

\newcommand{\Sb}{\mathbf{S}}

\newcommand{\Ub}{\mathbf{U}}
\newcommand{\Vb}{\mathbf{V}}
\newcommand{\Wb}{\mathbf{W}}
\newcommand{\Xb}{\mathbf{X}}
\newcommand{\Yb}{\mathbf{Y}}
\newcommand{\Zb}{\mathbf{Z}}

\newcommand{\Bcc}{\mathscr{B}}
\newcommand{\Fcc}{\mathscr{F}}
\newcommand{\Kcc}{\mathscr{K}}

\newcommand{\Dcc}{\mathscr{D}}

\newcommand{\ind}{\mathbf{1}}

\DeclareMathOperator{\trace}{tr}

\DeclareMathOperator{\Popt}{P^\star}

\usepackage{dashrule}

\title{\huge\bfseries A Recursive Domain\hspace{0.05em}- and Objective\hspace{0.1em}-Adaptive \\ Frank\hspace{0.05em}-Wolfe Algorithm}
\date{}

\newcommand{\authors}[5]{%
    \textbf{#1} \hfill {\small\textsc{#2}} \\
    {\small \textit{#3, #4, #5}}
}

\begin{document}

\maketitle

\vspace{-1cm}
\noindent
\authors{Marcel Kaiser}{marcel.kaiser@uni-konstanz.de}{Department of Computer Science}{University of Konstanz}{Germany}\\[0.2cm]
\authors{Tobias Sutter}{tobias.sutter@unisg.ch}{Department of Economics}{University of St.Gallen}{Switzerland}

\vspace{0.5cm}
\begin{center}
    \textbf{\large Abstract} \\[0.6cm]
    \begin{minipage}{42em}
        {\small We investigate a recursive variant of the classical Frank\hspace{0.05em}-Wolfe algorithm for minimizing a convex differentiable function over a convex compact domain.
        Unlike the traditional setting, we assume that both the problem domain and the objective function are initially unknown and must be learned from data. 
        To address this, we integrate estimators into the optimization process, allowing the algorithm to iteratively refine approximations of the problem domain and the objective function.
        Our approach maintains the projection\hspace{0.05em}-free nature of the classical Frank\hspace{0.05em}-Wolfe algorithm while adapting to the uncertainty inherent in data\hspace{0.05em}-\hspace{0.025em}driven settings.
        We establish convergence guarantees for the recursive method, showing that the optimization error scales with the accuracy of the learned estimators.
        Two experiments support our theoretical findings, demonstrating that the proposed method achieves convergence behavior comparable to that of the classical Frank\hspace{0.05em}-Wolfe algorithm under exact knowledge of the problem domain and objective function, while offering significant computational savings. \\
        
        \textbf{Keywords.} Frank\hspace{0.05em}-Wolfe Algorithm $\cdot$ Conditional Gradient Method $\cdot$ Projection\hspace{0.05em}-\hspace{-0.025em}Free Optimization $\cdot$ \hbox{Data\hspace{0.05em}-Driven} Optimization $\cdot$ Convergence Analysis \\
        
        } 
    \end{minipage}
\end{center}
\vspace{0.2cm}

\section{Introduction}
    Convex optimization lies at the core of modern machine learning, statistics, operations research, and control. 
    Among the many algorithms developed for dealing with structured constrained convex minimization problems, the Frank\hspace{0.05em}-Wolfe algorithm (also known as the conditional gradient method) stands out for its elegance and efficiency and has become an essential part of the algorithmic toolbox in convex optimization, see \cite{braun2025conditional} for a recent comprehensive textbook.
    Consider a general convex minimization problem 
    \begin{equation}\label{EQ:generalProblem}
        \text{minimize } \,f(x)\, \text{ subject to } \,x \in D\hspace{0.1em}, \tag{\rm{P}}
    \end{equation}
    where the problem domain $D$ is a compact and convex subset of a finite\hspace{0.1em}-\hspace{0.05em}dimensional Hilbert space $H$, equipped with the inner product $\langle \hspace{0.05em}\cdot\mid\hspace{-0.05em} \cdot\hspace{0.05em} \rangle$, and the objective function $f \colon \dom(f) \to \RR$, with $\dom(f) \subseteq H$ closed, is a continuously differentiable and convex function. 
    We denote the optimal value of the problem~\eqref{EQ:generalProblem} by $\Popt$.
    The Frank\hspace{0.05em}-Wolfe algorithm iteratively linearizes the objective function $f$ at the current iterate $x_n \in D$ and solves the linear minimization problem
    \begin{subequations}\label{eq:intro:FW:algo}
        \begin{equation}\label{LMO_problem}
            s_n \in \argmin\,\{\langle \hspace{0.05em} s \,|\, \nabla\hspace{-0.1em} f(x_n) \hspace{0.05em} \rangle \,\colon s \in D\} \tag{1.1}
        \end{equation}
    at each step $n \in \NN$ using a linear minimization oracle (LMO\hspace{-0.05em}), consequently updating along directions that remain feasible without requiring costly projections onto the problem domain $D$.
    Indeed, for many structured sets solving the linearized subproblem \eqref{LMO_problem} is significantly simpler than projecting onto it.
    The iterates are then updated using a given step\hspace{0.075em}-\hspace{0.05em}size
        \begin{equation}\label{eq:intro:update:rule}
            x_{n+1} \,=\, x_n + \lambda_{\hspace{0.025em}n}(s_n-x_n), \quad \lambda_{\hspace{0.025em}n} = \frac{2}{n+2}\hspace{0.1em}. \tag{1.2}
        \end{equation}
    \end{subequations}
    This algorithm was introduced in 1956 by Marguerite Frank and Philip Wolfe \cite{frank1956algorithm} and yields a sequence $(x_n)_{n \in \NN}$ with convergence guarantees of $\mathcal{O}(1/n)$ under standard smoothness assumptions on the objective function. 
    More specifically, it can be shown \cite[Theorem~1]{jaggi2013revisiting} that the generated sequence $(x_n)_{n \in \NN}$ satisfies
    \begin{equation} \label{eq:intro:vanilla:FW:complexity}
      0 \,\leq\, f(x_n) - \Popt \,\leq\, \frac{2\hspace{0.05em}C_{\hspace{-0.075em}f}}{n+2}\hspace{0.1em},
    \end{equation}
    where $C_{\hspace{-0.075em}f}$ is the curvature constant of $f$ as defined in \cite{jaggi2013revisiting}. 
    The assumption of a bounded curvature constant $C_{\hspace{-0.075em}f}$ closely corresponds to the smoothness of the objective function $f$. \\

    \noindent
    Despite its elegant geometry and strong guarantees, the classical Frank\hspace{0.05em}-Wolfe algorithm~\eqref{eq:intro:FW:algo} presumes \textit{full access} to both the objective function $f$ (more precisely its gradient) and the problem domain $D$. 
    However, this assumption breaks down in many data\hspace{0.075em}-\hspace{0.05em}driven, online, or interactive settings, where the problem domain and objective function may not be
    explicitly known, but can incrementally be learned from data. 
    That is, the objective function $f$ and the problem domain $D$ are only accessible via data\hspace{0.075em}-\hspace{0.05em}dependent sequences of estimators $(f_n)_{n\in\mathbb{N}}$ and $(D_n)_{n\in\mathbb{N}}$ which are constructed from incoming samples. 
    These estimators are typically noisy, biased, and evolving, raising significant challenges for optimization algorithms that must simultaneously learn and optimize. 
    Moreover, the approximation of the domain $D$ introduces additional complexity: Linear minimization steps, defined via \eqref{LMO_problem}, are no longer exact and must be executed on a surrogate set.

    \paragraph{Our Setting.}
    We propose and analyze a recursive domain\hspace{0.05em}- and objective\hspace{0.1em}-\hspace{0.05em}adaptive Frank\hspace{0.05em}-Wolfe algorithm, where both the objective function $f$ and the problem domain $D$ are initially unknown and must be approximated from data. 
    A rigorous formulation of this setting is provided in Section~\ref{SEC:stochasticFormulationAndAlgorithm}. 
    In particular, we consider the setting where at each iteration $n\in\mathbb{N}$ we only have access to estimators $f_n$ of the objective function $f$ and $D_n$ of the problem domain $D$. 
    A naive way to obtain a sequence of approximations on the optimal value $\Popt$ of the problem \eqref{EQ:generalProblem} would be to construct a sequence $(x_n)_{n \in \NN}$ defined as the minimizers of the optimization problems
    \begin{equation}\label{EQ:approximationProblems:naive}
        \text{minimize } \,f_n(x)\, \text{ subject to } \,x \in D_n
    \end{equation}
    for each $n\in\NN$. 
    However, since the minimization problem \eqref{EQ:generalProblem} and hence also the approximation problems~\eqref{EQ:approximationProblems:naive} may in general be time\hspace{0.1em}-\hspace{0.05em}consuming to solve, we introduce a computationally cheaper approach. 
    Motivated by the Frank\hspace{0.05em}-Wolfe algorithm~\eqref{eq:intro:FW:algo}, we define its recursive and adaptive counterpart as follows.
    
    \begin{algorithmOuter}{Recursive Adaptive Frank\hspace{0.05em}-Wolfe Algorithm}{onlineAdaptiveFrankWolfeAlgorithm}
        \begin{algorithmic}[1]
            \Require Approximation sequences $(\hspace{-0.05em} D_n\hspace{-0.05em})_{n \in \NN}$ and $(f_n)_{n \in \NN}$ of the problem domain $D$ and the objective function $f$, respectively. Initial iterate $x_0 \in \dom(f_1\hspace{-0.05em})$.
            \\
            \hspace*{-62.5pt}{\color{seeblau}\hdashrule{502pt}{0.75pt}{2pt}}\\[-4mm]
            \For{$n \in \NN_0$}
                \If{$x_n \in \dom(f_{n+\hspace{-0.05em}1}\hspace{-0.05em})$}
                    \State compute $s_{n+\hspace{-0.05em}1}\hspace{-0.15em}\in \argmin \hspace{0.2em}\{ \langle \hspace{0.1em} s \,|\, \nabla \hspace{-0.2em} f_{n+\hspace{-0.05em}1}(x_{n}\hspace{-0.05em}) \hspace{0.1em}\rangle\, \colon s \in D_{n+\hspace{-0.05em}1}\}$ \label{LINE:subproblem}
                    \Comment{requires LMOs}
                    \State set $\lambda_{\hspace{0.025em}n} = 2\hspace{0.025em}/(2+n)$
                    \State update $x_{n+\hspace{-0.05em}1} = \hspace{0.2em} x_{n} + \lambda_{\hspace{0.025em}n}(s_{n+\hspace{-0.05em}1} - x_{n}\hspace{-0.05em})$
                \Else \Comment{may happen only with small probability}
                    \State set $x_{n+\hspace{-0.05em}1} = x_n$
                \EndIf
            \EndFor
            \vspace{0.2cm}
        \end{algorithmic}
    \end{algorithmOuter}

    \noindent
    Clearly, the quality of a sequence $(x_n)_{n \in \NN}$ generated by Algorithm~\ref{ALG:onlineAdaptiveFrankWolfeAlgorithm} for approximating $\Popt$ depends on the approximation quality of two estimator sequences $(f_n)_{n\in\mathbb{N}}$ and $(D_n)_{n\in\mathbb{N}}$.
    In this paper, we analyze this approximation quality.
    We consider a setting in which the sequences of estimators $(f_n)_{n \in \mathbb{N}}$ and $(D_n)_{n \in \mathbb{N}}$ are modeled as stochastic processes. 
    These estimators are viewed as functions derived from an increasingly large set of training data. 
    As the sample size grows, they tend to concentrate around the true underlying objects $f$ and $D$, although convergence to these targets is not  guaranteed. 
    For example, we consider a scenario in which the estimators concentrate, with high probability $1 - \beta$, within an $\eta$\hspace{0.1em}-\hspace{0.05em}neighborhood of their respective nominal models, for some small constants $\beta, \eta > 0$, as the amount of training data increases. \\

    \noindent
    A motivating example arises in \emph{distributionally robust estimation}, which we illustrate in a simplified setting and discuss it in much greater detail in Section~\ref{SEC:examples}. 
    Let $\{\PP_{\hspace{-0.025em}\theta} \,\colon \theta \in \RR^d\}$ denote a parametric family of probability measures, and let $\Xb \in \RR^m$ be a random variable distributed according to the (unknown) ground\hspace{0.05em}-truth distribution $\PP_{\hspace{-0.025em}\theta^\star}$. 
    We observe independent and identically distributed samples $\Xb_1, \dots, \Xb_{\hspace{0.05em}n} \sim \PP_{\hspace{-0.025em}\theta^\star}$ and denote by $\hat{\PP}^{(n)}$ the associated empirical probability measure.
    For some abstract loss function $\ell:\RR^m \to \RR$, we define the objective function $f \colon \RR^d \to \RR$ via
    \begin{equation*}
        f(\theta) \,=\, \EE_{\hspace{0.05em}\PP_{\hspace{-0.025em}\theta}}[\hspace{0.025em}\ell(\Xb)\hspace{0.025em}]
    \end{equation*}
    for all $\theta \in \RR^d$.
    The feasible domain of interest is characterized by proximity to the true (unknown) distribution, 
    \begin{equation*}
        D \,=\, \bigl\{ \theta \in \RR^d : d\hspace{0.025em}(\hspace{0.05em}\PP_{\hspace{-0.025em}\theta}, \PP_{\hspace{-0.025em}\theta^\star}\hspace{-0.1em}) \leq r \bigr\}\hspace{0.1em},
    \end{equation*}
    where $r \geq 0$ is a robustness radius and $d$ is a distance (or divergence) on probability measures. 
    Since $\PP_{\hspace{-0.025em}\theta^\star}$ is unknown, the set $D$ cannot be computed directly. 
    Instead, one works with the empirical, data\hspace{0.075em}-\hspace{0.05em}driven approximations
    \begin{equation*}
        D_n \,=\, \bigl\{ \theta \in \RR^d : d\hspace{0.025em}(\hspace{0.05em}\PP_{\hspace{-0.025em}\theta}, \hat{\PP}^{(n)}) \leq r \bigr\}\hspace{0.1em},
    \end{equation*}
    for all $n \in \NN$.
    A distributionally robust estimation (or optimization) problem is then given by
    \begin{equation*}
        \text{minimize } \,f(\theta)\, \text{ subject to } \, \theta \in D,
    \end{equation*}
    which is intractable in practice due to the unknown domain $D$ but of the form \eqref{EQ:generalProblem}. 
    In applications, this motivates replacing $D$ with its data\hspace{0.075em}-\hspace{0.05em}adaptive surrogates $(D_n)_{n \in \mathbb{N}}$, a principle that underlies many modern distributionally robust optimization models \cite{kuhn2025distributionally}. \\

    \noindent    
    This paper develops the resulting approximation bounds for the sequence $(x_n)_{n \in \NN}$ generated by Algorithm~\ref{ALG:onlineAdaptiveFrankWolfeAlgorithm}.
    More specifically, we show how the concentration rates of the approximation sequences $(f_n)_{n \in \NN}$ and $(D_n)_{n \in \NN}$ around their respective targets $f$ and $D$ translate into concentration properties of the sequence $(f(x_n))_{n \in \mathbb{N}}$ towards the optimal value $\Popt$.

    \paragraph{Contributions.} 
    The main contributions of this work are as follows.
    \begin{itemize}
        \item We propose a novel projection\hspace{0.05em}-free and recursive optimization algorithm that accommodates noisy, \mbox{data\hspace{0.075em}-\hspace{0.05em}driven} approximations of both the problem domain and the objective function. 
        The algorithm is formulated within a general probabilistic framework that rigorously captures the interaction between estimator convergence and optimization progress, thereby extending classical Frank\hspace{0.05em}-Wolfe theory to dynamic and uncertain settings.
        
        \item We establish asymptotic convergence guarantees, showing that the error $f(x_n) - \Popt$ is governed by the statistical convergence rates of the underlying estimators under standard smoothness and convexity assumptions. 
        In addition, we derive explicit convergence rates under slightly stronger structural assumptions.
        
        \item We demonstrate how the established convergence rates can be further improved by exploiting strong convexity properties and by assuming that optimizers lie in the relative interior of the feasible sets, features that are well\hspace{0.05em}-\hspace{0.035em}known to accelerate the classical (non\hspace{0.05em}-adaptive) Frank\hspace{0.05em}-Wolfe algorithm.
        
        \item Empirical validation on two experiments, including data\hspace{0.075em}-\hspace{0.05em}driven robust linear\hspace{0.05em}-quadratic control, confirms that the theoretical convergence rates accurately reflect practical performance and that the algorithm remains robust even when the estimators are initially inaccurate.
    \end{itemize}

\paragraph{Related Work.}
Much of the literature on Frank\hspace{0.05em}-Wolfe algorithms has focused on improving their convergence rates.
It is well\hspace{0.05em}-\hspace{0.035em}known that, under the standard assumptions relevant to our setting, namely, a smooth convex objective function $f$ and a compact convex problem domain $D$, the convergence rate stated in~\eqref{eq:intro:vanilla:FW:complexity} cannot generally be improved beyond constants~\cite{jaggi2013revisiting, lan2014complexity}. 
Under stronger regularity conditions, however, such as strong convexity of the problem domain $D$, or gradient dominance properties of $f$, combined with the assumption that the optimal solution lies in the relative interior of the feasible region, the convergence rate~\eqref{eq:intro:vanilla:FW:complexity} can be improved to a linear rate, see~\cite[Section~2.2]{braun2025conditional} for details.
Recent results reveal that convergence rates of the Frank\hspace{0.05em}-Wolfe algorithm are fundamentally governed by geometric properties of the feasible set via the LMO, rather than solely by properties of the objective function \cite{pokutta2026frankwolfe1tconvergence}.
Modern Frank\hspace{0.05em}-Wolfe methods encompass a broad variety of algorithmic variants beyond the original formulation~\eqref{eq:intro:FW:algo}, primarily motivated by the goal of accelerating convergence or relaxing assumptions. 
Examples include adaptive step\hspace{0.075em}-\hspace{0.025em}size strategies~\cite{pedregosa2020linearly}, or structural modifications such as away\hspace{0.05em}-\hspace{0.025em}step Frank\hspace{0.05em}-Wolfe methods~\cite{lacoste2015global}, which aim to mitigate “zig\hspace{0.05em}-\hspace{0.05em}zag” behavior. 
A comprehensive treatment of these variants is provided in~\cite[Section~3]{braun2025conditional}. 
These works, however, all operate in a setting where the optimization problem is fully specified and static, that is, both the objective function $f$ and the problem domain $D$ are known.\\

\noindent
A large body of work extends the Frank\hspace{0.05em}-Wolfe framework to settings where the objective function is only partially accessible. 
One line of research studies Frank\hspace{0.05em}-Wolfe algorithms with inexact gradient information, where the true gradient is replaced by an approximation with bounded error. 
As shown in~\cite[Theorem~2]{jaggi2013revisiting}, the convergence rate~\eqref{eq:intro:vanilla:FW:complexity} is unaffected when the gradient is approximated up to some fixed accuracy $\delta$, resulting in the bound 
    \begin{equation*}
        f(x_n) - \Popt \,\leq\, \frac{2\hspace{0.05em}C_{\hspace{-0.075em}f}}{n+2}(1 + \delta)\hspace{0.1em} ,
    \end{equation*}
where $\delta \geq 0$ measures the deviation between the exact and the inexact gradient. \\
A similar challenge arises in stochastic optimization problems, where the objective function, or its gradient, is not directly accessible.
In such cases, the true gradient $\nabla\hspace{-0.1em}f(x_n)$ is replaced by an unbiased stochastic estimator $\tilde{\nabla} \hspace{-0.1em}f(x_n)$, giving rise to stochastic Frank\hspace{0.05em}-Wolfe algorithms and their variance\hspace{0.075em}-reduced variants~\cite{hazan2016variance}. Under standard assumptions, such as bounded variance of the stochastic gradient, one obtains the expected convergence guarantee
\begin{equation*}
      \EE[\hspace{0.025em}f(x_n) - \Popt] \in \mathcal{O}(1/n)\hspace{0.1em},
\end{equation*}
see~\cite[Section~4.1]{braun2025conditional} for an extensive discussion and further references.
Importantly, these methods still assume that the problem domain $D$ is fixed and exactly known, such that each linear minimization step is carried out over the true problem domain. \\

\noindent
Beyond classical stochastic Frank\hspace{0.05em}-Wolfe methods, a number of recent developments have focused on improving the efficiency of projection\hspace{0.05em}-\hspace{0.035em}free optimization by carefully balancing the use of gradient information and LMO calls.
A prominent example is the Conditional Gradient Sliding (CGS) algorithm~\cite{lan2016conditional}, which introduces a multi\hspace{0.05em}-level scheme that intermittently \emph{skips} gradient evaluations while continuing to perform LMO calls. This allows CGS to achieve the optimal convergence rate $\mathcal{O}(1/n)$ for smooth convex problems, while simultaneously attaining improved complexity bounds in terms of gradient evaluations. In particular, CGS decouples the computational roles of gradient computations and LMO calls, and can substantially reduce the total number of gradient evaluations compared to the classical Frank\hspace{0.05em}-Wolfe algorithm, without sacrificing the optimal order of LMO complexity. \\
A complementary line of work aims at closing the gap between stochastic and deterministic conditional gradient methods in terms of LMO complexity. The Generalized Stochastic Frank\hspace{0.05em}-Wolfe (GSFW) algorithm~\cite{lu2021generalized} achieves this goal in the context of empirical risk minimization with linear prediction by carefully combining stochastic gradient estimates with variance reduction techniques. Their results show that GSFW matches the optimal $\mathcal{O}(1/n)$ convergence rate while simultaneously achieving optimal dependence on both stochastic gradient evaluations and LMO calls, thereby resolving the classical “complexity gap” present in earlier stochastic Frank\hspace{0.05em}-Wolfe methods. \\
More broadly, these methods are part of a large and rapidly growing literature on stochastic and large\hspace{0.05em}-\hspace{0.035em}scale variants of conditional gradient methods. 
This includes variance\hspace{0.05em}-reduced methods~\cite{hazan2016variance,qu2018non}, coordinate\hspace{0.05em}-randomized Frank\hspace{0.05em}-Wolfe schemes~\cite{lacoste2013block}, and distributed projection\hspace{0.05em}-\hspace{0.035em}free algorithms designed for high\hspace{0.05em}-\hspace{0.035em}dimensional or decentralized settings~\cite{bellet2015distributed,wang2016distributed}. Recent unified frameworks~\cite{nazykov2024stochastic} further synthesize these developments by providing general stochastic formulations that recover many of the existing algorithms as special cases and clarify their underlying structural similarities. \\

\noindent
Despite their methodological differences, all of the approaches described above share a common modeling assumption: While the objective function (or its gradient) may be stochastic, noisy, or only partially observable, the problem domain $D$ is assumed to be fixed and exactly known, such that each linear minimization step is performed over the true problem domain.\\

\noindent
In contrast, uncertainty in the constraint set has been studied extensively in the context of data\hspace{0.05em}-\hspace{0.035em}driven and distributionally robust optimization (DRO), see \cite{kuhn2025distributionally, shapiro2014lectures}. 
In these settings, the feasible region is typically defined via statistical estimators, for instance through ambiguity sets constructed from empirical data. 
Such approaches are widely used in operations research, control, and machine learning, where constraints are often derived from uncertain or partially observed systems. 
However, the predominant paradigm in this literature is to separate estimation and optimization: 
One first constructs a surrogate feasible set based on data, and subsequently solves a deterministic optimization problem over this set. 
The interaction between the statistical error of the domain approximation and the optimization dynamics is typically not analyzed at the algorithmic level, and projection\hspace{0.05em}-\hspace{0.035em}free methods such as the Frank\hspace{0.05em}-Wolfe algorithm have received comparatively little attention in this context. \\

\noindent
Another related line of work is online learning~\cite{hazan2016introduction}, where the objective function may change over time and is revealed sequentially. The goal in this setting is to minimize regret relative to the best fixed decision in hindsight. While this framework shares the feature of evolving objective functions, it differs fundamentally from our setting in both modeling and objectives: 
Online learning is typically adversarial and focuses on regret minimization, whereas our approach is statistical in nature and aims at convergence to the true optimum of an underlying problem. Moreover, online learning algorithms generally assume a fixed and known decision set, and thus do not address uncertainty in the feasible domain.\\

\noindent
In summary, while the literature on Frank\hspace{0.05em}-Wolfe algorithms with unknown, noisy or evolving objective functions is extensive, the domain $D$ is typically assumed to be known and fixed.
However, in many data\hspace{0.075em}-\hspace{0.05em}driven optimization problems, common in operations research, control, and machine learning, the constraint set $D$ is derived from data and evolves as new information sequentially becomes available. 
To the best of our knowledge, this is the first work to propose and analyze a Frank\hspace{0.05em}-Wolfe algorithm for problems in which both the objective function and the constraint set change adaptively over time modeled as statistical estimators. 
The class of problems addressed by our approach is broad, encompassing many important settings in distributionally robust and data\hspace{0.075em}-\hspace{0.05em}driven optimization.

    \paragraph{Structure.} 
    The remainder of the paper is organized as follows. 
    Section~\ref{SEC:stochasticFormulationAndAlgorithm} formally introduces the stochastic setting of Algorithm~\ref{ALG:onlineAdaptiveFrankWolfeAlgorithm}.
    It constructs the stochastic approximation processes for modeling the approximation sequences $(f_n)_{n\in\mathbb{N}}$ and $(D_n)_{n\in\mathbb{N}}$ of the objective function $f$ and the problem domain $D$, creating a stochastic counterpart (Algorithm \ref{ALG:onlineAdaptiveStochasticFrankWolfeAlgorithm}), whose iterates are random variables. 
    It also introduces some underlying structural assumptions required later for our convergence analysis. 
    Section~\ref{SEC:theoreticalAnalysis} develops the theoretical results: First establishing an asymptotic convergence, and then analyzing a convergence rate under slightly stronger structural assumptions. 
    In Section~\ref{SEC:acceleratedConvergenceResult}, we study under which additional assumptions a convergence acceleration can be established, similar to the setting of the classical (non\hspace{0.05em}-adaptive) Frank\hspace{0.05em}-Wolfe methods.
    Finally, Section~\ref{SEC:examples} reports two numerical experiments: (i) a simple academic example illustrating the theoretical bounds and highlighting the impact of domain constructions and (ii) a more challenging, data\hspace{0.075em}-\hspace{0.05em}driven distributionally robust formulation of the linear quadratic regulator problem.
    
    \paragraph{Notation.}
    We denote $\NN = \{1, 2, 3, \ldots\}$ and $\NN_0 = \NN \cup \{0\}$ for the set of natural numbers and nonnegative integers, respectively. 
    For $n \in \NN$ we abbreviate $[n] = \{1, \ldots, n\}$ and $[n]_0 = [n] \cup \{0\}$. 
    For each $n \in \NN$, we denote $I_n \in \RR^{n \times n}$ for the identity matrix $\SS_+^{\hspace{0.025em}n} \subseteq \RR^{n \times n}$ for the cone of positive semidefinite and $\SS_{++}^{\hspace{0.025em}n} \subseteq \RR^{n \times n}$ for the cone of positive definite real matrices.
    For a topological space $X$ we denote the (relative) interior, (relative) boundary, and closure of a subset $A \subseteq X$ by $(\operatorname{rel})\interior(A)$, $(\operatorname{rel})\bd(A)$, and $\cl(A)$, respectively, and the Borel $\sigma$\hspace{0.05em}-\hspace{0.035em}algebra with $\Bcc(X)$.  
    For a metric space $(\hspace{-0.05em} X, d\hspace{0.05em})$, a point $x \in X$ and radius $\rho \geq 0$ we denote by $\BB(x, \rho)$ the \emph{closed} ball centered around $x$, that is, $\BB(x, \rho) \,=\, \{y \in X \,\colon d(x, y) \leq \rho\}$.
    We denote 
    \begin{equation*}
        \lVert \,\cdot\, \rVert \colon H \to \RR, \; x \,\mapsto\, \sqrt{\langle \hspace{0.05em}x \mid x \hspace{0.05em}\rangle}
    \end{equation*}
    for the induced norm and
    \begin{equation*}
        h \colon H \times H \to \RR, \; (x, y) \,\mapsto\, \lVert \hspace{0.05em}x-y\hspace{0.05em}\rVert
    \end{equation*}
    for the induced metric on $H$. 
    For any point $x \in H$ and any $Y \subseteq H$ we write $h(x, Y) = \inf\hspace{0.1em}\{h(x, y) \,\colon y \in Y\}$ for the minimal distance of $x$ to $Y$, $\diam(Y) = \sup\hspace{0.1em}\{\lVert \hspace{0.05em}y_1-y_2 \hspace{0.05em}\rVert \,\colon y_1,y_2 \in Y\}$ for the diameter of the set $Y$, and $\convex(Y)$ for its convex hull in $H$.
    Lastly, to distinguish clearly between deterministic and stochastic objects, we will write any random object in \emph{bold} case.

\section{Stochastic Formulation and Algorithm}\label{SEC:stochasticFormulationAndAlgorithm}

    In this section, we formalize the stochastic framework required to establish rigorous probabilistic convergence guarantees for Algorithm~\ref{ALG:onlineAdaptiveFrankWolfeAlgorithm}. 
    To this end, we model the sequences of data\hspace{0.05em}-\hspace{0.035em}driven approximations $(f_n\hspace{-0.035em})_{n \in \NN}$ and $(\hspace{-0.035em} D_n\hspace{-0.035em})_{n \in \NN}$ which are used as inputs of Algorithm~\ref{ALG:onlineAdaptiveFrankWolfeAlgorithm} as stochastic processes $\Fb$ and $\Db$, respectively. 
    We introduce the assumptions and regularity properties needed to ensure that this stochastic formulation is well\hspace{0.05em}-\hspace{0.035em}defined. 
    Building on these constructions, we then define an abstract stochastic version of Algorithm~\ref{ALG:onlineAdaptiveFrankWolfeAlgorithm}, which serves as a theoretical foundation for the subsequent convergence analysis. 
    The presentation is developed for the general case in which \emph{both} the objective function $f$ and the domain $D$ are unknown and must be approximated. 
    If either of them is known, the results can typically be adapted by replacing the corresponding stochastic approximation with its exact counterpart. 
    Simplified or less restrictive versions of the definitions and results that arise in such cases will be pointed out in the accompanying remarks.
    
    \subsection{Stochastic Approximation Processes}\label{SUBSEC:stochasticApproximation}
        Let $(\Omega, \Sigma, \PP)$ be a complete probability space.
        To model the data\hspace{0.05em}-\hspace{0.035em}driven approximation sequences $(f_n\hspace{-0.035em})_{n \in \NN}$ of the objective function $f$ and $(\hspace{-0.035em} D_n\hspace{-0.035em})_{n \in \NN}$ of the domain $D$ used in Algorithm~\ref{ALG:onlineAdaptiveFrankWolfeAlgorithm}, we consider stochastic processes 
        \begin{equation*}
            \Fb \colon \NN \times \Omega \to \Fcc
            \quad \text{and} \quad
            \Db \colon \NN \times \Omega \to \Dcc\hspace{0.1em},
        \end{equation*}
        where $\Fcc$ and $\Dcc$ denote suitable measurable spaces to be defined below.

        \paragraph{Definition of the Domain Space.} 
        Since the problem domain $D$ is nonempty, compact, and convex, it is natural to assume that the same holds for all its approximations. 
        Therefore, we set $\Dcc$ to be the set of all nonempty, compact and convex subsets of $H$ and define a corresponding metric on it.

        \begin{definition}{Hausdorff Distance}{hausdorffDistance}
            Let $X, Y \subseteq H$ be nonempty and compact sets.
            The \emph{Hausdorff distance} between $X$ and $Y$ is defined as 
            \begin{equation*}
                d_H(X, Y) \,\coloneqq\, \max\hspace{0.1em}\{\hspace{0.05em}\sup\hspace{0.1em}\{h(x, Y) \,\colon x \in X\},\hspace{0.05em} \sup\hspace{0.1em}\{h(y, X) \,\colon y \in Y\}\}\hspace{0.1em}.
            \end{equation*}
        \end{definition}

        \noindent
        Alternatively, it can be shown that for $X, Y \subseteq H$ nonempty and compact the Hausdorff distance $d_H(X, Y)$ can be written as
        \begin{equation}\label{EQ:hausdorffFormulation}
            d_H(X, Y) \,=\, \inf\hspace{0.1em}\{\hspace{0.05em}\rho \geq 0 \,\colon X \subseteq Y + \hspace{0.05em} \BB(0, \rho), Y \subseteq X + \hspace{0.05em} \BB(0, \rho)\}\hspace{0.1em},
        \end{equation}
        where in this case $+$ stands for the Minkowski addition of sets \cite[Chapter 3]{schneider2013convex}.
        It is well\hspace{0.05em}-\hspace{0.035em}known that $d$ defines a metric on the family $\Kcc$ of nonempty and compact subsets of $H$ and that the resulting metric space $(\Kcc\hspace{-0.1em}, d\hspace{0.05em})$ is Polish, that is, complete and separable, since $(\hspace{-0.05em}H, h)$ itself is Polish \cite[Chapter 3]{beer1993topologies}.
        As $\Dcc$ is a closed subset of $\Kcc$\hspace{-0.1em}, completeness and separability are inherited, making $(\hspace{-0.035em}\Dcc, d\hspace{0.05em})$ itself a Polish metric space.  
        Thus, we obtain the measurable space $(\hspace{-0.035em}\Dcc, \Bcc(\hspace{-0.05em}\Dcc))$ and can formally consider a stochastic domain approximation process $\Db \colon \NN \times \Omega \to \Dcc$, which can be interpreted as a sequence $(\hspace{0.035em}\Db_{\hspace{-0.025em}n}\hspace{-0.035em})_{n \in \NN}$ of $\Sigma$\hspace{0.1em}-\hspace{0.05em}measurable compact and convex random sets in the sense of \cite{molchanov2017theory}. 
        A natural goal would be to hope that for at least some sample points $\omega \in \Omega$ it holds that
        \begin{equation*}
            \lim_{n \hspace{0.05em}\to\hspace{0.05em} \infty} d_H(\hspace{0.035em}\Db_{\hspace{-0.025em}n}\hspace{-0.05em}(\omega), D) \,=\, 0\hspace{0.1em},
        \end{equation*}
        that is, as the number of samples increases, we would like to obtain better approximations of the domain $D$.
        However, in general, it may not be feasible to extract perfect approximations, even with an infinite amount of data, let alone for all sample points $\omega \in \Omega$. 
        Thus, we have to make the following assumption on the asymptotic convergence behavior of the stochastic  process $\Db$.

        \begin{assumption}{}{hausdorffConvergence}
            The domain approximation process $\Db$ is such that there exist constants $\beta_{\hspace{0.025em}1} \in [\hspace{0.025em}0,1]$ and $\eta_{\hspace{0.025em} 1} \geq 0$ with 
            \begin{equation*}
                \PP\hspace{-0.2em}\left[\hspace{0.05em}\limsup_{n \hspace{0.05em}\to\hspace{0.05em} \infty} \hspace{0.1em} d_H(\hspace{0.035em}\Db_{\hspace{-0.025em}n}\hspace{0.05em}, D) \,\leq\, \eta_{\hspace{0.025em} 1}\right] \,\geq\, 1-\beta_{\hspace{0.025em}1}\hspace{0.1em}.
            \end{equation*}
        \end{assumption}

        \noindent
        The constant $\beta_{\hspace{0.025em}1}$ represents an approximation failure tolerance (for example, $\beta_{\hspace{0.025em}1} = 0.05$), while the constant $\eta_{\hspace{0.05em} 1}$ quantifies the limiting approximation quality (for example, $\eta_{\hspace{0.05em} 1} = 0.1$) that can be obtained with a probability of at least $1-\beta_{\hspace{0.025em}1}$.
        If $\eta_{\hspace{0.05em} 1} = 0$, the limes superior can be replaced by a usual limes since $d_H(\hspace{0.035em}\Db_{\hspace{-0.025em}n}\hspace{0.05em}, D) \geq 0$ for all $n \in \NN$. 
        If additionally $\beta_{\hspace{0.025em}1} = 0$, then Assumption~\ref{ASS:hausdorffConvergence} implies almost sure convergence of the sequence $(\hspace{0.035em}\Db_{\hspace{-0.025em}n}\hspace{-0.035em})_{n \in \NN}$ to the problem domain $D$ in the sense of \cite{molchanov2017theory}. 
        
        \paragraph{Definition of the Function Space.} 
        Having established the stochastic framework for the domain process $\Db$, 
        we now turn to the construction of an appropriate function space $\Fcc$ for the objective approximation process $\Fb$. 
        While conceptually similar, the function setting requires additional care, as the effective domain on which the approximations are defined must remain compatible across iterations. 
        Since the objective function $f$ is convex and continuously differentiable, it again is natural to assume the same for all its approximations. 
        The main concern is now to find a common domain for these approximations.
        Indeed, one might first assume that $D_n \subseteq \dom(f_n\hspace{-0.035em})$ for all $n \in \NN$ would suffice to run Algorithm~\ref{ALG:onlineAdaptiveFrankWolfeAlgorithm}.
        However, starting from $x_0 \in \dom(f_1\hspace{-0.035em})$, we have $x_1 = s_1 \in D_1$ since $\lambda_{\hspace{0.025em}0} = 1$, but already $x_2$ could be located anywhere in $\convex(D_1 \cup D_2)$. 
        In general, for $n \in \NN$, we only know that it holds
        \begin{equation*}
            x_n \in \convex\hspace{-0.2em}\left(\hspace{0.075em}\bigcup\hspace{0.2em}\{\hspace{-0.035em}D_k \colon k \in [n]\hspace{0.025em}\}\hspace{-0.05em}\right)\hspace{-0.1em},
        \end{equation*}
        such that for $\nabla\hspace{-0.1em}f_{n+\hspace{-0.05em}1}\hspace{-0.05em}(x_n)$ to be well\hspace{0.05em}-\hspace{0.035em}defined we have to make sure that at least 
        \begin{equation*}
            \convex\hspace{-0.2em}\left(\hspace{0.075em}\bigcup\,\{\hspace{-0.035em}D_k \colon k \in [n]\hspace{0.025em}\}\hspace{-0.05em}\right) \,\subseteq \, \dom(f_{n+\hspace{-0.05em}1}\hspace{-0.035em})\hspace{0.1em}.
        \end{equation*}
        Moreover, the common domain used for the convergence analysis must not depend on a specific sample point.
        This motivates the following assumption on the existence of a uniform extension domain.

        \begin{assumption}{}{uniformDomainExtension}
            The domain approximation process $\Db$ is such that there exists a constant $\varepsilon \in [\hspace{0.025em}0,1]$ and a uniform extension domain $E \in \Dcc$ with $D \subseteq E \subseteq \dom(f)$ satisfying
            \begin{equation*}
                \PP\hspace{-0.2em}\left[\hspace{0.075em}\bigcup\hspace{0.2em} \{\Db_{\hspace{-0.025em}n} \hspace{0.05em}\colon n \in \NN\} \,\subseteq\, E\hspace{0.05em}\right]  \,\geq\, 1-\varepsilon\hspace{0.1em}.
            \end{equation*}
        \end{assumption}

        \noindent
        Note that the set
        \begin{equation}\label{EQ:eventN}
            N^+\,\coloneqq\, \left\{\hspace{0.05em}\bigcup\hspace{0.2em} \{\Db_{\hspace{-0.025em}n} \hspace{0.05em}\colon n \in \NN\} \,\subseteq\, E \right\} \,=\, \bigcap\hspace{0.2em}\{\{\hspace{0.05em}\Db_{\hspace{-0.025em}n} \,\subseteq\, E\hspace{0.05em}\} \,\colon n \in \NN\hspace{0.025em}\} \,\subseteq\, \Omega
        \end{equation}
        is indeed $\Sigma$\hspace{0.1em}-\hspace{0.05em}measurable due to the measurability of $\Db_{\hspace{-0.025em}n}$ for all $n \in \NN$ (see Appendix \ref{SEC:additionalProofs}), such that Assumption~\ref{ASS:uniformDomainExtension} is well\hspace{0.05em}-posed. 
        The probabilistic constant $\varepsilon$ in Assumption \ref{ASS:uniformDomainExtension} represents a tolerance for ill\hspace{0.05em}-\hspace{0.05em}behaved approximation processes $\Db$ reaching out farther than acceptable (for example, $\varepsilon = 0.05$).
        By Assumption~\ref{ASS:hausdorffConvergence}, we know that
        \begin{equation*}
            \PP\hspace{-0.2em}\left[\hspace{0.075em}\bigcup\hspace{0.2em} \{\Db_{\hspace{-0.025em}n} \hspace{0.05em}\colon n \in \NN\hspace{0.025em}\} \text{ is bounded}\hspace{0.05em}\right]  \,\geq\, \PP\hspace{-0.2em}\left[\hspace{0.05em}\limsup_{n \hspace{0.05em}\to\hspace{0.05em} \infty} \hspace{0.1em} d_H(\hspace{0.035em}\Db_{\hspace{-0.025em}n}\hspace{0.05em}, D) \,\leq\, \eta_{\hspace{0.025em} 1}\right] \,\geq\, 1-\beta_{\hspace{0.025em}1}\hspace{0.1em},
        \end{equation*}
        such that the most restrictive aspect of Assumption~\ref{ASS:uniformDomainExtension} is the existence of a deterministic extension domain $E$ that is independent of the actual sample points. 
        Furthermore, the domain $\dom(f)$ must be large enough to contain such an extension $E$.
        Note that, under additional regularity conditions of the objective function $f$, or more specifically its gradient $\nabla\hspace{-0.1em}f$, $\dom(f)$ may be extended to the whole space $H$, rendering this part redundant \cite[Theorem 1.8]{azagra2017whitney}. 
        When the data originates from bounded or light\hspace{0.05em}-\hspace{0.05em}tailed distributions, the tolerance $\varepsilon$ can often be chosen as zero or at least arbitrarily small, influencing the size of $E$.
        We denote $\Delta \coloneqq \diam(E)$ for the diameter of the uniform extension domain $E$, a quantity that will be used later. 

        \begin{remark}{}{weakerAssumptionKnownObjective}
            If the objective function $f$ is known, then Assumption~\ref{ASS:uniformDomainExtension} can be weakened: Instead of a fixed uniform extension domain $E \in \Dcc$, it suffices to assume the existence of a $\Sigma$\hspace{0.1em}-\hspace{0.05em}measurable map $\Eb \colon \Omega \to \Dcc$ and a constant $\varepsilon_{0} \in [\hspace{0.025em}0, 1]$ satisfying
            \begin{equation*}
                \PP\hspace{-0.2em}\left[\hspace{0.075em}\bigcup\hspace{0.2em} \{\Db_{\hspace{-0.025em}n} \hspace{0.05em}\colon n \in \NN\} \hspace{0.1em}\cup\hspace{0.05em} D\,\subseteq\, \Eb \,\subseteq\, \dom(f)\hspace{0.05em}\right]  \,\geq\, 1-\varepsilon_0\hspace{0.1em}.
            \end{equation*}
            Here, the set
            \begin{equation}\label{EQ:eventN0}
                N_{0}^+ \,\coloneqq\, \left\{\hspace{0.05em}\bigcup\hspace{0.2em} \{\Db_{\hspace{-0.025em}n} \hspace{0.05em}\colon n \in \NN\} \cup D\hspace{0.05em} \,\subseteq\, \Eb \,\subseteq\, \dom(f)\right\}
            \end{equation}
            is $\Sigma$\hspace{0.1em}-\hspace{0.05em}measurable due to the measurability of $\Db_{\hspace{-0.025em}n}$ for all $n \in \NN$ and the completeness of $(\Omega, \Sigma, \PP)$.
            Note that under Assumption \ref{ASS:hausdorffConvergence} the existence of a $\Sigma$\hspace{0.1em}-\hspace{0.05em}measurable map $\Eb$ satisfying 
            \begin{equation}\label{EQ:guaranteedExistance}
                \PP\hspace{-0.2em}\left[\hspace{0.075em}\bigcup\hspace{0.2em} \{\Db_{\hspace{-0.025em}n} \hspace{0.05em}\colon n \in \NN\} \hspace{0.1em}\cup\hspace{0.05em} D\,\subseteq\, \Eb \hspace{0.05em}\right]  \,\geq\, 1-\beta_{\hspace{0.025em}1}
            \end{equation}
            is guaranteed, such that the only remaining requirement is that $f$ is well\hspace{0.05em}-\hspace{0.035em}defined on this random extension domain $\Eb$ with high probability. 
            Furthermore, we can guarantee that $\varepsilon_0 \leq \varepsilon$ and for many cases it even holds that $\varepsilon_0$ can be chosen as zero.
            A detailed proof of this remark is provided in Appendix \ref{SEC:B.1}.
        \end{remark}

        \noindent
        Considering a uniform extension domain $E$ as in Assumption~\ref{ASS:uniformDomainExtension}, we define 
        \begin{equation*}
            \Fcc \coloneqq \{f \in C^{\hspace{0.025em}1}\hspace{-0.05em}(E, \RR) \,\colon f \text{ is convex}\}
        \end{equation*}
        and endow $C^{\hspace{0.025em}1}\hspace{-0.05em}(E, \RR)$ with the supremum norm
        \begin{equation}\label{EQ:supremumNorm}
            \lVert\hspace{0.05em} f \hspace{0.05em}\rVert_\infty \,\coloneqq\, \sup\hspace{0.1em}\{\hspace{0.05em}\lvert \hspace{0.05em}f(x)\hspace{0.05em} \rvert \,\colon x \in E\hspace{0.05em}\} \,+\, \sup\hspace{0.1em}\{\hspace{0.05em}\lVert \nabla\hspace{-0.1em}f(x) \hspace{0.05em}\rVert \,\colon x \in E\hspace{0.05em}\}\hspace{0.1em}.
        \end{equation}
        Since $C^{\hspace{0.025em}1}\hspace{-0.05em}(E, \RR)$ together with the supremum norm in \eqref{EQ:supremumNorm} forms a separable Banach space and, therefore, a Polish metric space, the same holds for $\Fcc$ as closed subspace \cite{aliprantis2006infinite, bogachev2007measure}.
        Thus, we obtain another measurable space $(\hspace{-0.05em}\Fcc, \Bcc(\Fcc))$ and can consider a stochastic objective approximation process $\Fb \colon \NN \times \Omega \to \Fcc$. 
        Similar to before, we can interpret the stochastic process $\Fb$ as a sequence $(\hspace{0.035em}\Fb_{\hspace{-0.075em}n}\hspace{-0.035em})_{n \in \NN}$ of $\Sigma$\hspace{0.1em}-\hspace{0.05em}measurable convex and continuously differentiable random functions.
        Again, rather than assuming exact convergence, we describe its limiting behavior probabilistically.

        \begin{assumption}{}{functionConvergence}
            The objective approximation process $\Fb$ is such that there exist constants $\beta_{\hspace{0.025em}2} \in [\hspace{0.025em}0,1]$ and $\eta_{\hspace{0.05em}2} \geq 0$ with
            \begin{equation*}
                \PP\hspace{-0.2em}\left[\hspace{0.05em}\limsup_{n \hspace{0.05em}\to\hspace{0.05em} \infty} \hspace{0.2em}\lVert \hspace{0.05em}\Fb_{\hspace{-0.075em}n} - f \hspace{0.05em}\rVert \,\leq\, \eta_{\hspace{0.05em}2}\hspace{0.05em}\right] \,\geq\, 1-\beta_{\hspace{0.025em}2}\hspace{0.1em}.
            \end{equation*}
        \end{assumption}
        
        \noindent
        The constant $\beta_{\hspace{0.025em}2}$ represents an approximation failure tolerance (for example, $\beta_{\hspace{0.025em}2} = 0.05$), while the constant $\eta_{\hspace{0.05em} 2}$ quantifies the limiting approximation quality (for example, $\eta_{\hspace{0.05em} 2} = 0.1$) that can be obtained with a probability of at least $1-\beta_{\hspace{0.025em}2}$. 
        If $\eta_{\hspace{0.05em}2} = 0$, then the limes superior can be replaced with the usual limes since $\lVert \hspace{0.05em}\Fb_{\hspace{-0.075em} n} -f \hspace{0.05em}\rVert \geq 0$ for all $n \in \NN$. 
        If additionally $\beta_{\hspace{0.025em}2} = 0$, then Assumption~\ref{ASS:functionConvergence} implies almost sure uniform convergence of the sequences $(\hspace{0.035em}\Fb_{\hspace{-0.075em}n}\hspace{-0.035em})_{n \in \NN}$ and $(\nabla\Fb_{\hspace{-0.075em}n}\hspace{-0.035em})_{n \in \NN}$ to the objective function $f$ and its gradient $\nabla\hspace{-0.1em}f$, respectively. \\

        \noindent
        Together, Assumption~\ref{ASS:hausdorffConvergence} and Assumption~\ref{ASS:functionConvergence} ensure that both stochastic approximation processes $\Db$ and $\Fb$, interpreted as approximation sequences, concentrate around their true counterparts $D$ and $f$ 
        with joint probability at least $1 - (\beta_{\hspace{0.025em}1} + \beta_{\hspace{0.025em}2})$, while Assumption~\ref{ASS:uniformDomainExtension} ensures the existence of the approximation process $\Fb$ on a common, deterministic state space $E$.
        With these measurable structures in place, we are now prepared to formulate the abstract stochastic version of Algorithm~\ref{ALG:onlineAdaptiveFrankWolfeAlgorithm}, which we introduce next and which will serve as the main object of the convergence analysis.

    \subsection{Stochastic Algorithm Definition and Interpretation}\label{SUBSEC:stochasticAlgorithmDefinitionAndInterpretation}

        Having established the stochastic approximation spaces $\Dcc$ and $\Fcc$ and the uniform extension domain $E$, we can now extend Algorithm~\ref{ALG:onlineAdaptiveFrankWolfeAlgorithm} to a fully stochastic setting in which all iterates and search directions are modeled as random variables on the underlying probability space $(\Omega, \Sigma, \PP)$.
        As a randomized version of an initial iterate, we assume to be given a $\Sigma$\hspace{0.1em}-\hspace{0.05em}measurable map $\Xb_0\colon \Omega \to H$.
        Since the only role of $\Xb_0$ is to enable the computation of the initial search direction $\nabla \Fb_{\hspace{-0.1em} 1}\hspace{-0.05em}(\hspace{0.035em}\Xb_0\hspace{-0.05em})$, it suffices to assume that $\Xb_0 \in E$ almost surely.
        In most cases, the random initial iterate $\Xb_0$ may be chosen as a constant $x_0 \in E$ or uniformly distributed on some sensible subset of $E$, such that this part is trivially satisfied.
        For the sake of formality, in the following and for the rest of this work, we have to restrict the set $N^+$ from \eqref{EQ:eventN} to the set 
        \begin{equation}\label{EQ:eventNUpdated}
            N \,\coloneqq\, N^+ \hspace{-0.025em}\cap \hspace{0.05em}\{\Xb_0 \in E\} \in \Sigma\hspace{0.1em},   
        \end{equation}
        which is no restriction since if $\Xb_0 \in E$ almost surely we have $\PP[N] \geq 1-\varepsilon$ by Assumption \ref{ASS:uniformDomainExtension}.

        \begin{remark}{}{setAdaptation}
            If the objective function $f$ is known, then in the setting of Remark \ref{REM:weakerAssumptionKnownObjective}, the initialization assumption that $\Xb_0 \in E$ almost surely can be reduced to assume that $\Xb_0 \in \Eb$ almost surely, where the event $\{\Xb_0 \in\Eb\hspace{0.025em}\}$ is measurable due to the completeness of  $(\Omega,\Sigma,\PP)$. 
            Here, as in \eqref{EQ:eventNUpdated} we restrict the set $N_0^+$ from \eqref{EQ:eventN0} to the set 
            \begin{equation*}
                N_0 \coloneqq N_0^+ \cap \{\Xb_0 \in \Eb\} \,\subseteq\, \Omega\hspace{0.1em},
            \end{equation*}
            which is $\Sigma$\hspace{0.1em}-\hspace{0.05em}measurable and satisfies $\PP[N_0] \geq 1-\varepsilon_0$ as well.
            The following constructions and results in this section then apply with $E$ replaced by $\Eb$, $N$ replaced by $N_0$, and $\varepsilon$ replaced by $\varepsilon_0$.
            A proof of this remark is provided in Appendix \ref{SEC:B.1}.

        \end{remark}
        
        \paragraph{Algorithmic Processes.} From the random initial iterate $\Xb_0$ onward, we can now aim to construct stochastic processes $\Sb$, modeling the solutions to the linear subproblems in Algorithm~\ref{ALG:onlineAdaptiveFrankWolfeAlgorithm}, and $\Xb$, modeling the iterates of Algorithm~\ref{ALG:onlineAdaptiveFrankWolfeAlgorithm}. 
        To this end, we first have to analyze how the computation of a solution to the linear subproblems, that is, line \ref{LINE:subproblem} in Algorithm~\ref{ALG:onlineAdaptiveFrankWolfeAlgorithm}, can be modeled in a probabilistic way. 
        Indeed, since we are dealing with an $\argmin$ problem that possibly attains multiple solutions, it is important to abstractly model the selection of such a solution in a way that preserves $\Sigma$\hspace{0.1em}-\hspace{0.05em}measurability of the resulting search direction. 
        If the reader is not yet familiar with correspondences (also known as multifunctions or set\hspace{0.05em}-\hspace{0.05em}valued maps) and the corresponding definitions of measurability, we suggest a look at Appendix \ref{SEC:setValuedMaps} for a brief introduction or a study of the corresponding chapters in the textbooks \cite{aliprantis2006infinite, rockafellar1998variational, castaing1977convex}.

        \begin{lemma}{}{subproblemSelector}
            Let $\Yb \colon \Omega \to H$ be a $\Sigma$\hspace{0.1em}-\hspace{0.05em}measurable map and let $A \in \Sigma$ with $A \subseteq \{\Yb \in E\hspace{0.05em}\}$. 
            Then, for all $n \in \NN$, the correspondence
            \begin{equation*}
                \Omega \,\rightrightarrows\, H, \; \omega \,\mapsto\, \argmin \hspace{0.1em}\{\ind_{\hspace{-0.05em} A}\hspace{-0.035em}(\omega)\hspace{0.1em}\langle \hspace{0.05em}s \,|\, \nabla \Fb_{\hspace{-0.075em} n}\hspace{-0.05em}(\omega)(\Yb(\omega))\hspace{0.05em }\rangle \,\colon s \in \Db_{\hspace{-0.025em}n}\hspace{-0.05em}(\omega)\}
            \end{equation*}
            is measurable and admits a $\Sigma$\hspace{0.1em}-\hspace{0.05em}measurable selector, that is, a map $\sigma\colon \Omega \to H$ satisfying 
            \begin{equation*}
                \sigma(\omega) \in \argmin \hspace{0.1em}\{\ind_{\hspace{-0.05em} A}\hspace{-0.035em}(\omega)\hspace{0.1em}\langle s \,|\, \nabla \Fb_{\hspace{-0.05em} n}\hspace{-0.05em}(\omega)(\Yb(\omega))\rangle \,\colon s \in \Db_{\hspace{-0.025em}n}\hspace{-0.05em}(\omega)\}
            \end{equation*}
            for all $\omega \in \Omega$.
        \end{lemma}

        \noindent
        Since by assumption we have that $\Xb_0$ is $\Sigma$\hspace{0.1em}-\hspace{0.05em}measurable, we can apply Lemma~\ref{LEM:subproblemSelector} to $\Yb = \Xb_0$ for $A = N$ and $n = 1$ to obtain a $\Sigma$\hspace{0.1em}-\hspace{0.05em}measurable selector $\Sb_1 \colon \Omega \to H$ satisfying
        \begin{equation*}
            \Sb_1\hspace{-0.05em}(\omega) \in \argmin \hspace{0.1em}\{\ind_{\hspace{-0.05em} N}\hspace{-0.035em}(\omega)\hspace{0.1em}\langle s \,|\, \nabla \Fb_{\hspace{-0.075em} 1}\hspace{-0.05em}(\hspace{0.035em}\Xb_0(\omega))\rangle \,\colon s \in \Db_1\hspace{-0.05em}(\omega)\hspace{0.05em}\}
        \end{equation*}
        for all $\omega \in \Omega$. 
        Then, we can construct an abstract new iterate $\Xb_1 \colon \Omega \to H$ by setting 
        \begin{equation*}
            \Xb_1\hspace{-0.05em}(\omega) \,=\, (1-\lambda_{\hspace{0.025em}0})\hspace{0.05em}\Xb_0(\omega) + \lambda_{\hspace{0.025em}0}\hspace{0.05em}\Sb_1\hspace{-0.05em}(\omega)
        \end{equation*}
        for all $\omega \in \Omega$, where $\lambda_{\hspace{0.025em}0} = 2/(2 + 0) = 1$ is the step\hspace{0.075em}-\hspace{0.05em}size chosen as in Algorithm~\ref{ALG:onlineAdaptiveFrankWolfeAlgorithm}.
        The new random iterate $\Xb_1$ is now $\Sigma$\hspace{0.1em}-\hspace{0.05em}measurable as convex combination of $\Sigma$\hspace{0.1em}-\hspace{0.05em}measurable maps.
        Furthermore, it holds
        \begin{equation*}
            \Xb_1\hspace{-0.05em}(\omega) \,=\, \Sb_1\hspace{-0.05em}(\omega) \in \Db_1\hspace{-0.05em}(\omega) \,\subseteq\, E
        \end{equation*}
        for all sample points $\omega \in N$.
        Therefore, we can proceed inductively and apply Lemma~\ref{LEM:subproblemSelector} to $\Yb = \Xb_{\hspace{0.025em}n-\hspace{-0.05em}1}$ for $A = N$ and $n \in \NN$ to construct the subsolution process $\Sb \colon \NN \times \Omega \to H$ satisfying 
        \begin{equation}\label{EQ:theoreticalSubsolutionProcess}
            \Sb_{n}\hspace{-0.05em}(\omega) \in \argmin \hspace{0.1em}\{\ind_{\hspace{-0.05em} N}\hspace{-0.035em}(\omega)\hspace{0.1em}\langle \hspace{0.05em}s \,|\, \nabla \Fb_{\hspace{-0.075em} n}\hspace{-0.05em}(\hspace{0.035em}\Xb_{\hspace{0.025em}n-\hspace{-0.05em}1}\hspace{-0.05em}(\omega))\hspace{0.05em}\rangle \,\colon s \in \Db_{\hspace{-0.025em}n}\hspace{-0.05em}(\omega)\}
        \end{equation}
        for all $\omega \in \Omega$ and the iterate process $\Xb \colon \NN \times \Omega \to H$ satisfying 
        \begin{equation}\label{EQ:theoreticalIterateProcess}
            \Xb_{\hspace{0.025em}n}\hspace{-0.05em}(\omega) \,=\, (1-\lambda_{\hspace{0.025em}n-\hspace{-0.05em}1})\hspace{0.05em}\Xb_{\hspace{0.025em}n-\hspace{-0.05em}1}\hspace{-0.05em}(\omega) + \lambda_{\hspace{0.025em}n-\hspace{-0.05em}1}\hspace{0.05em}\Sb_{n}\hspace{-0.05em}(\omega)
        \end{equation}
        for all $\omega \in \Omega$, where $\lambda_{\hspace{0.025em}n-\hspace{-0.05em}1} = 2/(2+n-1)$ as in Algorithm~\ref{ALG:onlineAdaptiveFrankWolfeAlgorithm}. 
        Note that defining $\Sb$ as in \eqref{EQ:theoreticalSubsolutionProcess} and $\Xb$ as in \eqref{EQ:theoreticalIterateProcess}, we have $\Sb_{n}\hspace{-0.05em}(\omega)\in \Db_{\hspace{-0.025em}n}\hspace{-0.05em}(\omega) \subseteq E$ as well as $\Xb_{\hspace{0.025em}n}\hspace{-0.05em}(\omega) \in E$ for all $n \in \NN$ and sample points $\omega \in N$, such that under Assumption \ref{ASS:uniformDomainExtension} we have
        \begin{equation*}
            \PP[\hspace{0.1em}\Sb_{n} \in E \text{ and } \Xb_{\hspace{0.025em}n} \in E \text{ for all } n \in \NN\hspace{0.1em}] \,\geq\, 1-\varepsilon\hspace{0.1em}.
        \end{equation*}
        In other words, we can guarantee that both, the subsolution process $\Sb$, and the iterate process $\Xb$ stay inside the uniform extension domain $E$ with high probability, such that all evaluations of $f$ and $\Fb_{\hspace{-0.075em}n}$ or their gradients are well\hspace{0.05em}-\hspace{0.035em}defined for all $n \in \NN$, that is, along the entire stochastic trajectory. 
        
        \paragraph{Abstract Stochastic Algorithm.} Now that we have defined the domain approximation processes $\Db$ and objective approximation process $\Fb$ and derived the resulting subsolution process $\Sb$ and iterate process $\Xb$, we have formally taken Algorithm~\ref{ALG:onlineAdaptiveFrankWolfeAlgorithm} to an abstract stochastic level. 
        A stochastic version of Algorithm~\ref{ALG:onlineAdaptiveFrankWolfeAlgorithm} is therefore given by Algorithm~\ref{ALG:onlineAdaptiveStochasticFrankWolfeAlgorithm} below, which can be read as a compact summary of the construction in \eqref{EQ:theoreticalSubsolutionProcess} and  \eqref{EQ:theoreticalIterateProcess}. 

        \begin{algorithmOuter}{Recursive Adaptive Stochastic Frank\hspace{0.05em}-Wolfe Algorithm}{onlineAdaptiveStochasticFrankWolfeAlgorithm}
            \begin{algorithmic}[1]
                \Require Domain approximation process $\Db$, objective approximation process $\Fb$, uniform extension domain $E$ and initial iterate $\Xb_{\hspace{0.025em}0}$ \\
                \hspace*{-62pt}{\color{seeblau}\hdashrule{502pt}{0.75pt}{2pt}}\\[-4mm]
                \State set $N = \left\{\hspace{0.075em}\bigcup\hspace{0.1em} \{\hspace{0.05em}\Db_{\hspace{-0.025em}n}\hspace{0.05em}\colon n \in \NN\hspace{0.05em}\} \hspace{0.1em}\subseteq\, E\hspace{0.025em}\right\} \cap \{\hspace{0.025em}\Xb_0 \in E\hspace{0.025em}\}$
                \For{$n \in \NN_0$}
                    \State select $\Sb_{\hspace{-0.01em}n+\hspace{-0.05em}1} \hspace{-0.05em}\in \argmin \hspace{0.1em}\{\ind_{\hspace{-0.05em}N} \langle \hspace{0.05em} s \,|\, \nabla \Fb_{\hspace{-0.075em}n+\hspace{-0.05em}1}\hspace{-0.05em}(\hspace{0.035em}\Xb_{\hspace{0.025em}n})\hspace{0.05em} \rangle\, \colon s \in \Db_{\hspace{-0.035em}n+\hspace{-0.05em}1}\}$
                    \State set $\lambda_{\hspace{0.025em}n} = 2\hspace{0.025em}/(2+n)$
                    \State update $\Xb_{\hspace{0.025em}n+\hspace{-0.05em}1} = \Xb_{\hspace{0.025em}n} + \lambda_{\hspace{0.025em}n}\hspace{0.05em}(\Sb_{\hspace{-0.01em}n+\hspace{-0.05em}1} - \Xb_{\hspace{0.025em}n})$
                \EndFor
            \end{algorithmic}
        \end{algorithmOuter}

        \noindent
        We stress that it is important to understand Algorithm~\ref{ALG:onlineAdaptiveStochasticFrankWolfeAlgorithm} not as an actual programmable algorithm, but as an entirely abstract algorithm which has the purpose of deriving rigorous probabilistic convergence results that can later be restricted to the deterministic case of Algorithm~\ref{ALG:onlineAdaptiveFrankWolfeAlgorithm}.
        Considering the returned stochastic iterate process $\Xb$ of Algorithm~\ref{ALG:onlineAdaptiveStochasticFrankWolfeAlgorithm}, for any sample point $\omega \in \Omega$,  we can interpret the corresponding sample sequence $(\hspace{0.035em}\Xb_{\hspace{0.025em}n}\hspace{-0.05em}(\omega))_{n \in \NN} \subseteq H$ as the return of Algorithm~\ref{ALG:onlineAdaptiveFrankWolfeAlgorithm}.
        Thus, any result stating probabilistic properties of $\Xb$ can be understood to hold for the deterministic counterpart $(\hspace{0.035em}\Xb_{\hspace{0.025em}n}\hspace{-0.05em}(\omega))_{n \in \NN}$ by viewing the stochastic analysis as providing high\hspace{0.05em}-\hspace{0.025em}probability guarantees for individual realizations of Algorithm~\ref{ALG:onlineAdaptiveFrankWolfeAlgorithm} fed with random data.
        In particular, given an iteration $n \in \NN$ and a sample point $\omega \in \Omega$, we can interpret 
        \begin{itemize}
            \item $\Db_{\hspace{-0.025em}n}\hspace{-0.05em}(\omega)$ as an $n$\hspace{0.1em}-\hspace{0.05em}th realized approximation of the problem domain $D$,
            \vspace{0.025em}
            \item $\Fb_{\hspace{-0.075em}n}\hspace{-0.05em}(\omega)$ as an $n$\hspace{0.1em}-\hspace{0.05em}th realized approximation of the objective function $f$,    
            \vspace{0.025em}
            \item $\Xb_{\hspace{0.025em}n}\hspace{-0.05em}(\omega)$ as the $n$\hspace{0.1em}-\hspace{0.05em}th iterate of Algorithm~\ref{ALG:onlineAdaptiveFrankWolfeAlgorithm} if run using the sequences $(\hspace{0.035em}\Db_{\hspace{-0.025em}n}\hspace{-0.05em}(\omega))_{n \in\NN}$ and $(\hspace{0.035em}\Fb_{\hspace{-0.075em}n}\hspace{-0.05em}(\omega))_{n \in \NN}$ and the starting point $\Xb_{\hspace{0.025em}0}(\omega) \in E = \dom(\hspace{0.035em}\Fb_{\hspace{-0.075em}1}\hspace{-0.05em}(\omega))$\hspace{0.1em},
            \vspace{0.025em}
            \item $\Sb_{n}\hspace{-0.05em}(\omega)$ as a solution to the linear subproblem with direction $\nabla\Fb_{\hspace{-0.075em} n}\hspace{-0.05em}(\omega)(\hspace{0.035em}\Xb_{\hspace{0.025em}n-\hspace{-0.05em}1}\hspace{-0.05em}(\omega))$ as given in \eqref{EQ:theoreticalSubsolutionProcess}, considered at the $n$\hspace{0.1em}-\hspace{0.05em}th iteration\hspace{0.05em}.
        \end{itemize}
        Therefore, for simplicity in notation, when fixing an iteration $n \in \NN$ and a sample point $\omega \in \Omega$, we will sometimes use the abbreviations 
        \begin{equation}\label{EQ:simplicticNotation}
            D_n \,=\, \Db_{\hspace{-0.025em}n}\hspace{-0.05em}(\omega), \quad f_n \,=\; \Fb_{\hspace{-0.075em}n}\hspace{-0.05em}(\omega), \quad x_n \,=\, \Xb_{\hspace{0.025em}n}\hspace{-0.05em}(\omega) \quad \text{and} \quad s_n \,=\, \Sb_{n}\hspace{-0.05em}(\omega)\hspace{0.1em},
        \end{equation}
        reminding of a deterministic setting.

        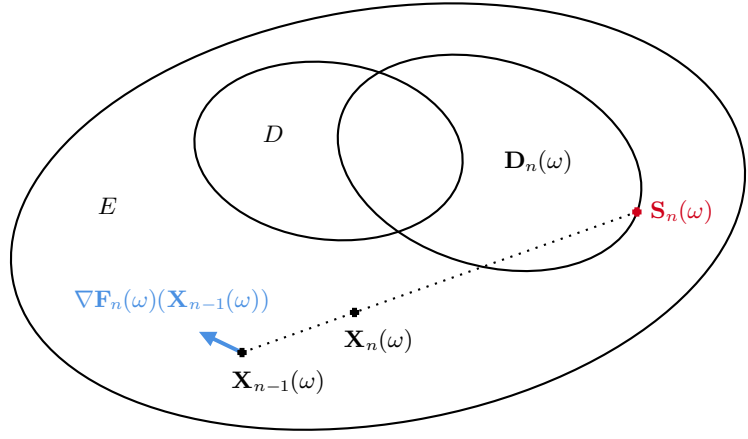
\begin{figure}[h]
            \begin{center}
                \tikzset{every picture/.style = {line width = 0.8pt}} 
                    \begin{tikzpicture}[x = 0.9pt, y = 0.9pt, yscale = -1, xscale = 1]
                        \draw                     (  0.70, 115.19) .. controls (-  7.33,   68.13) and 
                              ( 54.58,  18.32) .. (138.97,   3.92) .. controls ( 223.37, - 10.47) and 
                              (298.29,  16.01) .. (306.31,  63.06) .. controls ( 314.34,  110.12) and 
                              (252.43, 159.94) .. (168.04, 174.34) .. controls (  83.65,  188.73) and 
                              (  8.73, 162.25) .. (  0.70, 115.19) -- cycle;
                        \draw (76.91,54.78) .. controls (79.43,34.52) and (106.5,21.2) .. (137.37,25.04) .. controls (168.24,28.87) and (191.22,48.4) .. (188.71,68.66) .. controls (186.19,88.92) and (159.13,102.24) .. (128.26,98.4) .. controls (97.38,94.57) and (74.4,75.04) .. (76.91,54.78) -- cycle;
                        \draw (137.94,47.66) .. controls (144.87,25) and (178.34,15.15) .. (212.7,25.66) .. controls (247.06,36.16) and (269.3,63.05) .. (262.38,85.7) .. controls (255.45,108.36) and (221.98,118.21) .. (187.62,107.7) .. controls (153.26,97.2) and (131.02,70.31) .. (137.94,47.66) -- cycle;
                        
                        \draw  [dash pattern={on 0.84pt off 2.51pt}]  (96.63,145.97) -- (261.78,87.2) ;
                        \draw  [color={rgb, 255:red, 208; green, 2; blue, 27 }  ,draw opacity=1 ][fill={rgb, 255:red, 208; green, 2; blue, 27 }  ,fill opacity=1 ] (261.18,85.7) -- (262.38,85.7) -- (262.38,86.6) -- (263.28,86.6) -- (263.28,87.8) -- (262.38,87.8) -- (262.38,88.7) -- (261.18,88.7) -- (261.18,87.8) -- (260.28,87.8) -- (260.28,86.6) -- (261.18,86.6) -- cycle ;
                        \draw  [color={rgb, 255:red, 0; green, 0; blue, 0 }  ,draw opacity=1 ][fill={rgb, 255:red, 0; green, 0; blue, 0 }  ,fill opacity=1 ] (143.08,127.71) -- (144.28,127.71) -- (144.28,128.61) -- (145.18,128.61) -- (145.18,129.81) -- (144.28,129.81) -- (144.28,130.71) -- (143.08,130.71) -- (143.08,129.81) -- (142.18,129.81) -- (142.18,128.61) -- (143.08,128.61) -- cycle ;
                        
                        \draw [color={rgb, 255:red, 74; green, 144; blue, 226 }  ,draw opacity=1 ][line width=1.5]    (96.63,145.97) -- (82.21,139.04) ;
                        \draw [shift={(78.61,137.31)}, rotate = 25.67] [fill={rgb, 255:red, 74; green, 144; blue, 226 }  ,fill opacity=1 ][line width=0.08]  [draw opacity=0] (6.97,-3.35) -- (0,0) -- (6.97,3.35) -- cycle    ;
                        \draw  [fill={rgb, 255:red, 0; green, 0; blue, 0 }  ,fill opacity=1 ] (96.03,144.47) -- (97.23,144.47) -- (97.23,145.37) -- (98.13,145.37) -- (98.13,146.57) -- (97.23,146.57) -- (97.23,147.47) -- (96.03,147.47) -- (96.03,146.57) -- (95.13,146.57) -- (95.13,145.37) -- (96.03,145.37) -- cycle ;
                
                        \draw (205, 60) node [font=\footnotesize, anchor=north west][inner sep=0.75pt] {$\Db_{\hspace{-0.025em} n}\hspace{-0.05em}(\omega)$};
                        \draw (91,152) node [anchor=north west][inner sep=0.75pt]  [font=\footnotesize]  {$\Xb_{\hspace{0.025em}n-\hspace{-0.05em}1}\hspace{-0.05em}(\omega)$};
                        \draw (25,118) node [anchor=north west][inner sep=0.75pt]  [font=\footnotesize,color={rgb, 255:red, 74; green, 144; blue, 226 }  ,opacity=1 ]  {$\nabla\Fb_{\hspace{-0.075em} n}\hspace{-0.05em}(\omega)\hspace{-0.05em}(\hspace{0.035em}\Xb_{\hspace{0.025em}n-\hspace{-0.05em}1}\hspace{-0.05em}(\omega))$};
                        \draw (266, 81) node [anchor=north west][inner sep=0.75pt]  [font=\footnotesize,color={rgb, 255:red, 208; green, 2; blue, 27 }  ,opacity=1 ]  {$\Sb_{n}\hspace{-0.05em}(\omega)$};
                        \draw (138,135) node [anchor=north west][inner sep=0.75pt]  [font=\footnotesize,color={rgb, 255:red, 0; green, 0; blue, 0 }  ,opacity=1 ]  {$\Xb_{\hspace{0.025em}n}\hspace{-0.05em}(\omega)$};
                        \draw (35, 80) node [font=\footnotesize, anchor = north west][inner sep = 0.75pt]{$E$};
                        \draw (104, 50) node [font=\footnotesize, anchor = north west][inner sep = 0.75pt]{$D$};
                    \end{tikzpicture}
                \end{center}
                \vspace{-0.5cm}
            \caption{
            Conceptual illustration of a step ($n-\hspace{-0.05em}1 \to n$) of Algorithm~\ref{ALG:onlineAdaptiveStochasticFrankWolfeAlgorithm} for some fixed sample point $\omega \in N$\hspace{-0.1em}, ensuring that everything is contained in the uniform extension domain $E$. 
            By the previous analysis and using the notation \eqref{EQ:simplicticNotation} this can also be identified with a corresponding step of Algorithm \ref{ALG:onlineAdaptiveFrankWolfeAlgorithm}.
            }
            \label{fig:idea:stochastic:convergence}
        \end{figure}

\section{Theoretical Analysis and Convergence Results}\label{SEC:theoreticalAnalysis}

    In this section, we develop a convergence theory for the Algorithm~\ref{ALG:onlineAdaptiveStochasticFrankWolfeAlgorithm}.
    We start by recalling and tailoring standard tools from deterministic Frank\hspace{0.1em}-Wolfe analysis to our stochastic setting.
    Then, we derive some baseline \hbox{one\hspace{0.1em}-\hspace{0.05em}step} error bounds that are extended to asymptotic convergence bounds and explicit convergence rates under quantitative assumptions on the approximation processes.
    Finally, we briefly comment on how actual domain approximations may be generated via Carathéodory functions and sublevel sets.
    
        \subsection{Stochastic Convergence Analysis Tools}

            Before turning our attention to the probabilistic convergence analysis, we must introduce some commonly used definitions already used in the deterministic convergence analysis of Frank\hspace{0.1em}-Wolfe methods, however, in a slightly more complex environment, since we are dealing with uncertainty. 

            \paragraph{Curvature Constant.} First, we consider an extension of the so\hspace{0.1em}-\hspace{0.05em}called curvature constant, which is a tool for measuring the nonlinearity of a function over a given domain by measuring the deviation from its linear approximations.
            The curvature constant was originally introduced in \cite{clarkson2010coresets} and later revisited in \cite{jaggi2013revisiting} and has since become a standard smoothness surrogate in several Frank\hspace{0.1em}-Wolfe methods \cite{freund2016new, dvurechensky2023generalized}.
            
            \begin{definition}{Uniform Curvature Constant}{uniformCurvatureConstant}
                The \emph{uniform curvature constant} of the objective function $f$ on the uniform extension domain $E$ is defined as
                \begin{equation*}
                    C \,=\, \sup\hspace{-0.1em}\left\{\hspace{-0.1em}\frac{2\,}{\,\lambda^{\hspace{-0.05em}2}} (f(x + \lambda(s-x)) - \hspace{-0.05em}f(x) - \lambda\langle \hspace{0.05em} s-x \,|\, \nabla\hspace{-0.1em} f(x) \hspace{0.05em}\rangle)\hspace{0.1em} \,\colon x, s \in E, \lambda \in [\hspace{0.025em}0,1]\right\}\hspace{-0.025em}.
                \end{equation*}
            \end{definition}

            \noindent
            In general, the uniform curvature constant $C$ need not be finite.
            However, as already discussed in \cite{jaggi2013revisiting}, if the objective function $f$ is smooth, that is, if $\nabla\hspace{-0.1em}f$ is Lipschitz continuous with Lipschitz constant $L_\nabla \geq 0$, then it holds
            \begin{equation}\label{EQ:smoothnessConnection}
                C \,\leq\, \diam(E)^2L_\nabla\hspace{0.1em}.
            \end{equation}
            This inequality is a cornerstone of most Frank\hspace{0.05em}-Wolfe method convergence analyses, since it allows all nonlinear error terms to be controlled uniformly over the problem domain, independent of the iterates. 
            In our stochastic setting this plays an even more important role, because the feasible region or objective approximations may change sample\hspace{0.1em}-\hspace{0.05em}wise.

            \begin{remark}{}{adaptedCurvatureConstant}
                If the objective function $f$ is known, then Definition \ref{DEF:uniformCurvatureConstant} can be adapted to the setting of Remark \ref{REM:weakerAssumptionKnownObjective}.
                In this case, the uniform curvature constant $C$ has to be extended to the $\Sigma$\hspace{0.1em}-\hspace{0.05em}measurable \emph{random curvature constant} $\Cb \colon \Omega \to \RR \cup \{\infty\}$ defined by 
                \begin{equation*}
                    \Cb(\omega) \,=\, \sup\hspace{-0.1em}\left\{\hspace{-0.1em}\ind_{\hspace{-0.035em} N_0}\hspace{-0.1em}(\omega)\frac{2\,}{\,\lambda^{\hspace{-0.05em}2}} (f(x + \lambda(s-x)) - f(x) - \lambda\langle\hspace{0.05em} s-x \,|\, \nabla\hspace{-0.1em} f(x) \hspace{0.05em}\rangle)\,\colon x, s \in \Eb(\omega), \lambda \in (0,1]\hspace{-0.05em}\right\}
                \end{equation*}
                for all $\omega \in \Omega$.
                Here, the use of $(\hspace{-0.025em}0, 1]$ instead of $[\hspace{0.025em}0, 1]$ as in Definition \ref{DEF:uniformCurvatureConstant} has pure measurability reasons but does not affect any of the following results.
                Similar to \eqref{EQ:smoothnessConnection}, we obtain that
                \begin{equation*}
                    \Cb \,\leq\, \diam(\Eb)^{\hspace{-0.025em}2}L_\nabla
                \end{equation*}
                if the objective function $f$ is smooth.
                The random curvature constant plays the same structural role sample\hspace{0.1em}-wise as $C$ does deterministically.
                A detailed proof of this remark can be found in Appendix \ref{SEC:B.2}.
            \end{remark}

            \paragraph{Suboptimality Gap.} We consider a quantity originally used for certifying the closeness of $f(x)$ and $\Popt$ for some given point $x \in D$, called the suboptimality gap.
            Given a point $x \in D$ it is defined as 
            \begin{equation*}
                g(x) \,=\, \max\hspace{0.1em}\{\hspace{0.025em}\langle \hspace{0.05em}x-s \,|\, \nabla \hspace{-0.1em} f(x)\hspace{0.05em}\rangle \,\colon s \in D\hspace{0.025em}\}
            \end{equation*}
            and is of particular interest since it can be computed on the fly in the Frank\hspace{0.1em}-Wolfe algorithm as can be seen writing
            \begin{equation*}
                g(x) \,=\, \max\hspace{0.1em}\{\langle \hspace{0.05em} x-s \,|\, \nabla \hspace{-0.1em} f(x)\hspace{0.05em}\rangle \,\colon s \in D\} \,=\, \langle \hspace{0.05em}x \,|\, \nabla\hspace{-0.1em} f(x)\hspace{0.05em}\rangle - \min\hspace{0.1em}\{\langle \hspace{0.05em}s \,|\,  \nabla\hspace{-0.1em} f(x)\hspace{0.05em}\rangle \,\colon s \in D\}\hspace{0.1em},
            \end{equation*}
            since the second term is exactly the optimal value of the linear subproblem.
            Furthermore, the suboptimality gap satisfies $g(x) \geq f(x) - \Popt \geq 0$ such that it acts as a natural upper bound for the error of any point $x \in D$. 
            This tool is widely used in the Frank\hspace{0.1em}-Wolfe literature as a computable certificate~\cite{jaggi2013revisiting, lacoste2015global, freund2016new} and, in the following, we consider an abstract version that is defined on the iterates of Algorithm \ref{ALG:onlineAdaptiveStochasticFrankWolfeAlgorithm}.

            \begin{definition}{Suboptimality Gap Process}{suboptimalityGap}
                The \emph{suboptimality gap process} is defined as the stochastic process
                \begin{equation*}
                    \Gb \colon \NN \times \Omega \to \RR, \; \Gb\hspace{-0.05em}(n, \omega) \,\mapsto\, \ind_{\hspace{-0.05em} N}\hspace{-0.035em}(\omega)\max \hspace{0.1em}\{\langle \hspace{0.05em}\Xb_{\hspace{0.025em}n}\hspace{-0.05em}(\omega) - s \,|\, \nabla \hspace{-0.1em}f(\hspace{0.035em}\Xb_{\hspace{0.025em}n}\hspace{-0.05em}(\omega)) \hspace{0.05em}\rangle \,\colon s \in D\hspace{0.025em}\}\hspace{0.1em}.
                \end{equation*}
            \end{definition}

            \noindent
            It follows from standard measurability results \cite{rockafellar1998variational, castaing1977convex} that the suboptimality gap process forms a well\hspace{0.05em}-\hspace{0.035em}defined stochastic process.
            Similar to above, we can reformulate
            \begin{equation*}
                \Gb_{\hspace{-0.015em}n} \,=\, \ind_{\hspace{-0.05em}N}\langle \hspace{0.05em}\Xb_{\hspace{0.025em}n} \,|\, \nabla\hspace{-0.1em} f(\hspace{0.035em}\Xb_{\hspace{0.025em}n}\hspace{-0.05em})\hspace{0.05em}\rangle - \ind_{\hspace{-0.05em}N}\hspace{-0.05em}\min\hspace{0.1em}\{\langle \hspace{0.05em}s \,|\,  \nabla\hspace{-0.1em} f(\hspace{0.035em}\Xb_{\hspace{0.025em}n}\hspace{-0.05em})\hspace{0.05em}\rangle \,\colon s \in D\}\hspace{0.1em},
            \end{equation*}
            such that on $N$ there is not much of a difference to the deterministic case. 
            That the suboptimality gap process can be used as an upper bound as well is due to the following result.

            \begin{lemma}{}{suboptimalityGapIsUpperBound}
                Let $\varepsilon \in [\hspace{0.025em}0, 1]$ be as in Assumption \ref{ASS:uniformDomainExtension}.
                Then, it holds
                \begin{equation}\label{EQ:suboptimalityGapUpperBound}
                    \Gb_{\hspace{-0.015em}n} \,\geq\, f(\hspace{0.035em}\Xb_{\hspace{0.025em}n}\hspace{-0.05em}) - \Popt
                \end{equation}
                on $N$ for all $n \in \NN_0$.
                In particular, \eqref{EQ:suboptimalityGapUpperBound} holds with probability at least $1-\varepsilon$.
            \end{lemma}
            \begin{proof}[\textcolor{seeblau}{Proof}]
                Let $n \in \NN_0$ be arbitrary and $\omega \in N$ be a fixed sample point.
                The convexity of the objective function $f$ on $E$ implies that
                \begin{equation*}
                    f(s) - f(\hspace{0.035em}\Xb_{\hspace{0.025em}n}\hspace{-0.05em}(\omega)) \,\geq\, \langle\hspace{0.05em} s -\Xb_{\hspace{0.025em}n}\hspace{-0.05em}(\omega) \,|\, \nabla\hspace{-0.1em}f(\hspace{0.035em}\Xb_{\hspace{0.025em}n}\hspace{-0.05em}(\omega))\hspace{0.05em}\rangle 
                \end{equation*}
                for all $s \in D \subseteq E$, such that taking the minimum over $D$ on both sides yields
                \begin{equation*}
                    \begin{aligned}
                        \Popt \hspace{-0.1em}- \hspace{0.1em}f(\hspace{0.035em}\Xb_{\hspace{0.025em}n}\hspace{-0.05em}(\omega))
                        & \,\geq\, \min\hspace{0.1em}\{\langle\hspace{0.05em}s-\Xb_{\hspace{0.025em}n}\hspace{-0.05em}(\omega) \,|\, \nabla\hspace{-0.1em} f(\hspace{0.035em}\Xb_{\hspace{0.025em}n}\hspace{-0.05em}(\omega)) \hspace{0.05em}\rangle \,\colon s \in D\hspace{0.025em}\} \\
                        &\,=\, -\max\hspace{0.1em}\{\langle \hspace{0.05em}\Xb_{\hspace{0.025em}n}\hspace{-0.05em}(\omega)-s \,|\, \nabla\hspace{-0.1em} f(\hspace{0.035em}\Xb_{\hspace{0.025em}n}\hspace{-0.05em}(\omega))\hspace{0.05em}\rangle \,\colon s \in D\hspace{0.025em}\} \,=\, -\hspace{0.05em}\Gb_{\hspace{-0.015em}n}\hspace{-0.05em}(\omega)\hspace{0.1em}.
                    \end{aligned}
                \end{equation*}
                The second part now follows since $\PP[\hspace{-0.025em}N\hspace{0.025em}] \geq 1 - \varepsilon$.
            \end{proof}
            
            \paragraph{Lipschitz Constant.} Above, we already commented on the Lipschitz continuity of the derivative $\nabla\hspace{-0.1em}f$ of the objective function $f$.
            However, in this work we have to also consider the Lipschitz continuity of the objective function $f$ itself. 
            In the usual analysis of the Frank\hspace{0.1em}-Wolfe algorithm this is rather uncommon, however, in stochastic settings, Lipschitz continuity of $f$ is important to control value fluctuation between the random iterates.
            Luckily, this comes without further restrictions.

            \begin{lemma}{}{lipschitzConstant}
                The objective function $f$ is Lipschitz continuous on any compact subset of $\dom(f)$.
                In particular, $f$ is Lipschitz continuous on the uniform extension domain $E$ with Lipschitz constant 
                \begin{equation*}
                    L \,=\, \max\hspace{0.1em}\{\hspace{0.05em}\lVert \hspace{0.05em}\nabla\hspace{-0.1em} f(x) \hspace{0.05em}\rVert \hspace{0.05em}\colon x \in E\hspace{0.025em}\}\hspace{0.1em}.
                \end{equation*}
            \end{lemma}

            \begin{remark}{}{adaptedLipschitzConstant}
                If the objective function $f$ is known, then Lemma \ref{LEM:lipschitzConstant} can be adapted to the setting of Remark \ref{REM:weakerAssumptionKnownObjective}. 
                In this case, we have to define the $\Sigma$\hspace{0.1em}-\hspace{0.05em}measurable random Lipschitz constant
                \begin{equation*}
                    \Lb \colon \Omega \to \RR, \; \omega \,\mapsto\,  \max\hspace{0.1em}\{\ind_{\hspace{-0.05em}N_0}\hspace{-0.05em}(\omega)\hspace{0.05em}\lVert\hspace{0.05em} \nabla\hspace{-0.1em} f(x) \hspace{0.05em}\rVert \hspace{0.05em}\colon x \in \Eb(\omega)\}\hspace{0.1em}.
                \end{equation*}
                In this way, for all $\omega \in N_0$ and $x, y \in \Eb(\omega)$ it holds that
                \begin{equation*}
                    \lvert \hspace{0.05em }f(x)-f(y) \hspace{0.05em}\rvert \,\leq\, \Lb(\omega) \hspace{0.05em}\lVert \hspace{0.05em}x-y\hspace{0.05em}\rVert\hspace{0.1em},
                \end{equation*}
                that is, $f$ is Lipschitz continuous on $\Eb$ with probability at least $1-\varepsilon_0$. 
                A detailed proof of this remark is provided in Appendix \ref{SEC:B.2}.
            \end{remark}

        \subsection{Preliminaries and Intermediate Results}

            In this section we will provide some intermediate results on upper and lower iteration and distance bounds. 
            At first, we consider a general result on deriving upper bounds for a sequence that obeys a recursive structure. 
            A slightly simpler version of this result can be found in the proof of \cite[Theorem 1]{jaggi2013revisiting}, and will serve as the main tool to convert our descent inequalities into explicit finite\hspace{0.1em}-\hspace{0.035em}time convergence rates later on.

            \begin{proposition}{}{generalInductionBound}
                Let $(T_n\hspace{-0.035em})_{n \in \NN_0} \subseteq \RR$ be a sequence. 
                Let $r \in [\hspace{0.025em}0, 1]$, $A_1, A_2 > 0$ be some fixed constants and let $(\lambda_{\hspace{0.025em}n}\hspace{-0.035em})_{n \in \NN_0}$ be our usual step\hspace{0.075em}-\hspace{0.05em}size sequence.
                If there exists some $m \in \NN_0$ with
                \begin{equation}\label{EQ:upperBoundProof}
                    T_{n+\hspace{-0.05em}1} \,\leq\, (1-\lambda_{\hspace{0.025em}n})\hspace{0.05em}T_n + A_1\hspace{-0.05em}\lambda_{\hspace{0.025em}n}^{\hspace{-0.1em}1+r} + A_2\lambda_{\hspace{0.025em}n}\vspace{0.25em}
                \end{equation}
                for all $n \geq m$, then
                \begin{equation}\label{EQ:upperBoundProofSecond}
                    T_n \,\leq\, (m + 3)\max\hspace{0.1em}\{\hspace{0.05em}\lvert \hspace{0.05em}T_m\hspace{0.05em}\rvert, A_1\}\hspace{-0.05em}\lambda_{\hspace{0.025em}n}^{\hspace{-0.05em} r} + A_2
                \end{equation}
                for all $n \geq m + 1$. 
                If $m = 0$, then \eqref{EQ:upperBoundProofSecond} can even be simplified to 
                $T_n \,\leq\, 2A_1\hspace{-0.05em}\lambda_{\hspace{0.025em}n}^{\hspace{-0.05em} r} + A_2$ for all $n \in \NN$.
            \end{proposition}

            \paragraph{Bounding Iteration Error Differences.} As a first small result, we consider an adaptation of \cite[Lemma~5]{jaggi2013revisiting} to our stochastic and approximate setting. 
            This result implicitly shows that we can bound the iteration error difference $f(\hspace{0.035em}\Xb_{\hspace{0.025em}n+\hspace{-0.05em}1}\hspace{-0.05em}) - f(\hspace{0.035em}\Xb_{\hspace{0.025em}n}\hspace{-0.05em})$ from above for all $n \in \NN_0$ with high probability, where the additional error terms precisely quantify the effect of domain approximation and objective approximation.

            \begin{lemma}{}{upperBoundIterationDifference}
                Let $\lambda \in (\hspace{-0.025em}0,1]$ be a step\hspace{0.1em}-\hspace{0.05em}size and $\varepsilon \in [\hspace{0.025em}0, 1]$ as in Assumption \ref{ASS:uniformDomainExtension}. 
                Then, it holds
                \begin{equation}\label{EQ:upperBoundIterationDifference}
                    f(\hspace{0.035em}\Xb_{\hspace{0.025em}n} \hspace{-0.05em}+ \lambda\hspace{0.025em}(\Sb_{n+\hspace{-0.05em}1}\hspace{-0.05em}-\Xb_{\hspace{0.025em}n}\hspace{-0.05em})) -f(\hspace{0.035em}\Xb_{\hspace{0.025em}n}\hspace{-0.05em}) \,\leq\, -\hspace{0.05em}\Gb_{\hspace{-0.015em}n}\lambda + (d_H(\hspace{0.035em}\Db_{\hspace{-0.025em}n+\hspace{-0.05em}1}, D)L + 2\hspace{0.05em}\Delta\hspace{0.05em}\lVert \hspace{0.05em}\Fb_{\hspace{-0.075em}n+\hspace{-0.05em}1} - f \hspace{0.025em}\rVert_\infty)\lambda+\frac{C}{2}\lambda^{\hspace{-0.05em}2}
                \end{equation}
                on $N$ for all $n \in \NN_0$. 
                In particular, \eqref{EQ:upperBoundIterationDifference} holds with probability at least $1-\varepsilon$. 
            \end{lemma}
            \begin{proof}[\textcolor{seeblau}{Proof}]
                Let $n \in \NN_0$ be fixed and $\omega \in N$ be a sample point. 
                For simplicity, we use the notation \eqref{EQ:simplicticNotation} for the rest of the proof.
                By definition of the uniform curvature constant $C$ and since we know that $x_n, s_{n+\hspace{-0.05em}1} \in E$ we have
                \begin{equation}\label{EQ:curvatureInequality}
                    f(x_n + \lambda(s_{n+\hspace{-0.05em}1}-x_n)) \,\leq\, f(x_n) + \lambda\hspace{0.05em}\langle\hspace{0.05em} s_{n+\hspace{-0.05em}1}-x_n \,|\, \nabla \hspace{-0.1em}f(x_n)\hspace{0.05em}\rangle + \frac{C}{2}\lambda^{\hspace{-0.05em}2}\hspace{0.1em}.
                \end{equation}
                Using the Cauchy\hspace{0.1em}-\hspace{0.05em}Schwarz inequality we can see that
                \begin{equation}\label{EQ:innerProductBound}
                    \begin{aligned}
                         \lvert\langle\hspace{0.05em} s - x_n \,|\, ( \nabla\hspace{-0.1em} f_{n+\hspace{-0.05em}1} - \nabla\hspace{-0.1em} f\hspace{0.05em})(x_n)\hspace{0.05em}\rangle\rvert \,\leq\, \lVert \hspace{0.05em}s - x_n \hspace{0.05em}\rVert \, \lVert \hspace{0.05em}( \nabla\hspace{-0.1em} f_{n+\hspace{-0.05em}1} - \nabla\hspace{-0.1em} f\hspace{0.05em})(x_n) \hspace{0.05em}\rVert \,\leq\, \Delta \hspace{0.05em} \lVert \hspace{0.05em} f_{n+\hspace{-0.05em}1} - f \hspace{0.05em}\rVert_\infty\hspace{0.1em},
                    \end{aligned}
                \end{equation}
                for all $s \in E$ such that we can bound the inner product in \eqref{EQ:curvatureInequality} from above as
                \begin{align}\label{EQ:innerProductBoundOne}
                     \langle\hspace{0.05em} s_{n+\hspace{-0.05em}1} - x_n \,|\, \nabla\hspace{-0.1em} f(x_n)\hspace{0.05em}\rangle \,\leq\, \langle\hspace{0.05em} s_{n+\hspace{-0.05em}1} - x_n \,|\, \nabla\hspace{-0.1em} f_{n+\hspace{-0.05em}1}(x_n) \hspace{0.05em}\rangle + \Delta \hspace{0.05em} \lVert f_{n+\hspace{-0.05em}1} - f \hspace{0.05em}\rVert_\infty\hspace{0.1em}.
                \end{align}
                Furthermore, by definition of the subsolution process $\Sb$ and using \eqref{EQ:innerProductBound}, we can see that
                \begin{equation}\label{EQ:innerProductBound2}
                    \begin{aligned}
                        \langle\hspace{0.05em} s_{n+\hspace{-0.05em}1} - x_n \,|\, \nabla \hspace{-0.1em} f_{n+\hspace{-0.05em}1}(x_n) \hspace{0.05em}\rangle 
                        & \,=\, \min\hspace{0.1em}\hspace{-0.05em}\{\langle \hspace{0.05em} s - x_n \,|\, \nabla\hspace{-0.1em} f_{n+\hspace{-0.05em}1}(x_n) \hspace{0.05em}\rangle \,\colon s \in D_{n+\hspace{-0.05em}1}\} \\
                        & \,\leq\, \min\hspace{0.1em}\{\langle \hspace{0.05em}s - x_n \,|\, \nabla\hspace{-0.1em}f(x_n) \hspace{0.05em}\rangle \,\colon s \in D_{n+\hspace{-0.05em}1}\} + \Delta \hspace{0.05em} \lVert \hspace{0.05em} f_{n+\hspace{-0.05em}1} - f \hspace{0.05em}\rVert_\infty\hspace{0.1em},
                    \end{aligned}
                \end{equation}
                such that we can focus on reformulating the last minimization term.
                Revisiting representation \eqref{EQ:hausdorffFormulation} of the Hausdorff distance, we know that $D \subseteq D_{n+\hspace{-0.05em}1} \hspace{-0.1em} + \BB(0, d_H(D_{n+\hspace{-0.05em}1}, D) )$ for any $n \in \NN$, such that 
                \begin{equation}\label{EQ:minimizationProblemBound}
                    \begin{aligned}
                         -\hspace{0.05em}\Gb_{\hspace{-0.015em}n}\hspace{-0.05em}(\omega)& \,=\, \min\hspace{0.1em}\{\langle \hspace{0.05em }s - x_n \,|\, \nabla\hspace{-0.1em} f(x_n) \hspace{0.05em}\rangle\,\colon s \in D\} \\
                        & \,\geq\, \min\hspace{0.1em}\{\langle \hspace{0.05em} s -x_n \,|\, \nabla \hspace{-0.1em} f(x_n) \hspace{0.05em}\rangle\,\colon s\in D_{n+\hspace{-0.05em}1} \hspace{-0.1em} + \BB(0, d_H(D_{n+\hspace{-0.05em}1}, D))\} \\
                        & \,=\, \min\hspace{0.1em}\{\langle \hspace{0.05em}s - x_n \,|\, \nabla\hspace{-0.1em}f(x_n) \hspace{0.05em}\rangle\,\colon s\in D_{n+\hspace{-0.05em}1}\} - d_H(D_{n+\hspace{-0.05em}1}, D)\hspace{0.05em}\lVert \hspace{0.05em}\nabla \hspace{-0.1em} f(x_n) \hspace{0.05em} \rVert\hspace{0.1em}.
                    \end{aligned}
                \end{equation}
                Combining \eqref{EQ:curvatureInequality} and \eqref{EQ:innerProductBoundOne}\hspace{0.1em}-\hspace{0.05em}\eqref{EQ:minimizationProblemBound} we then obtain
                \begin{equation*}
                    f(x_n + \lambda(s_{n+\hspace{-0.05em}1}-x_n))\hspace{0.05em}-f(x_n) \,\leq\, -\hspace{0.1em}\Gb_{\hspace{-0.015em}n}\hspace{-0.05em}(\omega)\lambda + \Ab_{\hspace{0.015em}n}\hspace{-0.05em}(\omega)\lambda + \frac{C}{2}\lambda^{\hspace{-0.1em}2}\hspace{0.1em},
                \end{equation*}
                where
                \begin{equation*}
                    \Ab_{\hspace{0.015em}n}\hspace{-0.05em}(\omega) \,=\, d_H(D_{n+\hspace{-0.05em}1}, D)L + 2\hspace{0.05em}\Delta\hspace{0.05em}\lVert \hspace{0.05em}f_{n+\hspace{-0.05em}1} - f \hspace{0.05em} \rVert_\infty \,\geq\, d_H(D_{n+\hspace{-0.05em}1}, D) \hspace{0.05em}\lVert \hspace{0.05em}\nabla\hspace{-0.1em}f(x_n)\hspace{0.05em}\Vert + 2\hspace{0.05em}\Delta\hspace{0.05em}\lVert \hspace{0.05em}f_{n+\hspace{-0.05em}1} - f \hspace{0.05em}\rVert_\infty\hspace{0.1em}.
                \end{equation*}
                The second claim follows since $\PP[\hspace{-0.025em}N\hspace{0.025em}] \geq 1-\varepsilon$.
            \end{proof}

            \begin{remark}{}{iterationDistanceUpperBoundKnownObjective}
                If the objective function $f$ is known, then Lemma \ref{LEM:upperBoundIterationDifference} can be adapted to the setting of the previous remarks.
                By omitting the quantities that depend on the objective approximation process $\Fb$ and replacing the quantities dealt with in the remarks, we obtain that
                \begin{equation}\label{EQ:upperBoundIterationDifferenceAdapted}
                    f(\hspace{0.035em}\Xb_{\hspace{0.025em}n} \hspace{-0.05em}+ \lambda(\Sb_{n+\hspace{-0.05em}1}\hspace{-0.05em}-\Xb_{\hspace{0.025em}n}\hspace{-0.05em})) -f(\hspace{0.035em}\Xb_{\hspace{0.025em}n}\hspace{-0.05em}) \leq -\hspace{0.05em}\Gb_{\hspace{-0.015em}n}\lambda + d_H(\hspace{0.035em}\Db_{n+\hspace{-0.05em}1}, D)\hspace{0.025em}\Lb\hspace{0.025em}\lambda+\frac{\Cb}{2}\lambda^{\hspace{-0.1em}2}
                \end{equation}
                on $N_0$ for all $n \in \NN_0$. 
                Here, \eqref{EQ:upperBoundIterationDifferenceAdapted} holds with probability at least $1-\varepsilon_0$. 
            \end{remark}

            \paragraph{Bounding Iteration Errors from Below.} In the deterministic analysis of the Frank\hspace{0.1em}-Wolfe algorithm, one can use that $f(x) \geq \Popt$ for all $x \in D$ to canonically bound the error $f(x) -\Popt$ from below by zero.
            However, in our domain adaptive setting it may occur that $f(\hspace{0.035em}\Xb_{\hspace{0.025em}n}\hspace{-0.05em}) < \Popt$ on a whole event $A \subseteq \Omega$ with $\PP(A) > 0$ as it can happen that $\Xb_{\hspace{0.025em}n} \in E \setminus D$ on $A$ and since $\Popt$ is not necessarily a global minimum on the entire uniform extension domain. 
            To tackle this problem we have to also find a lower bound on the iteration error $f(\hspace{0.035em}\Xb_{\hspace{0.025em}n}\hspace{-0.05em})-\Popt$ that holds with an at least high probability.
            To obtain this lower bound we strongly rely on the convexity of the approximation domains. Recall that $h$ is the metric induced by the inner product on $H$ and that $h(x, D)$ describes the distance of a point $x \in H$ to the problem domain $D$.

            \begin{lemma}{}{convexDistance}
                Let $K \subseteq H$ be nonempty and convex. 
                Then, the map $K \to\, \RR, \, x \,\mapsto\, h(x, D)$ is convex. 
            \end{lemma}

            \begin{lemma}{}{lowerBoundError}
                Let $\varepsilon \in [\hspace{0.025em}0, 1]$ be as in Assumption \ref{ASS:uniformDomainExtension}.
                Then, it holds that
                \begin{equation}\label{EQ:lowerBoundError}
                    f(\hspace{0.035em}\Xb_{\hspace{0.025em}n}\hspace{-0.05em}) - \Popt \,\geq\, -h(\hspace{0.035em}\Xb_{\hspace{0.025em}n}, D)L
                \end{equation}
                on $N$ for all $n \in \NN_0$. 
                In particular, \eqref{EQ:lowerBoundError} holds with probability at least $1-\varepsilon$.
            \end{lemma}
            \begin{proof}[\textcolor{seeblau}{Proof}]
                Let $n \in \NN_0$ and $\omega \in N$ be a sample point.
                Since $D$ is compact, we can find $y \in D$ satisfying 
                \begin{equation*}
                    h(\hspace{0.035em}\Xb_{\hspace{0.025em}n}\hspace{-0.05em}(\omega), D) \,=\, h(\hspace{0.035em}\Xb_{\hspace{0.025em}n}\hspace{-0.05em}(\omega), y) \,=\, \lVert \hspace{0.05em}\Xb_{\hspace{0.025em}n}\hspace{-0.05em}(\omega) - y \hspace{0.05em}\rVert\hspace{0.1em}.
                \end{equation*}
               Since both $\Xb_{\hspace{0.025em}n}\hspace{-0.05em}(\omega), y \in E$ and $E$ is nonempty and convex, we can use Lemma \ref{LEM:lipschitzConstant} to obtain 
                \[
                    \lvert \hspace{0.025em}f(\hspace{0.035em}\Xb_{\hspace{0.025em}n}\hspace{-0.05em}(\omega)) - f(y) \hspace{0.025em}\rvert \,\leq\, L\hspace{0.05em}\lVert \hspace{0.05em}\Xb_{\hspace{0.025em}n}\hspace{-0.05em}(\omega) - y \hspace{0.05em}\rVert \,=\, h(\hspace{0.035em}\Xb_{\hspace{0.025em}n}\hspace{-0.05em}(\omega), D)L\hspace{0.1em},
                \]
                such that we find
                \begin{align*}
                    f(\hspace{0.035em}\Xb_{\hspace{0.025em}n}\hspace{-0.05em}(\omega)) - \Popt \,\geq\, f(\hspace{0.035em}\Xb_{\hspace{0.025em}n}\hspace{-0.05em}(\omega)) - f(y) \,\geq \, -\hspace{0.025em}h(\hspace{0.035em}\Xb_{\hspace{0.025em}n}\hspace{-0.05em}(\omega), D)L\hspace{0.1em},
                \end{align*}
                where we used that $f(y) \geq \Popt$ since $y \in D$.
            \end{proof}

        \subsection{Results on Convergence of Algorithm \ref*{ALG:onlineAdaptiveStochasticFrankWolfeAlgorithm}}\label{SEC:theoreticalResults}

            We can now state the first convergence result of Algorithm \ref{ALG:onlineAdaptiveStochasticFrankWolfeAlgorithm}, which is of pure asymptotic nature. 
            Later in this section, we will even derive results on the convergence rate of Algorithm \ref{ALG:onlineAdaptiveStochasticFrankWolfeAlgorithm} under some additional assumptions.
            
            \begin{theorem}{Asymptotic Convergence of Algorithm \ref*{ALG:onlineAdaptiveStochasticFrankWolfeAlgorithm}}{convergenceWithAdditionalAssumption}
                Let $\Xb$ be the iterate process generated by running Algorithm \ref{ALG:onlineAdaptiveStochasticFrankWolfeAlgorithm} with starting random variable $\Xb_0$, domain approximation process $\Db$ and objective approximation process $\Fb$.
                Let $\beta_{\hspace{0.025em}1}, \beta_{\hspace{0.025em}2}, \varepsilon \in [\hspace{0.025em}0, 1]$ and $\eta_{\hspace{0.025em} 1}, \eta_{\hspace{0.05em} 2} \geq 0$ as in Assumptions \ref{ASS:hausdorffConvergence}, \ref{ASS:uniformDomainExtension}, and  \ref{ASS:functionConvergence}.
                Then, it holds
                \begin{equation*}
                    \PP\hspace{-0.2em}\left[-\eta_{\hspace{0.025em} 1}\hspace{-0.05em}L \,\leq\, \liminf_{n \hspace{0.05em}\to \hspace{0.05em} \infty} \, f(\hspace{0.035em}\Xb_{\hspace{0.025em}n}\hspace{-0.05em}) - \Popt \,\leq\, \limsup_{n \hspace{0.05em}\to \hspace{0.05em} \infty} \, f(\hspace{0.035em}\Xb_{\hspace{0.025em}n}\hspace{-0.05em}) - \Popt\,\leq\, \eta_{\hspace{0.025em} 1}\hspace{-0.1em}L + 2\hspace{0.05em} \eta_{\hspace{0.05em} 2}\Delta\right] \,\geq\, 1-(\beta_{\hspace{0.025em}1}+\beta_{\hspace{0.025em}2}+\varepsilon)\hspace{0.1em}.
                \end{equation*}
            \end{theorem}
            \begin{proof}[\textcolor{seeblau}{Proof}]
                First, denote the event
                \begin{equation*}
                    M \,\coloneqq\, \left\{\limsup_{n \hspace{0.05em}\to\hspace{0.05em} \infty} \hspace{0.15em} d_H(\hspace{0.035em}\Db_{\hspace{0.025em}n}\hspace{0.05em}, D) \,\leq\, \eta_{\hspace{0.025em} 1}\right\} \,\cap\, \left\{\limsup_{n \hspace{0.05em}\to\hspace{0.05em} \infty} \hspace{0.2em}\lVert \hspace{0.05em}\Fb_{\hspace{-0.075em}n} - f \hspace{0.05em}\rVert \,\leq\, \eta_{\hspace{0.05em} 2}\right\}
                \end{equation*}
                for the sample points corresponding to Assumption \ref{ASS:hausdorffConvergence} and Assumption \ref{ASS:functionConvergence} and let $\omega \in M \cap N$ be fixed.
                For simplicity, we use the notation \eqref{EQ:simplicticNotation} for the rest of the proof.
                We denote $e_n = f(x_n) - \Popt$ for the approximation error at iteration $n \in \NN_0$.
                By Lemma \ref{LEM:upperBoundIterationDifference}, we know that
                \begin{equation*}
                    e_{n+\hspace{-0.05em}1} \,\leq\, e_n -\Gb_{\hspace{-0.015em}n}\hspace{-0.05em}(\omega)\lambda_{\hspace{0.025em}n} + (d_H(D_{n+\hspace{-0.05em}1}, D)L + 2\hspace{0.05em}\Delta\lVert \hspace{0.05em}\nabla \hspace{-0.1em} f_{n+\hspace{-0.05em}1} - \nabla \hspace{-0.1em} f \hspace{0.05em}\rVert_\infty)\lambda_{\hspace{0.025em}n}+\frac{C}{2}\lambda_{\hspace{0.025em}n}^{\hspace{-0.1em}2}
                \end{equation*}
                for all $n \in \NN_0$, such that using Lemma \ref{LEM:suboptimalityGapIsUpperBound} this implies
                \begin{equation}\label{EQ:THM1.1}
                    e_{n+\hspace{-0.05em}1} \,\leq\, (1-\lambda_{\hspace{0.025em}n}\hspace{-0.05em})\hspace{0.05em}e_n+(d_H(D_{n+\hspace{-0.05em}1}, D)L + 2\hspace{0.05em}\Delta\lVert\hspace{0.05em} \nabla \hspace{
                    -0.1em}f_{n+\hspace{-0.05em}1} - \nabla \hspace{-0.1em} f \hspace{0.05em}\rVert_\infty)\lambda_{\hspace{0.025em}n}+\frac{C}{2}\lambda_{\hspace{0.025em}n}^{\hspace{-0.1em}2}\hspace{0.1em}.
                \end{equation}
                By choice of $\omega$, we obtain that for all $k \in \NN$ there exists some $m \in \NN$ satisfying
                \begin{equation}\label{EQ:THM1.2}
                    d_H(D_{n+\hspace{-0.05em}1}, D) \,\leq\, \eta_{\hspace{0.025em} 1} + \frac{1}{k} \quad \text{ and } \quad  \lVert \hspace{0.05em}\nabla\hspace{-0.1em} f_{n+\hspace{-0.05em}1} - \nabla \hspace{-0.1em} f\hspace{0.05em} \rVert_\infty \,\leq\, \eta_{\hspace{0.05em} 2} + \frac{1}{k}
                \end{equation}
                for all $n \geq m$.
                Hence, combining \eqref{EQ:THM1.1} and \eqref{EQ:THM1.2} we have
                \begin{equation}\label{EQ:THM1.3}
                    e_{n+\hspace{-0.05em}1} \,\leq\, (1-\lambda_{\hspace{0.025em}n}\hspace{-0.05em})\hspace{0.05em} e_n+\left(\hspace{-0.1em}\left(\hspace{-0.1em}\eta_{\hspace{0.025em} 1} + \frac{1}{k}\right)\hspace{-0.15em}\hspace{-0.05em} L + 2\hspace{0.05em}\Delta\hspace{-0.15em}\left(\hspace{-0.1em}\eta_{\hspace{0.05em} 2} + \frac{1}{k}\right)\hspace{-0.1em}\right)\hspace{-0.2em}\lambda_{\hspace{0.025em}n} + \frac{C}{2}\lambda_{\hspace{0.025em}n}^{\hspace{-0.1em}2}
                \end{equation}
                for all $n \geq m$, such that we can apply Proposition \ref{PRO:generalInductionBound} to \eqref{EQ:THM1.3} to obtain that
                \begin{equation*}
                   e_{n} \,\leq\, (m+3)\max\left\{\hspace{-0.1em}\lvert \hspace{0.05em}e_{m}\hspace{0.05em}\rvert, \frac{C}{2}\right\}\hspace{-0.05em}\lambda_{\hspace{0.025em}n} + \left(\hspace{-0.1em}\eta_{\hspace{0.025em} 1} + \frac{1}{k}\right)\hspace{-0.2em} L + 2\hspace{0.05em}\Delta\hspace{-0.2em}\left(\eta_{\hspace{0.05em} 2} + \frac{1}{k}\right)
                \end{equation*}
                for all $n \geq m+1$. 
                Thus, overall we have 
                \begin{align}\label{EQ:THM1.4}
                    \limsup_{n \hspace{0.05em}\to \hspace{0.05em} \infty} \, e_n \,\leq\, \left(\hspace{-0.1em}\eta_{\hspace{0.025em} 1} + \frac{1}{k}\right)\hspace{-0.2em}L + 2\hspace{0.05em}\Delta\hspace{-0.2em}\left(\hspace{-0.1em}\eta_{\hspace{0.05em} 2} + \frac{1}{k}\right) \,=\, \eta_{\hspace{0.025em} 1}\hspace{-0.05em}L + 2\hspace{0.05em}\Delta\hspace{0.05em} \eta_{\hspace{0.05em} 2} + \frac{1}{k}(L + 2\Delta)
                \end{align}
                for all $k \in \NN$.
                Since \eqref{EQ:THM1.4} holds for all $k \in \NN$, we eventually obtain that the limes superior of $e_n$ as $n \to \infty$ is bounded above by $\eta_{\hspace{0.05em} 1}\hspace{-0.05em}L + 2\hspace{0.05em}\Delta\hspace{0.05em}\eta_{\hspace{0.05em} 2}$. 
                Analogously to \eqref{EQ:THM1.1}, since $E$ is nonempty and convex, by Lemma \ref{LEM:convexDistance} we know that
                \begin{equation*}
                    \begin{aligned}
                        h(x_{n+\hspace{-0.05em}1}, D) \,\leq\, (1-\lambda_{\hspace{0.025em}n}\hspace{-0.05em})\hspace{0.05em} h(x_n, D) + \lambda_{\hspace{0.025em}n} \hspace{0.05em} h(s_{n+\hspace{-0.05em}1}, D) \,\leq\, (1-\lambda_{\hspace{0.025em}n}\hspace{-0.05em})\hspace{0.05em} h(x_n, D) + \lambda_{\hspace{0.025em}n} \hspace{0.05em} d_H(D_{n+\hspace{-0.05em}1}, D)\hspace{0.1em},
                    \end{aligned}
                \end{equation*}
                where in the last inequality we used that $s_{n+\hspace{-0.05em}1} \in D_{n+\hspace{-0.05em}1}$ for all $n \in \NN_0$.
                Therefore, similar to \eqref{EQ:THM1.3}, we have 
                \begin{equation*}
                    h(x_{n+\hspace{-0.05em}1}, D) \,\leq\, (1-\lambda_{\hspace{0.025em}n}\hspace{0.05em})\hspace{0.05em} h(x_n, D) +\left(\hspace{-0.1em}\eta_{\hspace{0.025em} 1} + \frac{1}{k}\right)\hspace{-0.2em}\lambda_{\hspace{0.025em}n}
                \end{equation*}
                for all $n \geq m$, such that by applying Proposition \ref{PRO:generalInductionBound} and considering the limes superior this yields
                \begin{equation}\label{EQ:THM1.6}
                    \limsup_{n \hspace{0.05em}\to \hspace{0.05em}\infty} \, h(x_n, D) \,\leq\, \eta_{\hspace{0.025em} 1} + \frac{1}{k}
                \end{equation}
                for all $k \in \NN$.
                As before, since \eqref{EQ:THM1.6} holds for all $k \in \NN$, we obtain that the limes superior of $h(x_n, D)$ as $n \to \infty$ is bounded above by $\eta_{\hspace{0.025em} 1}$. 
                Hence, combining \eqref{EQ:THM1.6} and Lemma \ref{LEM:lowerBoundError}, we find
                \begin{equation*}
                    \liminf_{n \hspace{0.05em}\to \hspace{0.05em}\infty} \, e_n\,\geq\, \liminf_{n \hspace{0.05em}\to \hspace{0.05em} \infty}\, -h(x_n, D)L  \,=\, -\limsup_{n \hspace{0.05em}\to \hspace{0.05em}\infty} \,h(x_n, D) L\,\geq\, -\hspace{0.05em}\eta_{\hspace{0.05em} 1}L\hspace{0.1em},
                 \end{equation*}
                such that the claim follows since $\PP[M \cap N\hspace{0.05em}] \geq 1-(\beta_{\hspace{0.025em}1} + \beta_{\hspace{0.025em}2} + \varepsilon)$.
            \end{proof}

            \begin{remark}{}{adaptedAsymptoticConvergence}
                If the objective function $f$ is known, then Theorem \ref{THM:convergenceWithAdditionalAssumption} can be adapted to the setting of the previous remarks.
                By omitting the quantities that depend on the domain approximation process $\Fb$ and replacing the quantities dealt with in the remarks, we obtain that
                \begin{equation*}
                    \PP\hspace{-0.2em}\left[\hspace{0.1em}\limsup_{n \hspace{0.05em} \to \hspace{0.05em} \infty} \,\, \lvert\hspace{0.05em} f(\hspace{0.035em}\Xb_{\hspace{0.025em}n}\hspace{-0.05em})-\Popt \rvert \,\leq\, \eta_{\hspace{0.05em}1}\hspace{-0.05em}\Lb\hspace{0.05em}\right] \,\geq\, 1-(\beta_{\hspace{0.025em}1}+\varepsilon_0)\hspace{0.1em}.
                \end{equation*}
            \end{remark}

            \paragraph{Convergence Rates.} While Theorem \ref{THM:convergenceWithAdditionalAssumption} provides an asymptotic convergence result for the approximation error of $f(\hspace{0.035em}\Xb_{\hspace{0.025em}n}\hspace{-0.05em})-\Popt$ as $n \to \infty$, there remains a lack of control on the actual convergence rate. 
            To obtain results involving the convergence rate or to give explicit approximation error bounds depending on the number of iterations in Algorithm \ref{ALG:onlineAdaptiveStochasticFrankWolfeAlgorithm} (or its deterministic counterpart Algorithm \ref{ALG:onlineAdaptiveFrankWolfeAlgorithm}), we have to impose stronger assumptions on the convergence rate of the domain approximation process $\Db$ and the objective approximation process $\Fb$. 
            More precisely, we consider a more restrictive version of Assumption \ref{ASS:hausdorffConvergence} and Assumption \ref{ASS:functionConvergence} to obtain this control.

            \begin{assumption}{}{hausdorffConvergenceRateAndFunctionConvergenceRate}
                 The domain approximation process $\Db$ and objective approximation process $\Fb$ are such that there exist constants $\beta_{\hspace{0.025em}1}, \beta_{\hspace{0.025em}2} \in [\hspace{0.025em}0, 1], r_1, r_2 \in (0, 1], \eta_{\hspace{0.025em} 1}, \eta_{\hspace{0.05em} 2} \geq 0$ and $c_1, c_{\hspace{0.025em}2} > 0$ satisfying 
                 \begin{equation*}
                     \PP[\hspace{0.05em} d_H(\hspace{0.035em}\Db_{n+\hspace{-0.05em}1}, D) \,\leq\, c_1\lambda_{\hspace{0.025em}n}^{\hspace{-0.05em} r_1} + \eta_{\hspace{0.025em} 1} \text{ for all } n \in \NN\hspace{0.05em}] \,\geq\, 1-\beta_{\hspace{0.025em}1}
                 \end{equation*}
                and 
                \begin{equation*}
                    \PP[\hspace{0.05em} \lVert \hspace{0.05em}\nabla\Fb_{\hspace{-0.075em}n+\hspace{-0.05em}1} - \nabla \hspace{-0.1em}f\hspace{0.05em} \hspace{0.05em}\rVert_\infty \,\leq\, c_{\hspace{0.025em}2}\hspace{0.025em}\lambda_{\hspace{0.025em}n}^{\hspace{-0.05em} r_2} + \eta_{\hspace{0.05em} 2} \text{ for all } \in \NN\hspace{0.05em}] \,\geq\, 1-\beta_{\hspace{0.025em}2}\hspace{0.1em}.
                \end{equation*}
            \end{assumption}

            \noindent
            Note that Assumption \ref{ASS:hausdorffConvergenceRateAndFunctionConvergenceRate} indeed implies both Assumption \ref{ASS:hausdorffConvergence} and Assumption \ref{ASS:functionConvergence} since we have
            \begin{equation*}
                \limsup\limits_{n \hspace{0.05em}\to\hspace{0.05em} \infty}\, d_H(\hspace{0.035em}\Db_{n+\hspace{-0.05em}1}, D) \,\leq\, \limsup\limits_{n \hspace{0.05em}\to\hspace{0.05em} \infty} \, c_1\lambda_{\hspace{0.025em}n}^{\hspace{-0.05em} r_1} + \eta_{\hspace{0.05em} 1} \,=\, \eta_{\hspace{0.05em} 1}
            \end{equation*}
            with probability at least $1-\beta_{\hspace{0.025em}1}$ and 
            \begin{equation*}
                \limsup\limits_{n \hspace{0.05em}\to\hspace{0.05em} \infty}\, \lVert \hspace{0.05em}\nabla \Fb_{\hspace{-0.075em}n+\hspace{-0.05em}1} - \nabla\hspace{-0.1em} f\hspace{0.05em}\rVert_\infty \,\leq\, \limsup\limits_{n \hspace{0.05em}\to\hspace{0.05em} \infty} \hspace{0.2em} c_{\hspace{0.025em}2}\lambda_{\hspace{0.025em}n}^{\hspace{-0.05em} r_2} + \eta_{\hspace{0.05em} 2} \,=\, \eta_{\hspace{0.05em} 2}
            \end{equation*}
            with probability at least $1 - \beta_{\hspace{0.025em} 2}$. 
            Therefore, as mentioned before, the constants $\beta_{\hspace{0.025em}1}$ and $\beta_{\hspace{0.025em} 2}$ represent approximation failure tolerances and $\eta_{\hspace{0.025em} 1}$ and $\eta_{\hspace{0.05em}2}$ represent the approximation qualities which can be obtained in the limit of the approximation processes $\Db$ and $\Fb$ with a probability of at least $1-\beta_{\hspace{0.025em}1}$ and $1- \beta_{\hspace{0.025em} 2}$, respectively.
            The additional constants $r_1$ and $r_2$ are the corresponding convergence rates of the approximation processes (for example, $r_1 = r_2 = 0.5$). 
            The scaling parameters $c_1$ and $c_{\hspace{0.025em}2}$ can be used to influence the probabilities $1-\beta_{\hspace{0.025em}1}$ and $1-\beta_{\hspace{0.025em} 2}$, for instance via concentration inequalities \cite{boucheron2013concentration}.
            Using this new assumption, we can now improve Theorem \ref{THM:convergenceWithAdditionalAssumption}. 
            However, we first introduce an event $M \subseteq \Omega$ of sample points satisfying Assumption~\ref{ASS:hausdorffConvergenceRateAndFunctionConvergenceRate}. 
            To this end, for all $n \in \NN$ we define the event
            \begin{equation*}
                M_n \,\coloneqq\, \left\{\hspace{-0.035em}d_H(\hspace{0.035em}\Db_{\hspace{-0.025em}n+\hspace{-0.05em}1}, D) \,\leq\, c_1\lambda_{\hspace{0.025em}n}^{\hspace{-0.05em} r_1} + \eta_{\hspace{0.025em} 1}\right\} \,\cap\, \left\{\lVert \hspace{0.05em}\nabla\Fb_{\hspace{-0.075em}n+\hspace{-0.05em}1} - \nabla \hspace{-0.1em}f\hspace{0.05em} \hspace{0.05em}\rVert_\infty \,\leq\, c_{\hspace{0.025em}2}\hspace{0.025em}\lambda_{\hspace{0.025em}n}^{\hspace{-0.05em} r_2} + \eta_{\hspace{0.05em} 2}\right\}
            \end{equation*}
            and set 
            \begin{equation*}
                M \,=\, \bigcap \hspace{0.2em}\{M_n \,\colon n \in \NN\}\hspace{0.1em}.
            \end{equation*}
            In this way we have $\PP[\hspace{-0.025em}M\hspace{0.025em}] \geq 1-(\beta_{\hspace{0.025em}1} + \beta_{\hspace{0.025em}2})$.

            \begin{theorem}{Convergence Rate for Algorithm 2}{errorBoundWithAdditionalAssumption}
                Let $\Xb$ be the iterate process generated by running Algorithm \ref{ALG:onlineAdaptiveStochasticFrankWolfeAlgorithm} with starting random variable $\Xb_0$, domain approximation process $\Db$ and objective approximation process $\Fb$.
                Let $\varepsilon, \beta_{\hspace{0.025em}1}, \beta_{\hspace{0.025em}2} \in [\hspace{0.025em}0, 1]$, $r_1, r_2 \in (0, 1]$, $\eta_{\hspace{0.05em}1}, \eta_{\hspace{0.05em}2} \geq 0$ and  $c_1, c_{\hspace{0.025em}2} > 0$ as in Assumption \ref{ASS:uniformDomainExtension} and Assumption \ref{ASS:hausdorffConvergenceRateAndFunctionConvergenceRate}. 
                Then, there exist constants $A, B \geq  0$, such that with $r = \min\{r_1, r_2\}$ it holds
                \begin{equation*}
                    \PP\hspace{-0.15em}\left[\hspace{0.05em}-A\lambda_{\hspace{0.025em}n}^{\hspace{-0.05em}r_1} - \eta_{\hspace{0.025em}1}\hspace{-0.05em}L \,\leq\,  f(\hspace{0.035em}\Xb_{\hspace{0.025em}n}\hspace{-0.05em}) - \Popt  \leq B\lambda_{\hspace{0.025em}n}^{\hspace{-0.05em}r} + \eta_{\hspace{0.025em} 1}\hspace{-0.05em}L + 2\hspace{0.1em}\eta_{\hspace{0.05em} 2}\hspace{0.025em}\Delta \text{ for all } n \in \NN\hspace{0.05em}\right] \,\geq\, 1-\left(\beta_{\hspace{0.025em}1} + \beta_{\hspace{0.025em}2} + \varepsilon\right)\hspace{0.075em}.
                \end{equation*}
            \end{theorem}
            \begin{proof}[\textcolor{seeblau}{Proof}]
                This proof is a natural extension of the proof of Theorem       \ref{THM:convergenceWithAdditionalAssumption} based on \cite[Theorem 1]{jaggi2013revisiting}. 
                Let $\omega \in M \cap N$ be a fixed sample point corresponding to Assumption \ref{ASS:uniformDomainExtension} and Assumption \ref{ASS:hausdorffConvergenceRateAndFunctionConvergenceRate}. 
                For simplicity, we use the notation \eqref{EQ:simplicticNotation} for the rest of the proof. 
                We denote $e_n = f(x_n) - \Popt$ for the approximation error at iteration $n \in \NN_0$. 
                Analogously to the proof of Theorem \ref{THM:convergenceWithAdditionalAssumption}, we obtain that
                \begin{equation}\label{EQ:THM2.1}
                    e_{n+\hspace{-0.05em}1} \,\leq\, (1-\lambda_{\hspace{0.025em}n}\hspace{-0.05em})\hspace{0.05em}e_n+(d_H(D_{n+\hspace{-0.05em}1}, D)L + 2\hspace{0.05em}\Delta\lVert \hspace{0.05em}\nabla \hspace{-0.1em} f_{n+\hspace{0.05em}1} - \nabla \hspace{-0.1em} f \hspace{0.05em}\rVert_\infty)\lambda_{\hspace{0.025em}n} + \frac{C}{2}\lambda_{\hspace{0.025em}n}^{\hspace{-0.1em}2}
                \end{equation}
                and by choice of the sample point $\omega$ we have
                \begin{equation}\label{EQ:THM2.2}
                    d_H(\hspace{-0.05em}D_{n+\hspace{-0.05em}1}, D) \,\leq\, c_1\hspace{-0.05em}\lambda_{\hspace{0.025em}n}^{\hspace{-0.05em} r_1} + \eta_{\hspace{0.025em} 1} \quad \text{ and } \quad  \lVert \hspace{0.05em}\nabla\hspace{-0.1em} f_{n+\hspace{-0.05em}1} - \nabla \hspace{-0.1em} f \hspace{0.05em}\rVert_\infty \,\leq\, c_{\hspace{0.025em}2}\lambda_{\hspace{0.025em}n}^{\hspace{-0.05em} r_2} + \eta_{\hspace{0.05em}2} 
                \end{equation}
                for all $n \in \NN_0$. 
                Hence, combining \eqref{EQ:THM2.1} and \eqref{EQ:THM2.2} we have that
                \begin{equation}\label{EQ:THM2.3}
                    e_{n+\hspace{-0.05em}1} \,\leq\, (1-\lambda_{\hspace{0.025em}n}\hspace{-0.05em})\hspace{0.05em} e_n + c_1\hspace{-0.075em}L\hspace{0.05em}\lambda_{\hspace{0.025em}n}^{\hspace{-0.1em}1+\hspace{0.05em}r_1} + 2\hspace{0.05em} c_{\hspace{0.025em}2}\hspace{0.025em}\Delta\hspace{0.05em}\lambda_{\hspace{0.025em}n}^{\hspace{-0.1em}1+r_2} + (\eta_{\hspace{0.025em} 1}\hspace{-0.05em}L + 2\hspace{0.05em}\eta_{\hspace{0.05em} 2}\Delta)\lambda_{\hspace{0.025em}n}  + \frac{C}{2}\lambda_{\hspace{0.025em}n}^{\hspace{-0.1em}2}
                \end{equation}
                for all $n \in \NN_0$, such that we can apply Proposition \ref{PRO:generalInductionBound} to \eqref{EQ:THM2.3} to obtain that
                \begin{equation*}
                   e_{n} \,\leq\, B\lambda_{\hspace{0.025em}n}^{\hspace{-0.05em}\min\{r_1, r_2\hspace{-0.05em}\}} + \eta_{\hspace{0.025em} 1}\hspace{-0.05em}L + 2\hspace{0.05em}\eta_{\hspace{0.05em} 2}\Delta
                \end{equation*}
                for all $n \in \NN$, where $B = 2\hspace{0.05em}c_1\hspace{-0.05em}L + 4\hspace{0.05em} c_{\hspace{0.025em}2}\Delta + C$.
                To bound $e_n$ from below for all $n \in \NN$, we can proceed as in the proof of Theorem \ref{THM:convergenceWithAdditionalAssumption} to obtain that
                \begin{equation*}
                    h(x_{n+\hspace{-0.05em}1}, D) \,\leq\, (1-\lambda_{\hspace{0.025em}n}\hspace{-0.05em})\hspace{0.05em} h(x_n, D) + c_1\hspace{-0.05em}\lambda_{\hspace{0.025em}n}^{\hspace{-0.1em}1 + \hspace{0.05em}r_1} + \eta_{\hspace{0.025em} 1}\hspace{-0.05em}\lambda_{\hspace{0.025em}n}
                \end{equation*}
                for all $n \in \NN_0$, such that by applying Proposition \ref{PRO:generalInductionBound} this yields
                \begin{equation}\label{EQ:THM2.5}
                    h(x_n, D) \,\leq\, 2\hspace{0.05em}c_1\hspace{-0.05em}\lambda_{\hspace{0.025em}n}^{\hspace{-0.05em} r_1} + \eta_{\hspace{0.025em} 1}
                \end{equation}
                for all $n \in \NN$.
                Using Lemma \ref{LEM:lowerBoundError} and \eqref{EQ:THM2.5} we then obtain
                \begin{equation*}
                    e_n \,\geq\, -\hspace{0.05em}h(x_n, D)L \,\geq\, -A\lambda_{\hspace{0.025em}n}^{\hspace{-0.05em} r_1} - \eta_{\hspace{0.025em} 1}\hspace{-0.05em}L
                \end{equation*}
                for all $n \in \NN$, where $A =  2\hspace{0.05em}c_1\hspace{-0.05em}L$.
                The claim now follows since $\PP[M \cap N] \geq 1 - (\beta_{\hspace{0.025em}1} + \beta_{\hspace{0.025em} 2} + \varepsilon)$.
            \end{proof}

            \begin{remark}{}{adaptedConvergenceRate}
                If the objective function $f$ is known, then Theorem \ref{THM:errorBoundWithAdditionalAssumption} can be adapted to the setting of the previous remarks.
                By omitting the quantities that depend on the objective approximation process $\Fb$ and replacing the quantities dealt with in the remarks, we obtain that
                \begin{equation*}
                    \PP\hspace{-0.15em}\left[\hspace{0.1em}\lvert \hspace{0.05em}f(\hspace{0.035em}\Xb_{\hspace{0.025em}n}\hspace{-0.05em})-\Popt\rvert \,\leq\, \Ab\lambda_{\hspace{0.025em}n}^{\hspace{-0.05em}r_1} + \eta_{\hspace{0.025em}1}\hspace{-0.05em}\Lb \text{ for all } n \in \NN\hspace{0.1em}\right] \,\geq\, 1-(\beta_{\hspace{0.025em}1} + \varepsilon_0)\hspace{0.1em},
                \end{equation*}
                where $\Ab \,=\, 2\hspace{0.05em}c_1\hspace{-0.05em}\Lb + \Cb \geq 0$ now is a random variable. 
                Here, the resulting convergence rate is governed solely by the speed at which the domain approximation process $\Db$ approaches the problem domain $D$.
            \end{remark}

        \subsection{Comments on the Assumptions}

            \noindent
            We now briefly comment on the assumptions on the domain approximation process $\Db$ of this section.
            It is immediately clear that if the problem domain $D$ is given as a closed ball in $H$, that is, $D = \BB(y, \rho)$ for some center point $y \in H$ and radius $\rho \geq 0$, then a natural way to obtain domain approximations for each $n \in \NN$ is to consider the sets $D_n = \BB(y_n, \rho_n\hspace{-0.05em})$, where the sequences $(y_n\hspace{-0.05em})_{n \in \NN} \subseteq H$ and $(\rho_n\hspace{-0.05em})_{n \in \NN} \subseteq \RR_+$ approximate $y$ and $\rho$, respectively.
            However, this example is yet quite restrictive and can easily be extended to a larger collection of problem domains. 
            
            \paragraph{Domain Approximations via Common Sublevel Sets.} To generate a sequence of nonempty, compact, and convex domain approximations, our goal is to mimic the structure of the closed balls in $H$, centered around a fixed point, as common sublevel set of functions with suitable properties.

            \begin{definition}{Carathéodory Function}{}
                Let $(S, \Sigma_S\hspace{-0.05em})$ be a measurable space and $X$ be a topological space.
                A function $g \colon S \times X \to \RR$ is called \emph{Carathéodory function} if both
                \begin{itemize}
                    \item the function $s \mapsto g(s, x)$ is $\Sigma_S$\hspace{0.05em}-\hspace{0.05em}measurable for all $x \in X$,

                    \item the function $x \mapsto g(s, x)$ is continuous for all $s \in S$.
                \end{itemize}
                
            \end{definition}

            \noindent
            More specifically, fixing some $A \subseteq H$ closed, we want to consider Carathéodory functions $g \colon A \times A \to \RR$ additionally satisfying
            \begin{equation}\label{EQ:caratheodoryProperties}
                g(y, y) \,\leq\, 0 \qquad  \text{and} \qquad t \,\mapsto\, g(y, t) \text{ is convex}
            \end{equation}
            for all $y \in A$.
            Taking $n \in \NN$ such Carathéodory functions $g_{\hspace{0.015em}1}, \ldots, g_{\hspace{0.025em}n}$ and assuming that further $x \mapsto g_{\hspace{0.025em}1}\hspace{-0.05em}(y, x)$ is coercive, that is,
            \begin{equation}\label{EQ:upperBoundFunction}
                g_{\hspace{0.025em}1}\hspace{-0.05em}(y, x) \,\to\, \infty \quad \text{ as } \quad \lVert \hspace{0.05em}x \hspace{0.05em}\rVert \,\to\, \infty\hspace{0.1em},
            \end{equation}
            fixing a center point $y \in A$ and a radius $\rho \geq 0$ we can define the common sublevel set
            \begin{equation*}
                \SS(y, \rho) \,\coloneqq\, \{x \in A \,\colon g_{\hspace{0.025em}1}\hspace{-0.05em}(y, x) \leq \rho, \ldots, g_{\hspace{0.025em}n}\hspace{-0.05em}(y, x) \leq \rho\} \,\subseteq\, H\hspace{0.1em}.
            \end{equation*}
            By construction, we now have that the common sublevel set $\SS(y, \rho)$ is nonempty, compact, and convex and can be used as a baseline for domain approximation sequences as described above. 
            A detailed proof of the next result can be found in Appendix \ref{SEC:B.2}.

            \begin{lemma}{}{generalizedBall}
                Let $n \in \NN$ and fix Carathéodory functions $g_{\hspace{0.025em}1}, \ldots, g_n$ satisfying \eqref{EQ:caratheodoryProperties} and \eqref{EQ:upperBoundFunction}.
                Let $y \in H$ and $\rho \geq 0$. 
                Then, the corresponding common sublevel set $\SS(y, \rho)$ is nonempty, compact, and convex. 
                Furthermore, for $\Sigma$\hspace{0.075em}-\hspace{0.05em}measurable functions $\Yb \colon \Omega \to H$ and $\Pb \colon \Omega \to \RR_+$ the map
                \begin{equation}\label{EQ:measurableSublevelMap}
                    \Omega \to \Dcc, \; \omega \,\mapsto\, \SS(\Yb(\omega), \Pb(\omega))
                \end{equation}
                is well\hspace{0.05em}-\hspace{0.035em}defined and $\Sigma$\hspace{0.1em}-\hspace{0.05em}measurable.
            \end{lemma}

            \noindent
            Similar to above, Lemma \ref{LEM:generalizedBall} now yields an easy way to construct a domain approximation process if the problem domain $D$ is given as the common sublevel set of certain Carathéodory functions, that is, $D = \SS(y, \rho)$ for some center point $y \in H$ and radius $\rho \geq 0$. 
            Indeed, if we have access to a stochastic processes $\Yb$ and $\Pb$ approximating $y$ and $\rho$, respectively, in some probabilistic sense, then the process $\Db$ defined by $\Db_{\hspace{-0.025em}n} = \SS(\Yb_{\hspace{-0.1em}n}, \Pb_{\hspace{-0.1em}n})$ for all $n \in \NN$ forms a valid domain approximation process. 
            Note that the convergence behavior of this process $\Db$ is now directly influenced by the convergence behavior of the processes $\Yb$ and $\Pb$, but also depended on the explicit Carathéodory functions used.
            Therefore, in the general case, it is not possible to derive good upper bounds on the Hausdorff distance without further knowledge of the Carathéodory functions. 
            Hence, Lemma \ref{LEM:generalizedBall} only yields a way to determine general feasible domain approximations but it remains to check for the approximation quality in every case separately, typically by exploiting additional structure such as Lipschitz bounds in the first argument or explicit control of the level sets.

            \begin{remark}{}{exampleSublevelSetAndDistanceBound}
                The easiest example of a Carathéodory function satisfying \eqref{EQ:caratheodoryProperties} and \eqref{EQ:upperBoundFunction} is the induced metric of \emph{any} norm on $H$.
                In contrast to the previous paragraph, taking $n \in \NN$ norms on $H$ and considering their induced metrics $g_i$ for $i \in [n]$, we can find explicit upper bounds on the Hausdorff distance of the common sublevel sets. 
                Taking two center points $y_1, y_2 \in H$ and radii $\rho_1, \rho_{\hspace{0.025em}2} \geq 0$ we find that
                \begin{equation*}
                    d_H(\hspace{0.05em}\SS(y_1, \rho_1\hspace{-0.05em}), \SS(y_2, \rho_2)) \,\leq\, \frac{M}{m}\hspace{0.05em}\lVert\hspace{0.05em} y_1 - y_2 \rVert + \frac{1}{m}\hspace{0.05em}\lvert \hspace{0.05em}\rho_1 - \rho_{\hspace{0.025em}2}\rvert\hspace{0.1em},
                \end{equation*}
                where $0 < m < M$ are such that $m\hspace{0.05em}\lVert s-x\rVert \leq g_i\hspace{0.025em}(s, x) \leq M \lVert s-x \rVert$ for all $s, x \in H$ and $i \in [n]$.
                A proof of this remark can be found in Appendix \ref{SEC:B.2}.
            \end{remark}

    \section{Convergence Acceleration}\label{SEC:acceleratedConvergenceResult}

        A natural question arising in the Frank\hspace{0.1em}-Wolfe algorithm framework is whether additional assumptions on the objective function $f$ or on the problem domain $D$ can provide accelerated convergence results.
        Similar to some deterministic settings with open\hspace{0.05em}-\hspace{0.025em}loop step\hspace{0.075em}-\hspace{0.05em}sizes \cite{wirth2023acceleration, wirth2025accelerated} or line\hspace{0.1em}-\hspace{0.05em}search step\hspace{0.075em}-\hspace{0.05em}sizes \cite{garber2015faster, kerdreux2021projection}, this question can be answered positively and in a structurally similar manner.
        One key assumption in the following will be that the objective function $f$ is \emph{known}, such that our setting is slightly less general than in the previous section.
        In particular, we will find ourselves in the setting of the previous remarks, where we replace $E$ by $\Eb$, $N$ by $N_0$, $\varepsilon$ by $\varepsilon_0$, $C$ by $\Cb$, and $L$ by $\Lb$. 
        First, in the case of outer domain approximations, we will derive only improved upper convergence rates, before turning to settings with inner domain approximations and strong convexity, where genuinely accelerated rates can be obtained. 
        Lastly, we comment on how such outer and inner domain approximations may be constructed in practice.

        \subsection{Outer Domain Approximation}
        
            A relatively simple way to obtain an acceleration of Algorithm \ref{ALG:onlineAdaptiveStochasticFrankWolfeAlgorithm} is to assume that the domain approximation process $\Db$ converges towards the domain $D$ from the outside.

            \begin{assumption}{}{domainContainment}
                The domain approximation process $\Db$ is such that there exists a constant $\delta \in [\hspace{0.025em}0, 1]$ satisfying
                \begin{equation*}
                    \PP\hspace{-0.2em}\left[D \,\subseteq\, \bigcap\hspace{0.2em} \{\Db_{\hspace{-0.025em}n}\,\colon n \in \NN\hspace{0.025em}\}\right] \,\geq\, 1-\delta\hspace{0.1em}.
                \end{equation*}
            \end{assumption}
        
            \noindent
            First, similar to the event $N_0^+$ in Remark \ref{REM:weakerAssumptionKnownObjective}, we can see that the set
            \begin{equation}\label{EQ:exteriorMeasurable}
                \left\{\hspace{-0.1em}D \subseteq \bigcap\hspace{0.2em} \{\Db_{\hspace{-0.025em}n} \,\colon n \in \NN\hspace{0.025em}\}\hspace{-0.1em}\right\} \,=\, \bigcap\hspace{0.2em}\{\{\hspace{-0.05em}D \subseteq \Db_{\hspace{-0.025em}n}\hspace{-0.05em}\}\,\colon n \in \NN\hspace{0.05em}\}
            \end{equation}
            is $\Sigma$\hspace{0.1em}-\hspace{0.05em}measurable, such that Assumption \ref{ASS:domainContainment} is well\hspace{0.05em}-\hspace{0.025em}posed.
            In other words, Assumption \ref{ASS:domainContainment} tells us that for each $n \in \NN$ we can guarantee that $D \subseteq \Db_{\hspace{0.025em}n}$ with high probability (for example, $1-\delta = 0.95$). 
            The reason why this can help with accelerated convergence to the optimal value $\Popt$ from above is due to the fact that we can omit the Hausdorff distance term $d_H(\hspace{0.035em}\Db_{\hspace{-0.025em}n+\hspace{-0.05em}1}, D)\hspace{0.05em}\Lb\hspace{0.05em}\lambda$ in \eqref{EQ:upperBoundIterationDifference}, such that it does not matter whether we approximate the problem domain $D$ fast or at all. 
            However, convergence to the optimal value $\Popt$ from below is in general not improved by Assumption \ref{ASS:domainContainment}. 
            In fact, for many cases it will even be of negative effect since in practice outer domain approximations potentially slow down Hausdorff convergence and inflate the underlying curvature bounds. 

            \begin{theorem}{Accelerated Upper Convergence for Algorithm 2}{outerAcceleratedConvergence}
                Let $\Xb$ be the iterate process generated by running Algorithm \ref{ALG:onlineAdaptiveStochasticFrankWolfeAlgorithm} with starting random variable $\Xb_0$ and domain approximation process $\Db$. 
                Let $\beta_{\hspace{0.025em}1}, \varepsilon_0, \delta \in [\hspace{0.025em}0, 1], r_1 \in (\hspace{-0.025em}0, 1]$, $\eta_{\hspace{0.025em}1} \geq 0$ and $c_1 > 0$ be as in Remark~\ref{REM:setAdaptation}, Assumption~\ref{ASS:hausdorffConvergenceRateAndFunctionConvergenceRate} and Assumption~\ref{ASS:domainContainment}. 
                Then, there exists a random variable $\Ab \geq 0$, such that
                \begin{equation*}
                    \PP\hspace{-0.15em}\left[\hspace{0.05em}-\Ab\lambda_{\hspace{0.025em}n}^{\hspace{-0.05em} r_1} - \eta_{\hspace{0.025em}1}\hspace{-0.05em}\Lb \,\leq\,  f(\hspace{0.035em}\Xb_{\hspace{0.025em}n}\hspace{-0.05em}) - \Popt  \,\leq\, \Cb\lambda_{\hspace{0.025em}n} \,\text{ for all }\, n \in \NN \hspace{0.1em}\right] \,\geq\, 1-(\beta_{\hspace{0.025em}1} + \varepsilon_0 + \delta)\hspace{0.1em}.
                \end{equation*}
            \end{theorem}
            \begin{proof}[\textcolor{seeblau}{Proof.}]
                This proof is a simple adaptation of the proofs of Lemma \ref{LEM:upperBoundIterationDifference} and  Theorem \ref{THM:errorBoundWithAdditionalAssumption}, restricted to the setting of a known objective function $f$.
                Let 
                \begin{equation}\label{EQ:eventDefinitionSecond}
                    \omega \in \bigcap \left\{\hspace{-0.035em}D \subseteq\hspace{0.05em} \bigcap\, \{\Db_{\hspace{0.025em}n}\,\colon n \in \NN\}\right\} \,\cap\, M \,\cap\, N_0
                \end{equation}
                be a fixed sample point corresponding to Remark \ref{REM:setAdaptation}, Assumption \ref{ASS:hausdorffConvergenceRateAndFunctionConvergenceRate} and Assumption \ref{ASS:domainContainment}. 
                First, we follow the proof of Lemma~\ref{LEM:upperBoundIterationDifference} but additionally use that $D \subseteq \Db_{\hspace{-0.025em}n+\hspace{-0.05em}1}\hspace{-0.05em}(\omega)$ to obtain that 
                \begin{equation*}
                        - \hspace{0.05em}\Gb_{\hspace{-0.015em}n}\hspace{-0.05em}(\omega)
                            \,=\, \min\hspace{0.1em}\{\langle \hspace{0.05em }s - \Xb_{\hspace{0.025em}n}\hspace{-0.05em}(\omega) \,|\, \nabla\hspace{-0.1em} f(\hspace{0.035em}\Xb_{\hspace{0.025em}n}\hspace{-0.05em}(\omega)) \hspace{0.05em}\rangle\,\colon s \in D\} 
                            \,\geq\, \min\hspace{0.1em}\{\langle \hspace{0.05em} s -\Xb_{\hspace{0.025em}n}\hspace{-0.05em}(\omega) \,|\, \nabla\hspace{-0.1em} f(\hspace{0.035em}\Xb_{\hspace{0.025em}n}\hspace{-0.05em}(\omega))\hspace{0.05em}\rangle\,\colon s\in D_{n+\hspace{-0.05em}1}\}
                \end{equation*}
                instead of \eqref{EQ:minimizationProblemBound}.
                Hence, we can derive that
                \begin{equation}\label{EQ:iterationBoundAdapted}
                    f(\hspace{0.035em}\Xb_{\hspace{0.025em}n}\hspace{-0.05em}(\omega) \hspace{-0.05em}+ \lambda_{\hspace{0.025em}n}(\Sb_{n+\hspace{-0.05em}1}\hspace{-0.05em}(\omega)-\Xb_{\hspace{0.025em}n}\hspace{-0.05em}(\omega))) -f(\hspace{0.035em}\Xb_{\hspace{0.025em}n}\hspace{-0.05em}(\omega)) \,\leq\, -\hspace{0.05em}\Gb_{\hspace{-0.015em}n}\hspace{-0.05em}(\omega)\lambda_{\hspace{0.025em}n} +\frac{\Cb(\omega)}{2}\lambda_{\hspace{0.025em}n}^{\hspace{-0.075em}2}\hspace{0.01em}
                \end{equation}
                for all $n \in \NN_0$. 
                Then, inserting \eqref{EQ:iterationBoundAdapted} instead of Lemma \ref{LEM:upperBoundIterationDifference} in the proof of Theorem \ref{THM:errorBoundWithAdditionalAssumption}, we have
                \begin{equation*}
                    f(\hspace{0.035em}\Xb_{\hspace{0.025em}n}\hspace{-0.05em}(\omega)) - \Popt \,\leq\, (1-\lambda_{\hspace{0.025em}n}\hspace{-0.05em})\hspace{0.05em}(f(\hspace{0.035em}\Xb_{\hspace{0.025em}n}\hspace{-0.05em}(\omega))-\Popt)+ \frac{\Cb}{2}\lambda_{\hspace{0.025em}n}^{\hspace{-0.05em}2}
                \end{equation*}
                instead of \eqref{EQ:THM2.3}, such that applying Proposition \ref{PRO:generalInductionBound} yields
                \begin{equation*}
                    f(\hspace{0.035em}\Xb_{\hspace{0.025em}n}\hspace{-0.05em}(\omega)) - \Popt  \,\leq\, \Cb(\omega)\lambda_{\hspace{0.025em}n}
                \end{equation*}
                for all $n \in \NN$.
                The random variable $\Ab$ now can be chosen as in Theorem \ref{THM:errorBoundWithAdditionalAssumption} but replacing $L$ with $\Lb$, that is, $\Ab =  2\hspace{0.05em}c_1\hspace{-0.05em}\Lb$.
                The claim follows since the event in \eqref{EQ:eventDefinitionSecond} has probability at least $1-(\beta_{\hspace{00.025em}1} + \varepsilon_0 + \delta)$.
            \end{proof}

        \subsection{Inner Domain Approximation}

            As a counterpart to Assumption \ref{ASS:domainContainment}, we can also consider the case where the problem domain $D$ is not approximated from the outside, but rather from the inside.

            \begin{assumption}{}{domainContainmentExterior}
                The domain approximation process $\Db$ is such that there exists a constant $\delta \in [\hspace{0.025em}0,1]$ satisfying 
                \begin{equation*}
                    \PP\hspace{-0.2em}\left[\hspace{0.1em}\bigcup \hspace{0.2em} \{\Db_{\hspace{-0.025em}n} \,\colon n \in \NN\} \,\subseteq\, D\hspace{0.05em}\right] \,\geq\, 1-\delta\hspace{0.1em}.
                \end{equation*}
            \end{assumption}

            \noindent
            In this case, similar to \eqref{EQ:exteriorMeasurable}, we can see that the set 
            \begin{equation*}
                I \,\coloneqq\, \left\{\hspace{0.05em}\bigcup\hspace{0.2em}\{\hspace{0.05em}\Db_{\hspace{-0.025em}n} \,\colon n \in \NN\hspace{0.05em}\} \subseteq D\right\} \,=\, \bigcap\hspace{0.2em}\{\{\hspace{0.05em}\Db_{\hspace{-0.025em}n} \subseteq D\} \,\colon n \in \NN\hspace{0.05em}\}
            \end{equation*}
            is $\Sigma$\hspace{0.1em}-\hspace{0.05em}measurable, such that Assumption \ref{ASS:domainContainmentExterior} is well\hspace{0.05em}-\hspace{0.025em}posed.
            Assumption \ref{ASS:domainContainmentExterior} guarantees that with high probability (for example, $1-\delta = 0.95$) we have $\Db_{\hspace{-0.025em}n} \subseteq D$ for all $n \in \NN$.
            This property is of particular interest since it canonically bounds the approximation error $f(\hspace{0.035em}\Xb_{\hspace{0.025em}n}\hspace{-0.05em})-\Popt$ from below by zero for all $n \in \NN$, as in the deterministic Frank\hspace{0.1em}-Wolfe setting.
            Therefore, we can completely focus on finding accelerated upper bounds. 
            However, we already note that making sure that a domain approximation process $\Db$ satisfies Assumption \ref{ASS:domainContainmentExterior} is significantly harder than for Assumption \ref{ASS:domainContainment}, such that every stronger bound comes with a certain cost.
            Note that since
            \begin{equation*}
                \bigcup \hspace{0.2em} \{\hspace{0.05em}\Db_{\hspace{-0.025em}n}\hspace{-0.05em}(\omega) \,\colon n \in \NN\hspace{0.05em}\} \cup D \,\subseteq\, D
            \end{equation*}
            for all $\omega \in I$, we do not need an artificial random domain extension $\Eb$ on the event $I$.
            In particular, choosing $\Eb \equiv D$ we can replace $\varepsilon_0$ with $\delta$ from Assumption \ref{ASS:domainContainmentExterior}.
            As a consequence, we can replace the random curvature constant $\Cb$ with the classical Frank\hspace{0.1em}-Wolfe curvature constant
            \begin{equation*}
                C_{\hspace{-0.05em}f} \,=\, \sup\hspace{-0.1em}\left\{\hspace{-0.05em}\frac{2\,}{\,\lambda^{\hspace{-0.05em}2}} (f(x + \lambda(s-x)) - \hspace{-0.05em}f(x) - \lambda\langle \hspace{0.05em} s-x \,|\, \nabla\hspace{-0.1em} f(x) \hspace{0.05em}\rangle)\hspace{0.1em} \colon x, s \in D, \lambda \in [\hspace{0.025em}0,1]\right\}
            \end{equation*}
            as in \cite{jaggi2013revisiting} and the random Lipschitz constant $\Lb$ with the Lipschitz constant 
            \begin{equation*}
                L_{\hspace{-0.05em}f} \,\coloneqq\, \max\hspace{0.1em}\{\lVert \hspace{0.05em}\nabla\hspace{-0.1em}f(x)\hspace{0.05em}\rVert \,\colon x \in D\}\hspace{0.1em}.
            \end{equation*}
            Additionally, as done in many other convergence acceleration settings for the deterministic Frank\hspace{0.1em}-Wolfe algorithm \cite{wirth2023acceleration, garber2015faster, kerdreux2021projection}, we assume the objective function $f$ and the problem domain $D$ to be strongly convex.

            \begin{definition}{Strongly Convex Function}{stronglyConvexFunctions}
                Let $\mu > 0$. 
                A continuously differentiable function $g \colon \hspace{-0.2em}\dom(g) \to \RR$ is called $\mu$\hspace{0.1em}-\hspace{0.05em}strongly convex on a convex set $K \subseteq \dom(g)$ if it holds that
                \begin{equation*}
                    g(x) \,\geq\ g(y) + \langle \hspace{0.05em} x-y \,|\, \nabla\hspace{-0.05em} g(y) \hspace{0.05em} \rangle + \frac{\mu}{2}\lVert \hspace{0.05em}x-y\hspace{0.05em}\rVert^2
                \end{equation*}
                for all $x, y \in K$.
            \end{definition}

            \begin{definition}{Strongly Convex Set}{stronglyConvexSet}
                Let $\alpha > 0$. 
                A set $K \subseteq H$ is called $\alpha$\hspace{0.1em}-\hspace{0.05em}strongly convex if it holds 
                \begin{equation*}
                    \lambda \hspace{0.05em}x + (1-\lambda)\hspace{0.05em}y + \lambda(1-\lambda)\frac{\alpha}{2}\lVert \hspace{0.05em}x-y\hspace{0.05em}\rVert^2 z \in K
                \end{equation*}
                for all $x, y\in K$, $\lambda \in [\hspace{0.025em}0, 1]$ and $z \in H$ with $\lVert\hspace{0.05em} z \hspace{0.05em}\rVert = 1$.
            \end{definition}

            \noindent
            Since optimal points of strongly convex functions on convex sets are unique, in the following we denote by $x^\star \in D$ the optimal solution of problem \eqref{EQ:generalProblem}, that is, it holds $f(x^\star\hspace{-0.05em}) = \Popt$.
            The reason why strong convexity of both the objective function $f$ and the problem domain $D$ are beneficial for convergence acceleration is that instead of bounding 
            \begin{equation*}
                \langle \hspace{0.05em}\Sb_{n+\hspace{-0.05em}1} - \Xb_{\hspace{0.025em}n} \,|\, \nabla\hspace{-0.1em} f(\hspace{0.035em}\Xb_{\hspace{0.025em}n}\hspace{-0.05em}) \hspace{0.05em}\rangle \,\leq\, -\hspace{0.05em}\Gb_{\hspace{-0.015em}n} + d_H(\hspace{0.035em}\Db_{\hspace{-0.025em}n+\hspace{-0.05em}1}, D)L_{\hspace{-0.05em}f}
            \end{equation*}
            as done in the proof of Lemma \ref{LEM:upperBoundIterationDifference} (adapted to this restricted setting), we can derive sharper bounds.
            To this end, we consider the next two intermediate results.
            The first one is composed of parts from \cite{wirth2023acceleration} and the second one is strongly inspired by \cite[Lemma 1]{garber2015faster}, but was adapted to our setting.

            \begin{lemma}{}{improvedInnerProductBound1}
                Let $f$ be $\mu$\hspace{0.1em}-\hspace{0.05em}strongly convex on $D$ and $\delta \in [\hspace{0.025em}0,1]$ as in Assumption \ref{ASS:domainContainmentExterior}. 
                Then, it holds that
                \begin{equation}\label{EQ:upperBoundDistanceOptimalSolution}
                    \lVert \hspace{0.05em} \nabla\hspace{-0.1em}f(\hspace{0.035em}\Xb_{\hspace{0.025em}n}\hspace{-0.05em})\hspace{0.05em}\rVert^2 \,\geq\,\frac{\mu}{2}(f(\hspace{0.035em}\Xb_{\hspace{0.025em}n}\hspace{-0.05em}) - \Popt)
                \end{equation}
                on $I$ for all $n \in \NN_0$.
                In particular, \eqref{EQ:upperBoundDistanceOptimalSolution} holds with probability at least $1-\delta$.
            \end{lemma}
            \begin{proof}[\textcolor{seeblau}{Proof.}]
                Let $n \in \NN_0$ and $\omega \in I$ be a sample point. 
                For simplicity, we use the notation \eqref{EQ:simplicticNotation} for the rest of the proof.
                First, without loss of generality, we assume that $x_n \neq x^\star$ since otherwise the claim trivially holds.
                Since by Assumption \ref{ASS:domainContainmentExterior} we know that $x_n \in D$ using that $f$ is $\mu$\hspace{0.1em}-\hspace{0.05em}strongly convex on $D$ we obtain that 
                \begin{equation}\label{EQ:firstIntermediateResult}
                    f(x_n\hspace{-0.025em}) - \Popt \,\geq\, \langle \hspace{0.05em}x_n \hspace{-0.05em}- x^{\hspace{-0.05em}\star} \hspace{-0.05em}\,|\, \nabla\hspace{-0.1em}f(x^{\hspace{-0.05em}\star}\hspace{-0.05em})\hspace{0.05em}\rangle + \frac{\mu}{2}\lVert \hspace{0.05em}x_n\hspace{-0.05em}-x^{\hspace{-0.05em}\star}\rVert^2 \,\geq\, \frac{\mu}{2}\lVert \hspace{0.05em}x_n\hspace{-0.05em}-x^{\hspace{-0.05em}\star}\rVert^2,
                \end{equation}
                where we used that $\langle \hspace{0.05em}x_n\hspace{-0.05em} - x^{\hspace{-0.05em}\star} \hspace{-0.05em}\,|\, \nabla\hspace{-0.1em}f(x^\star\hspace{-0.05em})\hspace{0.05em}\rangle \geq 0$ due to the optimality of $x^{\hspace{-0.05em}\star}$ on $D$.
                Furthermore, we can see that
                \begin{equation}\label{EQ:secondIntermediateResult}
                    \lVert \hspace{0.05em}\nabla\hspace{-0.1em}f(x_n\hspace{-0.025em}) \hspace{0.05em} \rVert^2\hspace{0.05em}\lVert \hspace{0.05em} x_n\hspace{-0.05em} - x^{\star} \rVert^2 \,\geq\, \langle \hspace{0.05em} x_n\hspace{-0.05em} - x^{\hspace{-0.05em}\star}\hspace{-0.05em}\,|\, \nabla\hspace{-0.1em}f(x_n\hspace{-0.025em}) \hspace{0.05em}\rangle^2 \,\geq\, (f(x_n\hspace{-0.025em}) - \Popt)^2,
                \end{equation}
                where we used Cauchy\hspace{0.05em}-\hspace{0.05em}Schwarz in the first and the convexity of $f$ in the second inequality. 
                Dividing \eqref{EQ:secondIntermediateResult} by $\lVert \hspace{0.05em}x_n\hspace{-0.05em} - x^{\star}\rVert^2 > 0$ and using \eqref{EQ:firstIntermediateResult} then yields
                \begin{equation*}
                    \lVert \hspace{0.05em}\nabla\hspace{-0.1em}f(x_n) \hspace{0.05em} \rVert^2 \,\geq\, \frac{(f(x_n) - \Popt)^2}{\lVert \hspace{0.05em} x_n\hspace{-0.05em} - x^{\star} \hspace{0.05em} \rVert^2} \,\geq\, \frac{\mu}{2}(f(x_n) - \Popt)\hspace{0.1em}.
                \end{equation*}
                The second claim now follows since $\PP[I\hspace{0.075em}] \geq 1-\delta$.
            \end{proof}

            \begin{lemma}{}{improvedInnerProductBound2}
                Let $D$ be $\alpha$\hspace{0.1em}-\hspace{0.05em}strongly convex and $\delta \in [\hspace{0.025em}0,1]$ as in Assumption \ref{ASS:domainContainmentExterior}. 
                Then, it holds that
                \begin{equation}\label{EQ:upperBoundScalarProduct}
                    \langle \hspace{0.05em}\Sb_{n+\hspace{-0.05em}1} - \Xb_{\hspace{0.025em}n} \,|\, \nabla\hspace{-0.1em} f(\hspace{0.035em}\Xb_{\hspace{0.025em}n}\hspace{-0.05em}) \hspace{0.05em}\rangle \,\leq\, -\frac{1}{2}(f(\Xb_{\hspace{0.025em}n}\hspace{-0.025em})-\Popt) - \frac{\alpha}{8}\lVert \hspace{0.05em}\Sb_{n+\hspace{-0.05em}1} - \Xb_{\hspace{0.025em}n} \hspace{0.05em}\rVert^2 \hspace{0.05em}\lVert \hspace{0.05em}\nabla\hspace{-0.1em}f(\hspace{0.035em}\Xb_{\hspace{0.025em}n}\hspace{-0.05em}) \hspace{0.05em}\rVert + \frac{3}{2}\hspace{0.05em}d_H(\hspace{0.035em}\Db_{\hspace{-0.025em}n+\hspace{-0.05em}1}, D)L_{\hspace{-0.05em}f}
                \end{equation}
                on $I$ for all $n \in \NN_0$.
                In particular, \eqref{EQ:upperBoundScalarProduct} holds with probability at least $1- \delta$.
            \end{lemma}
            \begin{proof}[\textcolor{seeblau}{Proof.}]
                Let $n \in \NN_0$ and $\omega \in I$ be a sample point. 
                For simplicity, we use the notation \eqref{EQ:simplicticNotation} for the rest of the proof.
                First, since $x^\star \in D$, by definition of the subsolution process $\Sb$ and analogously to \eqref{EQ:minimizationProblemBound} we find 
                    \begin{equation}\label{EQ:lemmaExtendedNew}
                        \langle \hspace{0.05em}s_{n+\hspace{-0.05em}1}-\hspace{0.05em}x_n \,|\, \nabla\hspace{-0.1em}f(x_n\hspace{-0.025em})\hspace{0.05em}\rangle \,\leq\, \langle \hspace{0.05em}x^\star-\hspace{0.05em}x_n \,|\, \nabla\hspace{-0.1em}f(x_n\hspace{-0.025em})\hspace{0.05em}\rangle + d_H(D_{n+\hspace{-0.05em}1}, D)\hspace{0.05em}L_{\hspace{-0.05em}f} \,\leq\, \Popt - \hspace{0.05em}f(x_n\hspace{-0.025em}) + d_H(D_{n+\hspace{-0.05em}1}, D)\hspace{0.05em}L_{\hspace{-0.05em}f}\hspace{0.1em},
                    \end{equation}
                    where the last inequality follows from the convexity of $f$.
                    Note now that, without loss of generality, we can assume that $\lVert \hspace{0.05em}\nabla\hspace{-0.1em}f(x_n\hspace{-0.025em})\hspace{0.05em} \rVert > 0$ since otherwise the claim trivially holds.
                Since by Assumption \ref{ASS:domainContainmentExterior} we know that $x_n \in D$ and  $s_{n+\hspace{-0.05em}1} \in D_{n+\hspace{-0.05em}1} \subseteq D$, by the $\alpha$\hspace{0.1em}-\hspace{0.05em}strong convexity of $D$ we have that
                \begin{equation*}
                    y \,\coloneqq\, \frac{1}{2}s_{n+\hspace{-0.05em}1} + \frac{1}{2}x_n - \frac{\alpha}{8}\frac{\lVert \hspace{0.05em}s_{n+\hspace{-0.05em}1}-x_n\hspace{0.05em}\rVert^2}{\lVert \hspace{0.05em}\nabla\hspace{-0.1em}f(x_n) \hspace{0.05em}\rVert}\nabla\hspace{-0.1em}f(x_n) \in D\hspace{0.1em}.
                \end{equation*}
                Furthermore, by definition of the subsolution process $\Sb$ and analogously to \eqref{EQ:minimizationProblemBound}, we find 
                \begin{align*}
                    \hspace{0.05em}\langle \hspace{0.05em}s_{n+\hspace{-0.05em}1} - \hspace{0.05em}x_n\,|\, \nabla\hspace{-0.1em} f(x_n\hspace{-0.025em})\hspace{0.05em}\rangle & \,\leq\, \hspace{0.05em}\langle \hspace{0.05em}y -x_n\,|\, \nabla\hspace{-0.1em} f(x_n\hspace{-0.025em})\hspace{0.05em}\rangle + \hspace{0.05em}d_H(D_{n+\hspace{-0.05em}1}, D)\hspace{0.05em}L_{\hspace{-0.05em}f} \\
                    & \,=\, \frac{1}{2}\langle \hspace{0.05em}s_{n+\hspace{-0.05em}1}-\hspace{0.05em}x_n \,|\, \nabla\hspace{-0.1em}f(x_n\hspace{-0.025em})\hspace{0.05em}\rangle - \frac{\alpha}{8}\lVert \hspace{0.05em} s_{n+\hspace{-0.05em}1}-x_n\hspace{0.05em}\rVert^2\lVert \hspace{0.05em}\nabla\hspace{-0.1em} f(x_n\hspace{-0.025em})\hspace{0.05em}\rVert + \hspace{0.05em}d_H(D_{n+\hspace{-0.05em}1}, D)\hspace{0.05em}L_{\hspace{-0.05em}f}\hspace{0.1em},
                \end{align*}
                such that plugging in \eqref{EQ:lemmaExtendedNew} yields
                \begin{equation*}
                    \langle \hspace{0.05em}s_{n+\hspace{0.05em}1} - \hspace{0.05em} x_n \,|\, \nabla\hspace{-0.1em} f(x_n\hspace{-0.025em})\hspace{0.05em}\rangle \,\leq\, -\frac{1}{2}(f(x_n\hspace{-0.025em})-\Popt)-\frac{\alpha}{8}\lVert \hspace{0.05em}s_{n+\hspace{-0.05em}1}-x_n\hspace{0.05em}\rVert^2\hspace{0.05em}\lVert \hspace{0.05em}\nabla\hspace{-0.1em} f(x_n)\hspace{0.05em}\rVert + \frac{3}{2}\hspace{0.05em}d_H(D_{n+\hspace{-0.05em}1}, D)\hspace{0.05em}L_{\hspace{-0.05em}f}\hspace{0.1em}.
                \end{equation*}
                The second claim now follows since $\PP[I\hspace{0.075em}] \geq 1-\delta$.
            \end{proof}

            \noindent 
            Similar to Section \ref{SEC:theoreticalAnalysis}, we also have to provide a result on the accelerated convergence of a general nonnegative sequence that is bounded from above before we can turn to Algorithm \ref{ALG:onlineAdaptiveStochasticFrankWolfeAlgorithm} itself.
            This result is inspired by \hbox{\cite[Lemma 3.5]{wirth2023acceleration}} but is adapted to fit our setting and is slightly more restrictive. 
            A detailed proof of this next result is provided in Appendix \ref{SEC:B.3}.

            \begin{proposition}{}{advancedInductionBound}
                Let $(T_n\hspace{-0.05em})_{n \in \NN_0} \subseteq \RR_+$ be a sequence. 
                Let $A, B_1, B_2, B_3 > 0$ be some fixed constants, $(\hspace{-0.05em}A_{\hspace{0.025em}n}\hspace{-0.05em})_{n \in \NN} \subseteq \RR$ be a sequence satisfying $A_{\hspace{0.025em}n} \geq A $ for all $n \in \NN$, $r \in [\hspace{0.025em}0, 1]$ and let $(\lambda_{\hspace{0.025em}n}\hspace{-0.05em})_{n \in \NN_0}$ be our usual step\hspace{0.075em}-\hspace{0.05em}size sequence. 
                If there exists some $m \in \NN_0$ such that
                \begin{equation}\label{EQ:advancedInductionEquation}
                    T_{n+\hspace{-0.05em}1} \,\leq\, \left(\hspace{-0.15em}1-\frac{\lambda_{\hspace{0.025em}n}}{2}\hspace{-0.15em}\right)\hspace{-0.1em} T_n - A_{\hspace{0.025em}n}\sqrt{T_n}\hspace{0.025em}B_1\lambda_{\hspace{0.025em}n} + B_2\lambda_{\hspace{0.025em}n}^{\hspace{-0.1em}1+\hspace{0.05em}r} + B_3\lambda_{\hspace{0.025em}n} 
                \end{equation}
                for all $n \geq m$, then there exists some $N \in \NN$ with $N \geq m$ such that
                \begin{equation}\label{EQ:inductionHypothesis}
                    T_n \,\leq\, \left(\hspace{-0.15em}T_m+4\hspace{-0.15em}\left(\hspace{-0.1em}\hspace{-0.05em}\frac{B_2}{AB_1}\hspace{-0.1em}\right)^{\hspace{-0.25em}2} \hspace{-0.2em}+ B_2\hspace{-0.2em}\right)\hspace{-0.2em}\lambda_{\hspace{0.025em}n-\hspace{-0.05em}1}^{\hspace{-0.05em}2\hspace{0.025em}r} + B_3
                \end{equation}
                for all $n \geq N+1$.
                In particular, if $2\hspace{0.025em}r \leq 1$, then $N = m$.
            \end{proposition}

            \noindent
            Comparing Proposition \ref{PRO:advancedInductionBound} to \cite[Lemma 3.5]{wirth2023acceleration}, we can see that a main difference is that the sequence $(\hspace{-0.05em}A_{\hspace{0.025em}n}\hspace{-0.05em})_{n \in \NN}$ in our case has to be bounded from below and not from above.
            Since in the case where we want to apply Proposition \ref{PRO:advancedInductionBound} we have that $A_{\hspace{0.025em}n}$ is given as a sample of $\lVert \hspace{0.05em} \Sb_{n+\hspace{-0.05em}1}-\Xb_{\hspace{0.025em}n}\rVert$ a canonical upper bound is given by $\diam(D)$.
            However, bounding $\lVert \hspace{0.05em} \Sb_{n+\hspace{-0.05em}1}-\Xb_{\hspace{0.025em}n}\rVert$ or more specifically $\inf\hspace{0.1em}\{\hspace{0.05em}\lVert \hspace{0.05em}\Sb_{n+\hspace{-0.05em}1} \hspace{-0.05em}- \hspace{0.05em}\Xb_n \rVert^2 \,\colon n \in \NN\hspace{0.05em}\}$ from below may be hard or even impossible without further restrictions since the infimum in general may tend towards zero.
            To make sure that we can guarantee that $\lVert \hspace{0.05em} \Sb_{n+\hspace{-0.05em}1}-\Xb_{\hspace{0.025em}n}\rVert$ can be bounded by a positive constant at least eventually, we assume that the optimal solution $x^\star$ is contained in the relative interior of $D$, such that the distance to its relative boundary is positive.

            \begin{theorem}{Asymptotic Accelerated Convergence of Algorithm \ref*{ALG:onlineAdaptiveStochasticFrankWolfeAlgorithm}}{advancedConvergence}
                Let $f$ be $\mu$\hspace{0.1em}-\hspace{0.05em}strongly convex on $D$, $D$ be $\alpha$\hspace{0.1em}-\hspace{0.05em}strongly convex and $x^\star \in \relinterior(D)$. 
                Let $\Xb$ be the iterate process generated by running the Algorithm \ref{ALG:onlineAdaptiveStochasticFrankWolfeAlgorithm} with starting random variable $\Xb_0$ and domain approximation process $\Db$.
                Let $\beta_{\hspace{0.025em}1}, \delta \in [\hspace{0.025em}0, 1]$, $r_1 \in (\hspace{-0.025em}0, 1]$, $\eta_{\hspace{0.025em}1} \geq 0$ with $3\hspace{0.05em}\eta_{\hspace{0.025em}1} < h(x^\star, \relbd(D))$, and $c_1 > 0$ be as in Assumption \ref{ASS:hausdorffConvergenceRateAndFunctionConvergenceRate} and Assumption \ref{ASS:domainContainmentExterior}.
                Then, denoting 
                \begin{equation}\label{EQ:piEvent}
                    \pi \,=\, \PP[\hspace{0.025em}\nabla \hspace{-0.1em}f(\Xb_n\hspace{-0.05em}) = 0 \text{ for some } n \in \NN\hspace{0.05em}]\hspace{0.1em},
                \end{equation} 
                there exists some $N \in \NN$ and $B \geq 0$ such that
                \begin{equation*}
                    \PP\hspace{-0.2em}\left[\hspace{0.05em}0 \,\leq\,  f(\hspace{0.035em}\Xb_{\hspace{0.025em}n}\hspace{-0.05em}) - \Popt  \,\leq\, B\lambda_{\hspace{0.025em}n-\hspace{-0.05em}1}^{\hspace{-0.1em}2\hspace{0.05em}r_1} + \frac{3}{2}\eta_{\hspace{0.025em}1}\hspace{-0.05em}L_{\hspace{-0.05em}f}\,\text{ for all } \, n \geq N\right] \,\geq\, 1-(\beta_{\hspace{0.025em}1}+\delta + \pi)\hspace{0.1em}.
                \end{equation*}
            \end{theorem}
        \begin{proof}[\textcolor{seeblau}{Proof.}]
            Let $\omega \in M \cap I$ be a sample point corresponding to Assumption \ref{ASS:hausdorffConvergenceRateAndFunctionConvergenceRate} and Assumption \ref{ASS:domainContainmentExterior}. 
                For simplicity, we use the notation \eqref{EQ:simplicticNotation} for the rest of the proof. 
                We denote $e_n = f(x_n) - \Popt$ for the approximation error at iteration $n \in \NN_0$.
                First, by definition of the curvature constant, we have
                \begin{equation}\label{EQ:curvatureFormulaKnown}
                    f(x_{n+\hspace{-0.05em}1}\hspace{-0.025em}) \,\leq\, f(x_n\hspace{-0.025em}) + \langle \hspace{0.05em}s_{n+\hspace{-0.05em}1} -\hspace{0.05em}x_n \,|\, \nabla\hspace{-0.1em}f(x_n\hspace{-0.025em})\hspace{0.05em} \rangle\lambda_{\hspace{0.025em}n} + \frac{C_{\hspace{-0.05em}f}}{2}\lambda_{\hspace{0.025em}n}^{\hspace{-0.05em}2}
                \end{equation}
                for all $n \in \NN_0$.
                Combining Lemma \ref{LEM:improvedInnerProductBound1} and Lemma \ref{LEM:improvedInnerProductBound2}, we further obtain that 
                \begin{equation*}
                    \begin{aligned}
                        \langle \hspace{0.05em}s_{n+\hspace{-0.05em}1} - \hspace{0.05em}x_n \,|\, \nabla\hspace{-0.1em} f(x_n) \hspace{0.05em}\rangle 
                        & \,\leq\, -\frac{1}{2}e_n - \frac{\alpha}{8}\lVert \hspace{0.05em}s_{n+\hspace{-0.05em}1} - \hspace{0.05em}x_n \hspace{0.05em}\rVert^2 \hspace{0.05em}\lVert \hspace{0.05em}\nabla\hspace{-0.1em}f(x_n\hspace{-0.025em}) \hspace{0.05em}\rVert + \frac{3}{2}\hspace{0.05em}d_H(D_{n+\hspace{-0.05em}1}, D)L_{\hspace{-0.05em}f} \\
                        & \,\leq\, -\frac{1}{2}e_n-\frac{\alpha}{8}\lVert \hspace{0.05em}s_{n+\hspace{-0.05em}1} -\hspace{0.05em} x_n \hspace{0.05em}\rVert^2 \sqrt{\frac{\mu}{2}e_n} + \frac{3}{2}\hspace{0.05em}d_H(D_{n+\hspace{-0.05em}1}, D)L_{\hspace{-0.05em}f}
                    \end{aligned}
                \end{equation*}
                for all $n \in \NN_0$, such that together with \eqref{EQ:curvatureFormulaKnown} and $d_H(D_{n+\hspace{-0.05em}1}, D) \leq c_1\lambda_{\hspace{0.025em}n}^{\hspace{-0.05em}r_1}  +\eta_{\hspace{0.025em}1}$ for all $n \in \NN_0$ we obtain the recursion bound
            \begin{equation}\label{EQ:improvedUpperBound}
                e_{n+\hspace{-0.05em}1} \leq \left(\hspace{-0.15em}1-\frac{\lambda_{\hspace{0.025em}n}}{2}\hspace{-0.15em}\right)\hspace{-0.15em}e_n - \lVert \hspace{0.05em}s_{n+\hspace{-0.05em}1} \hspace{-0.05em}- \hspace{0.05em}x_n \hspace{0.05em}\rVert^2 \sqrt{e_n}\left(\hspace{-0.1em}\frac{\alpha^2\hspace{-0.05em}\mu}{128}\hspace{-0.1em}\right)^{\hspace{-0.3em}\frac{1}{2}}\hspace{-0.3em}\lambda_{\hspace{0.025em}n} + \left(\hspace{0em}\frac{3}{2}\hspace{0.05em}c_1\hspace{-0.05em}L_{\hspace{-0.05em}f} + \frac{C_{\hspace{-0.05em}f}}{2}\hspace{-0.1em}\right)\hspace{-0.1em}\lambda_{\hspace{0.025em}n}^{\hspace{-0.1em}1+r_1} + \frac{3}{2}\hspace{0.05em}\eta_{\hspace{0.025em}1}\hspace{-0.05em}L_{\hspace{-0.05em}f}\lambda_{\hspace{0.025em}n}
            \end{equation}
            for all $n \in \NN_0$.
            Now we turn to finding a lower bound for $\lVert \hspace{0.05em}s_{n+\hspace{-0.05em}1}-\hspace{0.05em}x_n\rVert$ for all $n \geq N$ for some $N \in \NN_0$.
            To this end, we set $\rho = h(x^\star, \relbd(D))$ and recall that by construction of the subsolution process $\Sb$ it holds that
            \begin{equation}\label{EQ:innerProductBoundGeneral}
                \langle \hspace{0.05em}s_{n+\hspace{-0.05em}1} -\hspace{0.05em}x_n \,|\, \nabla\hspace{-0.1em} f(x_n\hspace{-0.025em}) \hspace{0.05em}\rangle \,\leq\,  \langle\hspace{0.05em} t - x_n \,|\, \nabla\hspace{-0.1em} f(x_n\hspace{-0.025em}) \hspace{0.05em}\rangle
            \end{equation}
            for all $t \in D_{n+\hspace{-0.05em}1}$.
            Since by assumption $3\hspace{0.05em}\eta_1 < \rho$ and using again that $d_H(D_{n+\hspace{-0.05em}1}, D) \leq c_1\lambda_{\hspace{0.025em}n}^{\hspace{-0.05em}r_1}  +\eta_{\hspace{0.025em}1}$ for all $n \in \NN_0$, we have that for 
            \begin{equation*}
                N_0 \,\coloneqq\, \left\lceil 2\hspace{-0.2em}\left(\frac{6\hspace{0.05em}c_1}{\rho}\right)^{\hspace{-0.3em}\frac{1}{\hspace{0.05em}r_1}} \hspace{-0.3em}- 2\right\rceil \in \NN_0
            \end{equation*}
            it holds $2\hspace{0.05em}d_H(D_{n+\hspace{-0.05em}1}, D) < \rho$, such that $x^\star \in D_{n+\hspace{-0.05em}1}$ for all $n \geq N_0$.
            Restricting ourselves to the complement of the event in \eqref{EQ:piEvent}, without loss of generality, we can assume that $\lVert \nabla\hspace{-0.1em} f(x_n)\rVert > 0$ for all $n \geq N_0$, such that together with \eqref{EQ:innerProductBoundGeneral} and 
            \begin{equation*}
                t \,=\, x^\star - \frac{\rho}{2\hspace{0.05em}\lVert\nabla\hspace{-0.1em} f(x_n\hspace{0.05em})\rVert}\nabla\hspace{-0.1em} f(x_n\hspace{0.05em}) \in D_{n+\hspace{-0.05em}1}
            \end{equation*}
            we find that
            \begin{align}\label{EQ:lowerDistanceBound}
                \langle\hspace{0.05em} s_{n+\hspace{-0.05em}1}-\hspace{0.05em}x_n \,|\, \nabla\hspace{-0.1em} f(x_n\hspace{-0.025em})\hspace{0.05em}\rangle \,\leq\, \langle \hspace{0.05em}x^\star-\hspace{0.05em}x_n \,|\, \nabla\hspace{-0.1em} f(x_n\hspace{-0.025em})\hspace{0.05em} \rangle - \frac{\rho}{2}\lVert \hspace{0.05em}\nabla\hspace{-0.1em} f(x_n\hspace{-0.025em}) \hspace{0.05em}\rVert \,\leq\, - \frac{\rho}{2}\lVert \hspace{0.05em}\nabla\hspace{-0.1em} f(x_n\hspace{-0.025em})\hspace{0.05em} \rVert
            \end{align}
            for all $n \geq N_0$, where in the second inequality we used that $\langle \hspace{0.05em } x^\star-x_n \,|\, \nabla\hspace{-0.1em} f(x_n\hspace{-0.025em})\hspace{0.05em} \rangle \leq 0$ by optimality of $x^\star$.
            Using \eqref{EQ:lowerDistanceBound} we obtain that for all $n \geq N_0$ it holds that 
            \begin{equation*}
                0 \,<\, \frac{\rho}{2}\lVert \hspace{0.05em}\nabla \hspace{-0.1em} f(x_n) \rVert \,\leq\, \langle \hspace{0.05em} x_n-\hspace{0.05em}s_{n+\hspace{-0.05em}1} \,|\, \nabla\hspace{-0.1em} f(x_n\hspace{-0.025em})\hspace{0.05em}\rangle \,\leq\, \lVert \hspace{0.05em} x_n-\hspace{0.05em}s_{n+\hspace{-0.05em}1}\rVert \hspace{0.05em} \lVert \hspace{0.05em}\nabla\hspace{-0.1em} f(x_n\hspace{-0.025em})\hspace{0.05em}\rVert\hspace{0.1em},
            \end{equation*}
            such that dividing by $\lVert \hspace{0.05em}\nabla\hspace{-0.1em}f(x_n)\hspace{0.05em}\rVert$ yields
            \begin{equation}\label{EQ:lowerConstantBound}
                \lVert s_{n+1}-\hspace{0.05em}x_n\rVert \,\geq\, \frac{\rho}{2} \,>\, 0
            \end{equation}
            for all $n \geq N_0$.
            Thus, we can apply Proposition \ref{PRO:advancedInductionBound} to \eqref{EQ:improvedUpperBound}, using \eqref{EQ:lowerConstantBound}, to obtain the hitting time $\Nb \colon \Omega \to \NN$ defined by
            \begin{equation}\label{EQ:hittnigTime}
                \Nb \,\coloneqq\, \max\left\{\hspace{-0.1em}N_0, \ind_{\{r \,\geq\, 1/2\}}\hspace{-0.3em}\left\lceil 2\hspace{-0.15em}\left(\frac{2\hspace{0.05em}\Bb}{3\hspace{0.05em}c_1\hspace{-0.05em}L_{\hspace{-0.05em}f} + C_{\hspace{-0.075em}f}}\right)^{\hspace{-0.3em}r_1}\hspace{-0.2em} - 2\right\rceil\right\},
            \end{equation}
            where
            \begin{equation*}
                \Bb \,\coloneqq \, \hspace{-0.1em}f(\hspace{0.035em}\Xb_{N_0}\hspace{-0.05em}) -\Popt + \hspace{0.2em}\frac{512}{\alpha^2\rho^{2}\mu}\hspace{-0.1em}\left(3\hspace{0.05em}c_1\hspace{-0.05em}L_{\hspace{-0.05em}f} + C_{\hspace{-0.075em}f}\right)^{2} \hspace{-0.2em}+ \frac{3}{2}\hspace{0.05em}c_1\hspace{-0.05em}L_{\hspace{-0.05em}f} + \frac{1}{2}C_{\hspace{-0.075em}f}\hspace{0.1em},
            \end{equation*}
            such that
            \begin{equation}\label{EQ:boundPre}
                e_{n} \,\leq\, \Bb(\omega)\lambda_{\hspace{0.025em}n-\hspace{-0.05em}1}^{\hspace{-0.1em}2\hspace{0.05em}r_1} +  \frac{3}{2}\hspace{0.05em}\eta_{\hspace{0.025em}1}\hspace{-0.05em}L_{\hspace{-0.05em}f}\hspace{0.1em}
            \end{equation}
            for all $n \geq \Nb(\omega) + 1$.
            Furthermore, we can use Remark \ref{REM:adaptedConvergenceRate} to bound 
            \begin{equation*}
                f(\Xb_{N_0})-\Popt \,\leq\, (2\hspace{0.05em}c_1\hspace{-0.05em}L_{\hspace{-0.075em}f} + C_{\hspace{-0.075em}f})\lambda_{\hspace{-0.05em}N_0}^{\hspace{-0.05em}r_1} + \eta_{\hspace{0.025em}1}\hspace{-0.05em}L_{\hspace{-0.075em}f} \,\leq\, (2\hspace{0.05em}c_1\hspace{-0.05em} + \eta_{\hspace{0.025em}1})L_{\hspace{-0.075em}f} + C_{\hspace{-0.075em}f}\hspace{0.1em},
            \end{equation*}
            such that we can replace the random variable $\Bb$ with the constant
            \begin{equation*}
                B \,=\, \frac{512}{\alpha^2\rho^{2}\mu}\hspace{-0.1em}\left(3\hspace{0.05em}c_1\hspace{-0.05em}L_{\hspace{-0.05em}f} + C_{\hspace{-0.075em}f}\right)^{2} \hspace{-0.2em}+ \left(\frac{7}{2}\hspace{0.05em}c_1 + \hspace{-0.05em}\eta_{\hspace{0.025em}1}\hspace{-0.15em}\right)\hspace{-0.15em}L_{\hspace{-0.05em}f} + \frac{3}{2}C_{\hspace{-0.075em}f}
            \end{equation*}
            in \eqref{EQ:hittnigTime} and \eqref{EQ:boundPre}.
            The claim now follows since $\PP[M \cap I\hspace{0.075em}] \geq 1-(\beta_{\hspace{0.025em}1} + \delta)$ and the probability of the complement of the event in \eqref{EQ:piEvent} is given by $1-\pi$, such that their intersections have probability at least $1-(\beta_{\hspace{0.025em}1}+\delta+\pi)$.
        \end{proof}    

        \noindent
        At first, using the probability $\pi$ in Theorem \ref{THM:advancedConvergence} may seem restrictive for the overall result.
        However, this is more of a technical subtlety.
        In general it is even positive if $\pi$ is large since this means that we can reach optimality in $N_{\text{optimal}} \in \NN$ many steps in Algorithm \ref{ALG:onlineAdaptiveFrankWolfeAlgorithm}.
        Moreover, taking a closer look at the proof of Theorem \ref{THM:advancedConvergence}, we can see that if $N_{\text{optimal}}$ is large, accelerated convergence up to this point still holds.

    \subsection{Comments on the Assumptions}
        We briefly comment on the assumptions made in this section.
        More specifically, given a domain approximation process $\Db$, we will introduce a convenient way of constructing natural extended domain approximation processes satisfying the more restrictive Assumption \ref{ASS:domainContainment} and Assumption \ref{ASS:domainContainmentExterior} for convergence acceleration. 
        
        \paragraph{Extended Domain Approximations via Morphology.} Given a valid domain approximation process $\Db$ we can naturally extend it using tools from morphology called dilation and erosion.
        
        \begin{definition}{Dilation and Erosion}{dilationErosion}
            Let $X, Y \subseteq H$ be two sets. 
            The \emph{dilation} of $X$ by $Y$ is defined as
            \begin{equation*}
                X \oplus Y \,\coloneqq\, \bigcup\hspace{0.2em}\{x + y\,\colon x \in X, y \in Y\}
            \end{equation*}
            and the \emph{erosion} of $X$ by $Y$ is defined as
            \begin{equation*}
                X \ominus Y \,\coloneqq\, \{z \in H \,\colon z + Y \subseteq X\}\hspace{0.1em},
            \end{equation*}
            where for $z \in H$ we have $z + Y =\{z + y \,\colon y \in Y\}$.
        \end{definition}
        
        \noindent
        Note that in the literature, for two sets $X, Y \subseteq H$, there is no difference between the dilation $X \oplus Y$ and the Minkowski sum $X + Y$. 
        However, in some contexts the Minkowski difference $X - Y$ may be defined as $X + (-Y)$, such that it does not necessarily coincide with the erosion $X \ominus Y$.
        To avoid confusion, we therefore stick with the morphology terminology.
        The main motivation for the use of the concepts of dilation and erosion is that these constructions yield an interpretable extension and contraction of the set $X$, respectively, which additionally can preserve the properties that are important in our analysis.
        
        \begin{lemma}{}{dilationErosionProperties} 
            Let $X, Y \subseteq H$ be nonempty, compact, and convex sets.
            Then, the dilation $X \oplus Y$ is nonempty, compact, and convex and the erosion $X \ominus Y$ is compact and convex.
            If $X$ is $\alpha$\hspace{0.1em}-\hspace{0.05em}strongly convex for some $\alpha > 0$ and it holds 
            \begin{equation}\label{EQ:nonemptyErosionCondition}
                \diam(X)^2 \,\geq\, \frac{8}{\alpha}\diam(\hspace{0.05em}Y)\hspace{0.1em},
            \end{equation}
            then $X \ominus Y$ is also nonempty.
        \end{lemma}

        \noindent
        Having Lemma \ref{LEM:dilationErosionProperties} at hand, we define the domain extension map 
        \begin{equation}\label{EQ:extensionMap}
            \Db^{\hspace{-0.05em}+} \colon \NN \times \Omega \to \Dcc, \; (n, \omega) \,\mapsto\, \Db_{\hspace{0.025em}n}\hspace{-0.05em}(\omega) \,\oplus\, \BB(0, d_H(\hspace{0.035em}\Db_{\hspace{0.025em}n}\hspace{-0.05em}(\omega), D))\hspace{0.1em}.
        \end{equation}
        For the corresponding domain contraction map, as can be seen in Lemma \ref{LEM:dilationErosionProperties}, we have to assume that there exists a sequence of nonnegative random variables $(\boldsymbol{\alpha}_n)_{n \in \NN}$ and a constant $\gamma \in [\hspace{0.025em}0, 1]$ such that 
        \begin{equation}\label{EQ:diameterBound}
                \PP\hspace{-0.2em}\left[\hspace{0.05em}\Db_{\hspace{0.025em}n} \text{ \hspace{-0.05em}is } \boldsymbol{\alpha}_n\text{-\hspace{0.05em}strongly convex and }\diam(\hspace{0.035em}\Db_{\hspace{0.025em}n}\hspace{-0.05em})^2 \,\geq\, \frac{16}{\hspace{0.2em}\boldsymbol{\alpha}_n} d_H(\hspace{0.035em}\Db_{\hspace{0.025em}n}, D) \,\text{ for all }\, n \in \NN\hspace{0.05em}\right] \,\geq\, 1-\gamma\hspace{0.1em}.
        \end{equation}
        Denoting with $A \in \Sigma$
        the event corresponding to \eqref{EQ:diameterBound}, whose measurability is shown in detail in Appendix~\ref{SEC:B.3}, the domain contraction map is then given by 
        \begin{equation}\label{EQ:contractionMap}
            \Db^{\hspace{-0.05em}-} \colon \NN \times \Omega \to \Dcc, \; (n, \omega) \,\mapsto\, \ind_{\hspace{-0.05em}A}\hspace{-0.05em}(\omega)\cdot\hspace{0.05em}(\hspace{0.05em}\Db_{\hspace{-0.05em}n}\hspace{-0.05em}(\omega) \,\ominus\, \BB(0, d_H(\hspace{0.025em}\Db_{\hspace{-0.05em}n}\hspace{-0.05em}(\omega), D)))\hspace{0.1em},
        \end{equation}
        where the product is to be understood pointwise. 
        By Lemma \ref{LEM:dilationErosionProperties}, both maps $\Db^{\hspace{-0.05em}+}$ and $\Db^{\hspace{-0.05em}-}$ are well\hspace{0.05em}-\hspace{0.035em}defined.
        However, it remains to show that these maps form stochastic processes satisfying Assumption~\ref{ASS:domainContainment} and Assumption~\ref{ASS:domainContainmentExterior}, respectively.
        This follows from the next result which additionally uses the Hausdorff distance between $\Db_{\hspace{-0.025em}n}$ and $D$ for each $n \in \NN$ to derive upper bounds on the Hausdorff distance between the domain extension $\Db_{\hspace{-0.05em}n}^{\hspace{-0.05em}+}$ and the domain contraction $\Db_{\hspace{-0.05em}n}^{\hspace{-0.05em}-}$ and $D$, respectively.
        A proof od this result can be found in Appendix \ref{SEC:B.3}.
        
        \begin{lemma}{}{hausdorffExtensionBounds}
            Let $\Db$ be a domain approximation process. 
            Then, the extension map $\Db^{\hspace{-0.05em}+}$ defined as in \eqref{EQ:extensionMap} is a valid domain approximation process satisfying Assumption \ref{ASS:domainContainment} for $\delta = 0$ and
            \begin{equation*}
                \PP\hspace{-0.2em}\left[\hspace{0.05em}d_H(\hspace{0.025em}\Db_{\hspace{-0.05em}n}^{\hspace{-0.05em}+}, D) \,\leq\, 2\hspace{0.05em} d_H(\hspace{0.025em}\Db_{\hspace{-0.05em}n}, D) \,\text{ for all }\, n \in \NN\hspace{0.1em}\right] \,=\, 1\hspace{0.1em}.
            \end{equation*}
            Furthermore, the contraction map $\Db^{\hspace{-0.05em}-}$ defined as in \eqref{EQ:contractionMap} is a valid domain approximation process satisfying Assumption \ref{ASS:domainContainmentExterior} for $\delta = \gamma$ and
            \begin{equation*}
                \PP\hspace{-0.2em}\left[\hspace{0.05em}d_H(\hspace{0.035em}\Db_{\hspace{0.025em}n}^-, D) \,\leq\,  d_H(\hspace{0.035em}\Db_{\hspace{0.025em}n}, D) + \sqrt{\frac{8}{\hspace{0.2em}\boldsymbol{\alpha}_n}d_H(\hspace{0.035em}\Db_{\hspace{-0.025em}n}, D)} \,\text{ for all }\, n \in \NN\hspace{0.05em}\right] \,\geq\, 1-\gamma\hspace{0.1em}.
            \end{equation*}
        \end{lemma}

        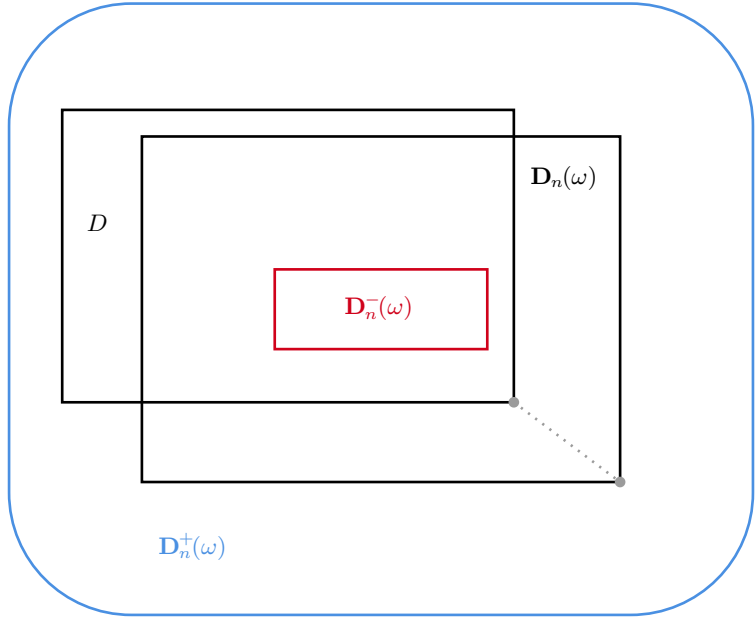
\begin{figure}[h]
            \centering
            \tikzset{every picture/.style = {line width = 1pt}} 
            \begin{tikzpicture}[x = 1pt, y = 1pt, yscale = -1, xscale = 1]
                
                \draw   (210,80) -- (380,80) -- (380,190) -- (210,190) -- cycle ;
                \draw   (240,90) -- (420,90) -- (420,220) -- (240,220) -- cycle ;
                \draw [color={rgb, 255:red, 155; green, 155; blue, 155 }  ,draw opacity=1 ] [dash pattern={on 0.84pt off 2.51pt}]  (380,190) -- (420,220) ;
                \draw  [color={rgb, 255:red, 155; green, 155; blue, 155 }  ,draw opacity=1 ][fill={rgb, 255:red, 155; green, 155; blue, 155 }  ,fill opacity=1 ] (378.5,190) .. controls (378.5,189.17) and (379.17,188.5) .. (380,188.5) .. controls (380.83,188.5) and (381.5,189.17) .. (381.5,190) .. controls (381.5,190.83) and (380.83,191.5) .. (380,191.5) .. controls (379.17,191.5) and (378.5,190.83) .. (378.5,190) -- cycle ;
                \draw  [color={rgb, 255:red, 155; green, 155; blue, 155 }  ,draw opacity=1 ][fill={rgb, 255:red, 155; green, 155; blue, 155 }  ,fill opacity=1 ] (418.5,220) .. controls (418.5,219.17) and (419.17,218.5) .. (420,218.5) .. controls (420.83,218.5) and (421.5,219.17) .. (421.5,220) .. controls (421.5,220.83) and (420.83,221.5) .. (420,221.5) .. controls (419.17,221.5) and (418.5,220.83) .. (418.5,220) -- cycle ;
                \draw  [color={rgb, 255:red, 208; green, 2; blue, 27 }  ,draw opacity=1 ] (290,140) -- (370,140) -- (370,170) -- (290,170) -- cycle ;
                \draw  [color={rgb, 255:red, 74; green, 144; blue, 226 }  ,draw opacity=1 ] (190,86) .. controls (190,60.59) and (210.59,40) .. (236,40) -- (424,40) .. controls (449.41,40) and (470,60.59) .. (470,86) -- (470,224) .. controls (470,249.41) and (449.41,270) .. (424,270) -- (236,270) .. controls (210.59,270) and (190,249.41) .. (190,224) -- cycle ;
                
                \draw (245,238) node [anchor=north west][inner sep=0.75pt]  [font=\footnotesize,color={rgb, 255:red, 74; green, 144; blue, 226 }  ,opacity=1 ]  {$\Db_{\hspace{-0.05em}n}^{\hspace{-0.05em}+}\hspace{-0.15em}( \omega )$};
                \draw (315,148) node [anchor=north west][inner sep=0.75pt]  [font=\footnotesize,color={rgb, 255:red, 208; green, 2; blue, 27 }  ,opacity=1 ]  {$\Db_{\hspace{-0.05em}n}^{\hspace{-0.05em}-}\hspace{-0.15em}(\omega)$};
                \draw (385, 100) node [anchor=north west][inner sep=0.75pt]  [font=\footnotesize]  {$\Db_{\hspace{-0.05em}n}\hspace{-0.1em}(\omega)$};
                \draw (218,118) node [anchor=north west][inner sep=0.75pt]  [font=\footnotesize]  {$D$};
            \end{tikzpicture}      
            \caption{Conceptual illustration of the introduced concepts of domain extension (dilation) and domain contraction (erosion) for a fixed sample point $\omega \in \Omega$. 
            The Hausdorff distance between $\Db_{\hspace{-0.025em}n}\hspace{-0.05em}(\omega)$ and $D$ is exactly the distance between the two gray points (dotted line). 
            The domain extension $\Db_{\hspace{-0.025em}n}^{\hspace{-0.05em}+}\hspace{-0.05em}(\omega)$ is given in blue and is, in particular, no longer polygonal. 
            The domain contraction $\Db_{\hspace{-0.025em}n}^{\hspace{-0.05em}-}\hspace{-0.05em}(\omega)$ is given in red.}
            \label{FIG:dilationExtension}
        \end{figure}
        
        \begin{remark}{}{practicalSetting}
            In a practical setting the definitions \eqref{EQ:extensionMap} and \eqref{EQ:contractionMap} are not directly applicable since the actual quantity $d_H(\hspace{0.035em}\Db_{\hspace{0.025em}n}, D)$ at each step $n \in \NN$ is most likely unknown. 
            However, if we assume that Assumption \ref{ASS:hausdorffConvergenceRateAndFunctionConvergenceRate} holds, then we can simply replace the Hausdorff distance $d_H(\Db_{\hspace{-0.025em}n}, D)$ with $c_1\hspace{-0.05em}\lambda_{\hspace{0.025em}n}^{\hspace{-0.05em}r_1} + \eta_{\hspace{0.025em}1}$ for each $n \in \NN$ in the constructions above.
            In this case, the upper bounds on the Hausdorff distance from Lemma \ref{LEM:hausdorffExtensionBounds} can be adapted accordingly, taking into consideration the probabilistic tolerance $\beta_{\hspace{0.025em}1}$. 
            Hence, this yields implementable surrogate extensions and contractions controlled by the known convergence rates.
        \end{remark}

\section{Numerical Examples}\label{SEC:examples}

    We present two numerical examples. 
    The first is a simple academic toy problem, which serves to illustrate all theoretical results and assumptions required in our paper.
    The second is a more complex distributionally robust Linear Quadratic Gaussian (LQG) problem, where we demonstrate the runtime efficiency of our proposed method compared to existing approaches. \\

    \noindent
    The \Julia code: \href{https://github.com/Nuwanda314/Recursive_Adaptive_Frank-Wolfe_Algorithm_Experiments}{https://github.com/Nuwanda314/Recursive\_\hspace{0.05em}Adaptive\_Frank\hspace{0.05em}-Wolfe\_\hspace{0.05em}Algorithm\_Experiments}

    \newpage
    \subsection{Simple Example: Quadratic Objective on Rectangular Domain}\label{SUBSEC:academicProblem}
        For the purpose of illustrating the details of the assumptions and results we derived in the previous sections, we consider a simple problem with a quadratic objective function on a rectangular domain.
        However, for the sake of simplicity, we will restrict ourselves to the one\hspace{0.1em}-\hspace{0.025em}dimensional case and note that an extension to higher dimensions would be possible with slight adjustments.
        More explicitly, we consider the objective function
        \begin{equation*}
            f \colon \RR \to \RR\hspace{0.025em}, \; x \,\mapsto\, (x-x^\star)^2
        \end{equation*}
        for some fixed point $x^\star \in \RR$ and the problem domain $D = [\hspace{0.025em} a, b\hspace{0.025em}]$, where $a, b \in \RR$ with $a < b$.
        To approximate the problem domain, we assume to have access to a sequence $(\Yb_{\hspace{-0.05em}n}\hspace{-0.025em})_{n \in \NN}$ of independent samples of a random variable $\Yb \sim \Unif(a, b)$, that is, we assume to be able to iteratively draw points from the problem domain. 
        To highlight the effect of the domain approximation on the convergence of Algorithm \ref{ALG:onlineAdaptiveFrankWolfeAlgorithm}, we present two different approaches. 
        The first approach uses (centered) moments of the random variable $\Yb$, while the second one is based on the convex hull of the samples.

        \paragraph{Moment\hspace{0.05em}-Based Domain Approximations.} 
        Recall that for $\Yb \sim \Unif(a, b)$ the first moment and second centered moment, that is, the mean and the variance, are given by 
        \begin{equation*}
            \mu \,\coloneqq\, \EE[\Yb] \,=\, \frac{a+b}{2} \quad \text{and} \quad \sigma^2\,\coloneqq\, \VV[\Yb] \,=\, \EE[(\Yb - \mu)^2]\,=\, \frac{(b-a)^2}{12\hspace{0.4em}}\hspace{0.05em}.
        \end{equation*}
        Hence, we can reformulate $a = \mu - \sqrt{3}\hspace{0.05em}\sigma$ and $b = \mu + \sqrt{3}\hspace{0.05em}\sigma$, such that considering the empirical estimators
        \begin{equation}\label{EQ:empiricalEstimators}
           \Eb_{\hspace{0.015em}n}\hspace{-0.025em}(\Yb) \,=\, \frac{1}{n}\sum_{i\hspace{0.025em}= 1}^n \Yb_{\hspace{-0.1em}i} \quad \text{and} \quad \Vb_{\hspace{-0.1em} n}\hspace{-0.025em}(\Yb) \,=\, \frac{1}{n}\sum_{i \hspace{0.025em} = 1}^n\hspace{0.05em}(\Yb_{\hspace{-0.1em} i} -\hspace{0.05em} \Eb_{\hspace{0.015em}n}\hspace{-0.05em}(\Yb)\hspace{-0.05em})^2
        \end{equation}
        of $\mu$ and $\sigma^2$, respectively, we obtain the (centered) moment based empirical estimators
        \begin{equation*}
            \Ab_{\hspace{0.015em}n}^{\hspace{-0.175em}\text{MB}} \,=\, \Eb_{\hspace{0.015em}n}\hspace{-0.025em}(\Yb) - \sqrt{3\Vb_{\hspace{-0.1em} n}\hspace{-0.025em}(\Yb)}\quad \text{and} \quad \Bb_n^\text{MB} \,=\, \Eb_{\hspace{0.015em}n}\hspace{-0.025em}(\Yb) + \sqrt{3\Vb_{\hspace{-0.1em} n}\hspace{-0.025em}(\Yb)}
        \end{equation*}
        of $a$ and $b$, respectively.
        Thus, an obvious moment\hspace{0.05em}-\hspace{0.035em}based (MB) choice for a candidate of a domain approximation process for the problem domain $D = [\hspace{0.025em}a, b\hspace{0.025em}]$ is
        \begin{equation}\label{EQ:exampleDomainApproximationProcess}
            \Db^{\hspace{-0.025em}\text{MB}} \colon \Omega \times \NN \to \Dcc, \; (\omega, n) \,\mapsto\, [\hspace{0.025em}\Ab_{\hspace{0.015em}n}^{\hspace{-0.175em}\text{MB}}\hspace{-0.05em}(\omega), \Bb_n^\text{MB}\hspace{-0.05em}(\omega)\hspace{0.025em}]\hspace{0.1em}.
        \end{equation}
        Clearly $\Db_n^{\hspace{-0.025em}\text{MB}}$ attains nonempty, compact, and convex values for all $n \in \NN$, such that $\Db^{\hspace{-0.025em}\text{MB}}$ is well\hspace{0.05em}-\hspace{0.035em}defined.
        Furthermore, by definition of $\Ab_{\hspace{0.015em}n}^{\hspace{-0.175em}\text{MB}}$ and $\Bb_n^\text{MB}$ we can simply rewrite
        \begin{equation*}
            \Db_{\hspace{-0.025em}n}^{\hspace{-0.025em}\text{MB}} \,=\, [\hspace{0.025em}\Ab_{\hspace{0.015em}n}^{\hspace{-0.175em}\text{MB}}, \Bb_n^\text{MB}\hspace{0.025em}] \,=\, \BB\hspace{-0.2em}\left(\hspace{-0.05em}\Eb_{\hspace{0.015em}n}\hspace{-0.025em}(\Yb), \sqrt{3\Vb_{\hspace{-0.1em} n}\hspace{-0.025em}(\Yb)}\hspace{-0.05em}\right)\hspace{-0.05em},
        \end{equation*}
        such that by Lemma \ref{LEM:generalizedBall} we know that $\Db_{\hspace{-0.025em}n}^{\hspace{-0.025em}\text{MB}}$ is $\Sigma$\hspace{0.1em}-\hspace{0.05em}measurable for all $n \in \NN$ and, therefore, $\Db^{\hspace{-0.025em}\text{MB}}$ is a stochastic process that can be used as domain approximation process.
        Hence, it remains to check whether $\Db^{\hspace{-0.025em}\text{MB}}$ satisfies the necessary assumptions of Section \ref{SEC:theoreticalAnalysis}.
        First, it is easy to see that $\Eb_{\hspace{0.015em}n}\hspace{-0.05em}(\Yb) \in [\hspace{0.025em} a, b\hspace{0.05em}]$ and $\Vb_{\hspace{-0.15em} n}\hspace{-0.05em}(\Yb) \in [\hspace{0.025em}0, (b-a)^2\hspace{0.05em}]$ for all $n \in \NN$, such that
        \begin{equation*}
            \Ab_{\hspace{0.015em}n}^{\hspace{-0.175em}\text{MB}} \,=\, \Eb_{\hspace{0.015em}n}\hspace{-0.05em}(\Yb) - \sqrt{3\Vb_{\hspace{-0.15em} n}\hspace{-0.05em}(\Yb)} \,\geq\, a - \sqrt{3\hspace{0.05em}}(b-a) \,=\, (1+\sqrt{3\hspace{0.05em}})\hspace{0.05em} a - \sqrt{3\hspace{0.05em}}b \,\geq\, -3(\lvert a\rvert + \lvert b \rvert)
        \end{equation*}
        and, analogously, $\Bb_n^\text{MB} \leq 3\hspace{0.05em}(\lvert a\rvert + \lvert b\rvert)$ for all $n \in \NN$. 
        Thus, we have
        \begin{equation*}
            \bigcup\hspace{0.2em}\{\Db_{\hspace{-0.025em}n}^{\hspace{-0.025em}\text{MB}} \,\colon n \in \NN\} \cup D \,\subseteq\, 3\hspace{0.05em}(\lvert a\rvert + \lvert b\rvert)\hspace{0.05em}[-1, 1]\hspace{0.1em},
        \end{equation*}
        such that defining the random domain extension map $\Eb^\text{MB} \equiv 3\hspace{0.025em}(\lvert a\rvert + \lvert b\rvert)[-1, 1]$ this directly implies that
        \begin{equation*}
            \PP\hspace{-0.2em}\left[\hspace{0.15em}\bigcup \,\{\Db_{\hspace{-0.025em}n}^{\hspace{-0.025em}\text{MB}}\,\colon n \in \NN\} \cup D \,\subseteq\, \Eb^\text{MB} \,\subseteq\, \dom(f)\hspace{0.1em}\right] \,=\, 1\hspace{0.1em},
        \end{equation*}
        such that $\Db^{\hspace{-0.025em}\text{MB}}$ satisfies Remark \ref{REM:weakerAssumptionKnownObjective} for $\varepsilon_0 = 0$.
        To show that $\Db^{\hspace{-0.025em}\text{MB}}$ also satisfies Assumption~\ref{ASS:hausdorffConvergenceRateAndFunctionConvergenceRate} we have to consider the next result, for which a proof can be found in Appendix \ref{app:sec5:proofs}.
        
        \begin{proposition}{}{basicApproximationProperties}
            Let $\Db^{\hspace{-0.025em}\text{MB}}$ be the domain approximation process defined as in \eqref{EQ:exampleDomainApproximationProcess}.
            Then, for all $r \in [\hspace{0.025em}0, 1/2)$ and all $\beta \in (0, 1]$ it holds that
            \begin{equation*}
                \PP\hspace{-0.2em}\left[\hspace{0.025em} d_H(\hspace{0.035em}\Db_{\hspace{-0.025em}n}^{\hspace{-0.025em}\text{MB}}, D) \,\leq\, c^\text{MB} n^{\hspace{-0.05em}-r} \text{ for all } n \in \NN\hspace{0.05em}\right] \,\geq\, 1-\beta\hspace{0.05em}, 
            \end{equation*}
            where for $s = 1-2\hspace{0.05em}r$ we have
            \begin{equation*}
                c^\text{MB} \,=\, 6\hspace{0.05em}(b-a)\hspace{-0.2em}\left(\hspace{-0.2em}2 + \hspace{-0.1em}\left(\hspace{-0.05em}\frac{8\hspace{0.05em}\Gamma(1/s)}{s\hspace{0.025em}\beta}\hspace{-0.05em}\right)^{\hspace{-0.25em} s/2}\right)^{\hspace{-0.275em}2}\hspace{-0.2em}
            \end{equation*}
            and $\Gamma$ denotes the complete Gamma function. 
        \end{proposition}

        \noindent
        By Proposition \ref{PRO:basicApproximationProperties} and since 
        \begin{equation*}
            c^\text{MB}\frac{1}{(n+1)^r} \,\leq\, c^\text{MB}\frac{2^{\hspace{0.025em}r}}{(n+2)^r} \,=\, c^\text{MB}\lambda_{\hspace{0.025em}n}^{\hspace{-0.05em}r}
        \end{equation*}
        for all $n \in \NN$, we directly obtain that
        \begin{equation*}
            \PP\hspace{-0.2em}\left[\hspace{0.025em} d_H(\hspace{0.035em}\Db_{\hspace{-0.025em}n+\hspace{-0.05em}1}^{\hspace{-0.025em}\text{MB}}, D) \,\leq\, c^\text{MB}\lambda_{\hspace{0.025em}n}^{\hspace{-0.05em}r} \text{ for all } n \in \NN\hspace{0.05em}\right] \,\geq\, 1-\beta\hspace{0.05em}, 
        \end{equation*}
        such that $\Db^{\hspace{-0.025em}\text{MB}}$ satisfies Assumption~\ref{ASS:hausdorffConvergenceRateAndFunctionConvergenceRate} with $r_1 = r, \beta_1 = \beta, c_1 = c^\text{MB}$ and $\eta_{\hspace{0.025em}1} = 0$ for all $r \in [\hspace{0.05em}0, 1/2)$ and $\beta \in (\hspace{-0.025em}0, 1]$.
        Since the complete Gamma function $\Gamma$ in general cannot be computed explicitly, we use the fact that for $m \in \NN$ it holds $\Gamma(m) = (m-1)!$ to derive a computable version of the constant $c^\text{MB}$ in Proposition~\ref{PRO:basicApproximationProperties}. 
        To this end, we choose $m \in \NN$ and $\tau \in \RR$ and set
        \begin{equation*}
            r^\text{MB} \,=\, \frac{m-1}{2\hspace{0.05em}m} \,<\, \frac{1}{2} \quad \text{ and } \quad \beta \,=\, 10^{-\tau} \in (0, 1]\hspace{0.1em},
        \end{equation*}
        such that we can simplify
        \begin{equation*}
            c^\text{MB} \,=\, 6\hspace{0.025em}(b-a)\hspace{-0.075em}\left(\hspace{-0.1em}(\hspace{-0.05em}2\cdot10^{\hspace{0.025em}\tau}\hspace{0.025em}m\hspace{0.025em}!\hspace{0.05em})^{ 1/(2\hspace{0.05em}m)}\hspace{-0.05em}+2\hspace{-0.05em}\right)^{\hspace{-0.1em}2}\hspace{-0.1em}.
        \end{equation*}
        Overall, we have shown that the MB domain approximation process $\Db^{\hspace{-0.025em}\text{MB}}$ satisfies the necessary assumption to apply Theorem \ref{THM:errorBoundWithAdditionalAssumption}, or in our case Remark~\ref{REM:adaptedConvergenceRate}.
        Computing the other constants appearing in Remark~\ref{REM:adaptedConvergenceRate}, that is, 
        \begin{itemize}
            \item $\diam(\Eb^\text{MB}) \,=\, 6\hspace{0.05em}(\lvert \hspace{0.05em}b\hspace{0.05em}\rvert + \lvert \hspace{0.05em}a\hspace{0.05em}\rvert)$

            \item $\Cb^\text{MB} \,\leq\, \diam(\Eb^\text{MB})^2L_{\nabla} \,=\, 72\hspace{0.05em}(\lvert \hspace{0.05em}b\hspace{0.05em}\rvert + \lvert \hspace{0.05em}a\hspace{0.05em}\rvert)^2$
    
            \item $\Lb^{\hspace{-0.15em}\text{MB}} \,=\, \max\hspace{0.1em}\{2\hspace{0.1em}\lvert\hspace{0.05em} x-x^\star\rvert \,\colon x \in \Eb^\text{MB}\} \,=\, 6\hspace{0.05em}(\lvert \hspace{0.05em}b\hspace{0.05em}\rvert + \lvert \hspace{0.05em}a\hspace{0.05em}\rvert) + 2\hspace{0.05em}\lvert \hspace{0.05em}x^\star\rvert$
        \end{itemize}
        we obtain that
        \begin{equation*}
            \PP\hspace{-0.15em}\left[\hspace{0.1em}\lvert \hspace{0.05em}f(\hspace{0.035em}\Xb_{\hspace{0.025em}n}^\text{MB}\hspace{-0.05em})-\Popt\rvert \,\leq\, A^{\hspace{-0.05em}\text{MB}}\lambda_{\hspace{0.025em}n}^{\hspace{-0.05em}r^\text{MB}}  \text{ for all } n \in \NN\hspace{0.1em}\right] \,\geq\, 1-\beta\hspace{0.1em},
        \end{equation*}
        where 
        \begin{equation*}
            \begin{aligned}
                A^{\hspace{-0.05em}\text{MB}} & \,=\, 2\hspace{0.05em}c^\text{MB}\Lb^{\hspace{-0.15em}\text{MB}} + \Cb^\text{MB}.
            \end{aligned}
        \end{equation*}

        \paragraph{Convex Hull Domain Approximation.} 
        For the second approach we first note that for $n \in \NN$ points $x_1, \ldots, x_n \in \RR$ we have 
        \begin{equation*}
            \convex(\{x_1,\ldots, x_n\}) \,=\, [\hspace{0.025em}\min\hspace{0.1em}\{x_1,\ldots, x_n\}, \max\hspace{0.1em}\{x_1,\ldots, x_n\}\hspace{0.025em}]\hspace{0.1em},
        \end{equation*}
        such that with the convex hull (CH) empirical estimators 
        \begin{equation*}
            \Ab_{\hspace{0.015em}n}^{\hspace{-0.2em}\text{CH}} \,=\, \min\hspace{0.1em}\{\Yb_{\hspace{-0.1em}i} \,\colon i \in [n]\} \quad \text{and} \quad \Bb_n^\text{CH} \,=\, \max\hspace{0.1em}\{\Yb_{\hspace{-0.1em}i} \,\colon i \in [n]\}\hspace{0.1em},
        \end{equation*}
        for $a$ and $b$, respectively, we can define the candidate of a domain approximation process \begin{equation}\label{EQ:exampleDomainApproximationProcessConvexHull}
            \Db^{\hspace{-0.025em}\text{CH}} \colon \Omega \times \NN \to \Dcc, \; (\omega, n) \,\mapsto\, [\hspace{0.025em}\Ab_{\hspace{0.015em}n}^{\hspace{-0.2em}\text{CH}}\hspace{-0.05em}(\omega), \Bb_n^\text{CH}\hspace{-0.05em}(\omega)\hspace{0.025em}] \,=\, \convex(\{\Yb_{\hspace{-0.1em}i}\hspace{-0.025em}(\omega) \,\colon i \in [n]\})\hspace{0.1em}.
        \end{equation}
        As before, we trivially have that $\Db_{\hspace{-0.025em}}^{\hspace{-0.025em}\text{CH}}$ has nonempty, compact, and convex values for all $n \in \NN$, such that $\Db^{\hspace{-0.025em}\text{CH}}$ is well\hspace{0.05em}-\hspace{0.035em}defined.
        Furthermore, we can rewrite
        \begin{equation*}
            \Db_{\hspace{-0.025em}n}^{\hspace{-0.025em}\text{CH}} \,=\, [\hspace{0.025em}\Ab_{\hspace{0.015em}n}^{\hspace{-0.2em}\text{CH}}, \Bb_n^\text{CH}\hspace{0.025em}] \,=\, \BB\hspace{-0.2em}\left(\hspace{-0.05em}\frac{\Ab_{\hspace{0.015em}n}^{\hspace{-0.2em}\text{CH}} + \Bb_n^{\text{CH}}}{2}, \frac{\Bb_n^{\text{CH}} - \Ab_{\hspace{0.015em}n}^{\hspace{-0.175em}\text{CH}}}{2}\hspace{-0.05em}\right)\hspace{-0.05em},
        \end{equation*}
        such that by Lemma \ref{LEM:generalizedBall} we know that $\Db_{\hspace{-0.025em}n}^{\hspace{-0.025em}\text{CH}}$ is $\Sigma$\hspace{0.1em}-\hspace{0.05em}measurable for all $n \in \NN$ and, therefore, a stochastic process that can be used as domain approximation process.
        It remains to check whether $\Db^{\hspace{-0.025em}\text{CH}}$ satisfies the necessary assumptions of Section \ref{SEC:theoreticalAnalysis}. 
        Since by construction $\Db_{\hspace{-0.025em}n}^{\hspace{-0.025em}\text{CH}}$ is a subset of $D$ for all $n \in \NN$, setting $\Eb^\text{CH} \equiv D$, we obtain that $\Db^{\hspace{-0.025em}\text{CH}}$ satisfies Remark~\ref{REM:weakerAssumptionKnownObjective} for $\varepsilon_0 = 0$.
        To show that $\Db^{\hspace{-0.025em}\text{CH}}$ also satisfies Assumption \ref{ASS:hausdorffConvergenceRateAndFunctionConvergenceRate} we have to consider the next result, whose proof is provided in Appendix \ref{app:sec5:proofs}.

        \begin{proposition}{}{convexHullApproximationQuality}
            Let $\Db^{\hspace{-0.025em}\text{CH}}$ be the domain approximation process defined as in \eqref{EQ:exampleDomainApproximationProcessConvexHull}.
            Then, for all $r \in [\hspace{0.025em}0, 1)$ and all $\beta \in (0, 1]$ it holds that
            \begin{equation*}
                \PP\hspace{-0.2em}\left[\hspace{0.025em} d_H(\hspace{0.035em}\Db_{\hspace{-0.025em}n}^{\hspace{-0.025em}\text{CH}}, D) \,\leq\, c^\text{CH}\hspace{0.05em} n^{\hspace{-0.05em}-r} \text{ for all } n \in \NN\hspace{0.05em}\right] \,\geq\, 1-\beta\hspace{0.05em}, 
            \end{equation*}
            where for $s = 1-r$ we have
            \begin{equation*}
                c^\text{CH}\,=\, (b-a) \hspace{-0.1em}\left(\hspace{-0.05em}\frac{2\hspace{0.05em}\Gamma(1/s)}{s\hspace{0.025em}\beta}\hspace{-0.05em}\right)^{\hspace{-0.25em}s}
            \end{equation*}
            and $\Gamma$ denotes the complete Gamma function.
        \end{proposition}

        \noindent
        Similar to before, from Proposition \ref{PRO:convexHullApproximationQuality} we find that
        \begin{equation*}
            \PP\hspace{-0.2em}\left[\hspace{0.025em} d_H(\hspace{0.035em}\Db_{\hspace{-0.025em}n+\hspace{-0.05em}1}^{\hspace{-0.025em}\text{CH}}, D) \,\leq\, c^\text{CH}\lambda_{\hspace{0.025em}n}^{\hspace{-0.05em}r} \text{ for all } n \in \NN\hspace{0.05em}\right] \,\geq\, 1-\beta\hspace{0.05em}, 
        \end{equation*}
        such that $\Db^{\hspace{-0.025em}\text{CH}}$ satisfies Assumption~\ref{ASS:hausdorffConvergenceRateAndFunctionConvergenceRate} with $r_1 = r, \beta_1 = \beta, c_1 = c^\text{CH}$ and $\eta_{\hspace{0.025em}1} = 0$ for all $r \in [\hspace{0.05em}0, 1)$ and $\beta \in (\hspace{-0.025em}0, 1]$.
        Again, we aim to derive a computable version of the constant $c^\text{CH}$ in Proposition~\ref{PRO:convexHullApproximationQuality}. 
        To this end, we choose $m \in \NN$ and $\tau \in \RR$ and set
        \begin{equation*}
            r^\text{CH} \,=\, \frac{m-1}{m} \,<\, 1 \quad \text{ and } \quad \beta \,=\, 10^{-\tau} \in (0, 1]\hspace{0.1em},
        \end{equation*}
        such that we can simplify
        \begin{equation*}
            c^\text{CH} \,=\, (b-a)\hspace{0.05em}(\hspace{-0.05em}2\cdot10^{\hspace{0.025em}\tau}\hspace{0.025em}m\hspace{0.025em}!\hspace{0.05em})^{ 1/m}\hspace{0.1em}.
        \end{equation*}

        \noindent
        Overall, we have shown that the CH domain approximation process $\Db^{\hspace{-0.025em}\text{CH}}$ satisfies the necessary assumption to apply Theorem \ref{THM:errorBoundWithAdditionalAssumption}, or in this case Remark~\ref{REM:adaptedConvergenceRate}.
        Computing the other constants appearing in Remark~\ref{REM:adaptedConvergenceRate}, that is, 
        \begin{itemize}
            \item $\diam(\Eb^\text{CH}) \,=\,  b-a$

            \item $\Cb^\text{CH} \,\leq\, \diam(\Eb^\text{CH})^2L_{\nabla} \,=\, 2\hspace{0.05em}(b-a)^2$

            \item $\Lb^{\hspace{-0.15em}\text{CH}} \,=\, \max\hspace{0.1em}\{2\hspace{0.1em}\lvert\hspace{0.05em} x-x^\star\rvert \,\colon x \in \Eb^\text{CH}\} \,=\, 2\hspace{0.05em}\max\hspace{0.1em}\{\lvert \hspace{0.05em}a-x^\star\rvert\hspace{0.025em}, \lvert \hspace{0.05em}b-x^\star\rvert\}$
        \end{itemize}
        we obtain that
        \begin{equation*}
            \PP\hspace{-0.15em}\left[\hspace{0.1em}\lvert \hspace{0.05em}f(\hspace{0.035em}\Xb_{\hspace{0.025em}n}^\text{CH}\hspace{-0.05em})-\Popt\rvert \,\leq\, A^{\hspace{-0.05em}\text{CH}}\lambda_{\hspace{0.025em}n}^{\hspace{-0.05em}r^\text{CH}}  \text{ for all } n \in \NN\hspace{0.1em}\right] \,\geq\, 1-\beta\hspace{0.1em},
        \end{equation*}
        where 
        \begin{equation*}
            \begin{aligned}
                A^{\hspace{-0.05em}\text{CH}} &\,=\, 2\hspace{0.05em}c^\text{CH}\Lb^{\hspace{-0.15em}\text{CH}} + \Cb^\text{CH}.
            \end{aligned}
        \end{equation*}

        \noindent
        In contrast to the moment\hspace{0.05em}-based domain approximation process $\Db^\text{MB}$, the convex hull domain approximation process $\Db^\text{CH}$ also satisfies Assumption \ref{ASS:domainContainmentExterior} for $\delta = 0$ as $\Db_{\hspace{-0.025em}n}^\text{CH} \subseteq D$ for all $n \in \NN$.
        Hence, we can also consider the accelerated convergence results from Section \ref{SEC:acceleratedConvergenceResult}, more explicitly Theorem \ref{THM:advancedConvergence}.
        Computing the remaining constants appearing in Theorem~\ref{THM:advancedConvergence}, that is, 
        \begin{itemize}
            \item $\diam(D) \,=\, b-a$

            \item $C_{\hspace{-0.05em}f} \,\leq\, \diam(D)^2L_{\nabla} \,=\, 2\hspace{0.05em}(b-a)^2$

            \item $L_{\hspace{-0.05em}f} \,=\, \max\hspace{0.1em}\{\hspace{0.05em}2\hspace{0.1em}\lvert\hspace{0.05em} x-x^\star\rvert \,\colon x \in D\} \,=\, 2\max\{\lvert\hspace{0.05em} a-x^\star\rvert, \lvert\hspace{0.05em} b-x^\star\rvert\}$

            \item $\alpha = 4\hspace{0.05em}/(b-a)$

            \item $\mu = 2$

            \item $\rho = \min\{\lvert \hspace{0.05em}a-x^\star\rvert, \lvert \hspace{0.05em}b-x^\star\rvert\}$
        \end{itemize}
        we can see that
        \begin{equation*}
            \PP\hspace{-0.2em}\left[\hspace{0.05em}0 \,\leq\,  f(\hspace{0.035em}\Xb_{\hspace{0.025em}n}\hspace{-0.05em}) - \Popt  \,\leq\, B^\text{CH}\lambda_{\hspace{0.025em}n-\hspace{-0.05em}1}^{\hspace{-0.1em}2\hspace{0.025em}r^\text{CH}} \,\text{ for all } \, n \geq N\right] \,\geq\, 1-(\beta + \pi)\hspace{0.1em},
        \end{equation*}
        where 
        \begin{equation*}
            \begin{aligned}
                B^\text{CH} &\,=\, \frac{512}{\alpha^2\rho^{2}\mu}\hspace{-0.1em}\left(3\hspace{0.05em}c^\text{CH}L_{\hspace{-0.05em}f} + C_{\hspace{-0.075em}f}\right)^{\hspace{-0.15em}2} \hspace{-0.2em}+ \frac{7}{2}\hspace{0.05em}c^\text{CH}L_{\hspace{-0.05em}f} + \frac{3}{2}C_{\hspace{-0.075em}f}
            \end{aligned}
        \end{equation*}
        and 
        \begin{equation*}
            \begin{aligned}
                N &\,=\, \left\lceil 2\hspace{-0.15em}\left(\frac{2\hspace{0.05em}B^\text{CH}}{3\hspace{0.05em}c^\text{CH}\hspace{-0.05em}L_{\hspace{-0.05em}f} + C_{\hspace{-0.075em}f}}\right)^{\hspace{-0.3em}r^\text{CH}}\hspace{-0.2em} - 2\right\rceil\hspace{-0.1em}.
            \end{aligned}
        \end{equation*}

        \paragraph{Numerical Results.} 
        First, note again that running Algorithm \ref{ALG:onlineAdaptiveFrankWolfeAlgorithm} can be interpreted as considering a single sample path of  Algorithm \ref{ALG:onlineAdaptiveStochasticFrankWolfeAlgorithm}.
        Hence, for a given sample point $\omega \in \Omega$, we compare the best\hspace{0.05em}-\hspace{0.035em}case (BC) errors given at iteration $n \in \NN$, that is,
        \begin{equation}\label{eq:ex1:best:case:error}
            \lvert \hspace{0.05em}\min\hspace{0.1em}\{f(x) \,\colon x \in \Db_{\hspace{-0.025em}n}^\text{MB}\hspace{-0.075em}(\omega)\} - \Popt \hspace{0.05em}\rvert \qquad \text{ and } \qquad \lvert \hspace{0.05em}\min\hspace{0.1em}\{f(x) \,\colon x \in \Db_{\hspace{-0.025em}n}^\text{CH}\hspace{-0.075em}(\omega)\} - \Popt \hspace{0.05em}\rvert
        \end{equation}
        with the corresponding recursive adaptive (RA) errors of Algorithm \ref{ALG:onlineAdaptiveFrankWolfeAlgorithm} at iteration $n \in \NN$, that is,
        \begin{equation}\label{eq:ex1:FW:error}
            \lvert \hspace{0.05em} f(\Xb_{\hspace{0.025em}n}^\text{MB}\hspace{-0.075em}(\omega)) - \Popt \rvert \qquad \text{ and } \qquad  \lvert \hspace{0.05em} f(\Xb_{\hspace{0.025em}n}^\text{CH}\hspace{-0.075em}(\omega)) - \Popt \rvert\hspace{0.1em},
        \end{equation}
        respectively.
        We consider the setting with the problem parameters $a = 0, b = 1$ and $x^\star = 2$, and choose $m = 20$ and $\tau = 2$.
        Thus, we can compute the convergence rates $r^\text{MB} = 0.475 < 0.5$ and $r^\text{CH} = 0.95 < 1$ as well as the constants
        \begin{equation*}
             A^{\hspace{-0.05em}\text{MB}} \,\approx\, 3578.83504 \qquad \text{ and } \qquad A^\text{CH} \,\approx\, 88.58601 \hspace{0.1em}.
        \end{equation*}
        Figure~\ref{FIG:simpleExample} shows the error comparisons for both (a) the MB domain approximation and (b) the CH domain approximation on the same scale for better comparison between both approximation methods.
        It can be seen that not only the theoretical bounds are influenced by the different approximation methods but also the convergence speed of the BC and the RA error.
        Additionally, we want to highlight the possible improvement of the convergence rate of Algorithm \ref{ALG:onlineAdaptiveFrankWolfeAlgorithm}. 
        Thus, we consider the case with $x^\star = 0.5 \in \relinterior([\hspace{0.025em}0, 1])$ and compute 
        \begin{equation*}
             B^\text{CH} \,\approx\, 304210.83314 \qquad \text{ and } \qquad N \,=\, 28472 \hspace{0.1em}.
        \end{equation*}
        In Figure \ref{FIG:simpleExample} (c) we can see that with the optimal solution in the interior of the problem domain, we can obtain a faster convergence rate, theoretically and practically. 
        Figure \ref{FIG:simpleExample} (d) shows that as soon as
        \begin{equation*}
            A^\text{CH}\lambda_{\hspace{0.025em}n}^{\hspace{-0.05em}r^\text{CH}} \,>\, B^\text{CH}\lambda_{\hspace{0.025em}n-\hspace{-0.05em}1}^{\hspace{-0.05em}2\hspace{0.025em}r^\text{CH}}
        \end{equation*}
        for some $n > N$, the improved theoretical convergence bound takes over and guarantees a convergence rate of $2\hspace{0.05em}r^\text{CH} = 1.9 < 2$.
        
        \begin{figure}[h]
            \centering
            \includegraphics[width=0.495\linewidth]{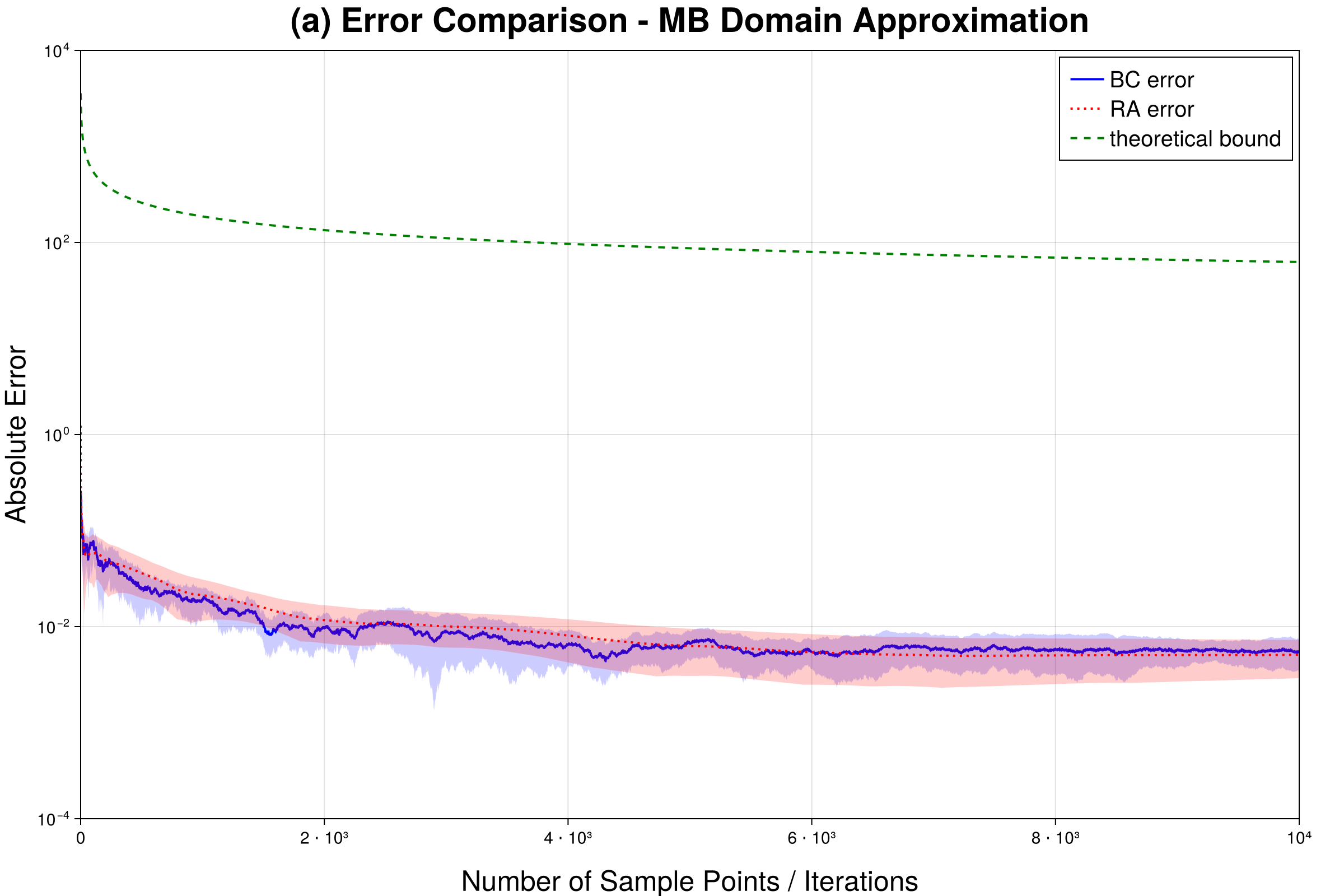}
            \includegraphics[width=0.495\linewidth]{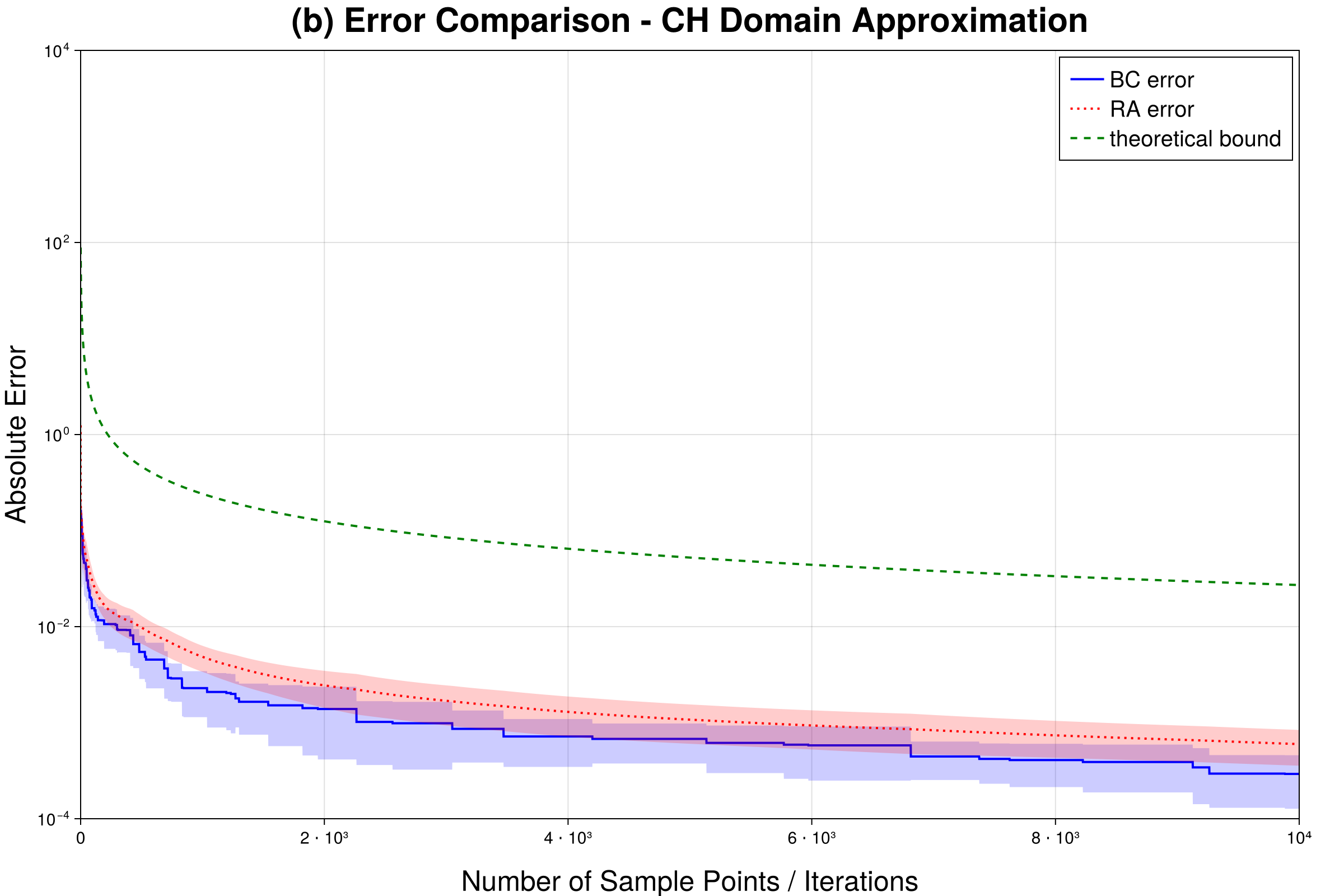}\\[0.25cm]
            \includegraphics[width=0.495\linewidth]{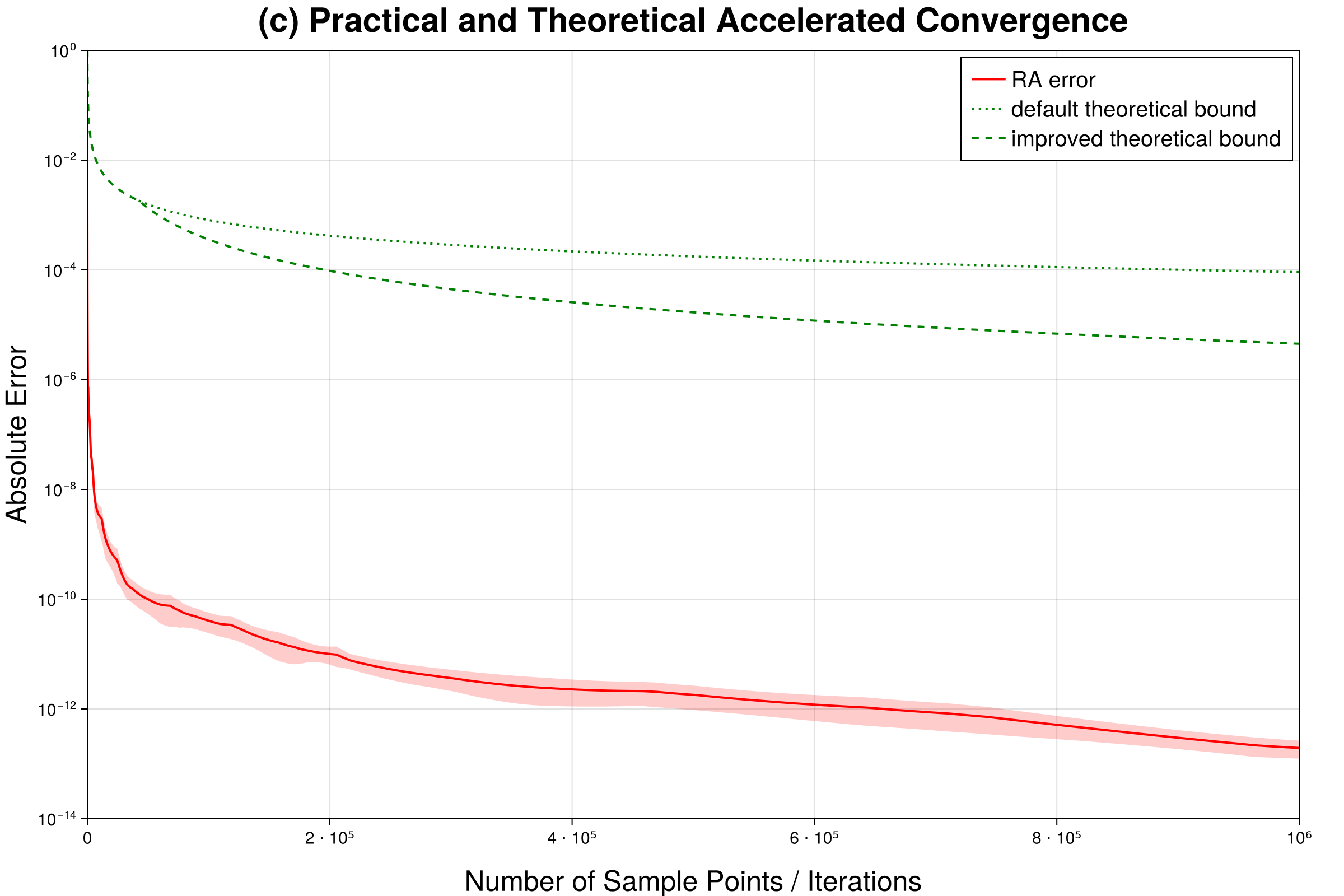}
            \includegraphics[width=0.495\linewidth]{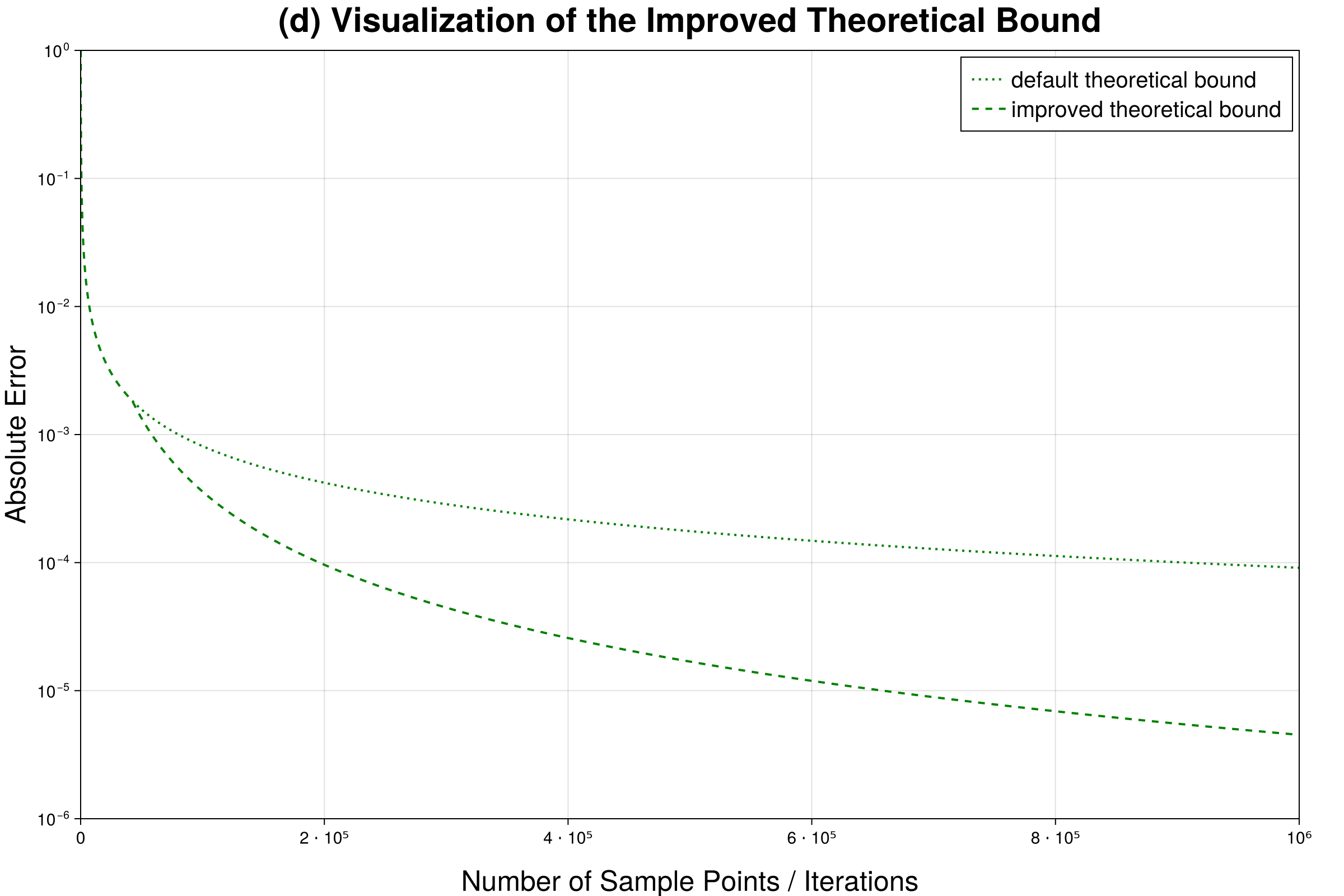 }
            \caption{Comparison of the BC error \eqref{eq:ex1:best:case:error}, the RA error \eqref{eq:ex1:FW:error}, and the theoretical upper bound trajectories for the (a) MB domain approximation and the (b) CH domain approximation.
            Depiction of the (c) practical and theoretical accelerated convergence of Theorem \ref{THM:advancedConvergence} and (d) comparison of the default and improved theoretical upper convergence bound.
            The error trajectories are computed as mean errors over 10 independent test runs and the shaded surrounding area shows two standard errors around this mean error.}
            \label{FIG:simpleExample}
        \end{figure}

    \subsection{Data\hspace{0.1em}-Driven Distributionally Robust Linear Quadratic Control Problem}

        As a second numerical example, we consider the classical Linear Quadratic Gaussian (LQG) problem, a cornerstone in optimal control and reinforcement learning, see, for example, ~\cite{anderson2007optimal, bertsekas2012dynamic} for a detailed treatment.
        We study a discrete\hspace{0.075em}-time, linear, time\hspace{0.075em}-\hspace{0.025em}varying system
        \begin{equation*}
            \Xb_{\hspace{0.025em}t+\hspace{-0.05em}1} \,=\, A_{\hspace{0.025em}t} \hspace{0.05em}\Xb_{\hspace{0.025em}t} + B_{\hspace{0.025em}t}\Ub_{\hspace{-0.025em}t} + \Wb_{\hspace{-0.1em}t}
        \end{equation*}
        for all $t \in [T-1]_0$, where $\Xb_{\hspace{0.025em}t} \in \RR^n$ are the states, $\Ub_{\hspace{-0.025em}t} \in \RR^m$ are the controls (or actions), $\Wb_{\hspace{-0.1em}t} \in \RR^n$ is the process noise, and $A_{\hspace{0.025em}t} \in \mathbb{R}^{n \times n}$ and $B_t \in \mathbb{R}^{n \times m}$ are the evolution matrices.
        Additionally, the state $\Xb_{\hspace{0.025em}t}$ is not directly observable, but instead we have access to measurements
        \begin{equation*}
            \Yb_{\hspace{-0.1em}t} \,=\, C_t\hspace{0.05em}\Xb_{\hspace{0.025em}t} + \Vb_{\hspace{-0.1em}t}
        \end{equation*}
        for all $t \in [T-1]_0$, where $\Vb_{\hspace{-0.1em}t} \in \RR^p$ is the measurement noise and $C_t \in \RR^{p \times n}$ is the measurement matrix.
        The initial condition $\Xb_0$ and the noise vectors $\Wb_{\hspace{-0.1em}t}$ and $\Vb_{\hspace{-0.1em}t}$ for all $t \in [T-1]_0$ are assumed to be normally distributed with zero mean. 
        For simplicity, we collect all noise vectors in a single noise variable $\Zb$, that is,
        \begin{equation*}
            \Zb = (\Xb_0, \Wb_{\hspace{-0.075em}0}, \ldots, \Wb_{\hspace{-0.05em}T-1}, \Vb_{\hspace{-0.075em}0}, \ldots, \Vb_{\hspace{-0.05em}T-1}) \in \Omega \,\coloneqq\, \RR^n \times (\RR^n)^T \times (\RR^p)^T
        \end{equation*}
        and store the corresponding dimension of the components of $\Zb$ in a vector $d \in \RR^{2\hspace{0.01em}T+1}$, that is, it holds $\Zb_k \in \RR^{d_k}$ for all $k \in [2\hspace{0.05em}T+1]$.
        Furthermore, we assume the components of $\Zb$ to be mutually independent, such that with
        \begin{equation*}
            \PP \,\coloneqq\, \PP_\Zb \,=\, \bigotimes_{k = 1}^{2\hspace{0.01em}T+1}\PP_{\Zb_k} \,=\, \PP_{\Xb_0} \otimes \left(\hspace{0.1em}\bigotimes_{t=0}^{T-1} \PP_{\Wb_{\hspace{-0.1em}t}}\hspace{-0.1em}\right)
            \otimes \left(\hspace{0.1em}\bigotimes_{t=0}^{T-1} \mathbb{P}_{\Vb_{\hspace{-0.1em}t}}\hspace{-0.1em}\right)
        \end{equation*}
        as the probability distribution of $\Zb$, we obtain the underlying probability space $(\Omega, \Bcc(\Omega), \PP\hspace{0.025em})$. 

        \paragraph{Robust Formulations. }
        Let $U \subseteq \RR^m$ denote the set of all feasible control inputs arising from causal policies, that is, policies where the control $\Ub_{\hspace{-0.025em}t}$ depends only on past observations $\Yb_{\hspace{-0.075em}0}, \dots, \Yb_{\hspace{-0.1em}t}$  for all $t \in [T-1]_0$. 
        For cost matrices $Q_t\in\SS^{\hspace{0.05em}n}_+$ and $R_t\in\SS^{\hspace{0.05em}m}_{++}$ for $t \in [T-1]_0$, the standard LQG problem is then
        \begin{equation} \label{eq:LQG}
         J^\star \,=\, \inf\hspace{0.1em}\left\{ 
            \hspace{0.1em}\sum_{t\hspace{0.05em}=\hspace{0.05em}0}^{T-1} 
            \EE_{\hspace{0.05em}\PP}\hspace{-0.2em}\left[ \Xb_{\hspace{0.025em}t}^\top \hspace{-0.1em}Q_t \hspace{0.05em}\Xb_{\hspace{0.025em}t} + \Ub_{\hspace{-0.025em}t}^\top \hspace{-0.1em}R_{\hspace{0.025em}t} \hspace{0.05em}\Ub_{\hspace{-0.025em}t} \right]
            + \EE_{\hspace{0.05em}\PP}\hspace{-0.2em}\left[ \Xb_T^\top Q_T \Xb_T \right] \,\colon \Ub_{\hspace{-0.025em}t} \in U\right\},
        \end{equation}
        which can be solved efficiently via the Kalman filter and dynamic programming \cite{bertsekas2012dynamic}. 
        Indeed, given the system matrices $A, B$ and $C$ and the cost matrices $Q$ and $R$ we can first solve the backward Ricatti equation
        \begin{equation*}
            P_t \,=\, A_t^\top P_{t+1}A_t + Q_+ - A_t^\top P_{t+1}B_t(R_t+B_t^\top P_{t+1}B_t)^{-1}B_t^\top P_{t+1}A_t
        \end{equation*}
        initialized at $P_T = Q_T$.
        Then, considering the covariance matrices $Z_k = \text{Var}(\Zb_k) \in \SS_+^{\hspace{0.025em}d_k}$ for all $k \in [2T+1]$ and writing 
        \begin{equation*}
            Z = (Z_1, \ldots, Z_{2\hspace{0.01em}T+1}) \in \SS_{\hspace{0.075em}\Omega} \,\coloneqq\, \SS_+^{\hspace{0.025em}n} \times (\SS_+^{\hspace{0.025em}n})^T \times (\SS_+^{\hspace{0.025em}p})^T,
        \end{equation*}
        where due to the product structure of $\PP$ we can identify $Z$ with $\text{Var}(\Zb)$, we can recursively compute the Kalman filter covariance estimations
                \begin{equation*}
            \Sigma_{\hspace{0.025em}t}(Z) \,=\, \Sigma_{\hspace{0.025em}t\hspace{0.01em}|\hspace{0.01em}t-1}(Z) - \Sigma_{\hspace{0.025em}t\hspace{0.01em}|\hspace{0.01em}t-1}(Z)C_t^\top(C_t\Sigma_{\hspace{0.025em}t\hspace{0.01em}|\hspace{0.01em}t-1}(Z)C_t^\top + Z_{T+t+2})^{-1}C_t\Sigma_{\hspace{0.025em}t\hspace{0.01em}|\hspace{0.01em}t-1}(Z),
        \end{equation*}
        and
        \begin{equation*}
           \Sigma_{t+1|t}(Z) \,=\, A_t\Sigma_{\hspace{0.025em}t}(Z)A_t^\top + Z_{t+2}
        \end{equation*}
        for all $t \in [T-1]_0$ initialized at $\Sigma_{0|-1}(Z) = Z_1$.
        Then, defining 
        \begin{equation}\label{EQ:objectiveFunctionLQG}
            f \colon \SS_{\hspace{0.075em}\Omega} \to \RR, \, Z \,\mapsto\, \sum_{t = 0}^{T-1} \trace((Q_t-P_t)\Sigma_{\hspace{0.025em}t}(Z)) + \sum_{t = 0}^T \trace(P_t\Sigma_{\hspace{0.025em}t\hspace{0.01em}|\hspace{0.01em}t-1}(Z))\hspace{0.1em},
        \end{equation}
        we have $J^\star = f(Z)$ \cite[Appendix~A]{taskesen2023distributionally}. 
        In practice, however, a control designer may wish to ensure robustness against potential misspecification of $\PP$. 
        This motivates the \emph{distributionally robust} LQG problem
        \begin{equation} \label{eq:DRLQG}
            J^\star\hspace{-0.075em}(\WW) \,\coloneqq\, \inf\hspace{0.1em}\left\{ \sup\hspace{0.1em}\left\{
            \hspace{0.1em}\sum_{t\hspace{0.05em}=\hspace{0.05em}0}^{T-1} 
            \EE_{\hspace{0.05em}\QQ}\hspace{-0.2em}\left[ \Xb_{\hspace{0.025em}t}^\top \hspace{-0.1em}Q_t \hspace{0.05em}\Xb_{\hspace{0.025em}t} + \Ub_{\hspace{-0.025em}t}^\top \hspace{-0.1em}R_{\hspace{0.025em}t} \hspace{0.05em}\Ub_{\hspace{-0.025em}t} \right]
            + \EE_{\hspace{0.05em}\QQ}\hspace{-0.2em}\left[ \Xb_T^\top Q_T \Xb_T \right] \,\colon \QQ \in \WW \right\} \,\colon \Ub_{\hspace{-0.025em}t} \in U\right\},
        \end{equation}
        where $\WW$ is an \emph{ambiguity set} containing all plausible distributions \cite{ taskesen2023distributionally}. 
        Assuming that at time $t \in \NN$ we are given an approximation $\hat{\PP}^{(t)}$ of the true distribution $\PP$ with the product structure
        \begin{equation*}
            \hat{\PP}^{(t)} \,\coloneqq\, \bigotimes_{k = 1}^{2T+1}\hat{\PP}_k^{(t)}\hspace{0.1em},
        \end{equation*}
        where each distribution $\hat{\PP}_k^{(t)}$ for $k \in [2\hspace{0.05em}T+1]$ is a centered Gaussian distribution, a possible ambiguity set $\WW_{\hspace{-0.1em}\rho}^{(t)}$ may then be constructed by first defining the Wasserstein ball
        \begin{equation*}
            \WW\hspace{-0.2em}\left(\hspace{-0.05em}\hat{\PP}_k^{(t)}\hspace{-0.1em}, \rho_k\hspace{-0.1em}\right) \,\coloneqq\, \left\{\QQ \in \mathcal{P}_2(\RR^{d_k}) \,\colon d_{\hspace{0.025em}W}\hspace{-0.3em}\left(\hat{\PP}_k^{(t)}, \QQ\right) \leq \rho_k \text{ and } \EE_{\hspace{0.025em}\QQ}[\hat{\Zb}_k] = 0 \text{ for all } \hat{\Zb}_k \sim \hat{\PP}_k^{(t)}\right\}
        \end{equation*}
        for all $k \in [2\hspace{0.05em}T+1]$ and then defining the product Wasserstein ball
        \begin{equation*}
            \WW_{\hspace{-0.1em}\rho}^{(t)} \,\coloneqq\, \WW\hspace{-0.2em}\left(\hat{\PP}^{(t)}, \rho\right) \,\coloneqq\, \prod_{k = 1}^{2T+1} \WW\hspace{-0.2em}\left(\hspace{-0.05em}\hat{\PP}_k^{(t)}\hspace{-0.1em}, \rho_k\hspace{-0.1em}\right)\hspace{-0.1em}.
        \end{equation*}
        Here, for $d \in \{n, p\}$ we denote $\mathcal{P}_2(\RR^d)$ for the set of all probability measures on $\RR^d$ with finite second moment
        and $d_{\hspace{0.025em}W}$ for the 2\hspace{0.1em}-Wasserstein distance and $\rho \in \RR^{2\hspace{0.01em}T+1}$ is a radius vector with component\hspace{0.05em}-\hspace{0.035em}wise nonnegative entries.
        It has been shown in \cite{taskesen2023distributionally} that, in the setting described, the problem \eqref{eq:DRLQG} is solved by a Gaussian product measure $\tilde{\QQ}^{(t)}$ on $\Omega$ determined by the ambiguity set $\WW_{\hspace{-0.1em}\rho}^{(t)}$ which is of the product form
        \begin{equation*}
            \tilde{\QQ}^{(t)} \,=\, \bigotimes_{k = 1}^{2\hspace{0.01em}T+1} \tilde{\QQ}_k^{(t)} \,=\, \bigotimes_{k = 1}^{2\hspace{0.01em}T+1} \mathcal{N}\hspace{-0.2em}\left(\hspace{-0.1em}0, \tilde{Z}_k^{(t)}\hspace{-0.05em}\right)\hspace{-0.1em},
        \end{equation*}
        for some covariance matrices $\tilde{Z}_k^{(t)} \in \SS_+^{d_k}$ for $k \in [2T+1]$.
        In particular, these covariance matrices can be computed by solving a maximization problem over the function $f$ defined in \eqref{EQ:objectiveFunctionLQG}, where we have to consider a matrix analogue $\GG_\rho^{(t)}$ of the product Wasserstein ball $\WW_{\hspace{-0.1em}\rho}^{(t)}$ as problem domain. 
        Specifically, instead of assuming that, for $t \in \NN$, we have access to an approximation $\hat{\PP}^{(t)}$ of the true distribution $\PP$, we assume to have access to an approximation $\hat{Z}^{(t)}$ of the covariance matrix $Z$ of $\PP$.
        That is, we can reconstruct the distribution $\hat{\PP}^{(t)}$ via
        \begin{equation*}
            \hat{\PP}_k^{(t)} = \mathcal{N}(0, \hat{Z}_k^{(t)})
        \end{equation*}
        for all $k \in [2T+1]$.
        To formulate a matrix analogue $\GG_\rho^{(t)}$ of the product Wasserstein ball $\WW_{\hspace{-0.1em}\rho}^{(t)}$ we need to define an analogue of the 2\hspace{0.1em}-Wasserstein distance, which is given by the Gelbrich distance (sometimes also called Bures distance or Bures\hspace{0.1em}-Wasserstein distance).

        \begin{definition}{Gelbrich Distance}{}
            Let $d \in \NN$.  
            The \emph{Gelbrich distance} $d_{\hspace{0.025em}G}$ on the set $\SS_+^d$ is defined by
            \begin{equation*}
                d_{\hspace{0.025em}G}(A, B) \,=\, \left(\trace(A) + \trace(B) - 2\hspace{-0.1em}\trace\hspace{-0.2em}\left(\hspace{-0.2em}\left(A^{1/2}BA^{1/2}\right)^{\hspace{-0.25em}1/2}\right)\hspace{-0.2em}\right)^{\hspace{-0.2em}1/2}
            \end{equation*}
            for $A, B \in \SS_+^{\hspace{0.025em}d}$.
        \end{definition}

        \noindent
        As for the product Wasserstein ball setting, we first define the restricted Gelbrich ball
        \begin{equation}\label{EQ:restrictedGelbrich}
            \GG\hspace{-0.2em}\left(\hspace{-0.1em}\hat{Z}_k^{(t)}\hspace{-0.1em}, \rho_k\hspace{-0.1em}\right) \,\coloneqq\, \left\{Y \in \SS_+^{\hspace{0.025em}d_k} \,\colon d_{\hspace{0.025em}G}\hspace{-0.25em}\left(\hspace{-0.1em}\hat{Z}_k^{(t)}, Y\hspace{-0.05em}\right)^{\hspace{-0.15em}2} \hspace{-0.1em}\leq \rho_k^2 \text{ and } Y \succeq \lambda_\text{min}\hspace{-0.2em}\left(\hat{\Zb}_k^{(t)}\right)\hspace{-0.2em}I\right\}
        \end{equation}
        for $k \in [2T+1]$ and then define the restricted product Gelbrich ball
        \begin{equation*}
            \GG_\rho^{(t)} \,\coloneqq\, \GG\hspace{-0.2em}\left(\hspace{-0.1em}\hat{Z}^{(t)}\hspace{-0.1em}, \rho\right):= \prod_{k \hspace{0.05em}= 1}^{2\hspace{0.05em}T+1} \GG\hspace{-0.2em}\left(\hspace{-0.05em}\hat{\Zb}_k^{(t)}\hspace{-0.1em}, \rho_k\hspace{-0.1em}\right)\hspace{-0.1em}.
        \end{equation*}
        Given the product Gelbrich ball $\GG_\rho^{(t)}$ we now can state one of the main results in \cite{taskesen2023distributionally}, namely that
        \begin{equation*}
            J^\star\hspace{-0.25em}\left(\hspace{-0.05em}\WW_{\hspace{-0.1em}\rho}^{(t)}\hspace{-0.1em}\right) \,=\, \max\hspace{0.0em}\left\{\hspace{-0.1em}f(Z) \,\colon Z \in \GG_\rho^{(t)}\hspace{-0.1em}\right\}
        \end{equation*}
        if we additionally assume that $Z_k$ is \emph{positive definite} for all $k \in \{T+2, \ldots, 2T+1\}$.
        Hence, we have derived a way to compute a solution to the distributionally robust LQG problem with ambiguity set $\WW_{\hspace{-0.1em}\rho}^{(t)}$ as a maximization problem of $f$ over the problem domain $\GG_\rho^{(t)}$ for all $t \in \NN$.
        
        \paragraph{Connection to our Setting.} In the above, we have derived a way to compute the optimal value of the distributionally robust LQG problem \eqref{eq:DRLQG} over $\WW_{\hspace{-0.1em}\rho}^{(t)}$ as a maximization problem of the function $f$ from \eqref{EQ:objectiveFunctionLQG} over the set $\GG_\rho^{(t)}$ for all $t \in \NN$. 
        Since the function $f$ is concave and smooth \cite[Proposition~4.2]{taskesen2023distributionally} we now want to apply Algorithm \ref{ALG:onlineAdaptiveFrankWolfeAlgorithm} to use an input stream of covariance matrix collection approximations to solve the limit problem 
        \begin{equation*}
            \text{maximize } \; f(X) \; \text{ subject to } \; X \in D\hspace{0.1em},
        \end{equation*}
        where $D \coloneqq \GG(Z, \rho)$ and the radius vector $\rho \in \RR^{2T+1}$ is fixed.
        Hence, theoretically, instead of assuming access to a single covariance matrix collection $\hat{Z}^{(t)}$ as before, we assume to have access to a measurable covariance matrix collection map $\hat{\Zb}^{(t)} \in \SS_{\hspace{0.075em}\Omega}$ at each time step $t \in \NN$. 
        That is, for any sample point $\omega \in \Omega$ we obtain an approximation $\hat{\Zb}^{(t)}\hspace{-0.1em}(\omega)$ of the true covariance matrix collection $Z$.
        Here, the domain approximation process $\Db^\GG$ is canonically given by our previous analysis and constructions, that is, we define
        \begin{equation}\label{EQ:gelbricgDomainApproximation}
            \Db^\GG \colon \NN \times \Omega \to \Dcc, \; (t, \omega) \,\mapsto\, \GG\hspace{-0.2em}\left(\hspace{-0.05em}\hat{\Zb}^{(t)}\hspace{-0.1em}, \rho\hspace{-0.05em}\right)\hspace{-0.1em}.
        \end{equation}
        That $\Db^\GG$ in fact is a stochastic process follows from Lemma \ref{LEM:DG:stochastic:process} in Appendix \ref{app:sec5:proofs}.
        Furthermore, the next theorem shows that if our covariance matrix collection approximations are of good quality, then this directly transfers to the domain approximations themselves, which is another motivation for using Algorithm \ref{ALG:onlineAdaptiveFrankWolfeAlgorithm}. 
        A detailed proof of the next theorem can be found in Appendix \ref{app:sec5:proofs}.

    \begin{theorem}{}{convergenceOfGelbrichBalls}
        Let $\Db^\GG$ be the domain approximation process defined as in \eqref{EQ:gelbricgDomainApproximation}.
        If there exist $\beta \in [\hspace{0.025em}0, 1)$ and $\eta \geq 0$ such that 
        \begin{equation}\label{EQ:LQGexampleAssumption}
            \PP\hspace{-0.2em}\left[\hspace{0.05em}\limsup_{t \hspace{0.05em}\to\hspace{0.05em} \hspace{0.05em}\infty} \hspace{0.25em}\lVert \hspace{0.05em}\hat{\Zb}^{(t)} - Z \hspace{0.05em} \rVert_F \,\leq\, \eta\right] \,\geq\, 1-\beta\hspace{0.1em},
        \end{equation}
        then there exists some $M \geq 0$ such that
        \begin{equation*}
            \PP\hspace{-0.2em}\left[\hspace{0.05em}\limsup_{t \hspace{0.05em}\to\hspace{0.05em} \infty} \hspace{0.25em} d_H(\hspace{0.025em}\Db_{t}^\GG, D) \,\leq\, M\hspace{-0.025em}\sqrt{\hspace{0.05em}\eta\hspace{0.1em}}\hspace{0.1em} \right] \geq 1-\beta\hspace{0.1em}.
        \end{equation*}
    \end{theorem}

        \paragraph{Numerical Results.}
        Note again that running Algorithm \ref{ALG:onlineAdaptiveFrankWolfeAlgorithm} can be interpreted as considering a single sample path of Algorithm \ref{ALG:onlineAdaptiveStochasticFrankWolfeAlgorithm}.
        Hence, for any sample point $\omega \in \Omega$ we compare the BC procedure, where at iteration $t \in \NN$ we are given the optimal value of the problem
        \begin{equation*}
            \max\hspace{-0.1em}\left\{f(X) \,\colon X \in \Db_{t}^\GG\hspace{-0.075em}(\omega)\right\}
        \end{equation*}
        with the RA procedure, where we are given the evaluation of the iterate of Algorithm \ref{ALG:onlineAdaptiveFrankWolfeAlgorithm}, that is,
        \begin{equation*}
            f\hspace{-0.2em}\left(\hspace{-0.05em}\Xb_{\hspace{0.025em}t}^\GG\hspace{-0.075em}(\omega)\hspace{-0.05em}\right)\hspace{-0.05em}.
        \end{equation*}
        We consider the setting where for $d, T \in \NN$ we construct an LQG system with $n = m = d$ and $p = \lceil d/2\hspace{0.05em}\rceil$.
        First, we define the matrices $A \in \RR^{n \times n}$ by setting $A_{ij} = 1$ if $i = j$ or $i = j-1$ for $i, j \in [n]$ and $0$ otherwise and $C \in \RR^{p \times n}$ by setting $C_{ij} = 1$ if $i = j$ for $i, j \in [p]$ and $0$ otherwise.
        We then define the system and cost matrices
        \begin{equation*}
            A_t \,=\, A, \; B_t \,=\, I_m, \; C_t \,=\, C, \;R_t \,=\, \frac{1}{m^2} \cdot I_m \quad \text{for} \quad t \in [T-1]_0 \quad \text{and} \quad Q_t \,=\, \frac{1}{n^2} \cdot I_n \quad \text{for} \quad t \in [T]_0
        \end{equation*}
        and the true covariance matrices as
        \begin{equation*}
            X_0 \,=\, I_n, \quad \text{and} \quad W_t \,=\, 0.05 \cdot I_n, \; V_t \,=\, 0.1 \cdot I_p \quad \text{for} \quad t \in [T-1]_0
        \end{equation*}
        and set $\rho \in \RR^{2T+1}$ to be $10^{-1}$ in each component. 
        We can now compare both procedures in approximating the optimal value of the problem
        \begin{equation}\label{EQ:optimalRobust}
            \text{maximize } \; f(X) \; \text{ subject to } \; X \in D\hspace{0.1em},
        \end{equation}
        where $D = \GG(Z, \rho)$, which is equal to the optimal value of the distributionally robust LQG problem \eqref{eq:DRLQG} over $\WW(\PP, \rho)$, over the incoming data stream \hspace{-0.05em}$\big(\hat{\Zb}^{(t)}\big)_{t \in \NN}$\hspace{0.05em}.
        In Figure \ref{fig:LQG} we set $d = T = 10$ and compare the (a) trajectories of the BC procedure and the RA procedure as well as the (b) error of both procedures with respect to the optimal solution to the distributionally robust limit problem \eqref{EQ:optimalRobust} for a number of sample points $N = 1000$.
        Furthermore, with Table \ref{TAB:runtime} we provide a runtime table of the BC procedure and the RA procedure which tracks the mean computation time $\pm$ one standard deviation (SD) and the minimal computation time of both procedures over all combinations $d, T \in \{5, 10, 15\}$. 
        Additionally, we provide a normalized root mean squared error (NRMSE) of the error of the RA procedure with respect to the error of the BC procedure over independent sampling rounds for some sampling set $A \subseteq \Omega$. 
        That is, we compute
        \begin{equation*}
            \left(\hspace{-0.2em}\frac{1}{N}\sum_{t \hspace{0.025em}= 1}^{N} \left(\hspace{-0.1em}M_t^\text{BC} - M_t^\text{RA}\right)^{\hspace{-0.1em}2}\right)^{\hspace{-0.3em}1/2}\hspace{-0.3em}\left(\hspace{-0.2em}\frac{1}{N}\sum_{t \hspace{0.05em}= 1}^N M_t^\text{BC}\right)^{\hspace{-0.2em}-1}\hspace{-0.5em},
        \end{equation*}
        where
        \begin{equation*}
            M_t^\text{BC} \,\coloneqq\, \frac{1}{\lvert A \rvert}\sum_{\omega \in A} \lvert \hspace{0.05em}\max\hspace{-0.1em}\left\{f(X) \,\colon X \in \Db_{t}^\GG\hspace{-0.075em}(\omega)\right\}-\Popt \hspace{0.05em}\rvert \quad \text{ and } \quad M_t^\text{RA} \,\coloneqq\, \frac{1}{\lvert A \rvert} \sum_{\omega \in A} \lvert \hspace{0.05em}f\hspace{-0.2em}\left(\hspace{-0.05em}\Xb_{\hspace{0.025em}t}^\GG\hspace{-0.075em}(\omega)\hspace{-0.05em}\right) - \Popt\hspace{0.05em}\rvert
        \end{equation*}
        are the mean values of the errors of the BC procedure and the RA procedure over $\lvert A \rvert = 10$ sampling rounds, to highlight that in spite of the benefit in computation time we have relatively similar behavior of the error trajectories, which coincides with what can be observed in Figure \ref{fig:LQG} (b).
        \vspace{0.25cm}

        \begin{figure}[h]
            \centering
            \includegraphics[width=0.49\linewidth]{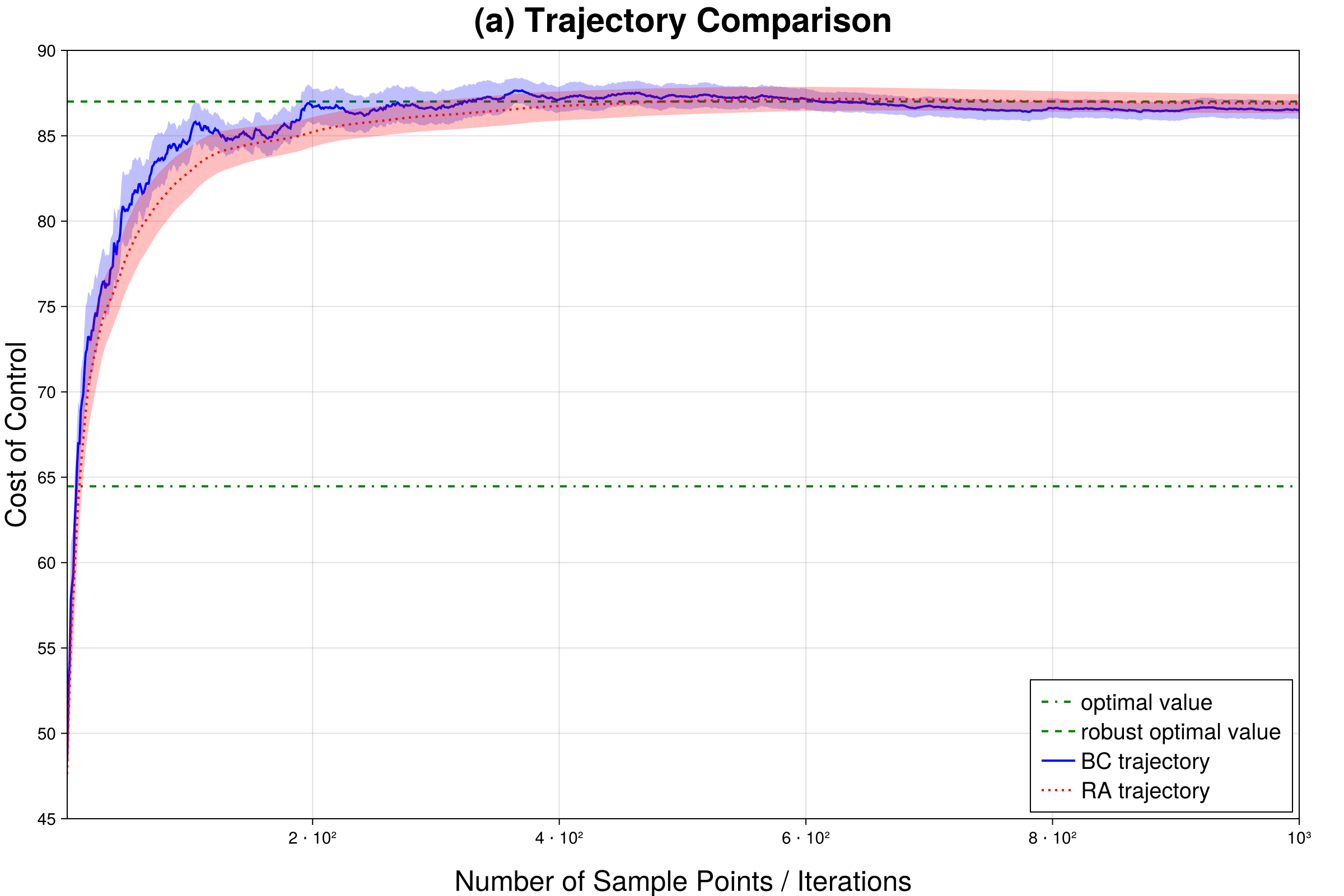}
            \includegraphics[width=0.49\linewidth]{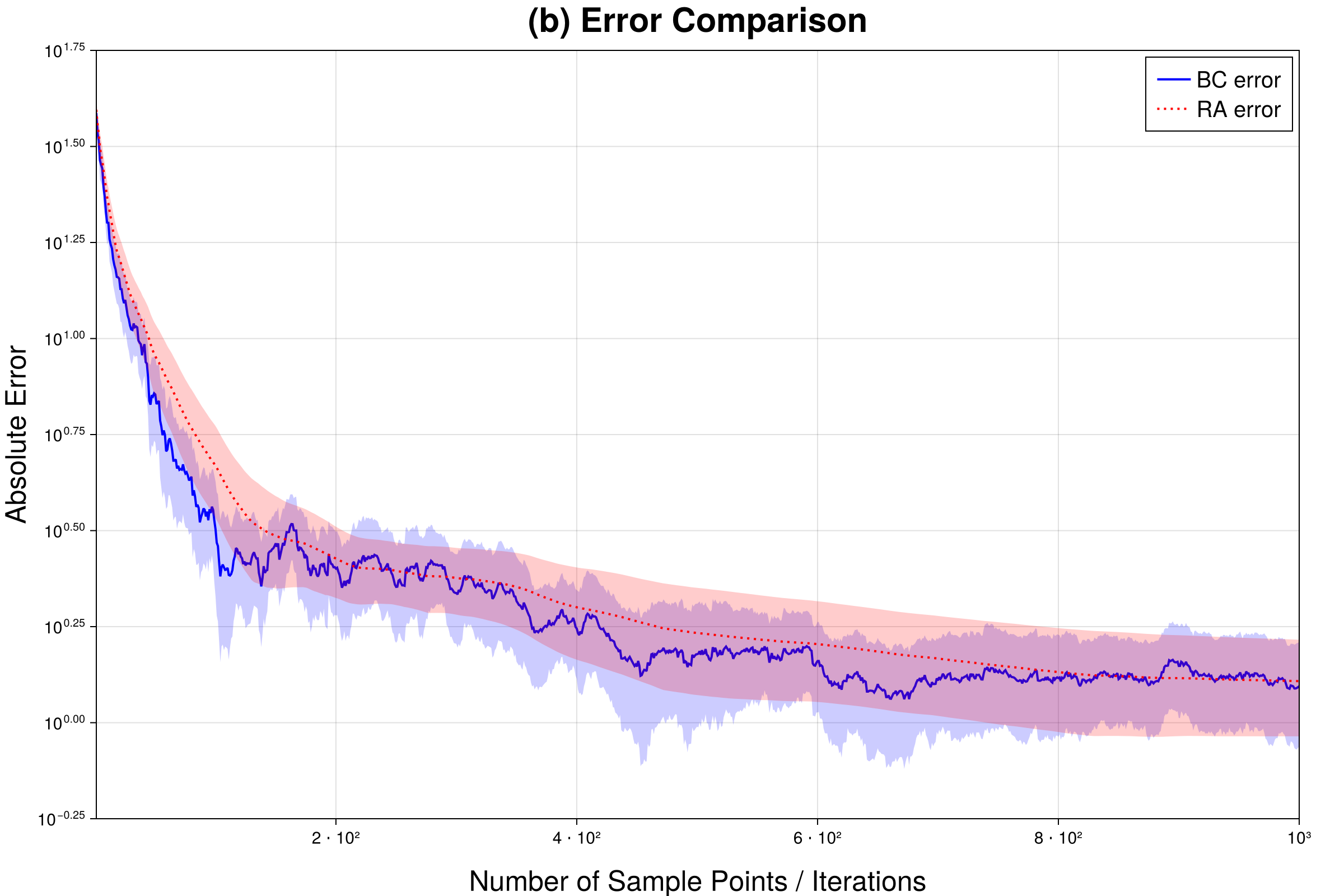}
            \caption{Comparison of the (a) control cost trajectories and the (b) error trajectories of the BC procedure and the RA procedure for approximating \eqref{EQ:optimalRobust}. 
            Both trajectories for both procedures are computed as the mean over 10 independent test runs and the shaded surrounding area shows two standard errors around these means.}
            \label{fig:LQG}
        \end{figure}

        \begin{table}[ht]
            \centering
            \renewcommand{\arraystretch}{1.15}
            \label{tab:error-comparison}
            \begin{tabular}{|c|c|c|c|c|c|c|}
                \hline
                \multirow{2}{*}{$d$} & \multirow{2}{*}{$T$} & \multicolumn{2}{|c|}{BC Procedure} 
                & \multicolumn{2}{|c|}{RA Procedure} & \multirow{2}{*}{NRMSE} \\
                \cline{3-6}
                &  & Mean $\pm$ SD & Minimum & Mean $\pm$ SD & Minimum  & \\
                \hline
                \multirow{3}{*}{$5$} & $5$  & 0.09282 $\pm$ 0.03906 & 0.05889 & 0.05185 $\pm$ 0.00241 & 0.04901 & 0.01096 \\
                \cline{2-7}
                & $10$ & 0.17910 $\pm$ 0.07796 & 0.10897 & 0.09859 $\pm$ 0.00327 & 0.09308 & 0.00851 \\
                \cline{2-7}
                & $15$ & 0.28926 $\pm$ 0.15313 & 0.15590 & 0.14648 $\pm$ 0.00345 & 0.13893 & 0.01022 \\
                \hline
                \multirow{3}{*}{$10$} & $5$ & 0.82361 $\pm$ 0.85618 & 0.22227 & 0.22892 $\pm$ 0.01642 & 0.21757 & 0.03713 \\
                \cline{2-7}
                & $10$ & 4.78202 $\pm$ 6.19723 & 0.84500 & 0.50689 $\pm$ 0.03364 & 0.48671 & 0.02825 \\
                \cline{2-7}
                & $15$ & 7.75606 $\pm$ 9.21825 & 1.23858 & 0.78101 $\pm$ 0.05832 & 0.75126 & 0.01764 \\
                \hline
                \multirow{3}{*}{$15$} & $5$ & 3.97992 $\pm$ 5.62694 & 1.01579 & 0.80105 $\pm$ 0.19034 & 0.71107 & 0.05468 \\
                \cline{2-7}
                & $10$ & 46.57378 $\pm$ 44.80941 & 6.65046 & 1.82394 $\pm$ 0.33993 & 1.67848 & 0.05865 \\
                \cline{2-7}
                & $15$ & 87.68202 $\pm$ 82.41199 & 6.15423 &  2.82033 $\pm$ 0.50042 & 2.58491 & 0.03428 \\
                \hline
            \end{tabular}
            \caption{Comparison of the mean $\pm$ SD and the minimal processing time in seconds of both the \mbox{BC} procedure and the RA procedure for all combinations of $d, T \in \{5, 10, 15\}$ and comparison of the approximation quality of the RA procedure with respect to the BC procedure using the NRMSE. 
            The mean, SD and minimum are taken with respect to a mean computation time trajectory that was computed running both procedures for 100 iterations over 10 independent test runs.}
            \label{TAB:runtime}
        \end{table}

    \printbibliography
    
    \appendix

\section{Measurable Correspondences}\label{SEC:setValuedMaps}

    In this section we briefly introduce the notion of correspondences, their measurability and some basic results on measurability and measurable selectors. 
    Great sources for this topic are the textbooks \cite{aliprantis2006infinite, rockafellar1998variational, castaing1977convex} from which we will cite most of the following results. 
    
    \paragraph{Definition and Basic Results.} 
    In the following, let $(\hspace{-0.05em}S, \Sigma_S)$ be a measurable space and $(\hspace{-0.05em}X, d_X)$ be a Polish metric space. With a correspondence (also known as multifunction or set\hspace{0.05em}-\hspace{0.025em}valued map) $\varphi$ from $S$ to $X$, we simply describe a map satisfying $\varphi(s) \subseteq X$ for all $s \in S$, in symbols $\varphi \colon S \rightrightarrows X$.
    That is, $\varphi$ simply is a map from $S$ to $\mathscr{P}(\hspace{-0.05em}X)$, where $\mathscr{P}(\hspace{-0.05em}X)$ denotes the power set of $X$.
    Since correspondences take values in subsets of $X$ rather than in points of $X$, measurability is more subtle than for ordinary functions.
    Hence, instead of defining the notion of measurability of a correspondence $\varphi$ in terms of the preimage, we use the concept of a \emph{lower inverse} of $\varphi$, which is defined as
    \begin{equation*}
        \varphi^{\hspace{0.025em}\ell}\hspace{-0.05em}(\hspace{-0.05em}A) \,=\, \{s \in S \,\colon \varphi(s) \cap A \neq \emptyset\hspace{0.05em}\}
    \end{equation*}
    for every $A \subseteq X$.
    Given the definition of a lower inverse, we can define the measurability of a correspondence in a similar way as done for usual maps.
    The two notions most relevant for the present paper are weak measurability and measurability with respect to closed sets. 
    The first is often easier to verify directly from the definition, whereas the second is more convenient in later arguments involving closed\hspace{0.05em}-\hspace{0.035em}valued correspondences and measurable graphs.

    \begin{definition}{Measurability}{}
        A correspondence $\varphi$ is called
        \begin{itemize}
            \item \emph{weakly measurable}, if $\varphi^{\hspace{0.025em}\ell}\hspace{-0.05em}(G) \in \Sigma_S$ for every \emph{open} subset $G$ of $X$.

            \item \emph{measurable}, if $\varphi^{\hspace{0.025em}\ell}\hspace{-0.05em}(F) \in \Sigma_S$ for every \emph{closed} subset $F$ of $X$.
        \end{itemize}
    \end{definition}

    \noindent 
    Although weak measurability and measurability are not identical in general, they coincide in the case that $\varphi$ has compact values, which is most relevant for this paper. 
    Since our stochastic domain approximations take values in nonempty, compact, and convex sets, this equivalence allows us to move between the two notions without additional effort.

    \begin{lemma}{}{correspondenceDefinitionsEqual}
        Let $\varphi$ be  a correspondence.  
        \begin{itemize}
            \item If $\varphi$ is measurable, then it is also weakly measurable. 

            \item If $\varphi$ is compact valued and weakly measurable, then it is also measurable.
        \end{itemize}
    \end{lemma}

    \noindent 
    Since the values of correspondences are itself sets, it is natural to question whether combining correspondences with set theoretic operations as closure, convex hull, product, union or intersection preserves measurability.
    This is given by the next result, which is a collection of multiple smaller results in \cite{aliprantis2006infinite, rockafellar1998variational}.

    \begin{lemma}{}{correspondenceSetOperations}
        Let $\varphi$ be a correspondence and  $(\varphi_n\hspace{-0.05em})_{n \in \NN}$ be a sequence of correspondences.
        \begin{itemize}
            \item The closure correspondence $s  \mapsto\cl(\varphi(s))$ is weakly measurable if and only if $\varphi$ is weakly measurable.

            \item If $\varphi$ is weakly measurable, then the convex hull correspondence $s \mapsto \convex(\varphi(s))$ is weakly measurable.

            \item Let $N \in \NN$. 
            If $\varphi_1, \ldots, \varphi_{\hspace{-0.025em}N}$  are weakly measurable, then the product correspondence 
            \begin{equation*}
                S \rightrightarrows H^{\hspace{-0.025em}N}, \; s \,\mapsto\, \varphi_1(s) \times \ldots \times \varphi_{\hspace{-0.025em}N}(s)
            \end{equation*}
            is weakly measurable.
            
            \item If $\varphi_n$ is (weakly) measurable for all $n \in \NN$, then the union correspondence 
            \begin{equation*}
               s \,\mapsto\, \bigcup\hspace{0.2em}\{\varphi_n(s)\,\colon n \in \NN\hspace{0.05em}\}
            \end{equation*}
            is (weakly) measurable.
            
            \item If $\varphi_n$ has closed values and is (weakly) measurable for all $n \in \NN$ and if for each $s \in S$ there exists some $k \in \NN$ such that $\varphi_k(s)$ is compact, then the intersection correspondence
            \begin{equation*}
                s \,\mapsto\, \bigcap\hspace{0.1em}\{\varphi_n(s)\,\colon n \in \NN\}
            \end{equation*}
            is measurable.
        \end{itemize}
    \end{lemma}

    \noindent 
    Beyond purely set\hspace{0.05em}-\hspace{0.05em}theoretic operations, measurable correspondences also arise through inverse images of jointly measurable and continuous maps. 
    The next result is therefore a basic device for deriving measurability of feasible set correspondences generated by Carathéodory functions.

    \begin{lemma}{}{correspondenceOpenSet}
        Let $Y$ be a topological space and $f \colon S \times X \to Y$ be a Carathéodory function.
        Then, for each open subset $G$ of $Y$ the correspondence $s \mapsto \{x \in X \,\colon f(s, x) \in G\hspace{0.05em}\}$ is measurable.
    \end{lemma}

    \paragraph{Measurability via Distance Functions.} Considering a nonempty valued correspondence $\varphi$, we can consider the associated distance function
    \begin{equation*}
        \delta \colon S \times X \to \RR,\; (s, x) \mapsto d_X(x, \varphi(s))\hspace{0.1em}.
    \end{equation*}
    The next theorem connects the weak measurability of $\varphi$ to properties of the induced distance function $\delta$, which is useful since $\delta$ simply is a scalar function. 
    A proof of this result is provided in \cite[Theorem 18.5]{aliprantis2006infinite}

    \begin{theorem}{}{measurabilityViaDistanceFunction}
        Let $\varphi$ be a correspondence with nonempty values. Then, $\varphi$ is weakly measurable if and only if its corresponding distance function $\delta$ is a Carathéodory function.
    \end{theorem}

    \paragraph{Correspondences with Measurable Graph.} For a given function $g \colon S \to X$, we know that $g$ is measurable if and only if its graph 
    \begin{equation*}
        \graph(g) \,=\, \{(s, x) \in S \times X\,\colon g(s) = x\} \,\subseteq\, S \times X
    \end{equation*}
    is measurable, that is, an element of the product $\sigma$\hspace{0.1em}-\hspace{0.05em}algebra $\Sigma_S \otimes \Bcc(X)$. 
    For a correspondence $\varphi$ and its counterpart of a graph
    \begin{equation*}
        \graph(\varphi) \,=\, \{(s, x) \in S \times X \,\colon x \in \varphi(s)\} \,\subseteq\, S \times X
    \end{equation*}
    such an equivalence does not hold.
    Nevertheless, weak measurability is still strong enough to imply measurable graph properties after passing to the closure correspondence.
    A proof of this next result can be found in \hbox{\cite[Theorem 18.6]{aliprantis2006infinite}}.
    
    \begin{theorem}{}{measurableGraph}
        Let $\varphi$ be a weakly measurable correspondence with nonempty values. Then, the closure correspondence $\cl(\varphi)$ has a measurable graph, that is, $\graph(\varphi) \in \Sigma_S \otimes \Bcc(X)$.
    \end{theorem}

    \paragraph{Measurable Correspondences as Measurable Functions.} As explained in more detail in Section \ref{SEC:stochasticFormulationAndAlgorithm} for the Hilbert space $H$ instead of the Polish metric space $X$, denoting with $\Kcc$ the set of all nonempty and compact subsets of $X$, we obtain a measurable space $(\Kcc, \Bcc(\Kcc))$. 
    The next result now shows that given a (weakly) measurable correspondence $\varphi$, we can derive measurability of its induced function on $\Kcc$.
    This is of particular interest since we have to connect the two different modeling approaches from Section \ref{SEC:stochasticFormulationAndAlgorithm} and Appendix \ref{SEC:additionalProofs} in the case where $\varphi$ has nonempty and compact values.
    A proof of this theorem is provided in
    \cite[Theorem 18.10]{aliprantis2006infinite}.

    \begin{theorem}{Equivalence of Measurabilities}{correspondenceEquivalences}
        Let $\varphi$ be  a correspondence with nonempty and compact values. 
        Then, the following are equivalent:
        \begin{itemize}
            \item The correspondence $\varphi$ is weakly measurable.

            \item The correspondence $\varphi$ is measurable.

            \item The function $S \to \Kcc, \, s \mapsto \varphi(s)$ is $\Bcc(\Kcc)$\hspace{0.1em}-\hspace{0.05em}measurable.
        \end{itemize}
    \end{theorem}

    \paragraph{Measurable Selectors and the Measurable Maximum Theorem.} A question that arises naturally in the context of correspondences is whether there exists a measurable selector for a given (weakly) measurable correspondence $\varphi \colon S \rightrightarrows X$, that is, a $\Sigma_S$\hspace{0.1em}-\hspace{0.05em}measurable map $f \colon S \to X$ satisfying $f(s) \in \varphi(s)$ for all \mbox{$s \in S$}.
    One such case, for example, is when the correspondence of interest is an optimizer set produced by a Carath\'eodory objective over a measurable compact\hspace{0.05em}-\hspace{0.035em}valued correspondence.
    In this case, the following theorem yields three facts simultaneously: Measurability of the optimal value function, measurability of the optimizer correspondence, and existence of a measurable selector. 
    It can be regarded as a measurable counterpart of the famous Berge Maximum Theorem \hbox{\cite[Chapter 6.3]{berge1963topological}} and a proof of it can be found in \cite[Theorem~18.19]{aliprantis2006infinite}
    
    \begin{theorem}{Measurable Maximum Theorem}{measurableMaximumTheorem}
        Let $\varphi$ be a (weakly) measurable correspondence with nonempty and compact values and $g \colon S \times X \to \RR$ be a Carathéodory function.
        Then, the optimal value function 
        \begin{equation*}
            m \colon S \to \RR, \; s \,\mapsto \max\hspace{0.1em}\{\hspace{0.05em}g(s, x) \,\colon x \in \varphi(s)\hspace{0.05em}\}
        \end{equation*}
        is $\Sigma_S$\hspace{0.1em}-\hspace{0.05em}measurable and the corresponding optimal solution correspondence
        \begin{equation*}
            \mu \colon S \rightrightarrows X, \; s \,\mapsto\, \argmax\hspace{0.1em}\{\hspace{0.05em}g(s, x) \,\colon x \in \varphi(s)\hspace{0.05em}\}
        \end{equation*}
        attains nonempty and compact values, is measurable, and admits a measurable selector.
    \end{theorem}

\section{Additional Proofs}\label{SEC:additionalProofs}

    Here, we provide some of the more technical proofs that occurred throughout the paper. 
    Some of the following proofs strongly rely on the concepts of correspondences and results on their measurability.
    Hence, it is suggested to take a look at Appendix \ref{SEC:setValuedMaps} or the textbooks \cite{aliprantis2006infinite, rockafellar1998variational, castaing1977convex} before continuing. \\

    \noindent
    To begin with, we point out a fundamental connection between our setting and Appendix~\ref{SEC:setValuedMaps}.
    Recall that from Section \ref{SEC:stochasticFormulationAndAlgorithm} we are given the complete probability space $(\Omega, \Sigma, \PP)$, the Polish metric space $(\hspace{-0.05em}H, h)$, and the measurable space $(\hspace{-0.05em}\Dcc, \Bcc(\hspace{-0.05em}\Dcc))$. 
    Then, since $(\hspace{-0.05em}H, h)$ is separable, due to Theorem \ref{THM:correspondenceEquivalences} we have a bijection between the $\Sigma$\hspace{0.1em}-\hspace{0.05em}measurable maps from $\Omega$ to $\Dcc$ and the measurable correspondences from $\Omega$ to $H$ with nonempty, compact and convex values. 
    Therefore, for any domain approximation process $\Db$, the maps $\Db_{\hspace{-0.025em}n}$ for all $n \in \NN$, or the random extension domain $\Eb$ may be seen as measurable correspondences, such that all notions and results from Appendix~\ref{SEC:setValuedMaps} are applicable.
    In particular, we note that for any $\Sigma$\hspace{0.1em}-\hspace{0.05em}measurable map $\varphi \colon \Omega \to \Dcc$ and any open or closed subset $G \subseteq H$ we have 
    \begin{equation}\label{EQ:correspondenceInclusionMeasurable}
        \Omega \hspace{-0.05em}\setminus\hspace{-0.1em} \{\varphi \subseteq G\hspace{0.05em}\} \,=\, \{\hspace{0.05em}\varphi \cap (H \hspace{-0.05em}\setminus\hspace{-0.1em} G) \neq \emptyset\hspace{0.05em}\} \,=\, \varphi^{\hspace{0.025em}\ell}\hspace{-0.05em}(H \hspace{-0.05em}\setminus\hspace{-0.1em} G) \in \Sigma\hspace{0.1em},
    \end{equation}
    such that implicitly $\{\varphi \subseteq G\} \in \Sigma$.
    Furthermore, we include a useful technical result.
    
    \begin{lemma}{}{convexCombinationConstantMeasurable}
        Let $\varphi \colon \Omega \rightrightarrows H$ be a measurable correspondence, $A \in \Sigma$ and $G \subseteq H$. 
        Then, the correspondence
        \begin{equation*}
            \varphi_{\hspace{-0.05em}A, G} \colon \Omega \rightrightarrows H, \; \omega \,\mapsto\, \ind_{\hspace{-0.05em}A}\hspace{-0.05em}(\omega)\hspace{0.075em}\varphi(\omega) + (1-\ind_{\hspace{-0.05em}A}\hspace{-0.05em}(\omega))\hspace{0.075em}G
        \end{equation*}
        is measurable.
    \end{lemma}
    \begin{proof}[\textcolor{seeblau}{Proof.}]
        Let $F \subseteq H$ be closed. 
        Then, 
        \begin{equation*}
            \begin{aligned}
                \varphi_{\hspace{-0.1em}A, G}^{\hspace{0.025em}\ell}(F) & \,=\, \{\omega \in \Omega \,\colon \varphi_{\hspace{-0.05em}A, G}(\omega) \cap F \neq \emptyset\hspace{0.05em}\} \cap (A \cup (\Omega \hspace{-0.05em}\setminus\hspace{-0.1em}A)) \\
                & \,=\, (\{\omega \in \Omega \,\colon \varphi(\omega) \cap F \neq \emptyset\hspace{0.05em}\} \cap A) \cup (\{\omega \in \Omega \,\colon G \cap F \neq \emptyset\hspace{0.05em}\}\cap (\Omega \hspace{-0.05em}\setminus\hspace{-0.1em}A)) \\
                & \,=\, (\varphi^{\hspace{0.025em}\ell}\hspace{-0.05em}(F) \cap A) \cup 
                \left\{\hspace{-0.35em}
                \begin{array}{cl}
                    \Omega \hspace{-0.05em}\setminus\hspace{-0.1em}A & \text{if } G \cap F \neq \emptyset \\
                    \hspace{0.1em}\emptyset & \text{if } G \cap F = \emptyset 
                \end{array}\hspace{-0.35em}\right\} \in \Sigma\hspace{0.1em},
            \end{aligned}
        \end{equation*}
        since $\varphi^{\hspace{0.025em}\ell}\hspace{-0.05em}(F) \in \Sigma$ due to the measurability of $\varphi$ and $A \in \Sigma$ by assumption.
    \end{proof}
    
    \subsection{Proofs of Section 2}\label{SEC:B.1}

        \paragraph{Stochastic Approximation Processes.} \textcolor{white}{.}

        \begin{proof}[\textcolor{seeblau}{Proof of Remark \ref{REM:weakerAssumptionKnownObjective}}]
            We first show that the set $N_0^+ \subseteq \Omega$ defined as in \eqref{EQ:eventN0} is measurable by showing that the sets $\{\Db_{\hspace{-0.025em}n} \hspace{-0.05em}\subseteq \Eb\hspace{0.025em}\} \subseteq \Omega$ are measurable for all $n \in \NN$.
            To this end, we fix $n \in \NN$ and note that for any $\omega \in \Omega$ it holds that
            \begin{equation}\label{EQ:equvalenceInclusionMeasurability}
                \Db_{\hspace{-0.025em}n}\hspace{-0.05em}(\omega) \subseteq \hspace{0.05em}\Eb(\omega) \quad \Longleftrightarrow \quad \Db_{\hspace{-0.025em}n}\hspace{-0.05em}(\omega) \cap (\Omega \hspace{-0.05em}\setminus\hspace{-0.1em} \Eb(\omega)) = \emptyset\hspace{0.1em}.
            \end{equation}
            Define the set 
            \begin{equation*}
                A \,=\, \{(\omega, x) \in \Omega \times H \,\colon x \in \Db_{\hspace{-0.025em}n}\hspace{-0.05em}(\omega) \text{ and } x \notin \Eb(\omega)\} \,=\, \graph(\Db_{\hspace{-0.025em}n}) \cap (\Omega \times H \hspace{-0.05em}\setminus\hspace{-0.1em}\hspace{-0.05em}\graph(\Eb)) \,\subseteq\, \Omega \times H
            \end{equation*}
            and note that $\Db_{\hspace{-0.025em}n}$ and $\Eb$ can be interpreted as closed valued measurable correspondences, hence $A \in \Sigma \otimes \Bcc(\hspace{-0.05em}H)$ due to Theorem \ref{THM:measurableGraph}.
            Using \eqref{EQ:equvalenceInclusionMeasurability} we can now find that
            \begin{equation}\label{EQ:connectionProjection}
                \Omega \hspace{-0.05em}\setminus\hspace{-0.1em} \{\Db_{\hspace{-0.025em}n} \hspace{-0.05em}\subseteq \Eb\hspace{0.025em}\} \,=\, \{\omega \in \Omega \,\colon \exists \hspace{0.05em}x \in H \hspace{0.1em}\colon x \in \Db_{\hspace{-0.025em}n}\hspace{-0.05em}(\omega) \text{ and } x \notin \Eb(\omega)\} \,=\, \pi(A)\hspace{0.1em},
            \end{equation}
            where $\pi \colon \Omega \times H \to \Omega$ denotes the projection onto the first component.
            Since by assumption $(\Omega, \Sigma, \PP)$ is complete probability space, we have that projections of $\Sigma \otimes \Bcc(\hspace{-0.05em}H)$\hspace{0.1em}-\hspace{0.05em}measurable sets are $\Sigma$\hspace{0.1em}-\hspace{0.05em}measurable, such that overall with \eqref{EQ:connectionProjection} we find $\{\Db_{\hspace{-0.025em}n} \hspace{-0.05em}\subseteq \Eb\hspace{0.025em}\} \in \Sigma$. 
            Furthermore, by \eqref{EQ:correspondenceInclusionMeasurable} we have that $\{\Eb \subseteq \dom(f)\} \in \Sigma$ and
            \begin{equation*}
                \{(H \hspace{-0.05em}\setminus\hspace{-0.1em} \Eb)\cap F \neq \emptyset\hspace{0.05em}\} \,=\, \{F \hspace{-0.05em}\setminus\hspace{-0.1em} (\Eb \cap F) \neq \emptyset\hspace{0.05em}\} = \{\Eb \not\subseteq F\} \in \Sigma
            \end{equation*}
            for any closed subset $F \subseteq H$.
            Hence, the complement correspondence $H \hspace{-0.05em}\setminus\hspace{-0.1em} \Eb$ is measurable, such that again by \eqref{EQ:correspondenceInclusionMeasurable} we have
            \begin{equation*}
                \{\hspace{-0.035em}D \subseteq \Eb\hspace{0.025em}\} \,=\, \{(H \hspace{-0.05em}\setminus\hspace{-0.1em} \Eb) \subseteq (H \hspace{-0.05em}\setminus\hspace{-0.1em} D)\} \in \Sigma\hspace{0.1em}.
            \end{equation*}
            Overall we now obtain that 
            \begin{equation*}
                N_0^+ \,=\, \bigcap\hspace{0.2em}\{\{\hspace{0.025em}\Db_{\hspace{-0.025em}n} \hspace{-0.05em}\subseteq \Eb\hspace{0.025em}\} \,\colon n \in \NN\hspace{0.025em}\} \,\cap\, \{\hspace{-0.035em}D \subseteq \Eb\hspace{0.025em}\} \,\cap\, \{\Eb \subseteq \dom(f)\}\in \Sigma\hspace{0.1em}.
            \end{equation*}
            For the second part of the statement we again note that $\Db_{\hspace{-0.025em}n}$ can be interpreted as a measurable correspondence for all $n \in \NN$, such that using Lemma \ref{LEM:correspondenceSetOperations}, we obtain that the union correspondence 
            \begin{equation*}
                \varphi \colon \Omega \rightrightarrows H, \; \omega \,\mapsto\, \bigcup\hspace{0.2em}\{\hspace{0.025em}\Db_{\hspace{-0.025em}n}\hspace{-0.05em}(\omega) \,\colon n \in \NN\hspace{0.025em}\} \cup D
            \end{equation*}
            is measurable.
            Hence, by Lemma \ref{LEM:correspondenceSetOperations} we also obtain that the closure correspondence $\cl(\varphi)$ and the closed convex hull correspondence $\convex(\cl(\varphi))$ are measurable correspondences.
            Furthermore, under Assumption \ref{ASS:hausdorffConvergence}, we know that
            \begin{equation}\label{EQ:unionIsBounded}
                \PP\hspace{-0.2em}\left[\hspace{0.075em}\bigcup\hspace{0.2em} \{\Db_{\hspace{-0.025em}n} \hspace{0.05em}\colon n \in \NN\hspace{0.025em}\} \text{ is bounded}\hspace{0.05em}\right]  \,\geq\, \PP\hspace{-0.2em}\left[\hspace{0.05em}\limsup_{n \hspace{0.05em}\to\hspace{0.05em} \infty} \hspace{0.1em} d_H(\hspace{0.035em}\Db_{\hspace{-0.025em}n}\hspace{0.05em}, D) \,\leq\, \eta_{\hspace{0.025em} 1}\hspace{-0.05em}\right] \,\geq\, 1-\beta_{\hspace{0.025em}1}\hspace{0.1em}.
            \end{equation}
            Thus, denoting with $B$ the measured event in \eqref{EQ:unionIsBounded}, we can consider the correspondence
            \begin{equation}\label{EQ:finalCorrespondence}
                \omega \,\mapsto\, \ind_{\hspace{-0.05em}B}\hspace{-0.025em}(\omega)\hspace{-0.025em}\convex(\cl(\varphi))(\omega) + (1-\ind_{\hspace{-0.05em}B}\hspace{-0.025em}(\omega))\hspace{0.075em}\{0\}\hspace{0.1em},
            \end{equation}
            which is measurable by Lemma \ref{LEM:convexCombinationConstantMeasurable}.
            Since the correspondence \eqref{EQ:finalCorrespondence} now attains nonempty, compact, and convex values it can be interpreted as a $\Sigma$\hspace{0.1em}-\hspace{0.05em}measurable map from $\Omega$ to $\Dcc$ satisfying \eqref{EQ:guaranteedExistance} and the claim follows.
        \end{proof}
        
        \paragraph{Stochastic Algorithm Definition and Interpretation.} \textcolor{white}{.}

        \begin{proof}[\color{seeblau}{Proof of Remark \ref{REM:setAdaptation}}]
            
            Since $\Xb_0$ is $\Sigma$\hspace{0.1em}-\hspace{0.05em}measurable we have $\graph(\Xb_0\hspace{-0.05em}) \in \Sigma \otimes \Bcc(\hspace{-0.05em}H)$ and since $\Eb$ can be interpreted as a closed valued measurable correspondence, by Theorem \ref{THM:measurableGraph}, we also have that $\graph(\Eb\hspace{0.025em}) \in \Sigma \otimes \Bcc(\hspace{-0.05em}H)$.
            Writing
            \begin{equation*}
                \begin{aligned}
                    \{\Xb_{\hspace{0.025em}0} \in \Eb\hspace{0.025em}\} \,=\,  \{\omega \in \Omega \,\colon \exists \hspace{0.05em}x \in H \colon \Xb_{\hspace{0.025em}0}\hspace{-0.05em}(\omega) = x \text{ and } x \in \Eb(\omega)\} \,=\, \pi(\graph(\Xb_0\hspace{-0.05em}) \cap \graph(\Eb\hspace{0.025em}))\hspace{0.1em},
                \end{aligned}
            \end{equation*}
            where $\pi \colon \Omega \times H \to \Omega$ denotes the projection onto the first component, now yields the claim since by the completeness of $(\Omega, \Sigma, \PP)$ we have that projections of $\Sigma \otimes \Bcc(\hspace{-0.05em}H)$\hspace{0.1em}-\hspace{0.05em}measurable sets are $\Sigma$\hspace{0.1em}-\hspace{0.05em}measurable.
        \end{proof}

        \begin{proof}[\textcolor{seeblau}{Proof of Lemma \ref{LEM:subproblemSelector}}]
            Let $n  \in \NN$ be fixed. 
            We consider the measurable space $(\Omega, \Sigma)$ as well as the Polish space $H$ and define the correspondence
            \begin{equation*}
                \varphi\colon \Omega \,\rightrightarrows\, H, \; \omega \,\mapsto\, \Db_{\hspace{-0.025em}n}\hspace{-0.05em}(\omega)
            \end{equation*}
            as well as the function
            \begin{equation*}
                g \colon \Omega \times H \to\, \RR, \; (\omega, s) \,\mapsto\, -\ind_{\hspace{-0.05em} A}(\omega)\langle \hspace{0.05em}s \,|\, \nabla \Fb_{\hspace{-0.05em} n}\hspace{-0.05em}(\omega)(\Yb(\omega)) \hspace{0.05em}\rangle\hspace{0.1em}.
            \end{equation*}
            The correspondence $\varphi = \Db_{\hspace{-0.025em}n}$ is weakly measurable and admits nonempty and compact values. 
            On the other hand, we know that $A \in \Sigma$, such that $\ind_{\hspace{-0.05em}A}$ is $\Sigma$\hspace{0.1em}-\hspace{0.05em}measurable and since $\nabla \Fb_{\hspace{-0.075em} n}\hspace{-0.05em}(\omega)(\Yb(\omega))$ is well\hspace{0.05em}-\hspace{0.035em}defined for all $\omega \in A$ this yields that $\omega \mapsto g(\omega, s)$ is $\Sigma$\hspace{0.1em}-\hspace{0.05em}measurable for all $s \in H$. 
            Furthermore, for fixed $\omega \in \Omega$, we have that $s \mapsto g(\omega, s)$ is continuous, such that $g$ is a Carathéodory function.
            Hence, we can apply Theorem~\ref{THM:measurableMaximumTheorem} to obtain that the map 
            \begin{equation*}
                \omega \,\mapsto\, - \max \hspace{0.1em} \{g(\omega, s) \,\colon s \in \varphi(\omega)\} \,=\, \min \hspace{0.1em}\{\ind_{\hspace{-0.05em} A}\hspace{-0.05em}(\omega)\hspace{0.1em}\langle \hspace{0.05em} s \,|\, \nabla \Fb_{\hspace{-0.075em} n}\hspace{-0.05em}(\omega)(\Yb(\omega))\hspace{0.05em}\rangle \,\colon s \in \Db_{\hspace{0.025em}n}\hspace{-0.05em}(\omega)\}
            \end{equation*}
            is $\Sigma$\hspace{0.1em}-\hspace{0.05em}measurable and its corresponding optimal solution correspondence
            \begin{equation*}
                \Omega \,\rightrightarrows\, H, \,\, \omega \,\mapsto\, \argmin \hspace{0.1em}\{\ind_{\hspace{-0.05em} A}\hspace{-0.05em}(\omega)\hspace{0.1em}\langle \hspace{0.05em}s \,|\, \nabla \Fb_{\hspace{-0.1em} n}\hspace{-0.05em}(\omega)(\Yb(\omega))\hspace{0.05em}\rangle \,\colon s \in \Db_{\hspace{0.025em}n}\hspace{0.05em}(\omega)\}
            \end{equation*}
            is measurable and admits a $\Sigma$\hspace{0.1em}-\hspace{0.05em}measurable selector.
        \end{proof}

    \subsection{Proofs of Section 3}\label{SEC:B.2}

        \paragraph{Stochastic Convergence Analysis Tools.}  \textcolor{white}{.}

        \begin{proof}[\textcolor{seeblau}{Proof of Remark \ref{REM:adaptedCurvatureConstant}}]
            That $\Cb$ is well\hspace{0.05em}-\hspace{0.035em}defined follows from Remark~\ref{REM:weakerAssumptionKnownObjective} together with the indicator function $\ind_{\hspace{-0.05em}N_0}$ and the convexity of the objective function $f$. 
            To prove the $\Sigma$\hspace{0.1em}-\hspace{0.05em}measurability of $\Cb$ we use Theorem~\ref{THM:measurableMaximumTheorem}.
            We consider the measurable space $(\Omega, \Sigma)$ as well as the Polish space $X_{\hspace{-0.025em}n} = \dom(f)^2 \times [1/n, 1]$ for some fixed $n \in \NN$ and define the correspondence
            \begin{equation*}
                \varphi_n \colon \Omega \,\rightrightarrows\, X_{\hspace{-0.025em}n}, \; \omega \,\mapsto\, \ind_{\hspace{-0.05em}N_0}\hspace{-0.05em}(\omega)(\hspace{0.05em}\Eb(\omega)^2 \times [1/n, 1]\hspace{0.05em}) + (1 - \ind_{\hspace{-0.05em}N_0}\hspace{-0.05em}(\omega))(\hspace{0.025em}\{y\}^2 \times [1/n, 1]\hspace{0.05em})\hspace{0.1em},
            \end{equation*}
            where $y \in \dom(f)$ is an arbitrary point, as well as the function $g \colon \Omega \times (\dom(f)^2 \times (\hspace{-0.025em}0, 1]) \,\to\, \RR$, defined by
            \begin{equation*}
                g(\omega, (x, s, \lambda)) \,=\, \ind_{\hspace{-0.05em} N_0}\hspace{-0.05em}(\omega)\frac{2\,}{\,\lambda^2} (f(x + \lambda(s-x)) - f(x) - \lambda\langle \hspace{0.05em}s-x \,|\, \nabla\hspace{-0.1em} f(x) \rangle)
            \end{equation*}
            for all $\omega \in \Omega$ and $(x, s, \lambda) \in \dom(f)^2 \times (\hspace{-0.025em}0, 1]$.
            The correspondence $\varphi_n$ is well\hspace{0.05em}-\hspace{0.035em}defined and attains nonempty and compact values since we have that $\Eb \subseteq \dom(f)$ on $N_0$ and $\Eb$ attains nonempty and compact values by assumption.
            Since, $\Eb$ can be interpreted as a measurable correspondence, using Lemma \ref{LEM:convexCombinationConstantMeasurable}, it is straightforward to verify that $\varphi_n$ is weakly measurable.
            On the other hand, the function $g$ restricted to the domain $\Omega \times X_{\hspace{-0.025em}n}$ is a Carathéodory function. 
            Indeed, for any tuple $(x, s, \lambda) \in X_{\hspace{-0.025em}n}$ we have that $\omega \mapsto g(\omega, (x, s, \lambda))$ is $\Sigma$\hspace{0.1em}-\hspace{0.05em}measurable since $N_0 \in \Sigma$, and for any $\omega \in \Omega$ we have that $(x, s, \lambda) \mapsto g(\omega, (x, s, \lambda))$ is continuous due to the continuity of $f$ and $\nabla\hspace{-0.1em}f$ on $\dom(f)$.
            Hence, we can apply Theorem~\ref{THM:measurableMaximumTheorem} to obtain that the function
            \begin{equation*}
                m_{\hspace{0.025em}n} \colon \Omega \,\to\, \RR, \; \omega \,\mapsto\, \max\hspace{0.1em}\{\hspace{0.05em}g(\omega, (x, s, \lambda)) \,\colon (x, s, \lambda) \in \varphi_n(\omega)\} 
            \end{equation*}
            is $\Sigma$\hspace{0.1em}-\hspace{0.05em}measurable.
            Note that since for $\omega \notin N_0$ it holds $g(\omega, (x, s, \lambda)) = 0$ independent of $(x, s, \lambda) \in X_{\hspace{-0.025em}n}$ we have that $m_{\hspace{0.025em}n}(\omega) = 0$, such that for any $\omega \in \Omega$ we can write
            \begin{align*}
                m_{\hspace{0.025em}n}(\omega) & \,=\, \max\hspace{0.1em}\{\hspace{0.05em}g(\omega, (x, s, \lambda)) \,\colon x, s, \in \Eb(\omega), \lambda \in [1/n, 1]\hspace{0.05em}\} \\
                & \,=\, \max\hspace{0.1em}\{\hspace{0.05em}\max\hspace{0.1em}\{\hspace{0.05em}g(\omega, (x, s, \lambda)) \,\colon x, s, \in \Eb(\omega)\} \,\colon \lambda \in [1/n, 1]\hspace{0.05em}\}\hspace{0.1em}.
            \end{align*}
            Finally, using the identity
            \begin{equation*}
                (\hspace{-0.025em}0, 1] \,=\, \bigcup\hspace{0.2em}\{\hspace{0.05em}[1/n, 1] \,\colon n \in \NN\hspace{0.05em}\}\hspace{0.1em},
            \end{equation*}
            we obtain that $\Cb = \sup\hspace{0.1em}\{m_{\hspace{0.025em}n} \,\colon n \in \NN\hspace{0.05em}\}$ is $\Sigma$\hspace{0.1em}-\hspace{0.05em}measurable.
            The second part of the statement can be shown analogously to the deterministic case.
            We use that for an $L_\nabla$\hspace{0.05em}-\hspace{0.05em}smooth objective function $f$ it holds that
            \begin{equation*}
                f(x) \,\leq\, f(y) + \langle \hspace{0.05em}x-y \,|\, \nabla\hspace{-0.1em}f(y)\rangle + \frac{L_{\nabla}}{2}\lVert \hspace{0.05em}x-y\hspace{0.05em}\rVert^2
            \end{equation*}
            for all $x, y \in \dom(f)$, such that we can bound
            \begin{equation}\label{EQ:boundSmoothness}
                \frac{2\,}{\,\lambda^2} (f(x + \lambda(s-x)) - f(x) - \lambda\langle \hspace{0.05em}s-x \,|\, \nabla\hspace{-0.1em} f(x) \hspace{0.05em}\rangle) \,\leq\, \frac{2\,}{\,\lambda^2} \frac{L_\nabla}{2}\lambda^2\lVert \hspace{0.05em}s-x\hspace{0.05em}\rVert^2 \,=\, L_\nabla\lVert \hspace{0.05em}s-x\hspace{0.05em}\rVert^2
            \end{equation}
            for all $x, s \in \dom(f)$ and $\lambda \in (\hspace{-0.025em}0, 1]$.
            Hence, noting that $\Cb = 0$ on $\Omega \setminus \hspace{-0.05em}N_0$ and using \eqref{EQ:boundSmoothness}, we obtain that
            \begin{equation*}
                \Cb \,\leq\, \sup\hspace{0.1em}\{L_\nabla\lVert \hspace{0.05em}s-x\hspace{0.05em}\rVert^2 \,\colon s, x, \in \Eb\hspace{0.05em}\} \,=\, \diam(\Eb)^2L_\nabla
            \end{equation*}
            and the overall claim follows.
        \end{proof}

        \begin{proof}[\textcolor{seeblau}{Proof of Remark \ref{REM:adaptedLipschitzConstant}}]
            That $\Lb$ is well\hspace{0.05em}-\hspace{0.035em}defined follows from Remark~\ref{REM:weakerAssumptionKnownObjective} together with the indicator function $\ind_{\hspace{-0.05em}N_0}$ and the compactness of $\Eb(\omega)$ for all $\omega \in \Omega$. 
            To prove the $\Sigma$\hspace{0.1em}-\hspace{0.05em}measurability of $\Lb$ we use Theorem~\ref{THM:measurableMaximumTheorem}.
            To this end, we consider the measurable space $(\Omega, \Sigma)$ as well as the Polish space $X = \dom(f)$ and define the correspondence
            \begin{equation*}
                \varphi \colon \Omega \,\rightrightarrows\, X, \; \omega \,\mapsto\, \ind_{\hspace{-0.05em}N_0}\hspace{-0.05em}(\omega)\hspace{0.05em}\Eb(\omega) + (1 - \ind_{\hspace{-0.05em}N_0}\hspace{-0.05em}(\omega))\hspace{0.05em}\{y\}\hspace{0.1em},
            \end{equation*}
            where $y \in \dom(f)$ is an arbitrary point, as well as the function
            \begin{equation*}
                g \colon \Omega \times X \,\to\, \RR, \; (\omega, x) \,\mapsto\, \ind_{\hspace{-0.05em} N_0}\hspace{-0.05em}(\omega)\hspace{0.05em}\lVert\hspace{0.05em}\nabla\hspace{-0.1em}f(x)\hspace{0.05em}\rVert\hspace{0.1em}.
            \end{equation*}
            The correspondence $\varphi$ is well\hspace{0.05em}-\hspace{0.035em}defined and attains nonempty and compact values since it holds $\Eb \subseteq \dom(f)$ on $N_0$ and $\Eb$ attains nonempty and compact values by assumption.
            Furthermore, since $\Eb$ can be seen as a measurable correspondence, by Lemma \ref{LEM:convexCombinationConstantMeasurable}, we have that $\varphi$ is also weakly measurable.
            On the other hand, the function $g$ is a Carathéodory function. 
            Indeed, for any $x \in \dom(f)$ we have that $\omega \mapsto g(\omega, x)$ is $\Sigma$\hspace{0.1em}-\hspace{0.05em}measurable since $N_0 \in \Sigma$, and for any $\omega \in \Omega$ we have that $x \mapsto g(\omega, x)$ is continuous due to the continuity of $\nabla\hspace{-0.1em}f$ on $\dom(f)$.
            Hence, we can apply Theorem~\ref{THM:measurableMaximumTheorem} to obtain that the function 
            \begin{equation*}
                m \colon \Omega \,\to\, \RR, \;\omega \,\mapsto\, \max\hspace{0.1em}\{g(\omega, x) \,\colon x \in \varphi(\omega)\}
            \end{equation*}
            is $\Sigma$\hspace{0.1em}-\hspace{0.05em}measurable. 
            Note that since for $\omega \notin N_0$ it holds $g(\omega, x) = 0$ independent of $x \in X$ we have $m(\omega) = 0$, such that for any $\omega \in \Omega$ we can write
            \begin{equation*}
                m(\omega) \,=\, \max\hspace{0.1em}\{g(\omega, x) \,\colon x \in \Eb(\omega)\} \,=\, \Lb(\omega)\hspace{0.1em},
            \end{equation*}
            such that $\Lb$ is $\Sigma$\hspace{0.1em}-\hspace{0.05em}measurable.
        \end{proof}

        \paragraph{Preliminaries and Intermediate Results.} \textcolor{white}{.}

        \begin{proof}[\textcolor{seeblau}{Proof of Proposition \ref{PRO:generalInductionBound}}]
            We prove the result by induction on $n \geq m+1$. 
            For the base case, we obtain
            \begin{equation*}
                T_{m+\hspace{-0.05em}1} \,\leq\, (1-\lambda_{\hspace{0.025em}m})\hspace{0.05em}T_m + A_1\hspace{-0.05em}\lambda_{\hspace{0.025em}m}^{\hspace{-0.1em}1+r} + A_2\lambda_{m} \,\leq\, 2\hspace{-0.05em}\max\hspace{0.1em}\{\lvert\hspace{0.05em}T_m\hspace{0.05em}\rvert, A_1\hspace{-0.05em}\} + A_2 \,\leq\, (m+3)\hspace{-0.05em}\max\hspace{0.1em}\{\lvert\hspace{0.05em}T_m\hspace{0.05em}\rvert, A_1\hspace{-0.05em}\}\lambda_{\hspace{0.025em}m+\hspace{-0.05em}1}^{\hspace{-0.05em}r} + A_2\hspace{0.1em},
            \end{equation*}
            where we used \eqref{EQ:upperBoundProof} in the first, $\lambda_{\hspace{0.025em}m} \in (\hspace{-0.025em}0, 1]$ in the second, and $\lambda_{\hspace{0.025em}m+\hspace{-0.05em}1} = 2\hspace{0.05em}(m+3)^{-1}$ and $r \leq 1$ in the third inequality.
            For the induction step $(n \to n+1)$ we abbreviate $B = (m+3)\hspace{-0.05em}\max\hspace{0.1em}\{\lvert\hspace{0.05em}T_m\hspace{0.05em}\rvert, A_1\hspace{-0.05em}\}$ and note that $B \geq 2A_1$ such that
            \begin{equation}\label{EQ:induction1}
                \begin{aligned}
                    T_{n+\hspace{-0.05em}1} & \,\leq\, (1-\lambda_{\hspace{0.025em}n})\hspace{0.05em} T_n + A_1\hspace{-0.05em}\lambda_{\hspace{0.025em}n}^{\hspace{-0.1em}1+r} + A_2\lambda_{\hspace{0.025em}n} \\
                    & \,\leq\, (1-\lambda_{\hspace{0.025em}n})(B\hspace{-0.05em}\lambda_{\hspace{0.025em}n}^{\hspace{-0.05em} r} + A_2) + A_1\hspace{-0.05em}\lambda_{\hspace{0.025em}n}^{\hspace{-0.1em} 1+r} + A_2\lambda_{\hspace{0.025em}n} \\
                    & \,\leq\, B\lambda_{\hspace{0.025em}n}^{\hspace{-0.05em}r}(1-\lambda_{\hspace{0.025em}n}) + 2A_1\hspace{-0.05em}\lambda_{\hspace{0.025em}n}^{\hspace{-0.05em}r}\frac{1}{2}\lambda_{\hspace{0.025em}n} + (1-\lambda_{\hspace{0.025em}n})A_2 + \lambda_{\hspace{0.025em}n}A_2 \\
                    & \,\leq\, B\hspace{-0.05em}\lambda_{\hspace{0.025em}n}^{\hspace{-0.05em} r}\hspace{-0.2em}\left(\hspace{-0.15em}1-\frac{1}{2}\lambda_{\hspace{0.025em}n}\hspace{-0.2em}\right) + A_2\hspace{0.1em},
                \end{aligned}
            \end{equation}
            where we used \eqref{EQ:upperBoundProof} in the first and the induction hypothesis in the second inequality.
            Furthermore, by definition of the step\hspace{0.075em}-\hspace{0.05em}sizes, we know that
            \begin{equation}\label{EQ:induction2}
                \lambda_{\hspace{0.025em}n}^{\hspace{-0.05em} r}\hspace{-0.2em}\left(\hspace{-0.15em}1-\frac{1}{2}\lambda_{\hspace{0.025em}n}\hspace{-0.2em}\right) \,=\, \frac{2^{\hspace{0.025em}r}}{(n+2)^r}\frac{n+1}{n+2} \,\leq\, \frac{2^{\hspace{0.025em}r}}{(n+2)^r}\frac{(n+2)^r}{(n+3)^r} \,=\, \frac{2^{\hspace{0.025em}r}}{(n+3)^r} \,=\, \lambda_{\hspace{0.025em}n+\hspace{-0.05em}1}^{\hspace{-0.05em} r}\hspace{0.1em},
            \end{equation}
            where we used that $r \leq 1$ in the inequality.
            Combining \eqref{EQ:induction1} and \eqref{EQ:induction2} then yields the first claim.
            The second part of the claim now follows by noticing that for $n = 1$ in the base case we have
            \begin{equation*}
                T_{1} \,\leq\, (1-\lambda_0)\hspace{0.05em}T_0 + A_1\hspace{-0.05em}\lambda_{\hspace{0.025em}0}^{\hspace{-0.1em}1+r} + A_2\lambda_{0} \,=\, A_1 + A_2 \,\leq\, 2A_1\hspace{-0.05em}\lambda_{\hspace{0.025em}1}^{\hspace{-0.05em}r} + A_2\hspace{0.1em},
            \end{equation*}
            where we used \eqref{EQ:upperBoundProof} in the first inequality, $\lambda_0 = 1$ in the equality, and $\lambda_{1} = 2 / 3$ and $r \leq 1$ in the last inequality.
            Then, for $B = 2A_1$ in \eqref{EQ:induction1} and with \eqref{EQ:induction2} the overall claim follows.
        \end{proof}

        \begin{proof}[\textcolor{seeblau}{Proof of Lemma \ref{LEM:convexDistance}}]
            Let $x_1, x_2 \in K$ and $\lambda \in [\hspace{0.025em}0,1]$.
            Since $D$ is compact, there exist unique $y_1, y_{\hspace{0.015em}2} \in D$ satisfying
            \begin{equation*}
                h(x_1, D) \,=\, \lVert \hspace{0.05em}x_1 - y_1 \rVert \quad \text{and} \quad h(x_2, D) \,=\, \lVert \hspace{0.05em}x_2 - y_{\hspace{0.015em}2} \rVert \hspace{0.1em}.
            \end{equation*}                
            Thus, for $z = \lambda\hspace{0.05em}y_1 + (1-\lambda)\hspace{0.05em}y_{\hspace{0.015em}2}$ we obtain that
            \begin{align*}
                h(\lambda\hspace{0.05em}x_1 + (1-\lambda)\hspace{0.05em}x_2, D) \,\leq\, \lVert \hspace{0.05em}\lambda\hspace{0.05em}x_1 + (1-\lambda)\hspace{0.05em}x_2-z \hspace{0.05em}\rVert \,\leq\, \lambda \lVert \hspace{0.05em}x_1-y_1\rVert + (1-\lambda) \lVert \hspace{0.05em}x_2-y_{\hspace{0.015em}2} \rVert\hspace{0.1em},
            \end{align*}
            where we used that $z \in D$ due to the convexity of $D$.
        \end{proof}

        \paragraph{Comments on the Assumptions.} \textcolor{white}{.}

        \begin{proof}[\textcolor{seeblau}{Proof of Remark \ref{LEM:generalizedBall}}]
            First, we have that $\SS(y, \rho)$ is nonempty since $g_i\hspace{-0.05em}(y, y) \leq 0 \leq \rho$ for all $i \in [n]$ by \eqref{EQ:caratheodoryProperties}.
            Furthermore, the convexity of $\SS(y, \rho)$ follows from the convexity of the Carathéodory functions in their second argument \eqref{EQ:upperBoundFunction}.
            Since $H$ is finite\hspace{0.1em}-\hspace{0.05em}dimensional, by the \emph{Heine\hspace{0.1em}-Borel Theorem}, to show that $\SS(y, \rho)$ is compact it suffices to show that it is both bounded and closed. 
            To show the boundedness, take $(x_k)_{k\in \NN} \subseteq \SS(y, \rho)$ and assume that $\lVert \hspace{0.05em}x_k \rVert \to \infty$ as $k \to \infty$. 
            Then, the coercivity of $g_{\hspace{0.025em}1}$ implies that $g_{\hspace{0.025em}1}\hspace{-0.1em}(y, x_k) \to \infty$ as $k \to \infty$, contradicting $g_{\hspace{0.025em}1}\hspace{-0.1em}(y, x_k) \leq \rho$ for all $k \in \NN$.
            Lastly, the closedness follows from the continuity of $g_i$ in the second argument for all $i \in [n]$, such that overall the first part of the claim follows. 
            For the second part, let $\Yb \colon \Omega \to H$ and $\Pb \colon \Omega \to \RR$ be $\Sigma$\hspace{0.1em}-\hspace{0.05em}measurable.
            Considering the correspondences
            \begin{equation*}
                \varphi_{\hspace{0.025em}i} \colon \Omega \,\rightrightarrows\, H,\; \omega \,\mapsto\, \{x \in H \,\colon g_i\hspace{-0.05em}(\Yb(\omega), x) < \Pb(\omega)\}\hspace{0.1em},
            \end{equation*}   
            we can apply Lemma \ref{LEM:correspondenceOpenSet} to the Carathéodory function $(\omega, x) \mapsto g_i\hspace{-0.05em}(\Yb(\omega), x) - \Pb(\omega)$ and the open subset $G = (-\infty, 0)$ of $\RR$ to obtain that $\varphi_{\hspace{0.025em}i}$ is measurable for all $i \in [n]$. 
            Now, using Lemma \ref{LEM:correspondenceDefinitionsEqual} and Lemma \ref{LEM:correspondenceSetOperations} also implies that the closure correspondence $\cl(\varphi_i)$ is weakly measurable for all $i \in [n]$.
            Considering the intersection correspondence $\varphi \colon \Omega \rightrightarrows H$ defined by
            \begin{equation*}
                \varphi(\omega) \,\coloneqq\, \bigcap\,\{\hspace{0.05em}\cl(\varphi_i)(\omega) \,\colon i \in [n]\hspace{0.05em}\} \,=\, \{x \in H \,\colon g_{\hspace{0.025em}1}\hspace{-0.1em}(\Yb(\omega), x) \leq \Pb(\omega), \ldots, g_{n}\hspace{-0.05em}(\Yb(\omega), x) \leq \Pb(\omega)\}
            \end{equation*}
            for all $\omega \in \Omega$ and noticing that $\cl(\varphi_i)$ is nonempty and compact valued for all $i \in [m]$ by the first part, we can use Lemma~\ref{LEM:correspondenceSetOperations} to obtain that $\varphi$ is measurable. 
            Since $\varphi$ is again nonempty and compact valued we can use Theorem~\ref{THM:correspondenceEquivalences} to derive that $\varphi$ interpreted as a map from $\Omega$ to $\Dcc$, that is, \eqref{EQ:measurableSublevelMap}, is measurable.
        \end{proof}

        \begin{proof}[\textcolor{seeblau}{Proof of Lemma \ref{REM:exampleSublevelSetAndDistanceBound}}]
            For the first part of the statement let $g$ be the induced metric of some norm $\lVert \hspace{0.1em}\cdot\hspace{0.1em}\rVert_g$ on $H$. 
            Since all norms are equivalent in finite\hspace{0.1em}-\hspace{0.05em}dimensional Hilbert spaces, we can find constants $0 < m_g < M_{\hspace{-0.025em}g}$ with 
            \begin{equation}\label{EQ:equivalenceNorms}
                m_g\lVert \hspace{0.05em}s-x\hspace{0.05em}\rVert \,\leq\, \lVert \hspace{0.05em}s-x\hspace{0.05em}\rVert_g \,=\, g(s, x) \,=\, \lVert \hspace{0.05em}s-x\hspace{0.05em}\rVert_g \,\leq\, M_{\hspace{-0.025em}g}\lVert \hspace{0.05em}s-x\hspace{0.05em}\rVert
            \end{equation}
            for all $x, s \in H$. 
            Hence, for all $s, x \in H$ with $\lVert \hspace{0.05em}x \hspace{0.05em}\rVert \to \infty$ we have $g(s, x) \geq m_g\lVert \hspace{0.05em}s-x\hspace{0.05em}\rVert \to \infty$, such that $g$ is coercive.
            Furthermore, $g$ is a Carathéodory function due to its joint continuity and the convexity of $g$ in the second argument follows from the subadditivity and absolute homogeneity of norms.
            For the second part of the statement on the upper bound for the Hausdorff distance, without loss of generality, we can assume that $\rho_1 > 0$.
            We denote with $\lVert \hspace{0.1em}\cdot\hspace{0.1em}\rVert_{\hspace{0.025em}i}$ the norm on $H$ inducing the metric $g_i$ for all $i \in [n]$.
            By \cite[Theorem~20]{wills2007hausdorff} the Hausdorff distance between two nonempty compact and convex sets can be reduced to their boundaries. 
            Hence, we choose $x_1 \in \bd(\SS(y_1, \rho_1\hspace{-0.075em}))$ and since $\SS(y_1, \rho_1\hspace{-0.075em})$ is the intersection of the sets $\{x \in H \,\colon g_i\hspace{0.025em}(y_1, x) \leq \rho_1\hspace{-0.05em}\}$, there exists an index $k \in [n]$ with $\lVert\hspace{0.05em} y_1-x_1\rVert_k = g(y_1, x_1\hspace{-0.075em}) = \rho_1$, such that we can write $x_1 = y_1 + \rho_1u$ for some $u \in H$ with $\lVert \hspace{0.05em}u\hspace{0.05em} \rVert_k = 1$. 
            Furthermore, we know that 
            \begin{equation*}
                \rho_1\lVert\hspace{0.05em} u\hspace{0.05em}\rVert_i \,=\, \lVert\hspace{0.05em} y_1 - x_1\rVert_i \,=\, g_i\hspace{0.025em}(y_1, x_1\hspace{-0.075em}) \,\leq\, \rho_1\hspace{0.1em},
            \end{equation*}
            such that $\lVert\hspace{0.05em} u \hspace{0.05em}\rVert_i \leq 1$ for all $i \in [n]$. 
            Defining $x_2 = y_{\hspace{0.015em}2}  +\rho_2u$ we then obtain that
            \begin{equation*}
                g_i\hspace{0.025em}(y_{\hspace{0.015em}2}, x_2\hspace{-0.075em}) \,=\, \lVert\hspace{0.05em} y_{\hspace{0.015em}2} - x_2 \rVert_i \,=\, \rho_2\lVert u\rVert_i \,\leq\, \rho_2
            \end{equation*}
            for all $i \in [n]$, such that
            $x_2 \in \SS(y_{\hspace{0.015em}2}, \rho_2\hspace{-0.05em})$.
            Now, taking constants $0 < m < M$ as in \eqref{EQ:equivalenceNorms} but for all norms at once, we find that
            \begin{equation}\label{EQ:normProof}
                \lVert\hspace{0.05em} x_1-x_2 \rVert \,\leq\, \frac{1}{m}\lVert\hspace{0.05em}x_1-x_2\rVert_k \,\leq\, \frac{1}{m}(\lVert\hspace{0.05em} y_1-y_{\hspace{0.015em}2} \rVert_k + \lvert\hspace{0.05em}\rho_1 - \rho_2\rvert) \,\leq\, \frac{M}{m}\lVert\hspace{0.05em} y_1-y_{\hspace{0.015em}2} \rVert + \frac{1}{m}\lvert\hspace{0.05em}\rho_1-\rho_2\rvert\hspace{0.1em},
            \end{equation}
            where we used that $\lVert\hspace{0.05em} u \hspace{0.05em}\rVert_k = 1$ in the second inequality.
            Therefore, using \eqref{EQ:normProof} and the compactness of $\SS(y_1, \rho_1\hspace{-0.075em})$ we find
            \begin{equation*}
                \begin{aligned}
                    \sup\hspace{0.1em}\{\hspace{0.05em}h(x, \bd(\SS(y_{\hspace{0.015em}2}, \rho_2\hspace{-0.075em}))) \,\colon x \in \bd(\SS(y_1, \rho_1\hspace{-0.075em}))\} & \,=\, \sup\hspace{0.1em}\{\hspace{0.05em}h(x, \SS(y_{\hspace{0.015em}2}, \rho_2\hspace{-0.075em})) \,\colon x \in \bd(\SS(y_1, \rho_1\hspace{-0.075em}))\} \\
                    & \,\leq\, \frac{M}{m}\lVert\hspace{0.05em} y_1-y_{\hspace{0.015em}2} \rVert + \frac{1}{m}\lvert\hspace{0.05em}\rho_1-\rho_2\rvert\hspace{0.1em}\hspace{0.1em},
                \end{aligned}
            \end{equation*}
            such that the claim follows by symmetry in the proof.
        \end{proof}
        
    \subsection{Proofs of Section 4}\label{SEC:B.3}

        \paragraph{Inner Domain Approximation.} \textcolor{white}{.}
        
        \begin{proof}[\textcolor{seeblau}{Proof of Proposition \ref{PRO:advancedInductionBound}}]
            We prove the result by induction on $n \in \NN$. 
            For the base case ($n = m+1$), we note that
            \begin{equation*}
                T_{m+\hspace{-0.05em}1} \,\leq\, \left(\hspace{-0.15em}1-\frac{\lambda_{\hspace{0.025em}m}}{2}\hspace{-0.15em}\right)\hspace{-0.1em}T_m - A_m\sqrt{T_m}\hspace{0.025em}B_1\lambda_{\hspace{0.025em}m} + B_2\lambda_{\hspace{0.025em}m}^{\hspace{-0.1em}1+r} + B_3\lambda_{\hspace{0.025em}m} \,\leq\,\left(\hspace{-0.15em}T_m+4\hspace{-0.15em}\left(\hspace{-0.1em}\hspace{-0.05em}\frac{B_2}{AB_1}\hspace{-0.1em}\right)^{\hspace{-0.25em}2} \hspace{-0.2em}+ B_2\hspace{-0.2em}\right)\hspace{-0.2em}\lambda_{\hspace{0.025em}m}^{\hspace{-0.1em}2\hspace{0.025em}r} + B_3\hspace{0.1em},
            \end{equation*}
            where we used \eqref{EQ:advancedInductionEquation} in the first and $\lambda_m \in [\hspace{0.025em}0, 1]$ and $2\hspace{0.05em}r \leq 1+r$ in the second inequality.
            For the induction step ($n \to n+1$) we distinguish between two cases. 
            First, assume that 
            \begin{equation}\label{EQ:acceleratedConvergenceCase}
                T_n \,\leq\, \left(\hspace{-0.1em}\frac{2B_2}{AB_1}\lambda_{\hspace{0.025em}n}^{\hspace{-0.05em}r}\hspace{-0.1em}\right)^{\hspace{-0.25em}2} \quad \Longleftrightarrow \quad \sqrt{T_n} \,\leq\, \frac{2B_2}{AB_1}\lambda_{\hspace{0.025em}n}^{\hspace{-0.05em}r}\hspace{0.1em}.
            \end{equation}
            Then,
            \begin{align*}
                T_{n+\hspace{-0.05em}1} 
                & \,\leq\, \left(\hspace{-0.15em}1-\frac{\lambda_{\hspace{-0.025em}n}}{2}\hspace{-0.15em}\right)\hspace{-0.1em}T_n - A_n\sqrt{T_n}\hspace{0.025em}B_1\lambda_{\hspace{0.025em}n} + B_2\lambda_{\hspace{0.025em}n}^{\hspace{-0.1em}1+r} + B_3\lambda_{\hspace{0.025em}n} \\
                & \,\leq\,\left(\hspace{-0.1em}\frac{2B_2}{AB_1}\lambda_{\hspace{0.025em}n}^{\hspace{-0.05em}r}\hspace{-0.15em}\right)^{\hspace{-0.25em}2} + B_2\lambda_{\hspace{0.025em}n}^{\hspace{-0.05em}2\hspace{0.025em}r} + B_3\lambda_{\hspace{0.025em}n} \\
                & \,\leq\,\left(\hspace{-0.15em}T_m+4\hspace{-0.15em}\left(\hspace{-0.1em}\hspace{-0.05em}\frac{B_2}{AB_1}\hspace{-0.1em}\right)^{\hspace{-0.25em}2} \hspace{-0.2em}+ B_2\hspace{-0.2em}\right)\hspace{-0.2em}\lambda_{\hspace{0.025em}n}^{\hspace{-0.05em}2\hspace{0.025em}r} + B_3\hspace{0.1em},
            \end{align*}
            where we used \eqref{EQ:advancedInductionEquation} in the first, \eqref{EQ:acceleratedConvergenceCase} and $2\hspace{0.05em}r \leq 1+r$ in the second, and $T_m \geq 0$ in the third inequality.
            Assume now that \eqref{EQ:acceleratedConvergenceCase} does not hold. 
            Then,
            \begin{equation}\label{EQ:inductionLastPart}
                \begin{aligned}
                    T_{n+\hspace{-0.05em}1} & \,\leq\, \left(\hspace{-0.15em}1-\frac{\lambda_{\hspace{-0.025em}n}}{2}\hspace{-0.15em}\right)\hspace{-0.1em}T_n - A_n\sqrt{T_n}\hspace{0.025em}B_1\lambda_{\hspace{0.025em}n} + B_2\lambda_{\hspace{0.025em}n}^{\hspace{-0.1em}1+\hspace{-0.025em}r} + B_3\lambda_{\hspace{0.025em}n} \\
                    & \,\leq\, \left(\hspace{-0.15em}1-\frac{\lambda_{\hspace{-0.025em}n}}{2}\hspace{-0.15em}\right)\hspace{-0.1em}T_n-2B_2\lambda_{\hspace{0.025em}n}^{\hspace{-0.1em}1+\hspace{-0.025em}r} + B_2\lambda_{\hspace{0.025em}n}^{\hspace{-0.1em}1+r} + B_3\lambda_{\hspace{0.025em}n} \\
                    & \,\leq\, \left(\hspace{-0.15em}1-\frac{\lambda_{\hspace{-0.025em}n}}{2}\hspace{-0.15em}\right)\hspace{-0.25em}\left(\hspace{-0.15em}T_m+4\hspace{-0.15em}\left(\hspace{-0.1em}\hspace{-0.05em}\frac{B_2}{AB_1}\hspace{-0.1em}\right)^{\hspace{-0.25em}2} \hspace{-0.2em}+ B_2\hspace{-0.2em}\right)\hspace{-0.2em}\lambda_{\hspace{0.025em}n-\hspace{-0.1em}1}^{\hspace{-0.05em}2\hspace{0.025em}r} - B_2\lambda_{\hspace{0.025em}n}^{\hspace{-0.1em}1+r} + B_3\hspace{0.1em},
                \end{aligned}
            \end{equation}
            where we used \eqref{EQ:advancedInductionEquation} in the first, the negation of \eqref{EQ:acceleratedConvergenceCase} and $A_n \geq A$ in the second, and the induction hypothesis \eqref{EQ:inductionHypothesis} in the third inequality.  
            Again, we distinguish between two cases.
            First, assume that $r \in [\hspace{0.025em}0, 1/2\hspace{0.025em}]$.
            The claim then follows combining \eqref{EQ:inductionLastPart} and 
            \begin{equation*}
                \lambda_{\hspace{0.025em}n-\hspace{-0.05em}1}^{\hspace{-0.05em}2\hspace{0.025em}r}\hspace{-0.2em}\left(\hspace{-0.15em}1-\frac{\lambda_{\hspace{-0.025em}n}}{2}\hspace{-0.15em}\right) \,=\, \frac{\hspace{-0.4em}4^{\hspace{0.025em}r}}{(n+1)^{2\hspace{0.025em}r}}\frac{n+1}{n+2} \,\leq\, \frac{\hspace{-0.4em}4^{\hspace{0.025em}r}}{(n+1)^{2\hspace{0.025em}r}}\frac{(n+1)^{2\hspace{0.025em}r}}{(n+2)^{2\hspace{0.025em}r}} \,=\, \frac{\hspace{-0.4em}4^{\hspace{0.025em}r}}{(n+2)^{2\hspace{0.025em}r}} \,=\, \lambda_{\hspace{0.025em}n}^{\hspace{-0.05em}2\hspace{0.025em}r}\hspace{0.05em},
            \end{equation*}
            where we used that $2\hspace{0.05em}r \leq 1$ in the inequality.   
            On the other hand, if $r \in [1/2, 1]$, then we have
            \begin{equation}\label{EQ:advancedProposition1}
                \lambda_{\hspace{0.025em}n-\hspace{-0.05em}1}^{\hspace{-0.05em}2\hspace{0.025em}r}\hspace{-0.2em}\left(\hspace{-0.15em}1-\frac{\lambda_{\hspace{0.025em}n}}{2}\hspace{-0.15em}\right) \,=\, \left(\frac{2}{n+1}\right)^{\hspace{-0.25em}2\hspace{0.025em}r}\frac{n+1}{n+2} \,=\, \left(\frac{2}{n+1}\right)^{\hspace{-0.25em}2\hspace{0.025em}r}\hspace{-0.3em}\left(\frac{n+1}{n+2}\right)^{\hspace{-0.25em}2\hspace{0.025em}r}\hspace{-0.3em}\left(\frac{n+1}{n+2}\right)^{\hspace{-0.25em}1-2\hspace{0.025em}r}  \hspace{-0.5em}\,=\, \lambda_{\hspace{0.025em}n}^{\hspace{-0.05em}2\hspace{0.025em}r}\left(1+\frac{1}{n+1}\right)^{\hspace{-0.25em}2\hspace{0.025em}r-1}\hspace{-1.5em}.
            \end{equation}
            Furthermore, by concavity of the function $x \mapsto (1+x)^s$ for $x, s \in [0, 1]$, we have $(1+x)^s \leq 1 + sx$, such that with \eqref{EQ:advancedProposition1} we have
            \begin{equation}\label{EQ:advancedProposition2}
               \lambda_{\hspace{0.025em}n-\hspace{-0.05em}1}^{\hspace{-0.05em}2\hspace{0.025em}r}\hspace{-0.2em}\left(\hspace{-0.15em}1-\frac{\lambda_{\hspace{0.025em}n}}{2}\hspace{-0.15em}\right) \,\leq\, \lambda_{\hspace{0.025em}n}^{\hspace{-0.05em}2\hspace{0.025em}r}\left(1+\frac{2\hspace{0.05em}r-1}{n+1}\right) \,\leq\, \lambda_{\hspace{0.025em}n}^{\hspace{-0.05em}2\hspace{0.025em}r} + \lambda_{\hspace{0.025em}n}^{\hspace{-0.05em}2\hspace{0.025em}r}\frac{2}{n+2} \,=\, \lambda_{\hspace{0.025em}n}^{\hspace{-0.05em}2\hspace{0.025em}r} + \lambda_{\hspace{0.025em}n}^{\hspace{-0.05em}2\hspace{0.025em}r+\hspace{-0.05em}1}\hspace{0.1em}.
            \end{equation}
            Writing 
            \begin{equation*}
                C \,=\, T_m+4\hspace{-0.15em}\left(\hspace{-0.1em}\hspace{-0.05em}\frac{B_2}{AB_1}\hspace{-0.1em}\right)^{\hspace{-0.25em}2} \hspace{-0.2em}+ B_2
            \end{equation*}
            and choosing 
            \begin{equation*}
                N \,=\, \max\left\{\left\lceil 2\hspace{-0.15em}\left(\frac{C}{B_2}\right)^{\hspace{-0.3em}1/r}\hspace{-0.4em}-2\right\rceil\hspace{-0.25em}, \hspace{0.05em}m\right\} \in \NN
            \end{equation*}
            together with \eqref{EQ:advancedProposition2} we the find
            \begin{equation*}
                \lambda_{\hspace{0.025em}n-\hspace{-0.05em}1}^{\hspace{-0.05em}2\hspace{0.025em}r}\hspace{-0.2em}\left(\hspace{-0.15em}1-\frac{\lambda_{\hspace{-0.025em}n}}{2}\hspace{-0.15em}\right)\hspace{-0.15em}C  - B_2\lambda_{\hspace{0.025em}n}^{\hspace{-0.1em}1+r} \,\leq\, C\lambda_{\hspace{0.025em}n}^{\hspace{-0.05em}2\hspace{0.025em}r} + \left(C\lambda_{\hspace{0.025em}n}^{\hspace{-0.05em}r} - B_2\right)\lambda_{\hspace{0.025em}n}^{\hspace{-0.1em}1+r} \,\leq\, C\lambda_{\hspace{0.025em}n}^{\hspace{-0.05em}2\hspace{0.025em}r} + \left(C\lambda_{\hspace{0.025em}N}^{\hspace{-0.05em}r} - B_2\right)\lambda_{\hspace{0.025em}n}^{\hspace{-0.1em}1+r} \,\leq\, C\lambda_{\hspace{0.025em}n}^{\hspace{-0.05em}2\hspace{0.025em}r}
            \end{equation*}
            for all $n \geq N$, such that together with \eqref{EQ:inductionLastPart} the claim follows.
        \end{proof}
        
        \paragraph{Comments on the Assumptions.} \textcolor{white}{.}

        \begin{proof}[\textcolor{seeblau}{Proof of Lemma \ref{LEM:dilationErosionProperties}.}]
            First, we consider the dilation $X \oplus Y$. 
            It is straightforward to verify that $X \oplus Y$ is nonempty, bounded and convex.
            To show the closedness, let $(z_{\hspace{0.025em}n}\hspace{-0.025em})_{n \in \NN} \subseteq X \oplus Y$ be a sequence with $z_{\hspace{0.025em}n} \to z \in H$ as $n \to \infty$. 
            Then, there exist sequences $(x_n\hspace{-0.025em})_{n \in \NN} \subseteq X$ and $(y_n\hspace{-0.025em})_{n \in \NN} \subseteq Y$ with $z_{\hspace{0.025em}n} = x_n + y_n$ for all $n \in \NN$.
            Since $X$ and $Y$ are bounded, by Bolzano\hspace{0.1em}-Weierstraß there exist joint subsequences $(x_{n_k}\hspace{-0.05em})_{k \in \NN} \subseteq X$ and $(y_{n_k}\hspace{-0.05em})_{k \in \NN} \subseteq Y$ as well as $x \in X$ and $y \in Y$ with $x_{n_k} \to x$ and $y_{n_k} \to y$ as $k \to \infty$.
            Overall this yields
            \begin{equation*}
                z \,=\, \lim_{n \to \infty} z_{\hspace{0.025em}n} \,=\, \lim_{k \to \infty} z_{n_k} \,=\, \lim_{k \to \infty} x_{n_k} + y_{n_k} \,=\, x + y \in X \oplus Y\hspace{0.1em},
            \end{equation*}
            such that $X \oplus Y$ is bounded and closed and therefore compact. 
            Now, we consider the erosion $X \ominus Y$. 
            Let $z \in X \ominus Y$ and $\rho_X, \rho_Y \geq 0$ with $X \subseteq \BB(0, \rho_X\hspace{-0.025em})$ and $Y \subseteq \BB(0, \rho_Y\hspace{-0.025em})$.
            Then, for each $y \in Y$ we have $z+y \in X \subseteq \BB(0, \rho_X\hspace{-0.025em})$ such that
            \begin{equation*}
                \rho_X \,\geq\, \lVert\hspace{0.05em}z+y\hspace{0.05em}\rVert \,\geq\, \lVert\hspace{0.05em}z\hspace{0.05em}\rVert - \lVert\hspace{0.05em}y\hspace{0.05em}\rVert \,\geq\, \lVert\hspace{0.05em}z\hspace{0.05em}\rVert - \rho_Y\hspace{0.1em},
            \end{equation*}
            showing that $X \ominus Y \subseteq \BB(0, \rho_X+\rho_Y\hspace{-0.025em})$ is bounded.
            For the closedness let $(z_{\hspace{0.025em}n}\hspace{-0.025em})_{n \in \NN} \subseteq X \ominus Y$ be a sequence with $z_{\hspace{0.025em}n} \to z \in H$ as $n \to \infty$. 
            Then, for each $y \in Y$ we have that the sequence $(z_{\hspace{0.025em}n}+y)_{n \in \NN} \subseteq X$ is bounded, such that by Bolzano\hspace{0.1em}-Weierstraß there exists a subsequence $(z_{n_k}+y)_{k \in \NN} \subseteq X$ and $x \in X$ with $z_{n_k} + y \to x$ as $k \to \infty$.
            Overall, by the laws of limits, this yields
            \begin{equation*}
                z + y \,=\, \lim_{n \to \infty} z_{\hspace{0.025em}n} + y \,=\, \lim_{k \to \infty} z_{n_k} + y \,=\, x\hspace{0.1em},
            \end{equation*}
            such that $z \in X \ominus Y$ since $y \in Y$ was arbitrary.
            Hence, together with the previously shown boundedness of $X\ominus Y$ we obtain that $X \ominus Y$ is compact.
            To show the convexity of $X \ominus Y$ let $z_1, z_2 \in X \ominus Y$ and $\lambda \in [\hspace{0.025em}0, 1]$. 
            Then, for all $y \in Y$ we have 
            \begin{equation}\label{EQ:convexErosion}
                \lambda\hspace{0.05em}z_1 + (1-\lambda)\hspace{0.05em}z_2 + y \,=\, \lambda (z_1 + y) + (1-\lambda)\hspace{0.025em}(z_2 + y) \,=\, \lambda \hspace{0.05em}x_1 + (1-\lambda)\hspace{0.05em}x_2 \in X\hspace{0.1em},
            \end{equation}
            where we denoted $z_1 + y = x_1 \in X$, $z_2 + y =x_2 \in X$ and used the convexity of $X$.
            Since $y \in Y$ was arbitrary by \eqref{EQ:convexErosion} we have $z + Y \subseteq X$ such that $z \in X \ominus Y$ and $X \ominus Y$ is convex.
            Let us now additionally assume that $X$ is $\alpha$\hspace{0.1em}-\hspace{0.05em}strongly convex for some $\alpha > 0$ and that \eqref{EQ:nonemptyErosionCondition} holds.
            If $X$ is a point set, then $Y$ is a point set as well by \eqref{EQ:nonemptyErosionCondition}. 
            In this case, writing $X = \{x\}$ and $Y = \{y\}$ for some $x, y \in H$ we have that $x-y \in \{x-y\} = X \ominus Y$.
            Assume now that $X$ is no point set.
            If $Y$ is still a point set, then again $x-y \in X \ominus Y$ for all $x \in X$, where $Y = \{y\}$ for some $y \in H$.
            Hence, we can now also assume that $Y$ is no point set.
            Due to the compactness of $X$ and $Y$ we can find $x_1, x_2 \in X$ with $R_X \coloneqq \lVert\hspace{0.05em}x_1-x_2\rVert = \diam(X)> 0$ and $y_1,y_2 \in Y$ with $R_Y \coloneqq \lVert\hspace{0.05em}y_1-y_2\rVert = \diam(Y)> 0$.
            Defining $u = (x_1 - x_2\hspace{-0.025em}) / R_X \in H$ we have $\lVert u \rVert = 1$ and $x_1 - R_Xu = x_2 \in X$.
            Hence, by the $\alpha$\hspace{0.1em}-\hspace{0.05em}strong convexity of $X$ we find
            \begin{equation*}
                \begin{aligned}
                    x_1-\frac{R_X}{2}u + \frac{\alpha R_X^{\hspace{0.05em}2}}{8}z \,=\, \frac{1}{2}x_1 + \frac{1}{2}(x_1-R_Xu) + \frac{\alpha}{8}\lVert R_Xu\hspace{0.05em}\rVert^2z \,=\, \frac{1}{2}x_1 + \frac{1}{2}x_2 + \frac{\alpha}{8}\lVert\hspace{0.05em} x_1-x_2 \rVert^2z \in X
                \end{aligned}
            \end{equation*}
            for all $z \in H$ with $\lVert \hspace{0.05em}z\hspace{0.05em}\rVert = 1$, such that
            \begin{equation}\label{EQ:property1}
                \left(\hspace{-0.15em}x_1-\frac{R_X}{2}u \hspace{-0.2em}\right) + \BB\hspace{-0.25em}\left(\hspace{-0.15em}0, \frac{\alpha R_X^{\hspace{0.05em}2}}{8}\hspace{-0.2em}\right) \,\subseteq\, X\hspace{0.1em}.
            \end{equation}
            On the other hand, for all $y \in Y$ we have that
            \begin{equation*}
                \left\lVert\hspace{0.05em}y-\frac{y_1+y_2}{2}\hspace{0.05em}\right\rVert \,\leq\, \frac{1}{2}\lVert y-y_1\rVert + \frac{1}{2}\lVert y-y_2\rVert \,\leq\, R_Y\hspace{0.1em},
            \end{equation*}
            such that
            \begin{equation}\label{EQ:property2}
                Y \,\subseteq\, \BB\hspace{-0.25em}\left(\hspace{-0.1em}\frac{y_1+y_2}{2}, R_Y\hspace{-0.2em}\right)\hspace{0.1em}.
            \end{equation}
            Combining \eqref{EQ:property1} and \eqref{EQ:property2} then yields
            \begin{equation*}
                \begin{aligned}
                    \left(\hspace{-0.15em}x_1-\frac{R_X}{2}u + \frac{y_1+y_2}{2}\hspace{-0.1em}\right) + Y & \,\subseteq\,  \left(\hspace{-0.15em}x_1-\frac{R_X}{2}u + \frac{y_1+y_2}{2}\hspace{-0.1em}\right) + \BB\hspace{-0.25em}\left(\hspace{-0.1em}\frac{y_1+y_2}{2}, R_Y\hspace{-0.2em}\right) \\
                    & \,=\, \left(\hspace{-0.15em}x_1-\frac{R_X}{2}u \hspace{-0.2em}\right) + \BB(0, R_Y) \\
                    & \,\subseteq\, \left(\hspace{-0.15em}x_1-\frac{R_X}{2}u \hspace{-0.2em}\right) + \BB\hspace{-0.25em}\left(\hspace{-0.15em}0, \frac{\alpha R_X^{\hspace{0.05em}2}}{8}\hspace{-0.2em}\right) \\
                    & \,\subseteq\, X\hspace{0.1em},
                \end{aligned}
            \end{equation*}
            where we used \eqref{EQ:nonemptyErosionCondition} in the second inclusion.
            Therefore, $X \ominus Y$ is nonempty and the overall claim follows.
        \end{proof}

        \noindent
        As already mentioned in Section \ref{SEC:acceleratedConvergenceResult}, we have to prove that the set
        \begin{equation*}
            A \,=\, \left\{\Db_{\hspace{0.025em}n} \text{ \hspace{-0.05em}is } \boldsymbol{\alpha}_n\text{-\hspace{0.05em}strongly convex and }\diam(\hspace{0.035em}\Db_{\hspace{0.025em}n}\hspace{-0.05em})^2 \,\geq\, \frac{16}{\hspace{0.2em}\boldsymbol{\alpha}_n} d_H(\hspace{0.035em}\Db_{\hspace{0.025em}n}, D) \,\text{ for all }\, n \in \NN\right\}
        \end{equation*}
        corresponding to \eqref{EQ:diameterBound} is measurable.
        Rewriting 
        \begin{equation*}
            A \,=\, \bigcup\left\{\{\hspace{0.05em}\Db_{\hspace{0.025em}n} \text{ \hspace{-0.05em}is } \boldsymbol{\alpha}_n\text{-\hspace{0.05em}strongly convex}\hspace{0.05em}\} \,\cap\, \left\{\diam(\hspace{0.035em}\Db_{\hspace{0.025em}n}\hspace{-0.05em})^2 \,\geq\, \frac{16}{\hspace{0.2em}\boldsymbol{\alpha}_n} d_H(\hspace{0.035em}\Db_{\hspace{0.025em}n}, D)\right\}\colon n \in \NN\hspace{0.05em}\right\}
        \end{equation*}
        we can see by standard results of measure theory this reduces to showing that $\{\hspace{0.05em}\Db_{\hspace{0.025em}n} \text{ is } \boldsymbol{\alpha}_n\text{-\hspace{0.05em}strongly convex}\hspace{0.05em}\}$ is measurable for all $n \in \NN$.
        This follows from the next two results.

        \begin{lemma}{}{measurableEvent}
            Let $\varphi \colon \Omega \rightrightarrows H$ be a measurable correspondence with nonempty and compact values.
            Then, the set $\{\varphi \text{ is } \alpha\text{\hspace{0.1em}-\hspace{0.05em}strongly convex}\hspace{0.05em}\} \subseteq \Omega$ is measurable for all $\alpha > 0$.
        \end{lemma}
        \begin{proof}[\textcolor{seeblau}{Proof.}]
            We fix $\alpha > 0$ and define the map
            \begin{equation*}
                f_\alpha \colon H^2 \times \BB(0, 1) \times [\hspace{0.025em}0, 1] \to H, \; (x, y, z, \lambda) \,\mapsto\, \lambda\hspace{0.05em}x + (1-\lambda)\hspace{0.05em}y + \frac{\alpha}{2}\lambda(1-\lambda)\hspace{0.05em}\lVert\hspace{0.05em}x-y\hspace{0.05em}\rVert^2z\hspace{0.1em},
            \end{equation*}
            which is jointly continuous.
            By definition of strong convexity and since $\varphi$ has closed values it now follows that
            \begin{equation*}
                \begin{aligned}
                    \varphi(\omega) \text{ is } \alpha\text{\hspace{0.1em}-\hspace{0.05em}strongly convex} & \quad \Longleftrightarrow \quad \forall\hspace{0.025em} x, y \in \varphi(\omega) \,\colon \forall\hspace{0.025em} z \in \BB(0, 1) \,\colon \forall \hspace{0.025em}\lambda \in [\hspace{0.025em}0, 1] \,\colon f_\alpha(x, y, z, \lambda) \in \varphi(\omega) \\
                    & \quad \Longleftrightarrow \quad \forall\hspace{0.025em} x, y \in \varphi(\omega) \,\colon \forall\hspace{0.025em} z \in \BB(0, 1) \,\colon \forall \hspace{0.025em}\lambda \in [\hspace{0.025em}0, 1] \,\colon h(f_\alpha(x, y, z, \lambda), \varphi(\omega)) = 0
                \end{aligned}
            \end{equation*}
            for all $\omega \in \Omega$, where as before $h(x, \varphi(\omega)) = \inf\hspace{0.05em}\{\lVert \hspace{0.05em} x-y\hspace{0.05em}\rVert \,\colon y \in \varphi(\omega)\}$ describes the distance between a point $x \in H$ and the set $\varphi(\omega) \subseteq H$.
            Furthermore, since $x \mapsto h(x, \varphi(\omega))$ is nonnegative on $H$ for all $\omega \in \Omega$, we overall obtain that
            \begin{equation*}
                \varphi(\omega) \text{ is } \alpha\hspace{0.1em}\text{-\hspace{0.05em}strongly convex} \quad \Longleftrightarrow \quad \sup\hspace{0.1em}\{\hspace{0.05em}h(f_\alpha(x, y, z, \lambda), \varphi(\omega)) \,\colon x, y \in \varphi(\omega), z \in \BB(0, 1), \lambda \in [\hspace{0.025em}0, 1]\hspace{0.05em}\} \,=\, 0
            \end{equation*}
            for all $\omega \in \Omega$.
            Hence, showing that the map
            \begin{equation*}
                m \colon \Omega \to \RR, \; \omega \,\mapsto\, \sup\hspace{0.1em}\{\hspace{0.05em}h(f_\alpha(x, y, z, \lambda), \varphi(\omega)) \,\colon x, y \in \varphi(\omega), z \in \BB(0, 1), \lambda \in [\hspace{0.025em}0, 1]\hspace{0.05em}\}
            \end{equation*}
            is $\Sigma$\hspace{0.1em}-\hspace{0.05em}measurable yields the claim since then $\{\varphi(\omega) \text{ is } \alpha\hspace{0.1em}\text{-\hspace{0.05em}strongly convex}\} \,=\, m^{-1}(\{0\}) \in \Sigma$.
            To this end, we consider the measurable space $(\Omega, \Sigma)$ as well as the Polish space $H^2 \times \BB(0, 1) \times [\hspace{0.025em}0, 1]$ and define the correspondence
            \begin{equation*}
                \tilde{\varphi} \colon \Omega \rightrightarrows H^2 \times \BB(0, 1) \times [\hspace{0.025em}0, 1], \; \omega \,\mapsto\, \varphi(\omega)^2 \times \BB(0, 1) \times [\hspace{0.025em}0, 1]
            \end{equation*}
            as well as the function
            \begin{equation*}
                g \colon \Omega \times H^2 \times \BB(0, 1) \times [\hspace{0.025em}0, 1] \to \RR,\; (\omega, x, y, z, \lambda) \,\mapsto\, h(f_\alpha(x, y, z, \lambda), \varphi(\omega))\hspace{0.1em}.
            \end{equation*}
            Since $\varphi$ is measurable, the measurability of the correspondence $\tilde{\varphi}$ follows immediately.
            Furthermore, we can see that $g$ is a Carathéodory function.
            Indeed, for fixed $(x, y, z, \lambda) \in H^2 \times \BB(0, 1) \times [\hspace{0.025em}0, 1]$ we have that the map $\omega \mapsto g(\omega, x, y, z, \lambda)$ is measurable since $\varphi(\omega) \in \Kcc$ and $K \mapsto h(x, K)$ is Lipschitz continuous on $\Kcc$ for all $x \in H$.
            On the other hand, we have that $(x, y, z, \lambda) \mapsto g(\omega, x, y, z, \lambda)$ is continuous for each $\omega \in \Omega$ since $f_\alpha$ is continuous and $x \mapsto h(x, K)$ is Lipschitz continuous on $H$ for all $K \in \Kcc$.
            Thus, we can apply Theorem~\ref{THM:measurableMaximumTheorem} to obtain that the function 
            \begin{equation*}
                m \colon \Omega \to \RR,\; \omega \,\mapsto\, \sup\hspace{0.1em}\{\hspace{0.05em}g \,\colon (x, y, z, \lambda) \in \tilde{\varphi}(\omega)\hspace{0.05em}\} 
            \end{equation*}
            is $\Sigma$\hspace{0.1em}-\hspace{0.05em}measurable such that the claim follows.
        \end{proof}

        \begin{lemma}{}{measurableEventExtended}
            Let $\varphi \colon \Omega \rightrightarrows H$ be a measurable correspondence with nonempty and compact values and $\boldsymbol{\alpha}$ be a nonnegative random variable.
            Then, the set $\{\varphi \text{ is } \boldsymbol{\alpha}\text{\hspace{0.1em}-\hspace{0.05em}strongly convex}\hspace{0.05em}\} \subseteq \Omega$ is measurable.
        \end{lemma}
        \begin{proof}[\textcolor{seeblau}{Proof.}]
            First, note that if a set is $\alpha_{\hspace{0.025em}0}$\hspace{0.05em}-\hspace{0.05em}strongly convex for some $\alpha_{\hspace{0.025em}0} > 0$, then it is $\alpha$\hspace{0.1em}-\hspace{0.05em}strongly convex for all $\alpha \in (\hspace{-0.05em}0, \alpha_{\hspace{0.025em}0}]$.
            Therefore, we consider function
            \begin{equation*}
                \boldsymbol{\alpha}_{\hspace{0.025em}0} \colon \Omega \to \RR, \, \omega \,\mapsto\, \sup\hspace{0.1em}\{\alpha > 0 \,\colon \varphi(\omega) \text{ is } \alpha\text{\hspace{0.1em}-\hspace{0.05em}strongly convex}\}\hspace{0.1em},
             \end{equation*}
            which is $\Sigma$\hspace{0.1em}-\hspace{0.05em}measurable since by Lemma \ref{LEM:measurableEvent} it holds
            \begin{equation*}
                \{\boldsymbol{\alpha}_{\hspace{0.025em}0} > t\} \,=\, \bigcup\hspace{0.2em}\{\{\varphi \text{ is } t\text{\hspace{0.1em}-\hspace{0.05em}strongly convex}\}\,\colon q \in \QQ \text{ and } q > t\} \in \Sigma
            \end{equation*}
            for all $t > 0$.
            Hence, we directly obtain $\{\varphi \text{ is } \boldsymbol{\alpha}\text{\hspace{0.1em}-\hspace{0.05em}strongly convex}\hspace{0.05em}\} \,=\, \{\boldsymbol{\alpha} \leq \boldsymbol{\alpha}_{\hspace{0.025em}0}\hspace{-0.025em}\} \in \Sigma$.
        \end{proof}

        \begin{lemma}{}{dilationErosionMeasurable}
            Let $\varphi \colon \Omega \rightrightarrows H$ be a measurable correspondence with nonempty and closed values and $\Pb \colon \Omega \to \RR_+$ be a nonnegative random variable. 
            Then, both the dilation correspondence $\varphi \oplus \BB(0, \Pb)$ and the erosion correspondence $\varphi \ominus \BB(0, \Pb)$ are measurable.
        \end{lemma}
        \begin{proof}[\textcolor{seeblau}{Proof.}]
            First, we consider the dilation correspondence $\varphi \oplus \BB(0, \Pb)$ and note that for each $\omega \in \Omega$ it holds
            \begin{equation}\label{EQ:representationDilationBall}
                \varphi(\omega) \oplus \BB(0, \Pb(\omega)) \,=\, \{x \in H \,\colon h(x, \varphi(\omega)) \leq \Pb(\omega)\}\hspace{0.1em}.
            \end{equation}
            Indeed, if $x \in \varphi(\omega) + \BB(0, \Pb(\omega))$, then there exist $y \in \varphi(\omega)$ and $u \in \BB(0, \Pb(\omega))$ with $x = y + u$, such that 
            \begin{equation*}
                h(x, \varphi(\omega)) \,\leq\, h(y, \varphi(\omega)) + \lVert u \hspace{0.05em}\rVert \,=\, \lVert u \hspace{0.05em}\rVert \,\leq\, \Pb(\omega)
            \end{equation*}
            for all $\omega \in \Omega$.
            Vice versa, if $h(x, \varphi(\omega)) \leq \Pb(\omega)$ then there exists some $y \in \varphi(\omega)$ with $\lVert \hspace{0.05em}x-y\hspace{0.05em}\rVert \leq \Pb(\omega)$, such that $x = y + (x-y) \in \varphi(\omega) \oplus \BB(0, \Pb(\omega))$ for all $\omega \in \Omega$. 
            Hence, \eqref{EQ:representationDilationBall} follows and as a direct consequence we obtain that
            \begin{equation}\label{EQ:representationDistanceAsMax}
                h(x, \varphi\oplus \BB(0, \Pb)) \,=\, \max\{h(x, \varphi) - \Pb, 0\}
            \end{equation}
            for all $x \in H$.
            Since $\varphi$ is measurable with nonempty values, Theorem \ref{THM:measurabilityViaDistanceFunction} yields that $(s, x) \mapsto h(x, \varphi(s))$ is a Carathéodory function.
            Using the representation \eqref{EQ:representationDistanceAsMax} this already implies that 
            \begin{equation*}
                (s, x) \,\mapsto\, h(x, \varphi(s) \oplus \BB(0, \Pb(s))) \,=\, \max\{h(x, \varphi(s)) - \Pb(s), 0\}
            \end{equation*}
            is a Carathéodory function.
            Now, since $\varphi \oplus \BB(0, \Pb)$ attains nonempty and compact values, we can again use Theorem \ref{THM:measurabilityViaDistanceFunction} to obtain that $\varphi \oplus \BB(0, \Pb)$ is weakly measurable and by Lemma \ref{LEM:correspondenceDefinitionsEqual} even measurable.
            Let us now consider the erosion correspondence $\varphi \ominus \BB(0, \Pb)$.
            We want to show that $(\varphi \ominus \BB(0, \Pb))^{\hspace{0.05em}\ell}(F) \in \Sigma$ for all closed $F \subseteq H$. 
            To this end, let $F \subseteq X$ be closed and note that since
            \begin{equation*}
                F \,=\, \bigcup\hspace{0.2em}\{F \cap \BB(0, n)\,\colon n \in \NN\hspace{0.05em}\}
            \end{equation*}
            and since the lower inverse of a correspondence is compatible with unions, without loss of generality, we can assume that $F$ is even compact.
            We then obtain that
            \begin{equation}\label{EQ:equivalentFormulationErosionMeasurability}
                \begin{aligned}
                    \varphi(\omega) \ominus \BB(0, \Pb(\omega)) \cap F \neq \emptyset & \quad \Longleftrightarrow \quad \exists \hspace{0.05em} x \in F \,\colon \forall u \in \BB(0, \Pb(\omega)) \,\colon x+u \in \varphi(\omega) \\
                    & \quad \Longleftrightarrow \quad \exists \hspace{0.05em} x \in F \,\colon \forall u \in \BB(0, \Pb(\omega)) \,\colon h(x+u, \varphi(\omega)) \,=\, 0 \\
                    & \quad \Longleftrightarrow \quad \inf\hspace{0.1em}\{\sup\hspace{0.1em}\{ \hspace{0.05em}h(x+u, \varphi(\omega)) \,\colon u \in \BB(0, \Pb(\omega))\} \,\colon x \in F\} \,=\, 0
                \end{aligned}
            \end{equation}
            for all $\omega \in \Omega$, where we used that $\varphi(\omega)$ is closed in the second and $x \mapsto h(x, \varphi(\omega))$ is nonnegative on $H$ as well as the compactness of $F$ in the third equivalence.
            Hence, defining the function 
            \begin{equation*}
                m \colon \Omega \to \RR, \; \omega \,\mapsto\, \inf\hspace{0.1em}\{\sup\hspace{0.1em}\{\hspace{0.05em}h(x+u, \varphi(\omega)) \,\colon u \in \BB(0, \Pb(\omega))\} \,\colon x \in F\} \hspace{0.1em},
            \end{equation*}
            by \eqref{EQ:equivalentFormulationErosionMeasurability} it holds $(\varphi \ominus \BB(0, \Pb))^{\hspace{0.05em}\ell}(F) = m^{-1}(\{0\})$, such that it suffices to show the $\Sigma$\hspace{0.1em}-\hspace{0.05em}measurability of $m$.
            To this end, we define the function
            \begin{equation*}
                g_{\hspace{0.025em}0} \colon \Omega \times H^2 \to \RR, \; (\omega, x, u) \,\mapsto\, h(x+u, \varphi(\omega))
            \end{equation*}
            and note that since $\varphi$ is measurable, by Theorem~\ref{THM:measurabilityViaDistanceFunction}, we know that $(\omega, x) \mapsto h(x, \varphi(\omega))$ is a Carathéodory function, such that $(x, u) \mapsto g(\omega, x, u)$ is jointly continuous for each $\omega \in \Omega$.
            Therefore, we can apply \emph{Berges Maximum Theorem} \cite[Theorem 17.31]{aliprantis2006infinite} to obtain that
            \begin{equation}\label{EQ:caratheodory1}
                x \,\mapsto\, \sup\hspace{0.1em}\{g_{\hspace{0.025em}0}(\omega, x, u) \,\colon u \in \BB(0, \Pb(\omega))\}
            \end{equation}
            is continuous for each $\omega \in \Omega$.
            By Lemma \ref{LEM:generalizedBall} the correspondence $\omega \mapsto \BB(0, \Pb(\omega))$ is measurable, such that we can also use Theorem \ref{THM:measurableMaximumTheorem} to obtain that 
            \begin{equation}\label{EQ:caratheodory2}
                \omega \,\mapsto\, \sup\hspace{0.1em}\{g_{\hspace{0.025em}0}(\omega, x, u) \,\colon u \in \BB(0, \Pb(\omega))\}
            \end{equation}
            is $\Sigma$\hspace{0.1em}-\hspace{0.05em}measurable for each $x \in H$.
            Hence, combining \eqref{EQ:caratheodory1} and \eqref{EQ:caratheodory2}, the function 
            \begin{equation*}
                g_{\hspace{0.015em}1} \colon \Omega \times H \mapsto \RR,\; (\omega, x) \,\mapsto\, \sup\hspace{0.1em}\{g_{\hspace{0.025em}0}(\omega, x, u) \,\colon u \in \BB(0, \Pb(\omega))\}
            \end{equation*}
            is again a Carathéodory function.
            Now, using Theorem \ref{THM:measurableMaximumTheorem}, we obtain that
            \begin{equation*}
                m_1 \colon \Omega \to \RR, \; \omega \,\mapsto\, \sup\hspace{0.1em}\{-g_{\hspace{0.015em}1}(\omega, x) \,\colon x \in F\}
            \end{equation*}
            is $\Sigma$\hspace{0.1em}-\hspace{0.05em}measurable.
            Since $m = -\hspace{0.05em}m_1$ the claim overall follows.
        \end{proof}

        \begin{proof}[\textcolor{seeblau}{Proof of Lemma \ref{LEM:hausdorffExtensionBounds}.}]
            First, we consider the extension map $\Db^{\hspace{-0.05em}+}$.
            By Lemma \ref{LEM:dilationErosionMeasurable} we know that for all $n \in \NN$ the map $\Db_{\hspace{-0.025em}n}^{\hspace{-0.05em}+}$ is measurable as a correspondence and therefore $\Sigma$\hspace{0.1em}-\hspace{0.05em}measurable by Theorem \ref{THM:correspondenceEquivalences}.
            Hence, $\Db^{\hspace{-0.05em}+}$ itself is a stochastic process.
            Furthermore, for each $x \in D$ we can find $y \in \Db_{\hspace{-0.025em}n}\hspace{-0.05em}(\omega)$ with $\lVert \hspace{0.05em}x-y\hspace{0.05em}\rVert \leq d_H(\Db_{\hspace{-0.025em}n}\hspace{-0.05em}(\omega), D)$, such that 
            \begin{equation*}
                x \,=\, y + (x-y) \in \Db_{\hspace{-0.025em}n}\hspace{-0.05em}(\omega) \oplus \BB(0, d_H(\Db_{\hspace{-0.025em}n}\hspace{-0.05em}(\omega), D)) \,=\, \Db_{\hspace{-0.025em}n}^{\hspace{-0.05em}+}\hspace{-0.05em}(\omega)
            \end{equation*}
            for all $\omega \in \Omega$ and $n \in \NN$.
            Thus, $\Db^{\hspace{-0.05em}+}$ satisfies Assumption \ref{ASS:domainContainment} for $\delta = 0$.
            For the upper bound on the Hausdorff distance we use that the Hausdorff distance is a metric on $\Dcc$ to obtain that
            \begin{equation*}
                d_H(\hspace{0.035em}\Db_{\hspace{-0.025em}n}^{\hspace{-0.05em}+}\hspace{-0.05em}(\omega), D) \,\leq\, d_H(\hspace{0.035em}\Db_{\hspace{-0.025em}n}^{\hspace{-0.05em}+}\hspace{-0.05em}(\omega), \Db_{\hspace{-0.025em}n}\hspace{-0.05em}(\omega))  + d_H(\hspace{0.035em}\Db_{\hspace{-0.025em}n}\hspace{-0.05em}(\omega), D) \,\leq\, 2\hspace{0.05em}d_H(\hspace{0.035em}\Db_{\hspace{-0.025em}n}\hspace{-0.05em}(\omega), D)
            \end{equation*}
            for all $\omega \in \Omega$ and $n \in \NN$,
            where in the second inequality we used that for any $y \in \Db_{\hspace{-0.025em}n}^{\hspace{-0.05em}+}\hspace{-0.05em}(\omega)$ there exist $x \in \Db_{\hspace{-0.025em}n}\hspace{-0.05em}(\omega)$, $u \in \BB(0, 1)$ and $r \leq d_H(\hspace{0.035em}\Db_{\hspace{-0.025em}n}\hspace{-0.05em}(\omega), D)$ with $y = x + ru$, such that 
            \begin{equation*}
                h(y, \Db_{\hspace{-0.025em}n}\hspace{-0.05em}(\omega)) \,\leq\, h(x, \Db_{\hspace{-0.025em}n}\hspace{-0.05em}(\omega)) + \lVert \hspace{0.05em}y-x \hspace{0.05em}\rVert \,=\, r \,\leq\, d_H(\hspace{0.035em}\Db_{\hspace{-0.025em}n}\hspace{-0.05em}(\omega), D)\hspace{0.1em}.
            \end{equation*}
            Now, we consider the contraction map $\Db^{\hspace{-0.05em}-}$.
            As before, by Lemma \ref{LEM:dilationErosionMeasurable} we know that for all $n \in \NN$ the map $\Db_{\hspace{-0.025em}n}^{\hspace{-0.05em}-}$ is measurable as a correspondence and therefore $\Sigma$\hspace{0.1em}-\hspace{0.05em}measurable by Theorem \ref{THM:correspondenceEquivalences}.
            Therefore, $\Db^{\hspace{-0.05em}-}$ is a stochastic process. 
            Fix now $\omega \in A$ and $n \in \NN$, where $A$ is defined as the event corresponding to \eqref{EQ:diameterBound}. 
            Let $x \in \Db_{\hspace{-0.025em}n}^{\hspace{0.05em}-}\hspace{-0.05em}(\omega)$. Then, due to the compactness of $D$, there  exists $y \in D$ with $h(x, D) = \lVert \hspace{0.05em}x - y\hspace{0.05em}\rVert$.
            Defining $u = (x-y) / \lVert \hspace{0.05em}x-y\hspace{0.05em}\rVert$ we have $y = x - \lVert \hspace{0.05em}x-y\hspace{0.05em}\rVert u$ but since $x \in \Db_{\hspace{-0.025em}n}^{\hspace{0.05em}-}\hspace{-0.05em}(\omega)$ we also have that $z = x + d_H(\hspace{0.035em}\Db_{\hspace{-0.025em}n}\hspace{-0.05em}(\omega), D)u \in \Db_{\hspace{-0.025em}n}\hspace{-0.05em}(\omega)$. 
            Hence, we find that
            \begin{equation*}
                d_H(\hspace{0.035em}\Db_{\hspace{-0.025em}n}\hspace{-0.05em}(\omega), D) \,\geq\, h(z, D) \,=\, \lVert \hspace{0.05em}z-y\hspace{0.05em}\rVert \,=\, \lVert \hspace{0.05em}d_H(\hspace{0.035em}\Db_{\hspace{-0.025em}n}\hspace{-0.05em}(\omega), D)u + \lVert\hspace{0.05em}x-y\hspace{0.05em} \rVert u\hspace{0.05em}\rVert \,=\, d_H(\hspace{0.035em}\Db_{\hspace{-0.025em}n}\hspace{-0.05em}(\omega), D) + \lVert\hspace{0.05em}x-y\hspace{0.05em} \rVert\hspace{0.1em},
            \end{equation*}
            such that $x = y \in D$.
            Therefore, $\Db^{\hspace{-0.05em}-}$ satisfies Assumption \ref{ASS:domainContainmentExterior} for $\delta = \gamma$ since we restricted ourselves to sample points of the event $A \subseteq \Omega$ with $\PP[\hspace{-0.05em}A\hspace{0.05em}] \geq 1-\gamma$.
            For the upper bound on the Hausdorff distance, similar to before, we have
            \begin{equation*}
                d_H(\hspace{0.035em}\Db_{\hspace{-0.025em}n}^{\hspace{-0.05em}-}\hspace{-0.05em}(\omega), D) \,\leq\, d_H(\hspace{0.035em}\Db_{\hspace{-0.025em}n}^{\hspace{-0.05em}-}\hspace{-0.05em}(\omega), \Db_{\hspace{-0.025em}n}\hspace{-0.05em}(\omega))  + d_H(\hspace{0.035em}\Db_{\hspace{-0.025em}n}\hspace{-0.05em}(\omega), D)\hspace{0.1em},
            \end{equation*}
            such that it suffices to bound $d_H(\hspace{0.035em}\Db_{\hspace{-0.025em}n}^{\hspace{-0.05em}-}\hspace{-0.05em}(\omega), \Db_{\hspace{-0.025em}n}\hspace{-0.05em}(\omega))$ from above, where $\omega \in A$ and $n \in \NN$ remain fixed. 
            By \cite[Theorem~20]{wills2007hausdorff} the Hausdorff distance between two nonempty compact and convex sets can be reduced to their boundaries.
            Hence, we can consider $x \in \bd(\Db_{\hspace{0.025em}n}\hspace{-0.05em}(\omega))$ and there exists $y \in \bd(\Db_{\hspace{0.025em}n}^{\hspace{-0.05em}-}\hspace{-0.05em}(\omega)) \subseteq \Db_{\hspace{0.025em}n}\hspace{-0.05em}(\omega)$ with $r \coloneqq h(x, \Db_{\hspace{0.025em}n}^{\hspace{-0.05em}-}\hspace{-0.05em}(\omega)) = \lVert \hspace{0.05em}x-y\hspace{0.05em}\rVert$. 
            Since by choice of $\omega$ we have that $\Db_{\hspace{0.025em}n}\hspace{-0.05em}(\omega)$ is $\alpha_n$-\hspace{0.05em}strongly convex, writing $u = (x-y) / r$, analogously to \eqref{EQ:property1}, we obtain that 
            \begin{equation*}
                \left(\hspace{-0.1em}x-\frac{r}{2}u \hspace{-0.1em}\right) + \BB\hspace{-0.25em}\left(\hspace{-0.15em}0, \frac{\alpha_n r^2}{8}\hspace{-0em}\right) \,\subseteq\, \Db_{\hspace{0.025em}n}\hspace{-0.05em}(\omega)\hspace{0.1em}.
            \end{equation*}
            If it would now hold that
            \begin{equation}\label{EQ:negationBound}
                h(x, \Db_{\hspace{-0.025em}n}\hspace{-0.05em}(\omega)) \,=\, r \,\geq \, \sqrt{\frac{8}{\alpha_n}d_H(\hspace{0.035em}\Db_{\hspace{-0.025em}n}\hspace{-0.05em}(\omega), D)}\hspace{0.1em},
            \end{equation}
            then we would obtain
            \begin{equation*}
                \left(\hspace{-0.1em}x-\frac{r}{2}u \hspace{-0.1em}\right) + \BB(0, d_H(\hspace{0.035em}\Db_{\hspace{-0.025em}n}\hspace{-0.05em}(\omega), D))  \,\subseteq\, \left(\hspace{-0.1em}x-\frac{r}{2}u \hspace{-0.1em}\right) + \BB\hspace{-0.25em}\left(\hspace{-0.15em}0, \frac{\alpha_n r^2}{8}\hspace{-0em}\right) \,\subseteq\, \Db_{\hspace{0.025em}n}\hspace{-0.05em}(\omega)\hspace{0.1em},
            \end{equation*}
            such that $x - (r/2)\hspace{0.05em}u \in \Db_{\hspace{0.025em}n}^{\hspace{-0.05em}-}\hspace{-0.05em}(\omega)$.
            However, this is a contradiction to the choice of $y \in \Db_{\hspace{0.025em}n}^{\hspace{-0.05em}-}(\omega)$, such that the negation of \eqref{EQ:negationBound} must hold.
            The claim now follows since we restricted ourselves to sample points $\omega \in A$ and it holds $\PP[\hspace{-0.05em}A\hspace{0.05em}]\geq 1-\gamma$ by assumption.
        \end{proof}

    \subsection{Proofs of Section 5} \label{app:sec5:proofs}
    
    \paragraph{Simple Example: Quadratic Objective on Rectangular Domain.}
    To begin with the derivation of uniform concentration bounds for the domain approximation processes $\Db^\text{MB}$ and $\Db^\text{CH}$ defined as in Section \ref{SEC:examples}, we first cite a well\hspace{0.05em}-\hspace{0.035em}known concentration bound for independent bounded random variables \cite[Theorem 2.8]{boucheron2013concentration}.
    
    \begin{theorem}{Hoeffding's Inequality}{hoeffdingInequality}
        Let $\Yb_{\hspace{-0.15em}1},\ldots, \Yb_{\hspace{-0.125em}n}$ be independent random variables such that $a_i \leq \Yb_{\hspace{-0.15em}i} \leq b_i$ almost surely for all $i \in [n]$. 
        Then, for all $\varepsilon > 0$ it holds that
        \begin{equation}\label{EQ:hoeffdingInequality}
            \PP\hspace{-0.2em}\left[\hspace{0.1em}\sum_{\hspace{0.1em}i\hspace{0.05em} = 1}^n(\Yb_{\hspace{-0.15em}i} -\EE[\Yb_{\hspace{-0.15em}i}]\hspace{0.05em}) \,\geq\, \varepsilon\hspace{0.1em}\right] \,\leq\, \hspace{-0.05em}\exp\hspace{-0.2em}\left(\hspace{-0.2em}-\frac{2\hspace{0.05em}\varepsilon^2}{\sum\limits_{i = 1}^n(b_i-a_i)}\right)\hspace{-0.2em}.
        \end{equation}
    \end{theorem}

    \noindent
    In the following, we use Hoeffding's inequality \eqref{EQ:hoeffdingInequality} to derive the constant seen in Proposition \ref{PRO:basicApproximationProperties} and Proposition~\ref{PRO:convexHullApproximationQuality}.
    Recall that for $x > 0$ the (complete) Gamma function $\Gamma$ (see \cite[Section 1]{andrews1999special} for more details) can be written as
    \begin{equation*}
        \Gamma(x) \,=\, \int_0^\infty t^{\hspace{0.05em}x-1}e^{-t} \,dt\hspace{0.1em}
    \end{equation*}
    and satisfies $\Gamma(n) = (n-1)!$ for all $n \in \NN$. 
    Furthermore, recall the definition of the empirical mean estimator $\Eb_{\hspace{0.015em}n}\hspace{-0.025em}(\Yb)$ and the empirical variance estimator $\Vb_{\hspace{-0.1em}n}\hspace{-0.025em}(\Yb)$ for any random variable $\Yb$ and $n \in \NN$ as given in \eqref{EQ:empiricalEstimators}.

    \begin{lemma}{}{empiricalMeanConstant}
        Let $\Yb$ be a random variable such that $a \leq \Yb \leq b$ almost surely. 
        Then, for all $r \in [\hspace{0.025em}0, 1/2)$ and for all $\beta \in (0, 1]$ there exists a constant $c > 0$ such that
        \[
            \PP\hspace{-0.2em}\left[\hspace{0.05em}\lvert\hspace{0.05em} \Eb_{\hspace{0.015em}n}\hspace{-0.05em}(\Yb) \hspace{-0.1em}-\hspace{0.05em} \EE[\Yb]\hspace{0.05em}\rvert \,\leq\, c\hspace{0.05em} n^{\hspace{-0.1em}-r} \text{ for all } n \in \NN\hspace{0.05em}\right] \,\geq\, 1-\beta\hspace{0.1em}.
        \]
        In particular, writing $s = 1-2\hspace{0.05em}r$, we can choose 
        \[
            c \,=\, (b-a)\hspace{-0.15em}\left(\hspace{-0.1em}\frac{\Gamma(1/s)}{s\hspace{0.05em}\beta}\hspace{-0.1em}\right)^{\hspace{-0.3em} s/2}\hspace{-0.25em}.
        \]
    \end{lemma}
    \begin{proof}[\textcolor{seeblau}{Proof.}]
        Since we know that $\Yb$ has values in $[\hspace{0.025em} a, b\hspace{0.05em}]$, we can use Hoeffding's inequality \eqref{EQ:hoeffdingInequality} to obtain that
        \[
            \PP[\hspace{0.05em}\lvert\hspace{0.05em} \Eb_{\hspace{0.015em}n}\hspace{-0.05em}(\Yb)-\hspace{0.05em}\EE[\Yb]\hspace{0.05em}\rvert \geq \varepsilon\hspace{0.05em}] \,\leq\, 2\hspace{-0.05em}\exp\hspace{-0.25em}\left(\hspace{-0.2em}-\frac{\hspace{-0.1em} 2\hspace{0.05em}\varepsilon^2}{(b-a)^2}\hspace{0.1em} n\hspace{-0.2em}\right) 
        \]
        for all $\varepsilon > 0$. 
        Hence, for any $r \in \RR$ and any constant $c > 0$ this implies
        \[
            \PP[\hspace{0.05em}\lvert\hspace{0.05em} \Eb_{\hspace{0.015em}n}\hspace{-0.05em}(\Yb)-\hspace{0.05em}\EE[\Yb]\hspace{0.05em}\rvert \geq c\hspace{0.05em} n^{\hspace{-0.05em}-r}\hspace{0.05em}] \,\leq\, 2\hspace{-0.05em}\exp\hspace{-0.25em}\left(\hspace{-0.2em}-\frac{\hspace{-0.1em} 2\hspace{0.05em} c^{\hspace{0.05em}2}}{(b-a)^2}\hspace{0.1em} n^{1-2\hspace{0.025em}r}\hspace{-0.2em}\right)\hspace{-0.1em},
        \]
        such that using the union bound we can derive that
        \begin{equation}\label{EQ:inverseResult}
            \begin{aligned}
                \PP[\hspace{0.05em}\lvert\hspace{0.05em} \Eb_{\hspace{0.015em}n}\hspace{-0.05em}(\Yb) -\hspace{0.05em} \EE[\Yb]\hspace{0.05em}\rvert \geq c\hspace{0.05em} n^{\hspace{-0.1em}-r} \text{ for some } n \in \NN\hspace{0.05em}] &\,\leq\, \sum_{n\in \hspace{0.05em}\NN} \,\PP[\hspace{0.05em}\lvert\hspace{0.05em} \Eb_{\hspace{0.015em}n}\hspace{-0.05em}(\Yb) - \hspace{0.05em}\EE[\Yb]\hspace{0.05em}\rvert \geq c\hspace{0.05em} n^{\hspace{-0.1em}-r}\hspace{0.05em}] \\
                & \,\leq\, 2\hspace{-0.1em} \sum_{n \in \hspace{0.05em}\NN}\,\exp(-Cn^s)\hspace{0.1em},
            \end{aligned}
        \end{equation}
        where $C = 2\hspace{0.05em}c^{\hspace{0.05em}2} / (b-a)^{-2}$ and $s = 1-2\hspace{0.05em}r$. 
        The last sum in \eqref{EQ:inverseResult} converges if and only if $C > 0$ and $s > 0$, which is equivalent to $c > 0$ and $r < 1/2$, what is assumed to be true in the following. 
        Since $x \mapsto \exp(-Cx^{s})$ is monotonically decreasing on $[\hspace{0.025em}0, \infty)$ we can bound the sum by its integral equivalent and use the substitution $t = Cx^s$ to obtain
        \begin{equation}\label{EQ:boundInfiniteSum}
            \sum_{n \in \hspace{0.05em}\NN} \,\exp(-Cn^{s}) \,\leq\, \int_0^\infty \hspace{-0.2em}\exp(-Cx^{s}) \; dx \,=\, C^{\hspace{0.05em}-1/s}\frac{\Gamma(1/s)}{s}\hspace{0.1em}.
        \end{equation}
        Furthermore, a simple transformation yields that
        \begin{equation*}
            2\hspace{0.05em} C^{\hspace{0.05em}-1/s}\frac{\Gamma(1/s)}{s} \,\leq\, \beta \quad \Longleftrightarrow \quad \left(\hspace{-0.1em}\frac{2\hspace{0.05em}\Gamma(1/s)}{s\hspace{0.05em}\beta}\hspace{-0.1em}\right)^{\hspace{-0.2em}s} \,\leq\, C \quad \Longleftrightarrow \quad \sqrt{\hspace{-0.05em}\frac{(b-a)^2}{\hspace{-0.3em}2}\hspace{-0.15em}\left(\hspace{-0.1em}\frac{2\hspace{0.05em}\Gamma(1/s)}{s\hspace{0.05em}\beta}\hspace{-0.1em}\right)^{\hspace{-0.2em}s}\hspace{0.1em}} \,\leq\, c\hspace{0.1em},
        \end{equation*}
        such that in combination with \eqref{EQ:inverseResult} and \eqref{EQ:boundInfiniteSum} we have
        \begin{equation}\label{EQ:event}
            \PP[\hspace{0.05em}\lvert\hspace{0.05em} \Eb_{\hspace{0.015em}n}\hspace{-0.05em}(\Yb) -\hspace{0.05em} \EE[\Yb]\hspace{0.05em}\rvert \,\geq\, c\hspace{0.05em} n^{\hspace{-0.1em}-r} \text{ for some } n \in \NN\hspace{0.05em}] \,\leq\, \beta \quad \Longleftrightarrow \quad c \,\geq\, \sqrt{\hspace{-0.05em}\frac{(b-a)^2}{\hspace{-0.3em}2}\hspace{-0.15em}\left(\hspace{-0.1em}\frac{2\hspace{0.05em}\Gamma(1/s)}{s\hspace{0.05em}\beta}\hspace{-0.1em}\right)^{\hspace{-0.2em}s}\hspace{0.1em}}\hspace{0.1em}.
        \end{equation}
        The first part of the statement now follows by taking the complement of the event in \eqref{EQ:event}. 
        The second part of the statement follows since
        \[
            \frac{(b-a)^2}{\hspace{-0.3em}2}\hspace{-0.15em}\left(\hspace{-0.1em}\frac{2\hspace{0.05em}\Gamma(1/s)}{s\hspace{0.05em}\beta}\hspace{-0.1em}\right)^{\hspace{-0.2em}s} \,\leq\, \frac{(b-a)^2}{\hspace{-0.3em}2^{\hspace{0.025em}s}}\hspace{-0.15em}\left(\hspace{-0.1em}\frac{2\hspace{0.05em}\Gamma(1/s)}{s\hspace{0.05em}\beta}\hspace{-0.1em}\right)^{\hspace{-0.2em}s} \,=\, (b-a)^2 \hspace{-0.15em}\left(\hspace{-0.1em}\frac{\Gamma(1/s)}{s\hspace{0.05em}\beta}\hspace{-0.1em}\right)^{\hspace{-0.2em}s}\hspace{-0.1em},
        \]
        where we used that $s = 1-2\hspace{0.05em}r \in (0, 1]$.
    \end{proof}

    \noindent
    We note that for $\Sigma$\hspace{0.1em}-\hspace{0.05em}measurable functions $g_{\hspace{0.025em}1}$ and $g_2$ and $\lambda \in [0, 1]$ and $\varepsilon \in \RR$ we have that
    \begin{equation}\label{EQ:probabilisticSetBound}
        \{g_{\hspace{0.025em}1} + g_2 \,\leq\, \varepsilon\} \,\supseteq\, \{g_{\hspace{0.025em}1} \,\leq\, \lambda\hspace{0.05em}\varepsilon\} \,\cap\, \{g_2 \,\leq\, (1-\lambda)\hspace{0.05em}\varepsilon\}\hspace{0.1em}.
    \end{equation}
    The statement in \eqref{EQ:probabilisticSetBound} will be used multiple times later on and is easier to grasp in this simplistic form.

    \begin{lemma}{}{empiricalVarianceConstant}
        Let $\Yb$ be a random variable such that $a \leq \Yb \leq b$ almost surely. 
        Then, for all $r \in [\hspace{0.025em}0, 1/2)$ and for all $\beta \in (0, 1]$ there exists a constant $c > 0$ such that
        \[
            \PP\hspace{-0.2em}\left[\hspace{0.05em}\lvert\hspace{0.05em} \Vb_{\hspace{-0.15em} n}\hspace{-0.05em}(\Yb) \hspace{-0.1em}-\hspace{0.05em}\VV[\Yb]\hspace{0.05em}\rvert \,\leq\, c\hspace{0.05em} n^{\hspace{-0.1em}-r} \text{ for all } n \in \NN\hspace{0.05em}\right] \,\geq\, 1-\beta\hspace{0.1em}.
        \]
        In particular, writing $s = 1-2\hspace{0.05em}r$, we can choose 
        \[
            c \,=\, (b-a)^2\hspace{-0.2em}\left(\hspace{-0.1em}1+\left(\hspace{-0.1em}\frac{2\hspace{0.05em}\Gamma(1/s)}{s\hspace{0.05em}\beta}\hspace{-0.1em}\right)^{\hspace{-0.25em} s/2}\right)^{\hspace{-0.275em}2}\hspace{-0.1em}.
        \]
    \end{lemma}
    \begin{proof}[\textcolor{seeblau}{Proof.}]
        First, we define the random variable $\Zb = (\Yb - \EE[\Yb]\hspace{0.05em})^2$ and note that $\EE[\Zb] = \VV[\Yb]$. Then, we obtain
        \begin{align*}
            \lvert \hspace{0.05em}\Vb_{\hspace{-0.15em} n}\hspace{-0.05em}(\Yb) - \hspace{0.05em}\Eb_{\hspace{0.015em}n}\hspace{-0.05em}(\Zb)\hspace{0.05em} \rvert 
            & \,=\, \left\lvert \,\frac{1}{n}\hspace{0.05em} \sum_{i \hspace{0.05em}= 1}^n \left((\hspace{0.035em}\Yb_{\hspace{-0.1em} i}-\Eb_{\hspace{0.015em}n}\hspace{-0.05em}(\Yb))^2 - (\Yb_{\hspace{-0.1em} i}-\EE[\Yb]\hspace{0.05em})^2\hspace{0.05em}\right)\right\rvert \\
            & \,=\, \left\lvert \,\frac{1}{n}\hspace{0.05em} \sum_{i \hspace{0.05em}= 1}^n \left(\hspace{0.1em} \Eb_{\hspace{0.015em}n}\hspace{-0.05em}(\Yb)^2 - \hspace{0.05em}\EE[\Yb]^2 - 2\hspace{0.05em} \Yb_{\hspace{-0.15em} i}(\hspace{0.035em}\Eb_{\hspace{0.015em}n} \hspace{-0.05em}(\Yb)-\hspace{0.05em} \EE[\Yb]\hspace{0.05em})\right)\right\rvert \\
            & \,=\, \lvert\hspace{0.05em}(\hspace{0.035em}\Eb_{\hspace{0.015em}n}\hspace{-0.05em}(\Yb)+\hspace{0.05em}\EE[\Yb]\hspace{0.05em}) (\hspace{0.035em}\Eb_{\hspace{0.015em}n}\hspace{-0.05em}(\Yb) -\hspace{0.05em} \EE[\Yb]\hspace{0.05em}) - 2\hspace{0.05em}\Eb_{\hspace{0.015em}n}\hspace{-0.05em}(\Yb)(\hspace{0.035em}\Eb_{\hspace{0.015em}n}\hspace{-0.05em}(\Yb)-\hspace{0.05em}\EE[\Yb])\hspace{0.05em}\rvert \\
            &\,=\, \lvert\hspace{0.05em}\Eb_{\hspace{0.015em}n}\hspace{-0.05em}(\Yb)-\hspace{0.05em}\EE[\Yb]\hspace{0.05em}\rvert^2 
        \end{align*}
        such that
        \begin{align*}
            \lvert \hspace{0.05em}\Vb_{\hspace{-0.15em} n}\hspace{-0.05em}(\Yb) - \VV[\Yb]\hspace{0.05em}\rvert \,\leq\, \lvert \hspace{0.05em}\Vb_{\hspace{-0.15em} n}\hspace{-0.05em}(\Yb) - \hspace{0.05em}\Eb_{\hspace{0.015em}n}\hspace{-0.05em}(\Zb)\hspace{0.05em} \rvert + \lvert \hspace{0.05em}\Eb_{\hspace{0.015em}n}\hspace{-0.05em}(\Zb) - \VV[\Yb] \hspace{0.05em}\rvert \,=\, \lvert\hspace{0.05em}\Eb_{\hspace{0.015em}n}\hspace{-0.05em}(\Yb)-\EE[\Yb]\hspace{0.05em}\rvert^2  + \lvert \hspace{0.05em}\Eb_{\hspace{0.015em}n}\hspace{-0.05em}(\Zb) -\hspace{0.05em} \EE[\Zb] \hspace{0.05em}\rvert\hspace{0.1em}.
        \end{align*}
        Hence, setting 
        \[
            x \,\coloneqq\, \left(\hspace{-0.1em}\frac{2\hspace{0.05em}\Gamma(1/s)}{s\hspace{0.05em}\beta}\hspace{-0.1em}\right)^{\hspace{-0.2em} s} \quad \text{ and } \quad c \,\coloneqq\, (b-a)^2\hspace{-0.1em}\left(x + \frac{\sqrt{x\hspace{0.1em}}\hspace{0.3em}}{\hspace{0.3em}4}\right)\hspace{-0.1em},
        \]
        where $s = 1-2\hspace{0.05em}r$, and using \eqref{EQ:probabilisticSetBound} we have
        \begin{align*}
            \left\{\hspace{0.05em}\lvert \hspace{0.05em}\Vb_{\hspace{-0.15em} n}\hspace{-0.05em}(\Yb) - \VV[\Yb]\hspace{0.05em}\rvert \leq c\hspace{0.05em} n^{\hspace{-0.1em}-r}\right\} & \supseteq \left\{\hspace{0.05em}\lvert\hspace{0.05em}\Eb_{\hspace{0.015em}n}\hspace{-0.05em}(\Yb)-\EE[\Yb]\hspace{0.05em}\rvert^2  + \lvert \hspace{0.05em}\Eb_{\hspace{0.015em}n}\hspace{-0.05em}(\Zb) -\hspace{0.05em} \EE[\Zb] \hspace{0.05em}\rvert \leq c\hspace{0.05em} n^{\hspace{-0.1em}-r}\right\} \\
            & \supseteq \left\{\hspace{0.05em}\lvert\hspace{0.05em}\Eb_{\hspace{0.015em}n}\hspace{-0.05em}(\Yb)-\EE[\Yb]\hspace{0.05em}\rvert^2 \leq (b-a)^2x\hspace{0.05em} n^{\hspace{-0.1em}-r}\right\} \cap \left\{\hspace{-0.05em}\lvert \hspace{0.05em}\Eb_{\hspace{0.015em}n}\hspace{-0.05em}(\Zb) -\hspace{0.05em} \EE[\Zb] \hspace{0.05em}\rvert \leq \frac{(b-a)^2\hspace{-0.1em}}{\hspace{-0.15em}4\hspace{0.15em}}\sqrt{x} n^{\hspace{-0.1em}-r}\hspace{-0.05em}\right\} \\
            & \supseteq \left\{\hspace{0.05em}\lvert\hspace{0.05em}\Eb_{\hspace{0.015em}n}\hspace{-0.05em}(\Yb)-\EE[\Yb]\hspace{0.05em}\rvert \leq (b-a)\sqrt{x}\hspace{0.05em} n^{\hspace{-0.1em}-r}\right\} \cap \left\{\hspace{-0.05em}\lvert \hspace{0.05em}\Eb_{\hspace{0.015em}n}\hspace{-0.05em}(\Zb) -\hspace{0.05em} \EE[\Zb] \hspace{0.05em}\rvert \leq \frac{(b-a)^2\hspace{-0.1em}}{\hspace{-0.15em}4\hspace{0.15em}}\sqrt{x} n^{\hspace{-0.1em}-r}\hspace{-0.05em}\right\}
        \end{align*}
        for all $n \in \NN$, such that taking the intersection over all $n \in \NN$ and applying Lemma \ref{LEM:empiricalMeanConstant} we obtain that
        \[
            \PP\hspace{-0.2em}\left[\hspace{0.05em}\lvert \hspace{0.05em}\Vb_{\hspace{-0.15em} n}\hspace{-0.05em}(\Yb) - \VV[\Yb]\hspace{0.05em}\rvert \leq c\hspace{0.05em} n^{\hspace{-0.1em}-r} \text{ for all } n \in \NN\hspace{0.05em}\right] \,\geq\, \left(1-\frac{\beta}{2}\right) + \left(1-\frac{\beta}{2}\right) - 1 \,=\, 1-\beta\hspace{0.1em},
        \]
        where we used that $\Zb \in [0, (b-a)^2/4]$.
        The second part of the statement holds since
        \[
            x + \frac{\sqrt{x\hspace{0.1em}}\hspace{0.3em}}{\hspace{0.3em}4} \,\leq\, x + 2\sqrt{x} + 1 \,=\, \left(\sqrt{x} + 1\right)^2\hspace{-0.1em},
        \]
        such that overall the claim follows.
    \end{proof}

    \noindent
    Now, we are ready to prove Proposition \ref{PRO:basicApproximationProperties} itself. 
    Recall that here $\Yb \sim \Unif(a, b)$.

    \begin{proof}[\textcolor{seeblau}{Proof of Proposition \ref{PRO:basicApproximationProperties}}]
        First, we set
        \begin{equation*}
            x \,\coloneqq\, \left(\hspace{-0.1em}\frac{8\hspace{0.05em}\Gamma(1/s)}{s\hspace{0.05em}\beta}\hspace{-0.1em}\right)^{\hspace{-0.2em} s} \quad \text{and} \quad c \,\coloneqq\, (b-a)\sqrt{x} + 6(b-a)(1+\sqrt{x})^2\hspace{0.1em},
        \end{equation*}
        where $s = 1-2\hspace{0.05em}r$, and use \eqref{EQ:probabilisticSetBound} to obtain
        \begin{equation}\label{EQ:inclusionUpper}
            \begin{aligned}
                \hspace{-0.75em}\left\{\lvert \Ab_{\hspace{0.015em}n}^{\hspace{-0.15em}\text{MB}}\hspace{-0.05em}-a\vert \leq c\hspace{0.05em} n^{\hspace{-0.1em}-r}\right\}
                & \supseteq \left\{\lvert \Eb_{\hspace{0.015em}n}\hspace{-0.05em}(\Yb)-\mu\rvert + \sqrt{3\hspace{0.1em}}\hspace{0.1em} \left\lvert \sqrt{\Vb_{\hspace{-0.15em} n}\hspace{-0.05em}(\Yb)}-\sigma\right\rvert \leq c\hspace{0.05em} n^{\hspace{-0.1em}-r}\right\} \\
                & \supseteq \left\{\lvert \Eb_{\hspace{0.015em}n}\hspace{-0.05em}(\Yb)-\mu\rvert \leq (b-a)\sqrt{x}n^{\hspace{-0.1em}-r}\right\} \cap \left\{\left\lvert \Vb_{\hspace{-0.15em} n}\hspace{-0.05em}(\Yb)-\sigma^2\right\rvert \leq \frac{6(b-a)(1+\sqrt{x})^2\sigma }{\sqrt{3}}\hspace{0.05em}n^{\hspace{-0.1em}-r}\right\} \\
                & \supseteq \left\{\lvert \Eb_{\hspace{0.015em}n}\hspace{-0.05em}(\Yb)-\mu\rvert \leq (b-a)\sqrt{x}n^{\hspace{-0.1em}-r}\right\} \cap \left\{\left\lvert \Vb_{\hspace{-0.15em} n}\hspace{-0.05em}(\Yb)-\sigma^2\right\rvert \leq (b-a)^2(1+\sqrt{x})^2\hspace{0.05em}n^{\hspace{-0.1em}-r}\right\}\hspace{-0.1em},
            \end{aligned}
        \end{equation}
        where for the second inclusion we used that
        \[
            \left\lvert \sqrt{\Vb_{\hspace{-0.15em} n\hspace{0.1em}}\hspace{-0.05em}(\Yb)}-\sigma\right\rvert \,=\,\frac{\left\lvert \Vb_{\hspace{-0.15em} n}\hspace{-0.05em}(\Yb)-\sigma^2\right\rvert}{\Vb_{\hspace{-0.15em} n}\hspace{-0.05em}(\Yb) + \sigma} \,\leq\, \frac{\left\lvert \Vb_{\hspace{-0.15em} n}\hspace{-0.05em}(\Yb)-\sigma^2\right\rvert}{\sigma} 
        \]
        since $(\sqrt{\Vb_{\hspace{-0.15em} n}\hspace{-0.05em}(\hspace{0.035em}\Xb)\hspace{0.1em}}-\sigma)(\sqrt{\Vb_{\hspace{-0.15em} n}\hspace{-0.05em}(\hspace{0.035em}\Xb)\hspace{0.1em}}+\sigma) = \Vb_{\hspace{-0.15em} n}\hspace{-0.05em}(\hspace{0.035em}\Xb)-\sigma^2$ and $\sqrt{\Vb_{\hspace{-0.15em} n}\hspace{-0.05em}(\hspace{0.035em}\Xb)\hspace{0.1em}} \geq 0$ for all $n \in \NN$.
        Hence, forming the intersection over all $n \in \NN$ of \eqref{EQ:inclusionUpper} and using Lemma \ref{LEM:empiricalMeanConstant} and Lemma \ref{LEM:empiricalVarianceConstant}, we obtain that
        \[
            \PP\hspace{-0.2em}\left[\hspace{0.05em}\lvert \hspace{0.05em}\Ab_{\hspace{0.015em}n}^{\hspace{-0.15em}\text{MB}}\hspace{-0.05em}(\Yb) - a\hspace{0.05em}\rvert \leq c\hspace{0.05em} n^{\hspace{-0.1em}-r} \text{ for all } n \in \NN\hspace{0.05em}\right] \,\geq\, \left(1-\frac{\beta}{8}\right) + \left(1-\frac{\beta}{4}\right) - 1 \,\geq\, 1-\frac{\beta}{2}\hspace{0.1em}.
        \]
        Analogously, we can derive that
         \[
            \PP\hspace{-0.2em}\left[\hspace{0.05em}\lvert \hspace{0.05em}\Bb_n^\text{MB}\hspace{-0.05em}(\Yb) - b\hspace{0.05em}\rvert \leq c\hspace{0.05em} n^{\hspace{-0.1em}-r} \text{ for all } n \in \NN\hspace{0.05em}\right] \,\geq\, 1-\frac{\beta}{2}
        \]
        and since by \cite[Theorem~20]{wills2007hausdorff} the Hausdorff distance between two nonempty, compact, and convex sets can be reduced to their boundaries, we have $d_H(\hspace{0.035em}\Db_{\hspace{-0.025em}n}^\text{MB}, D) \,=\, \max\hspace{0.1em}\{\lvert \Ab_{\hspace{0.015em}n}^{\hspace{-0.15em}\text{MB}}-a\rvert, \lvert \Bb_n^\text{MB}-b\rvert\}$.
        Overall this implies that
        \begin{equation*}
            \PP\hspace{-0.2em}\left[\hspace{0.05em}d_H(\hspace{0.035em}\Db_{\hspace{-0.025em}n}^\text{MB}, D) \leq c\hspace{0.05em} n^{\hspace{-0.1em}-r} \text{ for all } n \in \NN\hspace{0.05em}\right] \,\geq\, \left(1-\frac{\beta}{2}\right) + \left(1-\frac{\beta}{2}\right) - 1 \,=\, 1-\beta\hspace{0.1em},
        \end{equation*}
        such that by
        \begin{equation*}
            c \,\leq\, 6\hspace{0.05em}(b-a)\hspace{-0.15em}\left(\sqrt{x} + (1 + \sqrt{x})^2\right) \,\leq\, 6\hspace{0.05em}(b-a)(2 + \sqrt{x})^2
        \end{equation*}
        the claim follows.
    \end{proof}

    \noindent

    \begin{proof}[\textcolor{seeblau}{Proof of Proposition \ref{PRO:convexHullApproximationQuality}}]
        We set $L = b-a$. 
        Then, for $t \in [0, L]$ we have
        \begin{equation*}
            \PP[\Ab_{\hspace{0.015em}n}^{\hspace{-0.15em}\text{CH}} - a > t] \,=\, \PP[\Yb_i > t +a \text{ for all } i \in [n]] \,=\, \left(\frac{b-(a + t)}{L}\right)^{\hspace{-0.25em}n} \,=\, \left(1 - \frac{t}{L}\right)^{\hspace{-0.25em}n}\hspace{-0.2em},
        \end{equation*}
        for all $n \in \NN$.
        Using that $(1-x)^n \leq \exp(-nx)$ for all $x \in [0, 1]$ and $n \in \NN$ this then yields that
        \begin{equation*}
            \PP[\Ab_{\hspace{0.015em}n}^{\hspace{-0.15em}\text{CH}} - a > t] \,\leq\, \exp\hspace{-0.1em}\left(-\frac{n\hspace{0.05em}t}{L}\right)\hspace{-0.1em},
        \end{equation*}
        such that
        \begin{equation*}
            \PP[\Ab_{\hspace{0.015em}n}^{\hspace{-0.15em}\text{CH}} - a > c\hspace{0.05em}n^{-r}] \,\leq\, \exp\hspace{-0.2em}\left(-\frac{c}{L}n^{1-r}\right) \,\eqqcolon\, \exp\left(-C\hspace{0.05em}n^s\right)\hspace{0.05em},
        \end{equation*}
        where $C = c/L$ and $s = 1-r$. 
        Analogously we find that
        \begin{equation*}
            \PP[\hspace{0.05em}b - \Bb_n^{\text{CH}} > c\hspace{0.05em}n^{-r}] \,\leq\, \exp\left(-C\hspace{0.05em}n^s\right)
        \end{equation*}
        and since by \cite[Theorem~20]{wills2007hausdorff} the Hausdorff distance between two nonempty, compact, and convex sets can be reduced to their boundaries we have $d_H(\hspace{0.035em}\Db_{\hspace{-0.025em}n}^{\text{CH}}, D) \,=\, \max\hspace{0.1em}\{ \Ab_{\hspace{0.015em}n}^{\hspace{-0.15em}\text{CH}}-a, b - \Bb_n^\text{CH}\}$, such that
        \begin{equation*}
            \PP[d_H(\Db_{\hspace{-0.025em}n}^{\text{CH}}, D) > c\hspace{0.05em}n^{-r}] \,\leq\, \PP[\max\hspace{0.1em}\{\Ab_{\hspace{0.015em}n}^{\hspace{-0.15em}\text{CH}} - a, b - \Bb_n^{\text{CH}}\} > c\hspace{0.05em}n^{-r}] \,\leq\, 2\exp\left(-C\hspace{0.05em}n^s\right)\hspace{0.1em}.
        \end{equation*}
        Using the union bound we can then obtain that
        \begin{equation}\label{EQ:inverseResult2}
            \begin{aligned}
                \PP[\hspace{0.05em}d_H(\Db_{\hspace{-0.025em}n}^{\text{CH}}, D) > c\hspace{0.05em} n^{\hspace{-0.1em}-r} \text{ for some } n \in \NN\hspace{0.05em}] \,\leq\, \sum_{n\in \hspace{0.05em}\NN} \,\PP[\hspace{0.05em}d_H(\Db_{\hspace{-0.025em}n}^{\text{CH}}, D) > c\hspace{0.05em} n^{\hspace{-0.1em}-r}\hspace{0.05em}] \,\leq\, 2\hspace{-0.1em} \sum_{n \in \hspace{0.05em}\NN}\,\exp(-Cn^s)\hspace{0.1em},
            \end{aligned}
        \end{equation}
        where the last sum in \eqref{EQ:inverseResult2} converges if and only if $C > 0$ and $s > 0$, which is equivalent to $c > 0$ and $r < 1$.
        We can now proceed similarly as in the proof of Lemma \ref{LEM:empiricalMeanConstant} to obtain that
        \begin{equation}\label{EQ:lastBound}
            \PP[\hspace{0.05em}d_H(\Db_{\hspace{-0.025em}n}^{\text{CH}}, D) > c\hspace{0.05em} n^{\hspace{-0.1em}-r} \text{ for some } n \in \NN\hspace{0.05em}] \,\leq\, \beta \quad \Longleftrightarrow \quad c \,\geq\, (b-a)\hspace{-0.2em}\left(\hspace{-0.1em}\frac{2\hspace{0.05em}\Gamma(1/s)}{s\hspace{0.05em}\beta}\hspace{-0.1em}\right)^{\hspace{-0.25em}s}\hspace{-0.1em},
        \end{equation}
        such that the claim follows by taking the complement of the event in \eqref{EQ:lastBound}.
    \end{proof}

    \paragraph{Data\hspace{0.1em}-Driven Distributionally Robust Linear Quadratic Control Problem.} For the rest of this section let always $d \in \{n, p\}$.
    Similar to the definition of the restricted Gelbrich ball in \eqref{EQ:restrictedGelbrich}, for $X \in \SS_+^{\hspace{0.025em}d}$ and $\rho \geq 0$ we can define the restricted Gelbrich ball
    \begin{equation*}
        \GG(\hspace{-0.05em}X, \rho) \,=\, \left\{Y \in \SS_+^{\hspace{0.025em}d} \,\colon d_{\hspace{0.025em}G}(X, Y)^2 \leq \rho^2 \text{ and } Y \succeq \lambda_\text{min}(X)I\right\}.
    \end{equation*}
    First, we show that the candidate for a domain approximation process $\Db^\GG$ defined as in \eqref{EQ:gelbricgDomainApproximation} is a well\hspace{0.05em}-\hspace{0.035em}defined stochastic process. To this end, we have to derive some bounds on the Gelbrich distance and the Frobenius norm with respect to each other.
    Note that it is already well\hspace{0.05em}-\hspace{0.035em}known that
    \begin{equation}\label{EQ:knownEquivalence}
        \lVert X \rVert_2 \,\leq\, \lVert X \rVert_F \,\leq\, \trace(X) \,\leq\, \sqrt{\hspace{-0.1em}d\hspace{0.05em}}\hspace{0.1em}\lVert X \rVert_F \,\leq\, d\hspace{0.05em}\lVert X \rVert_2
    \end{equation}
    and consider the following result.

    \begin{lemma}{}{matrixBounds}
        Let $X, Y \in \SS_+^{\hspace{0.025em}d}$. 
        Then, it holds
        \begin{equation*}
            \lVert X-Y \rVert_F \,\leq\, \sqrt{\hspace{-0.1em}d\hspace{0.1em}}\hspace{-0.2em}\left(\hspace{-0.1em}\sqrt{\lVert X\rVert_F} + \sqrt{\lVert Y\rVert_F}\hspace{0.05em}\right)\hspace{-0.1em}d_{\hspace{0.025em}G}(X,Y) \quad \text{ and } \quad d_{\hspace{0.025em}G}(X, Y) \,\leq\, \sqrt{\hspace{-0.1em}d\hspace{0.05em}\lVert X-Y\rVert_F}\hspace{0.1em}.
        \end{equation*}
    \end{lemma}
    \begin{proof}[\textcolor{seeblau}{Proof.}]
        For the first bound, we use the representation 
        \begin{equation}\label{EQ:alternateFormGelbrich}
            d_{\hspace{0.025em}G}(X, Y) \,=\, \min\left\{\left\lVert \hspace{0.05em}X^{1/2} - Y^{1/2}\hspace{0.025em}U \hspace{0.05em}\right\rVert_{2}\colon U \in O(d)\right\}
        \end{equation}
        that can be found in \cite[Theorem 1]{bhatia2019bures}, where $O(d)$ stands for the orthogonal group.
        Let $U^\star \in O(d)$ be a minimizer of \eqref{EQ:alternateFormGelbrich} and write $P = X^{1/2}$ and $Q = Y^{1/2}\hspace{0.025em}U^\star$.
        Then, we have $X = PP^\top$ and $Y = Q\hspace{0.1em}Q^\top$, such that with 
        \begin{equation*}
            X-Y \,=\, PP^\top \hspace{-0.1em}- Q\hspace{0.1em}Q^\top \,=\, P(P-Q)^\top \hspace{-0.1em}- (P-Q)Q^\top
        \end{equation*}
        and \eqref{EQ:knownEquivalence} we obtain
        \begin{equation*}
            \lVert X-Y\rVert_F \,\leq\, \sqrt{\hspace{-0.1em}d\hspace{0.05em}}\hspace{0.1em}\lVert X-Y\rVert_2 \,=\, \sqrt{\hspace{-0.1em}d\hspace{0.05em}}\hspace{-0.05em}\left\lVert P(P-Q)^\top \hspace{-0.1em}- (P-Q)Q^\top\right\rVert_2 \,\leq\, \sqrt{\hspace{-0.1em}d\hspace{0.05em}}\hspace{0.1em}(\lVert P\rVert_2 + \lVert \hspace{0.05em}Q\rVert_2)\hspace{0.05em}\lVert P-Q\rVert_2\hspace{0.1em}.
        \end{equation*}
        Hence, noting that $\lVert P \rVert_2 = \sqrt{\lVert \hspace{0.05em}X\hspace{0.05em}\rVert_2}$ and $\lVert Q\rVert_2 = \sqrt{\lVert \hspace{0.05em}Y\hspace{0.05em}\rVert_2}$ we obtain
        \begin{equation*}
            \lVert X-Y\rVert_F \,\leq\, \sqrt{\hspace{-0.1em}d\hspace{0.1em}}\hspace{-0.2em}\left(\hspace{-0.1em}\sqrt{\lVert X\rVert_2} + \sqrt{\lVert Y\rVert_2}\hspace{0.05em}\right)\hspace{-0.2em}\lVert P-Q\rVert_2\hspace{0.1em} \,\leq\, \sqrt{\hspace{-0.1em}d\hspace{0.1em}}\hspace{-0.2em}\left(\hspace{-0.1em}\sqrt{\lVert X\rVert_F} + \sqrt{\lVert Y\rVert_F}\hspace{0.05em}\right)\hspace{-0.2em}d_{\hspace{0.025em}G}(X, Y)\hspace{0.1em},
        \end{equation*}
        where we used the optimality of $U^\star$ in the second inequality.
        For the second bound we again use \eqref{EQ:knownEquivalence} and the representation \eqref{EQ:alternateFormGelbrich} to obtain 
        \begin{equation*}
            d_{\hspace{0.025em}G}(X, Y) \,\leq\, \left\lVert X^{1/2} - Y^{1/2} \right\rVert_2 \,\leq\, \left\lVert X^{1/2} - Y^{1/2} \right\rVert_F,
        \end{equation*}
        such that using the Powers\hspace{0.1em}-St\o rmer inequality \cite[Lemma~4.1]{powers1970free} immediately yields
        \begin{equation*}
            d_{\hspace{0.025em}G}(X, Y) \,\leq\, \sqrt{\lVert X-Y\rVert_*} \,\leq\, \sqrt{\hspace{-0.1em}d\hspace{0.05em}\lVert X-Y\rVert_F}\hspace{0.1em},
        \end{equation*}
        where $\lVert \,\cdot\, \rVert_*$ denotes the nuclear norm.
    \end{proof}
    
    \begin{lemma}{}{gelbrichBallSublevelSet}
        Let $\Yb \colon \Omega \to \SS_+^{\hspace{0.025em}d}$ be $\Sigma$\hspace{0.1em}-\hspace{0.05em}measurable and $\rho \geq 0$.
        Then, the correspondence
        \begin{equation*}
            \Omega \,\rightrightarrows\, \SS_+^{\hspace{0.025em}d}, \; \omega \,\mapsto\, \GG(\Yb(\omega), \rho)
        \end{equation*}
        has nonempty, compact, and convex values and is measurable.
    \end{lemma}
    \begin{proof}[\textcolor{seeblau}{Proof.}]
        We consider the functions
        \begin{equation*}
            g_1 \colon \SS_+^{\hspace{0.025em}d} \times \SS_+^{\hspace{0.025em}d} \to \RR, \; (Y, X) \mapsto d_{\hspace{0.025em}G}(Y, X)^2-\rho^2
        \end{equation*}
        and
        \begin{equation*}
            g_2 \colon \SS_+^{\hspace{0.025em}d} \times \SS_+^{\hspace{0.025em}d} \to \RR, \; (Y, X) \mapsto \lambda_{\min}(Y) - \lambda_{\min}(X)\hspace{0.1em}.
        \end{equation*}
        Note that it holds
        \begin{equation*}
            \begin{aligned}
                \GG(\Yb\hspace{-0.05em}(\omega), \rho) & \,=\, \{X \in \SS_+^{\hspace{0.025em}d} \,\colon d_{\hspace{0.025em}G}(\Yb\hspace{-0.05em}(\omega), X)^2 \leq \rho^2 \text{ and } X \succeq \lambda_\text{min}(\Yb\hspace{-0.05em}(\omega))I\} \\
                & \,=\, \{X \in \SS_+^{\hspace{0.025em}d} \,\colon g_1(\Yb\hspace{-0.05em}(\omega), X) \leq 0 \text{ and } g_2(\Yb\hspace{-0.05em}(\omega), X) \leq 0\} \\
                & \,=\, \SS_g(\Yb\hspace{-0.05em}(\omega), 0)
            \end{aligned}
        \end{equation*}
        for all $\omega \in \Omega$, such that it suffices to show that $g_1$ and $g_2$ satisfy the assumptions of Lemma \ref{LEM:generalizedBall}.
        Since $X \mapsto \trace(X)$ and $X \mapsto X^{1/2}$ are continuous on $\SS_+^{\hspace{0.025em}d}$ the same holds for the Gelbrich distance $d_{\hspace{0.025em}G}$ in both arguments, such that $d_{\hspace{0.025em}G}$ is jointly continuous and hence a Carathéodory function.
        Furthermore, by Lemma \ref{LEM:matrixBounds} we have
        \begin{equation*}
            d_{\hspace{0.025em}G}(Y, X) \,\geq\, \frac{\lVert X-Y\rVert_F}{d\hspace{0.05em}(\hspace{-0.025em}\sqrt{\lVert X \rVert_F} + \sqrt{\lVert Y \rVert_F}\hspace{0.05em})} \,\geq\, \frac{\lVert X \rVert_F - \lVert Y\rVert_F}{d\hspace{0.05em}(\hspace{-0.025em}\sqrt{\lVert X \rVert_F} + \sqrt{\lVert Y \rVert_F}\hspace{0.05em})}  \,\geq\, \frac{1}{d}(\sqrt{\lVert X \rVert_F} - \sqrt{\lVert Y \rVert_F})
        \end{equation*}
        for all $X, Y \in \SS_+^{\hspace{0.025em}d}$, such that $g_1$ is coercive in the second argument.
        Lastly, we have
        \begin{equation*}
            g_1(X, X) \,=\, d_{\hspace{0.025em}G}(X, X)^2 -\rho^2 \,=\, -\rho^2 \,\leq\, 0
        \end{equation*}
        for all $X \in \SS_+^{\hspace{0.025em}d}$ and $X \mapsto g_1(Y, X)$ is convex on $\SS_+^{\hspace{0.025em}d}$ for all $Y \in \SS_+^{\hspace{0.025em}d}$ since $d_{\hspace{0.025em}G}$ is jointly convex by \hbox{\cite[Proposition~2.3]{nguyen2023bridging}}.
        Thus, $g_1$ satisfies both \eqref{EQ:caratheodoryProperties} and \eqref{EQ:upperBoundFunction} and we turn our attention to $g_2$. 
        By Weyl's inequality we have that 
        \begin{equation*}
            \lambda_{\min}(Y) - \lambda_{\min}(X) \,\leq\, \lambda_{\max}(Y-X) \,\leq\, \lVert Y-X \rVert_2 \,\leq\, \lVert Y-X \rVert_F
        \end{equation*}
        for all $X, Y \in \SS_+^{\hspace{0.025em}d}$, such that $X \mapsto \lambda_\text{min}(X)$ is 1\hspace{0.05em}-\hspace{0.05em}Lipschitz continuous on $\SS_+^{\hspace{0.025em}d}$. Hence, it follows that $g_2$ is jointly continuous and therefore a Carathéodory function.
        Furthermore, we have 
        \begin{equation*}
            g_2(X, X) \,=\, \lambda_\text{min}(X)-\lambda_\text{min}(X) \,=\, 0
        \end{equation*}
        for all $X \in \SS_+^{\hspace{0.025em}d}$ and by the representation 
        \begin{equation*}
            \lambda_\text{min}(X) \,=\, \min\{a^{\hspace{-0.15em}\top} \hspace{-0.2em}X a \,\colon \lVert \hspace{0.05em}a\hspace{0.05em}\rVert_2 = 1\}
        \end{equation*}
        we can see that $\lambda_\text{min}$ is concave on $\SS_+^{\hspace{0.025em}d}$, such that $X \mapsto g_2(Y, X)$ is convex on $\SS_+^{\hspace{0.025em}d}$ for all $Y \in \SS_+^{\hspace{0.025em}d}.$
        Thus, $g_2$ satisfies \eqref{EQ:upperBoundFunction} as well, such that overall Lemma \ref{LEM:generalizedBall} is applicable and the claim follows.
    \end{proof}

    \begin{lemma}{}{DG:stochastic:process}
        The map $\Db^\GG$ defined as in \eqref{EQ:gelbricgDomainApproximation} is a well\hspace{0.05em}-\hspace{0.035em}defined stochastic process.
    \end{lemma}
    \begin{proof}[\textcolor{seeblau}{Proof.}]
        Let $t \in \NN$ be fixed. 
        By assumption, the covariance matrix collection approximation $\hat{\Zb}^{(t)}$ is $\Sigma$\hspace{0.1em}-\hspace{0.05em}measurable, such that each component $\hat{\Zb}_k^{(t)}$ itself is $\Sigma$\hspace{0.1em}-\hspace{0.05em}measurable for all $k \in [\hspace{0.025em}2\hspace{0.05em}T + 1\hspace{0.025em}]$.
        Hence, by Lemma \ref{LEM:gelbrichBallSublevelSet} we have that the correspondence
        \begin{equation*}
            \varphi_k^{(t)} \colon \Omega \,\rightrightarrows\, \SS_+^{\hspace{0.025em}d}, \; \omega \,\mapsto\, \GG\hspace{-0.2em}\left(\hspace{-0.05em}\hat{\Zb}_k^{(t)}\hspace{-0.1em}(\omega), \rho\right)
        \end{equation*}
        has nonempty, compact, and convex values and is measurable for each $k \in [\hspace{0.025em}2\hspace{0.05em}T + 1\hspace{0.025em}]$.
        Therefore, by Lemma~\ref{LEM:correspondenceSetOperations}, we have that the product correspondence
        \begin{equation*}
            \varphi^{(t)} \colon \Omega \,\rightrightarrows\, \left(\SS_+^n \times \prod_{k \hspace{0.05em}=\hspace{0.05em} 0}^{T-1} \SS_+^{\hspace{0.025em}n} \times \prod_{k \hspace{0.05em}=\hspace{0.05em} 0}^{T-1} \SS_+^{\hspace{0.025em}p}\right)\hspace{-0.2em}, \; \omega \,\mapsto\, \prod_{k \hspace{0.05em}= 1}^{2\hspace{0.025em}T + 1} \GG\hspace{-0.2em}\left(\hspace{-0.05em}\hat{\Zb}_k^{(t)}\hspace{-0.1em}(\omega), \rho\right)
        \end{equation*}
        has nonempty, compact, and convex values and is weakly measurable.
        Since $\varphi^{(t)}$ implicitly has closed values, by Theorem \ref{THM:correspondenceEquivalences}, we can now identify $\varphi^{(t)}$ as map from $\Omega$ to $\Dcc$, such that $\Db_{t} = \varphi^{(t)}$.
        Overall, the claim follows since $t \in \NN$ was arbitrary.
    \end{proof}

    \noindent
    Since by Lemma \ref{LEM:DG:stochastic:process} we now know that we indeed have the domain approximation process $\Db^\GG$ it only remains to show that its corresponding domain approximation sequence $(\Db_{t}^\GG)_{t \in \NN}$ converges with respect to the Hausdorff distance.
    
    \begin{lemma}{}{gelbrichShiftBound}
        Let $X \in \SS_+^{\hspace{0.025em}d}$. Then, for all $t \geq 0$ we have $d_{\hspace{0.025em}G}(X, X+tI) \,\leq\, \sqrt{\hspace{-0.1em}d\hspace{0.025em}t}$.
    \end{lemma}
    \begin{proof}[\textcolor{seeblau}{Proof.}]
        Since $X$ and $X + tI$ commute, there exist an invertible matrix $S \in \RR^{d \times d}$ and diagonal matrices $D, D_t \in \RR^{d \times d}$ such that $X = S^{-1}\hspace{-0.05em}DS$ and $X + tI = S^{-1}\hspace{-0.05em}D_tS$.
        Therefore, we obtain that $\trace(X) = \trace(D)$, $\trace(X+tI) = \trace(D_t)$ and 
        \begin{equation*}
            \begin{aligned}
                \trace\hspace{-0.2em}\left(\hspace{-0.2em}\left((X+tI)^{1/2}X(X+tI)^{1/2}\right)^{\hspace{-0.2em}1/2}\right) 
                & \,=\, \trace\hspace{-0.2em}\left(\hspace{-0.2em}\left(S^{-1}\hspace{-0.05em}D_t^{1/2}SS^{-1}\hspace{-0.05em}DSS^{-1}\hspace{-0.05em}D_t^{1/2}S\right)^{\hspace{-0.2em}1/2}\right) \\
                & \,=\, \trace\hspace{-0.2em}\left(\hspace{-0.2em}\left(S^{-1}\hspace{-0.05em}DD_tS\right)^{1/2}\right) \\
                & \,=\, \trace\hspace{-0.2em}\left(\hspace{-0.1em}D^{1/2}D_t^{1/2}\right)\hspace{-0.1em}, 
            \end{aligned}
        \end{equation*}
        such that overall 
        \begin{equation*}
            \begin{aligned}
                d_{\hspace{0.025em}G}(X, X + tI)^2 
                & \,=\, \trace(X) + \trace(X + tI) - 2\trace\hspace{-0.2em}\left(\hspace{-0.2em}\left((X+tI)^{1/2}X(X+tI)^{1/2}\right)^{\hspace{-0.2em}1/2}\right) \\
                & \,=\, \trace(D) + \trace(D_t) - 2\trace\hspace{-0.2em}\left(\hspace{-0.1em}D^{1/2}D_t^{1/2}\right) \\
                & \,=\, \trace\hspace{-0.2em}\left(\hspace{-0.2em}\left(D^{1/2} - D_t^{1/2}\right)^{\hspace{-0.1em}2}\right) \\
                & \,=\, \sum_{i \hspace{0.05em}= 1}^d \left(\sqrt{\hspace{-0.05em}\lambda_i(X)} - \sqrt{\hspace{-0.05em}\lambda_i(X) + t}\right)^{\hspace{-0.1em}2}\hspace{-0.2em},
            \end{aligned}
        \end{equation*}
        where $\lambda_1\hspace{-0.05em}(X), \ldots, \lambda_d(X)$ are the eigenvalues of $X$.
        Since $X \in \SS_+^{\hspace{0.025em}d}$ we have $\lambda_i(X) \geq 0$ for all $i \in [n]$ and since 
        \begin{equation*}
            \left(\sqrt{a} - \sqrt{a+t}\hspace{0.05em}\right)^{\hspace{-0.05em}2} \,=\, \frac{t^2}{\left(\sqrt{a} + \sqrt{a+t}\hspace{0.05em}\right)^{\hspace{-0.05em}2}} \,\leq\, t
        \end{equation*}
        for all $a \geq 0$ this directly yields that $d_{\hspace{0.025em}G}(X, X+tI)^2 \leq d\hspace{0.025em}t$.
    \end{proof}

    \noindent
    For the following result we implicitly need the concept of geodesic metric spaces.

    \begin{definition}{Geodesic Metric Space}{}
        A metric space $(\hspace{-0.075em}X, d_X\hspace{-0.05em})$ is called \emph{geodesic} if for all $x, y \in X$ there exists a curve $\gamma \colon [0, 1] \to X$ with $\gamma(0) = x$ and $\gamma(1) = y$ satisfying 
        \begin{equation*}
            d_X(\gamma(s), \gamma(t)) \,=\, \lvert \hspace{0.05em}s-t\hspace{0.05em}\rvert \hspace{0.1em}d_X(x, y)\hspace{0.1em}.
        \end{equation*}
        The curve $\gamma$ then is called \emph{minimizing geodesic}.
    \end{definition}

    \noindent
    Note that the metric spaces $(\hspace{0.025em}\SS_+^{\hspace{0.025em}d}, d_{\hspace{0.025em}G}\hspace{-0.05em})$ and $(\hspace{0.025em}\SS_{++}^{\hspace{0.025em}d}, d_{\hspace{0.025em}G}\hspace{-0.05em})$ are geodesic \cite{bhatia2019bures, thanwerdas2023bures}.
    
    \begin{lemma}{}{generalGelbrichHausdorff}
        Let $X, Y \in \SS_+^{\hspace{0.025em}d}$.
        Then, for all $\rho \geq 0$ we have
        \begin{equation*}
            d_H(\GG(X, \rho), \GG(Y, \rho)) \,\leq\, d\hspace{0.05em}(d+2)\hspace{-0.15em}\left(2\rho + \sqrt{\lVert \hspace{0.05em}X\hspace{0.05em}\rVert_F} + \sqrt{\lVert \hspace{0.05em}Y\hspace{0.05em}\rVert_F}\right)\hspace{-0.05em}\sqrt{\lVert X-Y\rVert_F}\hspace{0.1em}.
        \end{equation*}
    \end{lemma}
    \begin{proof}[\textcolor{seeblau}{Proof.}]
        Let $\tilde{X} \in \GG(X, \rho)$ and set $t = \lvert \hspace{0.05em}\lambda_\text{min}(Y) - \lambda_\text{min}(\hspace{-0.05em}X) \hspace{0.05em} \rvert$.
        Then, by Lemma \ref{LEM:gelbrichShiftBound} and the triangle inequality for $d_{\hspace{0.025em}G}$, we have
        \begin{equation}\label{EQ:gelbrichStep}
            d_{\hspace{0.025em}G}(\tilde{X} + tI, Y) \,\leq\, d_{\hspace{0.025em}G}(\tilde{X} + tI, \tilde{X}) + d_{\hspace{0.025em}G}(\tilde{X}, X) + d_{\hspace{0.025em}G}(X, Y) \,\leq\, \sqrt{\hspace{-0.1em}d\hspace{0.025em}t\hspace{0.05em}} + \rho + d_{\hspace{0.025em}G}(X, Y)
        \end{equation}
        and furthermore 
        \begin{equation*}
            \tilde{X}+tI \,\succeq\, (\lambda_\text{min}(X) + t)I \,\succ\, \lambda_\text{min}(Y)I\hspace{0.1em}.
        \end{equation*}
        If $\lambda_\text{min}(Y) = 0$, then due to \cite{thanwerdas2023bures} we can find a minimizing geodesic $\gamma \colon [0, 1] \to \SS_+^{\hspace{0.025em}d}$ with $\gamma(0) = \tilde{X} + tI$ and $\gamma(1) = Y$.
        Denoting $\varepsilon = d_{\hspace{0.025em}G}(X, Y) + \sqrt{\hspace{-0.1em}d\hspace{0.025em}t\hspace{0.05em}}$ and $\tilde{Y} = \gamma(\varepsilon/(\rho + \varepsilon))$ we then find
        \begin{equation}\label{EQ:geodesicProperty}
            d_{\hspace{0.025em}G}(\tilde{Y}, Y) \,=\, \left\lvert\hspace{0.05em}\frac{\varepsilon}{\rho + \varepsilon} - 1\hspace{0.05em}\right\rvert \hspace{0.1em}d_{\hspace{0.025em}G}(\tilde{X}+tI, Y) \,\leq\, \frac{\rho}{\rho + \varepsilon}(\rho + \varepsilon) \,=\, \rho\hspace{0.1em},
        \end{equation}
        where we used \eqref{EQ:gelbrichStep} in the inequality.
        Hence, $d_{\hspace{0.025em}G}(\tilde{Y}, Y) \leq \rho$ and $\tilde{Y} \succeq 0 = \lambda_\text{min}(Y)I$ such that $\tilde{Y} \in \GG(Y, \rho)$.
        If $\lambda_\text{min}(Y) > 0$, then due to \cite[Section 4]{bhatia2019bures} we may choose the minimizing geodesic $\gamma \colon [0, 1] \to \SS_{++}^{\hspace{0.025em}d}$ with 
        \begin{equation*}
            \gamma(s) \,=\, s^2(\tilde{X}+tI) + (1-s)^2\hspace{0.025em}Y + s\hspace{0.025em}(1-s)\hspace{-0.1em}\left(\hspace{-0.2em}\left((\tilde{X}+tI)Y\right)^{\hspace{-0.2em}1/2} \hspace{-0.1em}+ \left(Y(\tilde{X}+tI)\right)^{\hspace{-0.2em}1/2}\right)
        \end{equation*}
        for which it is easy to check that $\gamma(s) \succeq \lambda_\text{min}(Y)I$ for all $s \in [0, 1]$, such that in this case $\tilde{Y} \in \GG(Y, \rho)$ as well.
        Analogously to \eqref{EQ:geodesicProperty} we find that $d_{\hspace{0.025em}G}(\tilde{X}+tI, \tilde{Y}) \leq \varepsilon$, such that with
        \begin{equation*}
            d_{\hspace{0.025em}G}(\tilde{X}, \tilde{Y}) \,\leq\, d_{\hspace{0.025em}G}(\tilde{X}, \tilde{X} + tI) + d_{\hspace{0.025em}G}(\tilde{X} + tI, \tilde{Y}) \,\leq\, d_{\hspace{0.025em}G}(X, Y) + 2\sqrt{\hspace{-0.1em}d\hspace{0.025em}t\hspace{0.05em}}
        \end{equation*}
        and Lemma \ref{LEM:matrixBounds} we have
        \begin{equation}\label{EQ:frobeniusBound1}
            \lVert \tilde{X}- \tilde{Y} \rVert_F \,\leq\, \sqrt{\hspace{-0.1em}d\hspace{0.1em}}\hspace{-0.2em}\left(\hspace{-0.1em}\sqrt{\lVert \tilde{X}\rVert_F} + \sqrt{\lVert \tilde{Y}\rVert_F}\hspace{0.05em}\right)\hspace{-0.2em}\left(d_{\hspace{0.025em}G}(X, Y) + 2\sqrt{\hspace{-0.1em}d\hspace{0.025em}t\hspace{0.05em}}\right)\hspace{-0.1em}.
        \end{equation}
        By the triangle inequality for $d_{\hspace{0.025em}G}$ we further have
        \begin{equation*}
            \sqrt{\lVert \tilde{X} \rVert_F} \,\leq\, \sqrt{\trace(\tilde{X})} \,=\, d_{\hspace{0.025em}G}(\tilde{X}, 0) \,\leq\, d_{\hspace{0.025em}G}(\tilde{X}, X) + d_{\hspace{0.025em}G}(X, 0) \,\leq\, \rho + \sqrt{\trace(X)} \,\leq\, \rho + \sqrt{\hspace{-0.1em}d\hspace{0.05em}\lVert X \rVert_F}
        \end{equation*}
        and analogously
        \begin{equation*}
            \sqrt{\lVert \tilde{Y} \rVert_F} \,\leq\, \rho + \sqrt{\hspace{-0.1em}d\hspace{0.05em}\lVert Y \rVert_F}\hspace{0.1em}.
        \end{equation*}
        Hence, considering \eqref{EQ:frobeniusBound1}, together with Lemma \ref{LEM:matrixBounds} and the 1-Lipschitz continuity of $\lambda_\text{min}$ on $\SS_+^{\hspace{0.025em}d}$ 
        we obtain that
        \begin{equation*}
            \begin{aligned}
                \lVert \tilde{X}- \tilde{Y} \rVert_F & \,\leq\, \sqrt{\hspace{-0.1em}d\hspace{0.1em}}\hspace{-0.2em}\left(2\rho + \sqrt{\hspace{-0.1em}d\hspace{0.05em}\lVert X\rVert_F} + \sqrt{\hspace{-0.1em}d\hspace{0.05em}\lVert Y\rVert_F}\hspace{0.05em}\right)\hspace{-0.2em}\left(\sqrt{\lVert X-Y\rVert_F}+ 2\sqrt{\hspace{-0.1em}d\hspace{0.05em}\lVert X-Y\rVert_F}\right) \\
                & \,\leq\, d\hspace{0.05em}(d+2)\hspace{-0.15em}\left(2\rho + \sqrt{\lVert \hspace{0.05em}X\hspace{0.05em}\rVert_F} + \sqrt{\lVert \hspace{0.05em}Y\hspace{0.05em}\rVert_F}\right)\hspace{-0.05em}\sqrt{\lVert X-Y\rVert_F}\hspace{0.1em},
            \end{aligned}
        \end{equation*}
        where we used that $1+2\sqrt{\hspace{-0.1em}d\hspace{0.05em}} \leq d+2$ for all $d \in \NN$ in the second inequality.
        Since the right hand side does not depend on $\tilde{Y}$ anymore this yields that 
        \begin{equation*}
            \begin{aligned}
                \sup\hspace{-0.1em}\left\{h(\tilde{X}, \GG(Y, \rho)) \,\colon \tilde{X} \in \GG(X, \rho)\right\} & \,\leq\, \sup\hspace{-0.1em}\left\{\lVert \tilde{X}-\tilde{Y}\rVert_F \,\colon \tilde{X} \in \GG(X, \rho)\right\} \\
                & \,\leq\, d\hspace{0.05em}(d+2)\hspace{-0.15em}\left(2\rho + \sqrt{\lVert \hspace{0.05em}X\hspace{0.05em}\rVert_F} + \sqrt{\lVert \hspace{0.05em}Y\hspace{0.05em}\rVert_F}\right)\hspace{-0.05em}\sqrt{\lVert X-Y\rVert_F}\hspace{0.1em}.
            \end{aligned}
        \end{equation*}
        The claim now follows by symmetry in the arguments.
    \end{proof}
    
    \begin{proof}[\textcolor{seeblau}{Proof of Theorem \ref{THM:convergenceOfGelbrichBalls}.}]
        Let $\omega \in \Omega$ be a fixed sample point corresponding to the event in \eqref{EQ:LQGexampleAssumption}.
        Then, there exists some $N \in \NN$ such that for all $k \in [2\hspace{0.05em}T+1]$ we have
        \begin{equation*}
            \lVert \hspace{0.05em}\hat{\Zb}_k^{(t)} - Z_k\hspace{0.05em}\rVert_F \,\leq\, \eta + 1\hspace{0.1em},
        \end{equation*}
        implying that
        \begin{equation*}
            \lVert \hspace{0.05em}\hat{\Zb}_k^{(t)}\rVert_F \,\leq\, \eta + 1 + \lVert \hspace{0.05em} Z_k \hspace{0.05em}\rVert_F
        \end{equation*}
        for all $t \geq N$.
        Hence, using Lemma \ref{LEM:generalGelbrichHausdorff} we find that
        \begin{equation*}
            d_H(\GG(\hat{\Zb}_k^{(t)}\hspace{-0.1em}(\omega), \rho), \GG(Z_k, \rho)) \,\leq\, M_k\hspace{0.025em}\lVert \hspace{0.05em}\hat{\Zb}_k^{(t)}\hspace{-0.1em}(\omega)-Z_k\hspace{0.05em}\rVert_F^{1/2}
        \end{equation*}
        for all $t \geq N$, where 
        \begin{equation*}
            M_k \,\coloneqq\, d_k\hspace{0.05em}(d_k+2)\hspace{-0.15em}\left(2\rho + \sqrt{\eta + 1 + \lVert \hspace{0.05em}Z_k\hspace{0.05em}\rVert_F} + \sqrt{\lVert \hspace{0.05em}Z_k\hspace{0.05em}\rVert_F}\right)
        \end{equation*}
        for all $k \in [2\hspace{0.05em}T+1]$.
        Since the Hausdorff distance of a product of sets is bounded above by the sum of the Hausdorff distance of each component, we furthermore obtain that
        \begin{equation*}
            \begin{aligned}
                d_H(\Db_{t}^\GG\hspace{-0.05em}\hspace{-0.05em}(\omega), D) \,\leq\, \sum_{k \hspace{0.05em}= 1}^{2\hspace{0.025em}T+1} d_H(\GG(\hat{\Zb}_k^{(t)}\hspace{-0.1em}(\omega), \rho), \GG(Z_k, \rho)) \,\leq\, \sum_{k \hspace{0.05em} = 1}^{2\hspace{0.025em}T+1} M_k\hspace{0.025em}\lVert \hspace{0.05em}\hat{\Zb}_k^{(t)}\hspace{-0.1em}(\omega)-Z_k\hspace{0.05em}\rVert_F^{1/2}\hspace{0.1em},
            \end{aligned}
        \end{equation*}
        such that
        \begin{equation*}
            \limsup_{t \hspace{0.05em}\to\hspace{0.05em} \infty} \hspace{0.25em} d_H(\Db_{t}^\GG\hspace{-0.05em}(\omega), D) \,\leq\, \limsup_{t \hspace{0.05em}\to\hspace{0.05em} \infty} \hspace{0.25em} \sum_{k \hspace{0.05em} = 1}^{2\hspace{0.025em}T+1} M_k\hspace{0.025em}\lVert \hspace{0.05em}\hat{\Zb}_k^{(t)}\hspace{-0.1em}(\omega)-Z_k\hspace{0.05em}\rVert_F^{1/2} \,\leq\, \underbrace{\left(\hspace{0.075em}\sum_{k \hspace{0.05em} = 1}^{2\hspace{0.025em}T+1} M_k\hspace{-0.1em}\right)}_{M \coloneqq }\hspace{-0.15em}\sqrt{\eta\hspace{0.1em}}\hspace{0.1em}.
        \end{equation*}
        The claim now follows since the probability of the event in \eqref{EQ:LQGexampleAssumption} by assumption is at least $1-\beta$.
    \end{proof}
\end{document}